\documentclass[11pt, reqno]{amsart}
\allowdisplaybreaks

\usepackage{comment}
\usepackage[rgb,dvipsnames]{xcolor}
\usepackage{tikz-cd}
\usepackage{amsfonts}
\usepackage{amssymb}
\usepackage{multicol}
\usepackage{amsmath,amsthm}
\usepackage{mathtools}
\usepackage{latexsym}
\usepackage{amscd}
\usepackage{esint}
\usepackage{mathrsfs}
\usepackage{dsfont}
\usepackage{setspace}
\usepackage{enumitem}
\usepackage[margin=1.2in]{geometry}
\usepackage{bm}

\usepackage{hyperref}
\usepackage{comment}
\usepackage{enumitem}

\newtheorem{prop}{Proposition}[section]
\newtheorem{thm}[prop]{Theorem}
\newtheorem{lemm}[prop]{Lemma}
\newtheorem{coro}[prop]{Corollary}

\newtheorem*{claim*}{Claim}
\newtheorem*{lemm*}{Lemma}
\newtheorem*{thm*}{Theorem}

\theoremstyle{definition}
\newtheorem{defin}[prop]{Definition}
\newtheorem{remark}[prop]{Remark}
\newtheorem*{rmk*}{Remark}
\newtheorem{exam}[prop]{Example}
\newtheorem*{example*}{Example}

\newenvironment{defi}{\begin{defin}%
  \pushQED{\qed}}%
  {\popQED\end{defin}}

\newenvironment{rmk}{\begin{remark}%
  \pushQED{\qed}}%
  {\popQED\end{remark}}

\newenvironment{example}{\begin{exam}%
  \pushQED{\qed}}%
  {\popQED\end{exam}}

\newcommand{\CC}{\mathbb{C}}

\newcommand{\HH}{\mathbb{H}}

\newcommand{\NN}{\mathbb{N}}

\newcommand{\RR}{\mathbb{R}}

\newcommand{\ZZ}{\mathbb{Z}}

\newcommand{\cC}{\mathcal C}
\newcommand{\cD}{\mathcal D}
\newcommand{\cE}{\mathcal E}

\newcommand{\cH}{\mathcal H}

\newcommand{\cK}{\mathcal K}
\newcommand{\cL}{\mathcal L}

\newcommand{\cR}{\mathcal R}

\newcommand{\cU}{\mathcal U}

\newcommand{\cY}{\mathcal Y}

\newcommand{\rg}{{\rm g}}

\def\fsu{\mathfrak{su}}

\DeclareMathOperator{\tr}{tr}

\DeclareMathOperator{\Span}{span}

\DeclareMathOperator{\supp}{supp}

\DeclareMathOperator{\loc}{loc}

\DeclareMathOperator{\Id}{Id}

\DeclareMathOperator{\diam}{diam}
\DeclareMathOperator{\dist}{dist}

\DeclareMathOperator{\Ric}{Ric}

\DeclareMathOperator{\Inte}{Int}

\DeclareMathOperator{\vol}{\text{vol}}
\DeclareMathOperator{\Vol}{\text{Vol}}

\DeclareMathOperator{\inj}{\text{inj}}

\newcommand{\ep}{\varepsilon}

\newcommand{\id}{\text{id}}

\newcommand{\bangle}[1]{\langle #1 \rangle}
\newcommand{\floor}[1]{\left\lfloor #1 \right\rfloor}

\newcommand{\Hom}{{\rm Hom}}

\newcommand{\rom}[1]{\expandafter\romannumeral #1}
\newcommand{\Rom}[1]{\uppercase\expandafter{\romannumeral #1}}

\setlist[enumerate]{leftmargin = 2em}
\setlist[itemize]{leftmargin = 2em}
\allowdisplaybreaks

\title[$SU(2)$ Yang--Mills--Higgs functional with Higgs self-interaction on $3$-manifolds]{$SU(2)$ Yang--Mills--Higgs functional with Higgs self-interaction on $3$-manifolds}

\author{Da Rong Cheng}
\address[Da Rong Cheng]{Department of Mathematics\\University of Miami\\1365 Memorial Drive\\
Coral Gables, FL 33146\\USA}
\email{darong.cheng@miami.edu}

\author{Daniel Fadel} 
\address[Daniel Fadel]{Instituto de Ci\^{e}ncias Matem\'{a}ticas e de Computa\c{c}\~{a}o (ICMC), Universidade de São Paulo (USP), São Carlos - SP, Brazil}
\urladdr{\href{https://sites.google.com/view/daniel-fadel-math-homepage/home}{sites.google.com/view/daniel-fadel-math-homepage/home}}
\email{daniel.fadel@icmc.usp.br}

\author{Luiz Lara}
\address[Luiz Lara]{Instituto de Matem\'{a}tica, Estat\'{i}stica e Computa\c{c}\~{a}o Cient\'{i}fica (IMECC), Universidade Estadual de Campinas (UNICAMP), Campinas - SP, Brazil}
\email{luizlara@ime.unicamp.br}


\begin{document}

\begin{abstract}
Given a positive coupling constant $\lambda>0$, for any parameter $\ep>0$ we study critical points of the ($\ep$-scaled) Yang--Mills--Higgs energy
\[
\cY_{\ep}(\nabla,\Phi) = \int_M \left(\ep^2|F_{\nabla}|^2 + |\nabla\Phi|^2 + \frac{\lambda}{4\ep^2}(1-|\Phi|^2)^2\right)\vol_g =: \int_{M} e_{\ep}(\nabla,\Phi) \vol_{g},
\] defined for pairs $(\nabla,\Phi)$ consisting of a connection $\nabla$ on an $SU(2)$-bundle over an oriented, Riemannian $3$-manifold $(M^3, g)$, and a section $\Phi$ of the associated adjoint bundle. When $M$ is closed, 
we use a $3$-parameter min-max construction to produce, for $\ep\ll_M 1$, non-trivial critical points 
of $\cY_{\ep}$ within the energy regime
\[
1 \lesssim_{\lambda}\ep^{-1}\cY_{\ep}(\nabla_{\ep},\Phi_{\ep}) \lesssim_{\lambda, M} 1.
\] Furthermore, 
when the first Betti number of $M$ is zero, the constructed critical points are guaranteed to be irreducible in the sense that $\nabla_{\ep}\Phi_{\ep}\neq 0$. Next, assuming that $M$ has bounded geometry but is not necessarily compact, and given a family of critical points with $\ep^{-1}\cY_{\ep}(\nabla_{\ep}, \Phi_{\ep})$ uniformly bounded, we show that as $\ep \to 0$, the associated energy measures $\ep^{-1} e_{\ep}(\nabla_{\ep}, \Phi_{\ep})\vol_{g}$ converge along a subsequence to
\[
|h|^2 \vol_g + \sum_{x \in S}\Theta(x)\delta_{x},
\]
where $h$ is an $L^2$ harmonic $1$-form on $M$, while $S$ is a finite set of points. Moreover, each $\Theta(x)$ is equal to the total energy of a finite collection of $\cY_{1}$-critical points on the Euclidean $3$-space $\mathbb{R}^3$. Finally, from the a priori estimates involved in proving the above statements, we obtain an energy gap for critical points on 
$3$-manifolds with bounded geometry, implying in particular that over $\mathbb{R}^3$, there is $\theta_{\text{gap}}>0$ such that if $\ep^{-1}\cY_{\ep}(\nabla,\Phi)\leqslant \theta_{\text{gap}}\cdot\min\{1,\lambda\}$ then in fact $\cY_{\ep}(\nabla,\Phi)=0$. As a byproduct of our results, we also deduce the existence of non-trivial critical points of $\cY_{1}$ over $\mathbb{R}^3$, for any $\lambda>0$.

\end{abstract}

\maketitle 
\tableofcontents

\section{Introduction}
Let $(M^3,g)$ be a complete, connected and oriented Riemannian $3$-manifold without boundary, and let $P$ be a principal $SU(2)$-bundle over $M$. The standard (faithful) representation of $SU(2)$ on $\mathbb{C}^2$ then gives rise to an associated complex vector bundle $E:=P\times_{SU(2)}\mathbb{C}^2$, which carries a Hermitian metric and an orientation. A connection on $E$ that is compatible with both these structures is said to be \emph{$SU(2)$-compatible}, and we denote by $\mathscr{A}(E)$ the space of smooth, $SU(2)$-compatible connections on $E$. In particular, $\mathscr{A}(E)$ is an affine space modeled on $\Omega^1(M,\mathfrak{su}(E))$, where $\mathfrak{su}(E)$ denotes the associated adjoint bundle of $E$, namely the real vector bundle of traceless, skew-Hermitian endomorphisms of $E$. In terms of the principal bundle $P$, we have $\mathfrak{su}(E)\cong P\times_{(\mathrm{Ad},SU(2))}\mathfrak{su}(2)$.

Given a constant $\lambda > 0$, for each parameter $\varepsilon>0$, 
we consider the $(\ep$-scaled) \textbf{Yang--Mills--Higgs energy} functional 
\begin{equation}\label{eq: YMH_energy}
\cY_{\ep}(\nabla, \Phi) := \int_{M} \ep^2 |F_{\nabla}|^2 + |\nabla \Phi|^2 + \frac{\lambda}{4\ep^2}(1 - |\Phi|^2)^2 \vol_{g},
\end{equation}
defined on the \emph{configuration space}
\[
\mathscr{C}(E):=\{(\nabla,\Phi)\in\mathscr{A}(E)\times\Gamma(\mathfrak{su}(E)):|F_{\nabla}|,\ |\nabla\Phi|,\ (1-|\Phi|^2)\in L^2(M)\}.
\] 
Here 
$F_{\nabla}\in\Omega^2(M,\mathfrak{su}(E))$ is the curvature of the connection $\nabla$, 
and the norms $|F_{\nabla}|$ and $|\nabla\Phi|$ are induced by the metric $g$ on $M$ together with the metric on $\mathfrak{su}(E)$ arising from the Ad$_{SU(2)}$-invariant inner product 
\[
(a,b)\mapsto -2\text{tr}(ab) = : \langle a, b \rangle
\]
on the Lie algebra $\mathfrak{su}(2)$. The section $\Phi$ of a pair $(\nabla,\Phi)\in\mathscr{C}(E)$ is called the \emph{Higgs field}, and the potential term $V(|\Phi|):=\frac{\lambda}{4}(1 - |\Phi|^2)^2$ appearing in~\eqref{eq: YMH_energy} is known as the \emph{Higgs self-interaction}, with $\lambda$ referred to commonly as the \emph{coupling constant}. For convenience, the integrand in~\eqref{eq: YMH_energy} will often be denoted by $e_{\ep}(\nabla, \Phi)$, so that given $(\nabla, \Phi) \in \mathscr{A}(E) \times \Gamma(\fsu(E))$, we have
\begin{equation}\label{eq:e-definition}
e_{\ep}(\nabla, \Phi) := \ep^2 |F_{\nabla}|^2 + |\nabla\Phi|^2 + \frac{\lambda}{4\ep^2} (1 - |\Phi|^2)^2.
\end{equation}
When the choice of $(\nabla, \Phi)$ is clear from the context, we simply write $e_{\ep}$ for $e_{\ep}(\nabla, \Phi)$.

Next, given $(\nabla, \Phi) \in \mathscr{C}(E)$ and a variation
\[
t \mapsto (\nabla + ta, \Phi + t\phi),
\]
where $(a,\phi)\in\Omega_c^1\oplus\Omega_c^0(M,\mathfrak{su}(E))$, a direct computation followed by integration by parts shows that 
\[
\begin{split}
\frac{d}{dt}\Big|_{t=0}\mathcal{Y}_{\varepsilon}(\nabla + ta,\Phi + t\phi) =\ & 2\langle \varepsilon^2 F_{\nabla},d_{\nabla}a\rangle_{L^2}  +  2\langle {\nabla}\Phi,\nabla\phi + [a,\Phi]\rangle_{L^2} -\frac{\lambda}{\varepsilon^2}\langle (1-|\Phi|^2)\Phi,\phi\rangle_{L^2}\\
=\ & 2\langle \varepsilon^2 d_{\nabla}^{\ast}F_{\nabla} - [\nabla\Phi,\Phi], a\rangle_{L^2} + 2\big\langle \nabla^{\ast}\nabla\Phi - \frac{\lambda}{2\varepsilon^2}(1-|\Phi|^2)\Phi,\phi\big\rangle_{L^2}.
\end{split}
\]
Here $d_{\nabla}^*$ is the adjoint of the exterior derivative $d_{\nabla}$ induced by $\nabla$, and $\nabla^*\nabla$ stands for the rough Laplacian, which is nonnegative definite according to the convention we adopt. Consequently, a configuration $(\nabla, \Phi) \in \mathscr{C}(E)$ is critical for $\cY_{\ep}$ subject to smooth, compactly supported variations if and only if it is a solution of the \emph{Yang--Mills--Higgs equations} 
(cf.~\cite[p.101]{Jaffe-Taubes}):
    \begin{equation}\label{eq: 2nd_order_crit_pt_intro}
        \begin{cases}
            \varepsilon^2 d_{\nabla}^{\ast}F_{\nabla} = [\nabla\Phi,\Phi],\\
            \nabla^{\ast}\nabla\Phi = \frac{\lambda}{2\varepsilon^2}(1-|\Phi|^2)\Phi.
        \end{cases}
    \end{equation}    
Below, when considering pairs in $\mathscr{C}(E)$, we use ``solution to~\eqref{eq: 2nd_order_crit_pt_intro}'' interchangeably with ``critical point of $\cY_\ep$''.

Finally, an important feature of the Yang--Mills--Higgs energy $\mathcal{Y}_{\varepsilon}$ is its invariance under the action of the group $\mathscr{G}(E)$ of gauge transformations of $E$. That is,
\begin{equation}\label{eq: gauge_invariance}
    \mathcal{Y}_{\varepsilon}({\rg}\cdot{(\nabla,\Phi)}) = \mathcal{Y}_{\varepsilon}(\nabla,\Phi),\quad \text{ for all } {\rg}\in\mathscr{G}(E),
\end{equation} where ${\rg}\cdot{(\nabla,\Phi)}:=({\rg}\circ\nabla\circ{\rg}^{-1},{\rg}\circ\Phi\circ{\rg}^{-1})$ denotes the gauge action. (This is due to the following relations 
\[
F_{{\rg}\circ\nabla\circ{\rg}^{-1}} = {\rg}\circ F_{\nabla}\circ{\rg}^{-1},\ \ \ \  ({\rg}\circ\nabla\circ{\rg}^{-1})({\rg}\circ\Phi\circ{\rg}^{-1}) = {\rg}\circ\nabla\Phi\circ {\rg}^{-1},
\]
and the fact that the metric in $\mathfrak{su}(E)$ comes from an Ad-invariant inner product.) Likewise, if $(\nabla, \Phi) \in \mathscr{C}(E)$ is a solution of~\eqref{eq: 2nd_order_crit_pt_intro}, then so is ${\rg} \cdot (\nabla, \Phi)$ for any ${\rg} \in \mathscr{G}(E)$. 

We are now ready to describe our main results. 
\subsection{Statements}\label{subsec: main_results}
To begin, using a variational construction via min-max families, similar to the one carried out in the work of Pigati--Stern~\cite{Pigati-Stern2021} on the $U(1)$ Yang--Mills--Higgs functional, we prove
\begin{thm}[Existence of critical points]\label{thm: existence}
Suppose $(M^3,g)$ is closed. Then there exist a universal constant $\cC_0>0$, and constants $\ep_M,\Lambda_M\in (0,\infty)$ depending only on $(M,g)$, such that for all $\ep\in(0,\ep_M)$, there is a 
solution $(\nabla_{\ep},\Phi_{\ep})\in\mathscr{C}(E)$ 
of~\eqref{eq: 2nd_order_crit_pt_intro} satisfying
    \begin{equation}\label{ineq: right_energy_regime}
    \min\{1,\lambda\}\cdot\cC_0\leqslant \ep^{-1}\cY_{\ep}(\nabla_{\ep},\Phi_{\ep})\leqslant \max\{1,\lambda\}\cdot\Lambda_M.
    \end{equation} Moreover, we can assume $\ep_M$ is sufficiently small, depending only on $M$, and possibly on $\lambda$ when $\lambda<1$, so that we also have $\Phi_{\ep}\not\equiv 0$. 
\end{thm}

A few comments are in order. First, several key components of the proof of Theorem~\ref{thm: existence}, most notably the choice of min-max families and the upper bound on the associated widths, are adapted from~\cite[Section 7]{Pigati-Stern2021}. Secondly, the assumption that $\lambda > 0$ is used in an essential way, especially in bounding the widths from below. (See Proposition~\ref{prop:min-max_lower_bound}.) Thirdly, as we describe in more detail in Section~\ref{subsec:context}, the normalization $\ep^{-1}\cY_{\ep}$ can be explained by a scaling argument that underlies the choice of energy regime in~\cite{Pigati-Stern2021} as well, and is at the same time equivalent to the normalization adopted in~\cite{Fadel-Oliveira2019} when analyzing the ``large mass'' limit of Yang--Mills--Higgs critical points over asymptotically conical $3$-manifolds.

In addition, note that the final part of the conclusion of Theorem \ref{thm: existence} is merely a consequence of the upper bound in \eqref{ineq: right_energy_regime}, because if $\Phi_{\ep}\equiv 0$ then 
    \[
    \ep^{-1}\cY_{\ep}(\nabla_{\ep},\Phi_{\ep})\geqslant\frac{\lambda}{4\ep^3}\mathrm{vol}(M),
    \] so up to further requiring
    \[
    \ep_M^3<\frac{\lambda\cdot\mathrm{vol}(M)}{4\max\{1,\lambda\}\cdot\Lambda_M},
    \] we must have $\Phi_{\ep}\not\equiv 0$. Now, an immediate follow-up question is whether one can also guarantee that the solutions $(\nabla_{\ep}, \Phi_{\ep})$ are \emph{irreducible}, in the sense that $\nabla_{\ep}\Phi_{\ep} \not\equiv 0$. Before clarifying the notion of irreducibility we adopt, and addressing the question just raised, we pause to make the following important remark: 
\begin{rmk}[Reducible solutions with $\Phi\not\equiv 0$]\label{rmk: reducible_sols}
Suppose $(\nabla,\Phi)$ is a solution of \eqref{eq: 2nd_order_crit_pt_intro} such that $\nabla\Phi \equiv 0$. On the one hand, by the first equation in~\eqref{eq: 2nd_order_crit_pt_intro}, we get that $\nabla$ satisfies the Yang--Mills equation:
    \begin{equation}\label{eq: YM}
    d_{\nabla}^{\ast}F_{\nabla} = 0,
    \end{equation}
    which, we emphasize, is independent of both the coupling constant $\lambda$ and the parameter $\ep$. On the other hand, $\nabla\Phi \equiv 0$ implies that $|\Phi|$ is constant, and the second equation in~\eqref{eq: 2nd_order_crit_pt_intro} forces either $\Phi\equiv 0$ or $|\Phi|\equiv 1$. In the latter case, we have an orthogonal splitting $\mathfrak{su}(E)=\langle\Phi\rangle\oplus\langle\Phi\rangle^{\perp}$, and since 
            \[
            [F_{\nabla},\Phi] = d_{\nabla}(\nabla\Phi) \equiv 0,
            \] it follows that $F_{\nabla} = \langle F_{\nabla},\Phi\rangle\Phi$. In particular, the fact that $\nabla$ is Yang--Mills (together with the Bianchi identity, $d_{\nabla}F_{\nabla}=0$) then reduces to the fact that $\bangle{F_{\nabla}, \Phi}$ is a harmonic $2$-form, with $\bangle{F_{\nabla}, \Phi}\in\ker(d\oplus d^\ast)$. In fact, in this case the connection $\nabla$ necessarily \emph{reduces} to a $U(1)$ Yang--Mills connection on $L:=\ker(\Phi - \frac{\sqrt{-1}}{2})$, and in an appropriate gauge one has
\[
            F_{\nabla}=\mathrm{diag}(F_L,-F_L),\ \ \ \Phi = \mathrm{diag}(\frac{\sqrt{-1}}{2},-\frac{\sqrt{-1}}{2}),
            \]
            where $F_L$ is the curvature of the reduced connection.
            
            Ignoring pure Yang--Mills solutions $(\nabla,0)$, henceforth we shall say that a solution $(\nabla,\Phi)$ of \eqref{eq: 2nd_order_crit_pt_intro} with $\Phi\not\equiv 0$ is \emph{reducible} when $\nabla\Phi\equiv 0$, and \emph{irreducible} otherwise. By the above discussion, a pair $(\nabla,\Phi\not\equiv 0)$ is a reducible solution of \eqref{eq: 2nd_order_crit_pt_intro} if and only if
            \begin{equation}\label{eq: reducible_pair_intro}
            |\Phi|\equiv 1,\quad\nabla\Phi\equiv 0,\quad\text{and}\quad F_{\nabla}=\langle F_{\nabla},\Phi\rangle\Phi,\quad\text{with}\quad \langle F_{\nabla},\Phi\rangle\in\ker(d\oplus d^*). 
            \end{equation}
            
            A configuration $(\nabla,\Phi)\in\mathscr{C}(E)$ is called \emph{trivial} when $\cY_{\ep}(\nabla,\Phi)=0$, that is, when
            \[
            |\Phi|\equiv 1,\quad\nabla\Phi\equiv 0,\quad\text{and}\quad F_{\nabla}\equiv 0.
            \] It follows that if $(M^3,g)$ admits no $L^2$-bounded harmonic $2$-forms (or, equivalently, $1$-forms), then all reducible solutions $(\nabla,\Phi\not\equiv 0)\in\mathscr{C}(E)$ of \eqref{eq: 2nd_order_crit_pt_intro} must be trivial. Conversely, as we explain in Section~\ref{subsec:reducible_solutions}, if $M$ is \emph{closed} and $b_2(M) \neq 0$, then there always exist reducible solutions which are non-trivial.
\end{rmk}

The discussion above leads to the following refinement of Theorem~\ref{thm: existence} when further conditions are placed on $M$.
\begin{thm}[Existence of irreducible critical points]\label{thm: irred}
Suppose $(M^3, g)$ is closed and $M^3$ is a rational homology 3-sphere, that is $b_1(M)=b_2(M)=0$. Then the solutions $(\nabla_{\ep},\Phi_{\ep})$ produced by Theorem~\ref{thm: existence} are \emph{irreducible} in the sense that $\nabla_{\ep}\Phi_{\ep}\not\equiv 0$.
\end{thm}
\begin{proof}[Proof of Theorem~\ref{thm: irred} assuming Theorem~\ref{thm: existence}]
Since $(\nabla_{\ep}, \Phi_{\ep}\not\equiv 0)\in\mathscr{C}(E)$ is a non-trivial solution of \eqref{eq: 2nd_order_crit_pt_intro}, the conclusion follows by Remark \ref{rmk: reducible_sols} and the assumption $b_2(M)=0$. 
\end{proof}

Next, we prove that on any complete, oriented, Riemannian $3$-manifold of  bounded geometry, there is an \textbf{energy gap} for irreducible solutions of \eqref{eq: 2nd_order_crit_pt_intro}, as long as the parameter $\ep$ is sufficiently small. Here, and throughout this paper, by \emph{bounded geometry} we mean the existence of a positive lower bound for the injectivity radius, together with bounds on the Riemann curvature tensor and its covariant derivatives of all orders.
\begin{thm}[Gap theorems]\label{thm: gap}
Suppose $(M^3,g)$ has bounded geometry, and let $\lambda_0$ be an upper bound for $\lambda$. 
Then, there exist constants $\theta_{\mathrm{gap}} = \theta_{\mathrm{gap}}(\lambda_0)$ and $\tau_{\mathrm{gap}} = \tau_{\mathrm{gap}}(\lambda_0, M, g)$ 
such that if 
\[
\ep < \tau_{\mathrm{gap}} \cdot \min\{\sqrt{\lambda}, 1\}
\]
and if $(\nabla, \Phi) \in \mathscr{C}(E)$ is a 
solution of~\eqref{eq: 2nd_order_crit_pt_intro} satisfying 
    \[
    \ep^{-1}\cY_{\ep}(\nabla,\Phi)\leqslant \theta_{\mathrm{gap}} \cdot \min\{\lambda, 1\},
    \] then $(\nabla,\Phi)$ is reducible as in \eqref{eq: reducible_pair_intro}. In particular, if $(M^3,g)$ admits no $L^2$-bounded harmonic $2$-forms (or, equivalently, $1$-forms), for instance if furthermore we impose
\vskip 1mm
    \begin{itemize} {\normalsize 
        \item[(i)] $M$ is closed and $b_1(M)=0$; or
        \vskip 1mm
        \item[(ii)] $M$ is noncompact and\footnote{It is well known that if $(M^n,g)$ is a complete noncompact Riemannian manifold with $\mathrm{Ric}(g)\geqslant 0$ then it admits no nonzero $L^2$-bounded harmonic $1$-forms; see \cite[Theorem 1]{greene1981harmonic}.} $\mathrm{Ric}(g)\geqslant 0$;}
    \end{itemize}  
\vskip 1mm
then in fact $(\nabla,\Phi)$ is trivial, that is, $\cY_{\ep}(\nabla,\Phi)=0$. 
\end{thm}
Theorem~\ref{thm: gap} is a consequence of local \textit{a priori} estimates obtained largely by following~\cite[Chapter IV]{Jaffe-Taubes}, which together with the smallness assumptions in Theorem~\ref{thm: gap} leads to a differential inequality on $|\nabla\Phi|^2$ that is favorable for the application of the maximum principle. (See especially Lemma~\ref{lemm:nablaPhi_exp_decay_base} and Proposition~\ref{prop:concentration_scale}.) Throughout this argument, the assumption $\lambda > 0$ is again used heavily. For reasons we elaborate on shortly, the case when $M$ is $\RR^3$ equipped with the standard flat metric $g_{\mathbb{R}^3}$ is worth singling out. Here, it turns out that we can remove the smallness assumption on $\ep$ by a scaling argument. This leads to the following gap result for solutions of~\eqref{eq: 2nd_order_crit_pt_intro} on $\RR^3$.

\begin{thm}[Gap theorem on $\RR^3$]\label{thm: gap_for_R3}
Suppose $\lambda \in (0, \lambda_0]$. For any $\ep > 0$, if $(\nabla,\Phi)$ is a smooth solution to~\eqref{eq: 2nd_order_crit_pt_intro} on an $SU(2)$-bundle $E\to\mathbb{R}^3$ over the Euclidean space $(\mathbb{R}^3,g_{\mathbb{R}^3})$, satisfying in addition that
    \[
    \ep^{-1}\cY_{\ep}(\nabla,\Phi)\leqslant \theta_{\mathrm{gap}}\cdot \min\{\lambda, 1\}, 
    \] then in fact
    \[
    \cY_{\ep}(\nabla,\Phi) = 0.
    \]
\end{thm}

To explain the interest of Theorem~\ref{thm: gap_for_R3}, we digress to recall a notion closely related to energy gaps, namely the magnetic charge. Suppose $M=\mathbb{R}^3$ and let $(\nabla,\Phi)\in\mathscr{C}(E)$. Then, it follows from the work of Taubes~\cite{Jaffe-Taubes} and Groisser~\cite{Groisser1984} that the \textbf{magnetic charge} of $(\nabla,\Phi)$, defined by 
\begin{equation}\label{eq: magnetic_charge}
k=k(\nabla,\Phi):=\frac{1}{4\pi}\int_{\mathbb{R}^3}\langle F_{\nabla}\wedge\nabla\Phi\rangle = \frac{1}{4\pi} \int_{\mathbb{R}^3}\langle * F_{\nabla}, \nabla\Phi \rangle \vol_{g_{\mathbb{R}^3}},
\end{equation} 
is always an integer. (See also Fadel~\cite{fadel2023asymptotics} for the integrality of the magnetic charge on general asymptotically conical $3$-manifolds.) The magnetic charge has the following interpretation. For sufficiently large $R$ depending on the configuration $(\nabla, \Phi)$, restricting $\Phi/|\Phi|$ to $\Sigma_R:=\partial \overline{B}_R(0)\cong\mathbb{S}^2$ determines a homotopy class of maps $\mathbb{S}^2\to\mathbb{S}^2$, and $k$ is the Brouwer degree of this class. Alternatively, the restrictions of the associated vector bundle $E=P\times_{SU(2)}\mathbb{C}^2$ over $\Sigma_R$ split as $L\oplus L^{-1}$, where $L$ is a complex line bundle over $\Sigma_{R}\cong\mathbb{S}^2$, corresponding to one of the eigenspaces of $\Phi$, and the degree of any such $L$ does not depend on $R$ and equals the charge $k$. 
Recalling also the formula (see~\cite[p.103]{Jaffe-Taubes} or~\cite[p.13]{Atiyah-Hitchin1988})
\begin{equation}\label{eq: energy_formula}
\mathcal{Y}_{\ep}(\nabla,\Phi) = \pm 8\pi k \ep + \|\ep F_{\nabla}\mp\ast\nabla\Phi\|_{L^2}^2 + \frac{\lambda}{4\ep^2}\|1-|\Phi|^2\|_{L^2}^2,
\end{equation} 
we arrive at the following well-known topological lower bound for the Yang--Mills--Higgs energy:
\begin{equation}\label{ineq: topological_energy_lower_bound}
\ep^{-1}\mathcal{Y}_{\ep}(\nabla,\Phi)\geqslant 8\pi |k|.
\end{equation}

When $\lambda=0$, equality in \eqref{ineq: topological_energy_lower_bound} holds if, and only if, $(\nabla,\Phi)$ is a solution to the first order \textbf{(anti-)monopole equations} 
\begin{equation}\label{eq: monopole_eq}
\ep F_{\nabla} = \pm\ast\nabla\Phi,
\end{equation}
which are easily seen to imply the second order equations~\eqref{eq: 2nd_order_crit_pt_intro}. In contrast, when $\lambda>0$, attaining the topological lower bound forces the configuration to be trivial in the sense that $\cY_{\ep}(\nabla,\Phi)=0$, in which case $k=0$. At any rate, one sees that when $M=\mathbb{R}^3$, the number $8\pi$ gives an energy gap for $\ep^{-1}\cY_{\ep}$ over the subset of $\mathscr{C}(E)$ consisting of configurations $(\nabla,\Phi)$ with $k(\nabla,\Phi)\neq 0$. On the other hand, Sibner--Talvacchia~\cite{Sibner-Talvacchia1994} has shown that, for any $\lambda>0$, there exists a finite-action solution $(\nabla,\Phi)$ of~\eqref{eq: 2nd_order_crit_pt_intro} on $\mathbb{R}^3$ with $\cY_{\ep}(\nabla, \Phi) > 0$ and $k(\nabla,\Phi)=0$. It is therefore interesting to find an energy gap that applies to configurations with zero charge as well, and that is what Theorem~\ref{thm: gap_for_R3} addresses.

Our next results concern the asymptotic behavior as $\ep \to 0$ of critical points of $\cY_{\ep}$ satisfying suitable energy bounds, and in particular are applicable to the family produced by Theorem~\ref{thm: existence}. Specifically, suppose $(M^3,g)$ has bounded geometry, $\lambda\in (0,\lambda_0]$, and let $(\nabla_{\ep},\Phi_{\ep})\in\mathscr{C}(E)$ be a family of critical points of $\mathcal{Y}_{\ep}$, satisfying a uniform energy bound 
\[
\ep^{-1}\cY_{\ep}(\nabla_{\ep},\Phi_{\ep})\leqslant\Lambda,
\] for some constant $\Lambda>0$ (possibly depending on $\lambda$ and $(M,g)$). We define the \textbf{blow-up set} of the sequence by
\begin{equation*}
S:=\bigcap_{0<r\leqslant  r_0}\left\{x\in M: \liminf_{\ep\to 0}\ep^{-1}\int_{B_r(x)}e_{\ep}(\nabla_{\ep},\Phi_{\ep})\geqslant\eta_{\ast}\right\},
\end{equation*}
where $\eta_{\ast} > 0$ is a threshold to be determined depending only on $(M, g)$, $\min\{1,\lambda\}$ and $\lambda_0$ (see Section \ref{sec:asymptotic}). Define also, for any $\beta\in (0,\frac{1}{2})$, the sets
\[
Z_{\beta}(\Phi_{\ep}):=\{x\in M: |\Phi_{\ep}(x)|^2\leqslant 1-2\beta\},
\] and let 
\begin{equation*}
Z_{\beta}:= \bigcap_{\kappa>0}\overline{\bigcup_{0<\ep<\kappa} Z_{\beta}(\Phi_{\ep})}.    
\end{equation*}

\begin{thm}[Asymptotic limit as $\ep\to 0$]
\label{thm: asymptotic}
In the above setting, we have the following.
\begin{enumerate}
\item[(a)] $Z_{\beta}\subset S$ and $\mathcal{H}^0(S)\leqslant\eta_\ast^{-1}\Lambda<\infty$. In particular both $Z_{\beta}$ and $S$ are finite sets.
\vskip 1mm
\item[(b)] Along a sequence of $\ep_i$'s converging to $0$, we have, in the sense of Radon measures,
\[
\mu_i:=\ep_i^{-1}e_{\ep_i}(\nabla_i, \Phi_i)\mathcal{H}^3 \rightharpoonup |h|^2\mathcal{H}^3+\sum_{x \in S}\Theta(x)\delta_{x},
\] and
\[
\kappa_{\ep_i} := 2\langle * F_{\nabla_i}, \nabla_i\Phi_i\rangle\mathcal{H}^3 \rightharpoonup \sum_{x \in S}\Xi(x)\delta_{x},
\] 
where $h$ is a harmonic $1$-form on $M$, and for all $x \in S$ there holds $\Theta(x)\geqslant\eta_{\ast}$, $\Xi(x) \in 8\pi \mathbb{Z}$, and $|\Xi(x)| \leqslant \Theta(x)$. Furthermore, $\Xi(x)=0$ for any $x\in S\setminus Z_{\beta}$.
\vskip 1mm
\item[(c)] Assume in addition that $M$ is closed and consider the Hodge decomposition
\begin{equation*}
\ep_i^{\frac{1}{2}}\langle * F_{\nabla_i}, \Phi_i\rangle = h_i + df_i + d^*\alpha_i,
\end{equation*}
where $f_i \in C^{\infty}(M)$, $\alpha_i \in \Omega^2(M)$, and $h_i$ is harmonic. Then both $df_i$ and $d^*\alpha_i$ subconverge smoothly to $0$ on compact subsets of $M\setminus S$, while $h_i$ subconverges smoothly on $M$ to the harmonic $1$-form $h$ from part (b).
\end{enumerate}
\end{thm}

Given the decomposition of the limiting measure in Theorem~\ref{thm: asymptotic}(b), and in view of the numerous precedents of \textbf{bubble tree} convergence theorems, a natural question is whether each of the densities $\Theta(x)$ is equal to the total energy of finitely many \textbf{bubbles}, or scaling limits of $(\nabla_i, \Phi_i)$ at points near $S$ where energy is concentrating most rapidly. To this, we are able to give a positive answer.

\begin{thm}[Bubbling]\label{thm: bubbling}
For all $x \in S$ there exists a critical point $(\nabla, \Phi)$ of $\cY_1^{g_{\mathbb{R}^3}}$ on $\mathbb{R}^3\cong (T_xM,g|_{T_xM})$ such that 
\begin{equation*}
0 < \cY_1^{g_{\mathbb{R}^3}}(\nabla, \Phi; \mathbb{R}^3) \leqslant \Theta(x).
\end{equation*}
 In fact, for each $x\in S$, there exists a finite collection of non-trivial, finite-action critical points of $\cY_1^{g_{\mathbb{R}^3}}$ on $\mathbb{R}^3$ whose energy and charge sum up to $\Theta(x)$ and $\Xi(x)$, respectively.

\end{thm}
\begin{rmk}\label{rmk: existence_by_blowup}
It follows from Theorem~\ref{thm: bubbling} and Theorem~\ref{thm: gap_for_R3} that $\Theta(x) \geqslant \theta_{\mathrm{gap}} \cdot \min\{\lambda, 1\}$ for all $x \in S$. Also, we can actually replace the upper bound on the number of elements in $S$ from Theorem \ref{thm: asymptotic}(a) by $\cH^0(S)\leqslant \theta_{\mathrm{gap}}^{-1}\cdot\max\{\lambda^{-1},1\}\cdot\Lambda$ (see Remark \ref{rmk:improved_upper_bound_S}). Finally, applying Theorem~\ref{thm: irred} to any closed $3$-manifold admitting no non-trivial harmonic $1$-forms, say $M=\mathbb{S}^3$, we deduce from Theorem~\ref{thm: asymptotic} and Theorem~\ref{thm: bubbling} the existence of non-trivial critical points of $\cY_1^{g_{\mathbb{R}^3}}$ on $\mathbb{R}^3$, for any $\lambda > 0$ (see Proposition \ref{prop:existence_R3}). Extracting information about the magnetic charge of critical points obtained this way is something we wish to take up in a future work.
\end{rmk}

A few comments on the proof of Theorem~\ref{thm: bubbling} might be helpful at this point. The standard procedure for extracting bubbles determines a sequence of rescaling rates whose ratio to $\ep_i$ cannot be prescribed beforehand. Nonetheless, thanks to the local estimates obtained in Section~\ref{sec:estimates}, in particular Proposition~\ref{prop:coarse_estimate} and Proposition~\ref{prop:local_convergence}, this rate turns out to be comparable to $\ep_i$ (Lemma~\ref{lemm:rate_of_rescaling}), and therefore the bubbles we obtain are non-trivial, finite energy solutions of~\eqref{eq: 2nd_order_crit_pt_intro} on $\RR^3$ with $\ep = 1$, as the statement asserts (Proposition~\ref{prop:multiplicity_lower_bound}). Next, identifying \textbf{neck regions} between bubbles is a routine matter, and we show that eventually the neck regions carry no energy by appealing to the exponential decay estimates on $|\nabla_i\Phi_i|$ and $1 - |\Phi_i|$ in Section~\ref{subsec:exp-decay}, and combining them with a local version of the equipartition theorem in~\cite[Corollary II.2.2]{Jaffe-Taubes} to control $|F_{\nabla_i}|$, taking advantage of the fact that we are working in dimension three (Lemma~\ref{lem:conservation_law} and Proposition~\ref{prop:no_neck}).

\subsection{Context}\label{subsec:context}
The main concerns of this paper are the construction of solutions to the SU(2) Yang--Mills--Higgs equation~\eqref{eq: 2nd_order_crit_pt_intro}, and the study of their limiting behavior as $\ep\to 0$. Below we briefly mention a number of previous works that, in our own biased view, are most relevant to our results, making no attempt to survey the many facets of the vast literature on monopoles and the Yang--Mills--Higgs equations.

In the case $\lambda = 0$, the earliest known solution is the celebrated Bogomol'nyi--Prasad--Sommerfield (BPS) monopole~\cite{Bogomolnyi1976,Prasad-Sommerfield1975}, which is a spherically symmetric, charge one solution of~\eqref{eq: monopole_eq} on $\RR^3$. Later, in what is perhaps the first instance of a gluing construction, Taubes established the existence of monopoles on $\RR^3$ with arbitrary charge~\cite[Chapter IV]{Jaffe-Taubes} by perturbing approximate solutions built out of BPS monopoles. Shortly thereafter, in a series of works~\cite{Taubes1982I,Taubes1982II,Taubes1983,Taubes1984,Taubes1985} that constituted a major tour de force, Taubes developed a min-max theory for the $SU(2)$ Yang--Mills--Higgs functional on $\RR^3$ and proved that for each prescribed charge, there exist infinitely many solutions to~\eqref{eq: 2nd_order_crit_pt_intro} which are not monopoles\footnote{Our focus here is on analytical approaches, but as is well-known, around the same time, an algebraic description of the space of $SU(2)$-monopoles with arbitrary fixed charge on $\RR^3$ emerged from the works of Hitchin~\cite{Hitchin1982,Hitchin1983}, Donaldson~\cite{Donaldson-CMP1984} and Hurtubise~\cite{Hurtubise1983,Hurtubise1985}. The interested reader is referred to the classical text~\cite{Atiyah-Hitchin1988} by Atiyah and Hitchin.}. Both the perturbative and variational approaches pioneered by Taubes were subsequently applied to produce solutions of~\eqref{eq: 2nd_order_crit_pt_intro} on other $3$-manifolds. For instance, as a crucial step in their construction of non-self-dual $SU(2)$ Yang--Mills connections on $S^4$, L. M. Sibner, R. J. Sibner and Uhlenbeck~\cite{Sibner-Sibner-Uhlenbeck1989} performed an analogue of Taubes' min-max construction on $\HH^3$. On the other hand, gluing constructions of monopoles were carried out by Floer and Ernst on asymptotically flat $3$-manifolds~\cite{Floer1987,Ernst1995}, by L. M. Sibner and R. J. Sibner on $\HH^3$~\cite{Sibner-Sibner2012}, by Foscolo on $\RR^2 \times S^1$~\cite{Foscolo2017}, and by Oliveira on asymptotically conical $3$-manifolds~\cite{Oliveira2016}, to name a few examples. Also, on closed $3$-manifolds, where taking $\lambda = 0 $ forces all monopoles to be trivial, Esfahani instead constructed monopoles with prescribed point singularities on rational homology $3$-spheres, by gluing together BPS solutions and liftings of Dirac monopoles~\cite{Esfahani2022}. 

When $\lambda > 0$, due partly to the absence of a first-order reduction such as~\eqref{eq: monopole_eq}, it appears that much fewer existence results for~\eqref{eq: 2nd_order_crit_pt_intro} are available, even on $\RR^3$, compared to the $\lambda = 0$ case. A part of Taubes' min-max theory was extended by Groisser~\cite{Groisser-thesis} to the case of sufficiently small positive $\lambda$. This restriction was later removed by L. M. Sibner and Talvacchia~\cite{Sibner-Talvacchia1994}, who proved that for any $\lambda > 0$, there exists a non-trivial solution of~\eqref{eq: 2nd_order_crit_pt_intro} on $\RR^3$ with charge zero. On the other hand, completing earlier work by Tyupkin, Fateev and Shvarts~\cite{Tyupkin-Fateev-Shvarts}, Plohr~\cite{Plohr-thesis} obtained spherically symmetric solutions with magnetic charge one by minimizing the Yang--Mills--Higgs functional over a class of symmetric configurations similar to the 't Hooft--Polyakov ansatz underlying the BPS monopole, and showing that the resulting minimizer is a solution of~\eqref{eq: 2nd_order_crit_pt_intro}. A similar construction was carried out by Schechter and Weder~\cite{Schechter-Weder1981}. Dostoglou~\cite{Dostoglou-thesis} then succeeded in finding solutions that minimize the Yang--Mills--Higgs functional over all spherically symmetric configurations. Our existence result noted in Remark~\ref{rmk: existence_by_blowup} adds to this list of approaches to solving~\eqref{eq: 2nd_order_crit_pt_intro} on $\RR^3$ when $\lambda > 0$.

Turning to the limiting behavior of solutions as $\ep \to 0$, we note that Theorem~\ref{thm: asymptotic} and Theorem~\ref{thm: bubbling} can be regarded as a three-dimensional and non-abelian analogue of the results of Hong, Jost and Struwe~\cite{Hong-Jost-Struwe1996} on the $U(1)$ Yang--Mills--Higgs functional over closed surfaces, in that the energy regimes considered in the two works arise from similar scaling arguments. In our context, suppose for simplicity that $M$ is $\RR^3$. Then, given any $(\nabla, \Phi) \in \mathscr{C}(E)$ and a sequence $\ep_i \to 0$, with $s_{i}$ denoting the map $x \mapsto \ep_i x$, and with $(\nabla_{i}, \Phi_{i})$ defined via $(\nabla, \Phi) = s_{i}^*(\nabla_{i}, \Phi_{i})$, we see that
\[
\cY_{1}(\nabla, \Phi) = \ep_i^{-1}\cY_{\ep_i}(\nabla_{i}, \Phi_{i}),
\]
relating a uniform bound on $\ep^{-1}\cY_{\ep}$ to concentration behavior resembling a blow-down process. The work~\cite{Hong-Jost-Struwe1996} was later vastly generalized by Pigati and Stern in~\cite{Pigati-Stern2021}, where $U(1)$ Yang--Mills--Higgs critical points with uniformly bounded energy are produced on arbitrary closed Riemannian $n$-manifolds, and shown to concentrate along the support of a stationary, integral $(n-2)$-varifold, thereby giving, in the codimension-$2$ case, an alternative proof of the fundamental existence result of Almgren. To find an analogue of this correspondence in the codimension-$3$ setting is one of our motivations for studying the $SU(2)$ Yang--Mills--Higgs functional. Although $\lambda = 0$ appears to be the more appropriate choice for such a search, since it is here that monopoles arise and could potentially fill the role played by vortices in~\cite{Pigati-Stern2021} (see especially Proposition 6.7 therein), we nonetheless regard the results in this paper as a first step in that direction, not least because the $\lambda > 0$ assumption allows us to obtain non-trivial critical points on closed $3$-manifolds. Motivated by the very recent work of Parise, Pigati and Stern~\cite{Parise-Pigati-Stern2025} on the $\Gamma$-convergence of $SU(2)$ Yang--Mills--Higgs (with $\lambda = 0$) to the $(n-3)$-volume, particularly Remark 1.5 therein, we hope to investigate in a future work the possibility of letting $\lambda$ tend to $0$ along with $\ep$ in our asymptotic analysis. Also, to go from either the $\Gamma$-convergence in~\cite{Parise-Pigati-Stern2025} or our analysis of critical points in dimension $3$ to a convergence result for critical points in general dimensions in the style of~\cite{Pigati-Stern2021}, one major obstacle is proving a monotonicity-type formula which would allow bounds on $r^{3-n} \int_{B_r(x)}e_{\ep}$ to be passed from one scale to smaller scales. This entails analyzing how energy is distributed among the terms in the integrand of~\eqref{eq: YMH_energy}, and is again something we wish to address eventually.

Our results are also related to another type of asymptotic analysis on the $SU(2)$ Yang--Mills--Higgs functional. Specifically, working over an asymptotically conical $3$-manifold, Fadel and Oliveira~\cite{Fadel-Oliveira2019} studied sequences of \emph{finite mass} monopoles\footnote{In \cite{Fadel-Oliveira2019}, Fadel and Oliveira worked with a definition of finite mass as in Oliveira's thesis \cite[Definition 1.4.1]{Oliveira-thesis}, which later was proved by Fadel \cite{fadel2023asymptotics} to be exactly the condition that the Higgs field norm $|\Phi|$ converges uniformly, along the conical end, to a constant at infinity; see \cite[Remark 1.10]{fadel2023asymptotics}. In particular, it follows from combining \cite[Theorems 1.1 and 1.4]{fadel2023asymptotics} with \cite[Proposition 1.4.4]{Oliveira-thesis} that the finite mass condition for a monopole is equivalent to finite energy.} with fixed magnetic charge and with mass tending to infinity. By a version of the formula~\eqref{eq: energy_formula}, the energy divided by the mass remains constant along such a sequence, a condition which is closely related to the energy bound considered in this paper by the following observation:  taking $\lambda = 0$ in~\eqref{eq: YMH_energy}, writing $m$ for $\ep^{-1}$, and then letting $\Psi = m\Phi$, one sees that
\[
\ep^{-1}\cY_{\ep}(\nabla, \Phi) = m^{-1}\int_{M} |F_{\nabla}|^2 + |\nabla\Psi|^2 \vol_g.
\]
In fact the relationship goes beyond this formal level, as can be seen by comparing Theorem~\ref{thm: asymptotic} and Theorem~\ref{thm: bubbling} above to Theorem 1.1 in~\cite{Fadel-Oliveira2019}, with the notable difference that we do not know in Theorem~\ref{thm: asymptotic} whether $Z_{\beta} = S$. By the first part of Theorem~\ref{thm: bubbling}, the failure of this equality would yield a non-trivial, finite energy critical point $(\nabla, \Phi)$ of $\cY_1^{g_{\RR^3}}$ (with $\lambda > 0$) on $\RR^3$ with a non-vanishing Higgs field. We are currently still investigating whether there could indeed be such a solution. Finally, while we shall not enter into a detailed discussion of this topic, we would be remiss not to mention that the monopole equation admits generalizations to Calabi--Yau and $G_2$-manifolds (see~\cite[Chapter 1]{Oliveira-thesis} for a succinct exposition of the background), and that sequences of these higher-dimensional monopoles, in the large mass limit, are expected to concentrate along codimension-$3$ calibrated submanifolds. We refer the reader to ~\cite{Fadel-thesis},~\cite{Fadel-Nagy-Oliveira2024},~\cite{Parise-Pigati-Stern2025} and~\cite{Li2025} for examples of recent progress in this direction.

\subsection{Notation and conventions}
As already mentioned after~\eqref{eq: YMH_energy}, we consider on $\mathfrak{su}(E)$ the metric $\langle\cdot{},\cdot{}\rangle$ induced by minus one-half the Cartan--Killing form of $\mathfrak{su}(2)$. That is, $\langle a,b\rangle := -2\tr(ab)$. If $\sigma_1,\sigma_2,\sigma_3$ denote the Pauli matrices, then
\begin{align}\label{eq:su2-on-basis}
T_1:=\frac{i\sigma_1}{2},\quad T_2:=\frac{i\sigma_2}{2},\quad T_3:=\frac{i\sigma_3}{2}    
\end{align}
gives an orthonormal basis of $\mathfrak{su}(2)$ with respect to $\langle\cdot{},\cdot{}\rangle$, satisfying
\[
[T_1,T_2]=-T_3,\quad [T_1,T_3]=T_2,\quad [T_2,T_3]=-T_1.
\] In particular, we get for all $a, b, c \in \fsu(2)$ that
\begin{equation}\label{eq: triple_cross_product}
[a,[b,c]] = b\langle a,c\rangle - c\langle a,b\rangle,
\end{equation}
and that
\begin{equation}\label{eq: bracket_norm}
\big|[a, b]\big|^2 + \big(\bangle{a, b}\big)^2 = |a|^2|b|^2.
\end{equation}
Given a Higgs field $0\neq\Phi\in\Gamma(\mathfrak{su}(E))$, we shall often write $w$ to mean $\frac{1}{2}(1 - |\Phi|^2)$. Also, we denote by 
\[
Z(\Phi):=\{x\in X: \Phi(x)=0\}
\]
the (gauge invariant\footnote{Note that $Z(\Phi)=Z({\rg}\circ\Phi\circ {\rg}^{-1})$ for any gauge transformation ${\rg}\in\mathscr{G}(E)$.}) zero locus of $\Phi$. Note that $Z(\Phi)$ is closed in $M$, and therefore would be compact if $M$ is. On the other hand, in the case where $M$ is complete, noncompact, if in addition $w$ decays to zero at infinity (which would occur if for instance $(\nabla,\Phi)$ is a finite energy critical point of $\cY_{\ep}$ with $\lambda > 0$, see Proposition~\ref{prop:maximum_principle_for_w} below), then $Z(\Phi)$ would be bounded, so by completeness of $M$ we again get that $Z(\Phi)$ is compact. At any rate, on the open set $V := M\setminus Z(\Phi)$, we have the decomposition
\begin{equation}\label{eq: adjoint_decomp}
\mathfrak{su}(E)|_{V}=\mathfrak{su}(E)^{||} \oplus \mathfrak{su}(E)^\perp,
\end{equation} where the \emph{longitudinal} line bundle $\mathfrak{su}(E)^{||}$ is given by
\begin{equation}
	\mathfrak{su}(E)^{||} = \ker \left(\mathrm{ad}(\Phi):\mathfrak{su}(E)|_{V}\to \mathfrak{su}(E)|_{V}\right) = \langle \Phi \rangle,
\end{equation}
and the \emph{transverse} rank $2$ bundle $\mathfrak{su}(E)^\perp$ is the orthogonal complement of $\mathfrak{su}(E)^{||}$. We note that
\begin{equation}\label{eq: decomp_Lie}
	[\mathfrak{su}(E)^{\perp},\mathfrak{su}(E)^{\perp}]\subset \mathfrak{su}(E)^{||}\quad\text{and}\quad [\mathfrak{su}(E)^{||},\mathfrak{su}(E)^{\perp}]\subset \mathfrak{su}(E)^{\perp}.
\end{equation}
Henceforth, over $M \setminus Z(\Phi)$, we split any section $\xi$ of $\mathfrak{su}(E)$ as $\xi =\xi^{||} + \xi^\perp$ according to the decomposition~\eqref{eq: adjoint_decomp}. More explicitly:
\begin{subequations}
	\begin{align}
		\xi^{||} &:=|\Phi|^{-2}\langle\xi,\Phi\rangle\Phi, \label{eq:long_part} \\
		\xi^{\perp} &:= |\Phi|^{-2}[\Phi,[\xi,\Phi]]. \label{eq:transv_part}
	\end{align} 
\end{subequations} It is clear that $\xi^{||}$ and $\xi^{\perp}$ are smooth on the complement of $Z(\Phi)$. For future use, we also note the following relations which hold outside of $Z(\Phi)$. Given sections $a, b, c$ of $\fsu(E)$, we first have by~\eqref{eq: bracket_norm} that
\begin{equation}\label{ineq:eigen_ad}
|[\Phi, a]| = |[\Phi, a^\perp]| = |\Phi| |a^{\perp}|.
\end{equation} 
Second, combining~\eqref{eq: bracket_norm} with the fact that $[a, b] = [a^{||}, b^{\perp}] + [a^{\perp}, b]$, we have
\begin{equation}\label{eq: bracket_ineq}
|[a, b]| \leqslant |a| |b^{\perp}| + |a^{\perp}| |b|,
\end{equation}
from which we get 
\begin{equation}\label{eq: double_bracket_ineq}
|[[a, b], \Phi]| \leqslant  |a||[b, \Phi]| + |[a,\Phi]||b|.
\end{equation}
Note that~\eqref{eq: double_bracket_ineq} holds even on $Z(\Phi)$. Finally, again using \eqref{eq: decomp_Lie}, and the Ad-invariance of the inner product, 
\begin{equation}\label{eq: inner_prod_decomp}
	\langle [a,b],c\rangle = \langle [a^{||},b^{\perp}],c^{\perp}\rangle + \langle [a^{\perp},b^{||}],c^{\perp}\rangle + \langle [a^{\perp},b^{\perp}],c^{||}\rangle\quad,
    \text{ for all }a,b,c\in \mathfrak{su}(E).
\end{equation} 

Next, given $\nabla \in \mathscr{A}(E)$, as noted above, $d_{\nabla}$ stands for the exterior covariant derivative induced by $\nabla$, and throughout the paper it mostly acts on $\fsu(E)$-valued forms. Since we are on a $3$-manifold, its adjoint, when acting on $p$-forms, is given by 
\begin{equation}\label{eq: d_star}
d_{\nabla}^* = (-1)^p * d_\nabla *.
\end{equation}
We then denote by $\Delta_{\nabla}$ the Hodge Laplacian induced by $\nabla$, that is, $\Delta_{\nabla} = d_{\nabla}d_{\nabla}^* + d_{\nabla}^* d_{\nabla}$. On the other hand, $\nabla^*\nabla$ denotes the rough Laplacian, given by $-\sum_{i = 1}^{3}\nabla^2_{e_i, e_i}$ in terms of a local orthonormal frame on $M$, and the same convention is adopted for the usual Laplacian acting on scalar-valued functions. In this notation, given an $\mathfrak{su}(E)$-valued tensor $S$, there holds
\begin{equation}\label{eq: Laplacian_conventions}
\Delta\big( \frac{|S|^2}{2} \big) = - |\nabla S|^2 + \bangle{S, \nabla^*\nabla S}.
\end{equation}
In addition, with subscripts denoting components with respect to a local orthonormal frame, the standard Weitzenb\"ock formulas for $1$-forms and $2$-forms with values in $\fsu(E)$ are given respectively by
\begin{equation}\label{eq: 1form_Weitzenbock}
(\Delta_{\nabla}a)_{i} = (\nabla^*\nabla a)_{i} + [(F_{\nabla})_{ki}, a_k] + \Ric_{ki}a_k,
\end{equation}
for $a \in \Omega^1(\fsu(E))$, and
\begin{equation}\label{eq: 2form_Weitzenbock}
(\Delta_{\nabla}\varphi)_{ij} = (\nabla^*\nabla \varphi)_{ij} + [F_{ki}, \varphi_{kj}] - [F_{kj}, \varphi_{ki}] + [\cR_2(\varphi)]_{ij}, 
\end{equation}
for $\varphi \in \Omega^2(\fsu(E))$, where $\cR_2:\Lambda^2T^*M \to \Lambda^2 T^*M$ denotes the bundle map given by
\begin{equation}\label{eq:R2-definition}
[\cR_{2}(\varphi)]_{ij} = \varphi_{\Ric(e_i), e_j} + \varphi_{e_i, \Ric(e_j)} + \varphi_{e_k, R_{e_i, e_j}e_k}.
\end{equation}

In deriving estimates, we allow constants such as ``$C$'' or ``$a$'' to change from line to line, unless otherwise stated. In addition, we use subscripts when we want to emphasize the dependence of a constant on other parameters. To shorten the statement of certain inequalities, we sometimes write ``$A \lesssim_{c_1, c_2, \cdots} B$'' to mean $A \leqslant C B$ with $C$ depending on $c_1, c_2, \cdots$. When we have both $A \lesssim_{c_1, c_2, \cdots} B$ and $B \lesssim_{c_1, c_2, \cdots }A $, we write ``$A \sim_{c_1, c_2, \cdots}B$''.

Finally, given $s\in [0,3]$, for any subset $S\subset (M^3,g)$ we write $\cH^s(S)$ for its Hausdorff $s$-dimensional measure with respect to the metric induced by $g$, normalized so that $\cH^3=\vol_g$ as measures. With no risk of confusion with the later notation, we sometimes write $\mathscr{H}^k(M)$ for the space of harmonic $k$-forms on $(M^3,g)$. That is,
\[
\mathscr{H}^k(M) = \{h \in \Omega^k(M):dh = 0,\quad d_g^{*}h = 0\}.
\]
\subsection{Organization}
Section~\ref{sec:existence} is devoted to the proof of Theorem~\ref{thm: existence}. We begin by recalling some well-known analytical properties of $\cY_{\ep}$, including the first variation formula of $\cY_{\ep}$ and the Palais--Smale condition up to change of gauge. Then we proceed to set up the min-max construction and establish the lower and upper bounds on the min-max values needed to invoke standard theory and conclude that they are critical values of $\cY_{\ep}$.

Section~\ref{sec:estimates} opens with the derivation of Bochner--Weitzenb\"ock formulas from~\eqref{eq: 2nd_order_crit_pt_intro}. Then, Sections \ref{subsec:coarse-estimates}--\ref{subsec:improved-estimates} consist essentially of a series of inductive arguments, whereby a priori estimates are obtained on solutions of~\eqref{eq: 2nd_order_crit_pt_intro}. The techniques involved are known to experts, and in most cases can be traced back to~\cite{Jaffe-Taubes}. A number of consequences of the estimates that are relevant for later parts of the paper are deduced in Sections \ref{subsec:local-convergence} and \ref{subsec:consequences-estimates}.

We begin Section~\ref{sec:gap} by proving Theorem~\ref{thm: gap}, and deducing from it Theorem~\ref{thm: gap_for_R3}. Then, to complement these gap results, we describe how to obtain examples of non-trivial, reducible solutions $(\nabla, \Phi \not\equiv 0)$ of~\eqref{eq: 2nd_order_crit_pt_intro} when $M$ is closed and $b_2(M) \neq 0$.

In Section~\ref{sec:asymptotic}, we take up the proofs of the last two main theorems. Specifically, we establish the statements in Theorem \ref{thm: asymptotic} throughout Sections \ref{subsec:blow-up}--\ref{subsec:hodge_decomp_longit_part}, while Theorem \ref{thm: bubbling} is proved in the remaining Sections \ref{subsec:rescaling}--\ref{subsec:energy_identity}. The analysis involved in the proofs of both theorems relies heavily on the main a priori estimates 
obtained in Section~\ref{sec:estimates}. Moreover, the gap result of Theorem \ref{thm: gap_for_R3}, proved in Section~\ref{sec:gap}, is also used in the bubbling analysis leading to the energy and charge identities for $\Theta(x)$ and $\Xi(x)$, respectively. 

Several standard facts which are invoked multiple times throughout the paper are gathered in the appendices for the sake of completeness. Appendix~\ref{sec:Coulomb} concerns the issue of finding Coulomb gauges locally in a continuous manner when a family of connections is involved. In Appendix~\ref{sec:Moser-iteration}, we trace the steps in Moser's iteration to state the resulting estimate in a way that suits our purposes in Sections~\ref{sec:estimates} and~\ref{sec:asymptotic}. Appendix~\ref{sec:commute} records a standard estimate on the commutator terms generated when moving $\nabla^*\nabla$ across covariant derivatives. Finally, Appendix~\ref{sec:proofs_derivative_formulas} collects a number of standard but lengthy computations involving repeated differentiation of the Bochner--Weitzenb\"ock formulas obtained from~\eqref{eq: 2nd_order_crit_pt_intro}. 
\vskip 1em
\noindent\textbf{Acknowledgements.} The authors are grateful to Henrique S\'a Earp for introducing them to each other, and thank Saman Esfahani, Gonçalo Oliveira, Daniel Stern, and Mark Stern for insightful conversations on the subject of this paper. In addition, D.C. wishes to thank Chris Scaduto for an inspiring topics course from which he learned a great deal about $SU(2)$ Yang--Mills--Higgs theory. 
D.F. was supported by
the MATH-AmSud project \emph{Symmetries in Geometry and Physics}
(SGP 24-MATH-12), CNPq Universal grant 406666/2023-7, the ANR--FAPESP
project \emph{BRIDGES---Brazil--France Interplays in Gauge Theory,
Extremal Structures and Stability} (ANR-21-CE40-0017 and FAPESP
2021/04065-6), and the University of S\~ao Paulo through its
\emph{Programa de Apoio aos Novos Docentes USP}. L.L. was supported by the São Paulo Research Foundation (FAPESP) [2020/15054-2]. 
\vskip 2mm
\noindent\textbf{AI declaration.} No artificial intelligence (AI) tools were used in the research or in the preparation of the manuscript.
\section{Existence of critical points}\label{sec:existence}
In this section, we assume that $(M^3,g)$ is a \emph{closed} Riemannian $3$-manifold, and we address the problem of existence of non-trivial critical points of $\mathcal{Y}_{\varepsilon}$ on a $SU(2)$-bundle $E$ over $(M^3,g)$. We shall prove that, as long as $\ep$ is sufficiently small depending only on the geometry of $(M,g)$, there always exist critical points within the energy regime $\ep\lesssim_{\lambda}\mathcal{Y}_{\ep}(\nabla,\Phi) \lesssim_{\lambda,M} \ep$, and satisfying furthermore that $\Phi \not \equiv 0$. These critical points are produced by a 3-parameter min-max procedure inspired by a similar construction done by Pigati--Stern \cite{Pigati-Stern2021} in the case of the $U(1)$-version of $\cY_{\ep}$. 

In \S\ref{subsec:variational_properties}, we prepare for the min-max construction by introducing the relevant function space and establishing some basic analytical facts about $\cY_{\ep}$. Much of this material is standard, and when proofs are included, it is only for the reader's convenience. In particular, we derive the first variation formula (Lemma~\ref{lem: first_variation_SU(2)}), and verify the Palais--Smale condition up to change of gauge (Proposition~\ref{prop:PS_up_to_gauge}). The min-max construction is carried out in \S\ref{subsec:min_max_construction}. Upon fixing an identification of $\fsu(2)$ with $\RR^3$, we define the min-max values, denoted $\omega_{\ep}$, using a collection of $3$-parameter families similar to the one used in~\cite[Definition 7.7]{Pigati-Stern2021}. We then establish that $\omega_\ep \sim \ep$ (Propositions~\ref{prop:min-max_lower_bound} and~\ref{prop:min-max_upper_bound}), which permits us to produce critical points of $\cY_{\ep}$ at the level $\omega_{\ep}$, and hence satisfying~\eqref{ineq: right_energy_regime}, by standard arguments (Proposition~\ref{prop:existence_critical_point}). The proof of Theorem~\ref{thm: existence} is completed at the end of \S\ref{subsec:min_max_construction}.

\subsection{Some variational properties of the functional}\label{subsec:variational_properties}
We start by noticing that the $SU(2)$-bundle $E\to M$ is trivializable in our setting. Indeed, since $SU(2)$ is a simply connected Lie group, it must be $2$-connected\footnote{Every connected Lie group $G$ has $\pi_2(G)=\{1\}$; see \cite[p. 183]{stern2010geometry} for an analytic proof (using Yang--Mills theory) of this fact when $G$ is compact.}, and it then follows from obstruction theory that a principal $SU(2)$-bundle over a $3$-manifold must be topologically trivializable. Since our underlying principal $SU(2)$-bundle $P\to M$ is smooth, it must actually be smoothly trivializable (see for instance \cite{muller2009equivalences}), and thus the same is true for the associated vector bundle $E$. Consequently, after choosing a trivialization, we can assume $E=M\times\mathbb{C}^2$, 
so that $\mathfrak{su}(E)$ is the trivial bundle $\underline{\mathfrak{su}(2)}=M\times\mathfrak{su}(2)$. Using the flat connection $d$ as a reference, 
each connection $\nabla\in\mathscr{A}(E)$ then corresponds bijectively to an $\mathfrak{su}(2)$-valued $1$-form $A$ such that $\nabla = d + A$,  and the space $\mathscr{C}(E)$ defined after~\eqref{eq: YMH_energy} becomes
\[
\mathscr{C}(E) = \{ (d + A, \Phi) \ |\ (A, \Phi) \in \Omega^1\oplus\Omega^0(M,\underline{\mathfrak{su}(2)})\},
\]
the integrability requirements being implied by smoothness since $M$ is closed.

In order to find critical points of $\cY_{\ep}$, we shall in fact work with configurations $(d+A,\Phi)$ on $E$ in which the pair $(A,\Phi)$ lies in the Hilbert space $X$ defined as the $W^{1,2}$ Sobolev completion of $\Omega^1\oplus\Omega^0(M,\underline{\mathfrak{su}(2)})$, with norm given by
\[
\|(A, \Phi)\|_{X} := \big(\|A\|_{2; M}^2 + \|DA\|_{2; M}^2 + \|\Phi\|_{2; M}^2 + \|d\Phi\|_{2; M}^2\big)^{\frac{1}{2}}.
\]
Here, to distinguish it from the exterior derivative, we have written $D$ for the connection acting on 
$\Omega^1(M,\underline{\mathfrak{su}(2)})$ induced by the flat connection on $E$ and the Levi--Civita connection of $g$. The $3$-dimensional Sobolev embedding $W^{1,2}\hookrightarrow L^6$ together with H\"older's inequality ensures that if $(A,\Phi)\in X$, then both $\Phi\in L^4$ and $A\in L^4$, and we get
\begin{align*}
    F_{d+A} &= dA + \frac{1}{2}[A,A]\in L^2,\\
    (d+A)\Phi &= d\Phi + [A,\Phi]\in L^2,\quad\text{and}\\
    (1-|\Phi|^2) &\in L^2.
\end{align*} We may therefore consider the Yang--Mills--Higgs energy $\mathcal{Y}_{\varepsilon}$ as a well-defined functional on $X$ given by
\[
\begin{split}
\cY_{\ep}(A, \Phi) =\ & \int_{M} \ep^2 |F_{d + A}|^2 + |(d + A)\Phi|^2 + \frac{\lambda}{4\ep^2}(1 - |\Phi|^2)^2 \vol_g\\
=\ & \int_{M} \ep^2 |dA + \frac{1}{2}[A, A]|^2 + |d\Phi + [A, \Phi]|^2 + \frac{\lambda}{4\ep^2}(1 - |\Phi|^2)^2 \vol_g.
\end{split}
\]
That is, $\cY_{\ep}$ can be thought of as being defined on the space of pairs $(A, \Phi)$ where $A$ is an $\fsu(2)$-valued $1$-form on $M$ of class $W^{1, 2}$, and $\Phi$ is a $W^{1, 2}$ map $M \to \fsu(2)$. 
It is standard to check that $\cY_{\ep}$, understood this way, is a smooth function from the Hilbert space $(X, \|\cdot \|_X)$ to $\RR$. 

Likewise, given $\rg \in W^{2, 2}(M; SU(2))$, we define
\begin{equation}\label{eq:gauge_action_on_X}
\rg \cdot (A, \Phi) := ({\rg}d({\rg}^{-1}) + {\rg}A {\rg}^{-1}, {\rg}\Phi {\rg}^{-1}), \quad \text{ for }(A, \Phi) \in X,
\end{equation}
where the right-hand side again lies in $X$ thanks to the following consequence of Sobolev embedding:
\begin{equation}\label{eq:W22_W12_multiplication}
\|uv\|_{1, 2} \leqslant C_M\|u\|_{1, 2} \|v\|_{2, 2},\quad \text{ for all }u \in W^{1,2}(M),\ v \in W^{2, 2}(M),
\end{equation}
which also implies that the affine map $(A, \Phi) \longmapsto \rg \cdot (A, \Phi)$ from $(X, \|\cdot\|_X)$ to itself is smooth. It is another standard exercise to check that
\begin{equation}\label{eq:gauge_invariance_on_X}
\cY_{\ep}(A, \Phi) = \cY_{\ep}(\rg \cdot (A, \Phi)),\quad \text{ for all } \rg \in W^{2, 2}(M; SU(2)),\ (A, \Phi) \in X.
\end{equation}
To continue, given $(A, \Phi) \in X$, we write $T_{(A, \Phi)}X$ for the space $X$ equipped with the norm
\[
\|(a, \phi)\|_{T_{(A, \Phi)}X} : = \big(\|a\|_{2; M}^2 + \|(D + A)a\|_{2; M}^2 + \|\phi\|_{2; M}^2 + \|(d + A)\phi\|_{2; M}^2\big)^{\frac{1}{2}},
\]
where $D + A$ is induced by the connection $d + A$ on $E$ and the Levi--Civita connection of $g$, and thus acts by 
\[
[(D + A)a]_{e_i, e_j} = (Da)_{e_i, e_j} + [A_{e_i}, a_{e_j}],
\]
in terms of a local orthonormal frame on $M$. Notice that, for $\rg\in W^{2,2}(M,SU(2))$, we have
\begin{equation}\label{eq:TX_norm_invariance}
\|( {\rg}a {\rg}^{-1}, {\rg}\phi {\rg}^{-1})\|_{T_{\rg \cdot (A, \Phi)}X} = \|(a, \phi)\|_{T_{(A, \Phi)}X}.
\end{equation}
The three lemmas that follow collect some basic properties of $X$ that permit us to invoke standard variational tools later this section.
\begin{lemm}\label{lemm:tangent-norm-equivalent}
Given $\Lambda > 0$, there exists a constant $C_{M,\Lambda} > 1$ such that whenever $\|(A, \Phi)\|_{X} \leqslant \Lambda$, we have 
\[
C^{-1}_{M, \Lambda}\|(a, \phi)\|_{T_{(A, \Phi)}X} \leqslant \|(a, \phi)\|_{X} \leqslant C_{M, \Lambda} \|(a, \phi)\|_{T_{(A, \Phi)}X}, \text{ for all }(a, \phi) \in T_{(A, \Phi)}X.
\]
\end{lemm}
\begin{proof}
By H\"older's inequality, the interpolation between $L^{2}$ and $L^6$, Young's inequality, and the Sobolev embedding $W^{1, 2} \hookrightarrow L^6$, we have
\[
\begin{split}
\|[A, a]\|_{2} \leqslant\ & \|A\|_{4} \|a\|_{4} \leqslant \|A\|_{4} \|a\|_{6}^{\frac{3}{4}}\|a\|_{2}^{\frac{1}{4}}\\
\leqslant\ & \|A\|_4 \cdot \big(\ep \|a\|_{6} + C\ep^{-3}\|a\|_2\big)\\
\leqslant\ & C_M\ep \|A\|_4 (\|a\|_{2} + \|Da\|_{2}) + C\ep^{-3}\|A\|_{4} \|a\|_{2}.
\end{split}
\]
Choosing 
\[
\ep = \frac{1}{2(C_M \|A\|_4 + 1)},
\]
we deduce from the above that 
\[
\begin{split}
\|Da\|_{2}\leqslant\ & \|(D + A)a\|_2 + \|[A, a]\|_{2}\\
\leqslant\ & \|(D + A)a\|_{2} + \frac{1}{2} \|Da\|_2 + C_{M}(1 + \|A\|_4)^4\|a\|_{2}\\
\leqslant\ & \|(D + A)a\|_{2} + \frac{1}{2} \|Da\|_2 + C_{M,\Lambda}\|a\|_{2},
\end{split}
\]
where in the last step we used
\[
\|A\|_{4} \leqslant C_M (\|A\|_{2} + \|DA\|_{2}) \leqslant C_{M}'\Lambda.
\]
Taking the two ends of the above string of inequalities, absorbing $\frac{1}{2}\|Da\|_2$ to the left-hand side, and adding $\|a\|_2$ to both sides, we obtain
\[
\|a\|_{2} + \|Da\|_{2} \leqslant C_{M, \Lambda}\|a\|_2 +  2\|(D + A)a\|_2.
\]
A bound to the reverse effect is much simpler to obtain. Indeed, by the triangle inequality, H\"older's inequality and Sobolev embedding, we have 
\[
\begin{split}
\|(D + A)a\|_2 \leqslant\ & \|Da\|_2 + \|A\|_4 \|a\|_4 \leqslant \|Da\|_2 +C_M \|A\|_4 \cdot(\|a\|_2 + \|Da\|_2) \\
\leqslant\ & C_{M, \Lambda}(\|a\|_{2} + \|Da\|_{2}).
\end{split}
\]
Similar arguments applied to $d\phi$ and $(d + A)\phi$ gives the desired equivalence of norms with the admissible dependence.
\end{proof}

\begin{lemm}\label{lemm:Finsler}
For all $(A_0, \Phi_0)$ and $\delta > 0$, there exists $\rho > 0$ such that whenever $\|(A, \Phi) - (A_0, \Phi_0)\|_{X} < \rho$, there holds
\[
1 - \delta \leqslant \frac{\|(a, \phi)\|_{T_{(A, \Phi)}X}}{\|(a, \phi)\|_{T_{(A_0, \Phi_0)}X}} \leqslant 1 + \delta,
\]
for all $(a, \phi) \in X \setminus \{(0, 0)\}$.
\end{lemm}
\begin{proof}
We first require that $\rho < 1$, so that, by the triangle inequality,
\begin{equation}\label{eq:Finsler_uniform_bound}
\|(A, \Phi)\|_X < \|(A_0, \Phi_0)\|_{X} + 1 =: \Lambda, \text{ for all }(A, \Phi) \in B_{\rho}^{\|\cdot\|_X}((A_0, \Phi_0)).
\end{equation}
Given $(a, \phi) \in X$ with $\|(a, \phi)\|_{T_{(A_0, \Phi_0)}X} = 1$, we have by the triangle inequality and Sobolev embedding that
\[
\begin{split}
\big|\|(a, \phi)\|_{T_{(A_0, \Phi_0)}X} - \|(a, \phi)\|_{T_{(A, \Phi)}X}\big| \leqslant\ & \big(\|[A - A_0, a]\|_{2}^2  + \|[A - A_0, \phi]\|_{2}^2\big)^{\frac{1}{2}}\\
\leqslant\ & \|A - A_0\|_{4} \big( \|a\|_{4}^2 + \|\phi\|_4^2 \big)^{\frac{1}{2}}
\leqslant C_M \rho \|(a, \phi)\|_{X}.
\end{split}
\]
By the bound~\eqref{eq:Finsler_uniform_bound} and Lemma~\ref{lemm:tangent-norm-equivalent}, we have
\[
\|(a, \phi)\|_{X} \leqslant C_{M, \Lambda} \|(a, \phi)\|_{T_{(A_0, \Phi_0)}X} = C_{M,\Lambda}.
\]
Substituting this back above yields 
\[
\big|\|(a, \phi)\|_{T_{(A_0, \Phi_0)}X} - \|(a, \phi)\|_{T_{(A, \Phi)}X}\big|  \leqslant C_{M, \Lambda}\rho.
\]
We get the desired bounds upon decreasing $\rho$ if necessary.
\end{proof}

Thanks to Lemma~\ref{lemm:tangent-norm-equivalent} and Lemma~\ref{lemm:Finsler}, the family of norms $\{\|\cdot\|_{T_{(A, \Phi)}X}\}_{(A, \Phi) \in X}$ defines a \emph{Finsler structure} on the tangent bundle $TX$ of $X$ (see~\cite[Chapter II, \S 3.7]{StrBook}). We can then introduce a distance metric $d_X$ on $X$ by letting
\[
d_X(p_0, p_1) = \inf \int_{0}^{1}\|\gamma'(t)\|_{T_{\gamma(t)}X} dt, \text{ for }p_0, p_1 \in X,
\]
where the infimum is taken over all $C^1$-paths $\gamma:[0, 1] \to (X, \|\cdot\|_X)$ with $\gamma(0) = p_0$ and $\gamma(1) = p_1$. 

\begin{lemm}[\cite{Palais1966}, Theorem 3.3]
\label{lemm:Finsler_topology}
With the above definition, we have:
\vskip 1mm
\begin{enumerate}
\item[(a)] $d_X$ is indeed a distance metric.
\vskip 1mm
\item[(b)] $d_X$ induces the same topology on $X$ as $\|\cdot\|_{X}$.
\vskip 1mm
\item[(c)] $d_X(p, q) = d_X({\rg} \cdot p, {\rg} \cdot q)$ for all ${\rg} \in W^{2, 2}(M; SU(2))$.
\end{enumerate}
\end{lemm}
\begin{proof}
For part (a), clearly we have $d_X(p, q) = d_X(q, p)$, and that $d_X(p, p) = 0$. Next, two $C^1$-paths with a common endpoint can be joined in a $C^1$-manner using cutoff functions as in~\cite[Lemma 3.1]{Palais1966}, from which it is not hard to prove the triangle inequality. It remains to show that $p \neq q$ implies $d_X(p, q) > 0$. Suppose $p_0, p_1 \in X$ are such that 
\[
d_X(p_0, p_1) = 0,
\] let $\Lambda = \|p_0\|_X + 1$, and denote by $C_{M, \Lambda}$ the constant given by Lemma~\ref{lemm:tangent-norm-equivalent}. We claim that whenever there is a $C^1$-path $\gamma:[0, 1] \to X$ from $p_0$ to $p_1$ such that
\begin{equation}\label{eq:p0_p1_short_curve}
L(\gamma) : = \int_{0}^{1}\|\gamma'(t)\|_{T_{\gamma(t)}X} dt < C_{M, \Lambda}^{-1},
\end{equation}
there holds
\begin{equation}\label{eq:norm_by_metric}
\|p_1 - p_0\|_X \leqslant C_{M, \Lambda}L(\gamma). 
\end{equation}
To see this, take any $\delta \in (L(\gamma), C_{M, \Lambda}^{-1})$ and set 
\[
t^* = \sup\{t \in [0, 1]\ |\ \gamma([0, t]) \subset B^{\|\cdot\|_X}_{C_{M, \Lambda}\delta}(p_0)\}.
\]
Assume towards a contradiction that $t^* < 1$. Then since $t \mapsto \|\gamma(t) - p_0\|_X$ is a continuous function, we must have 
\begin{equation}\label{eq:exit_at_boundary}
\|\gamma(t^*) - p_0\|_{X} = C_{M, \Lambda}\delta.
\end{equation}
On the other hand, by the triangle inequality we have 
\[
\|\gamma(t)\|_X \leqslant C_{M, \Lambda}\delta + \|p_0\|_X < \Lambda, \text{ for all }t \in [0, t^*],
\]
and hence, by Lemma~\ref{lemm:tangent-norm-equivalent},
\begin{equation}\label{eq:exit_time}
\int_{0}^{t^*} \|\gamma'(t)\|_{X} dt \leqslant C_{M, \Lambda}\int_{0}^{t^*}\|\gamma'(t)\|_{T_{\gamma(t)}X} dt \leqslant C_{M, \Lambda}L(\gamma) < C_{M, \Lambda}\delta.
\end{equation}
Combining~\eqref{eq:exit_time} and~\eqref{eq:exit_at_boundary} with the fact that 
\[
\|\gamma(t^*) - p_0\|_{X} \leqslant \int_{0}^{t^*} \|\gamma'(t)\|_{X} dt,
\]
we obtain a contradiction. Therefore $t^* = 1$, and hence
\[
\|p_1 - p_0\|_{X} \leqslant C_{M, \Lambda}\delta
\]
by continuity. Since $\delta \in (L(\gamma), C_{M, \Lambda}^{-1})$ is arbitrary, we get~\eqref{eq:norm_by_metric}, as claimed. Recalling the assumption $d_X(p_0, p_1) =0$ and the definition of $d_X$, it follows that $\|p_0 - p_1\|_X = 0$, that is, $p_0 = p_1$.

The proof that~\eqref{eq:p0_p1_short_curve} implies~\eqref{eq:norm_by_metric} actually demonstrates that given any $p \in X$, with $\Lambda = \|p\|_{X} + 1$ there holds
\begin{equation}\label{eq:norm_ball_metric_ball}
B_{\eta}^{d_X}(p) \subset B_{C_{M, \Lambda}\eta}^{\|\cdot\|_X}(p), \text{ whenever }\eta < C_{M, \Lambda}^{-1}.
\end{equation}
Conversely, for all $r < 1$ and $q \in B_{r}^{\|\cdot\|_X}(p)$, by considering the line segment from $p$ to $q$, which lies entirely in the convex set $B_r^{\|\cdot\|_{X}}(p)$, we have
\[
d_X(p, q) \leqslant \int_{0}^{1}\|q - p\|_{T_{tq + (1-t)p}X} \leqslant C_{M, \Lambda}\|q - p\|_{X} < C_{M, \Lambda}r,
\]
and thus
\begin{equation}\label{eq:metric_ball_norm_ball}
B_{r}^{\|\cdot\|_X}(p) \subset B_{C_{M, \Lambda}r}^{d_X}(p).
\end{equation}
The inclusions~\eqref{eq:metric_ball_norm_ball} and~\eqref{eq:norm_ball_metric_ball} implies that $d_X$ and $\|\cdot\|_X$ define the same collection of open sets, and we are done with part (b).

For part (c), given $p, q \in X$, a gauge transformation ${\rg} \in W^{2, 2}(M; SU(2))$, and a $C^1$-path $\gamma = (A, \Phi):[0, 1] \to X$ from $p$ to $q$, using~\eqref{eq:W22_W12_multiplication},  it is not hard to show that $t \longmapsto {\rg}\cdot \gamma(t)$ is still a $C^1$-path, and that in fact
\[
({\rg} \cdot \gamma)'(t) = ({\rg}A'(t){\rg}^{-1}, {\rg}\Phi'(t){\rg}^{-1}).
\]
Thus by~\eqref{eq:TX_norm_invariance} we have
\[
\|({\rg} \cdot \gamma)'(t)\|_{T_{{\rg}\cdot\gamma(t)}X} = \|\gamma'(t)\|_{T_{\gamma(t)}X},
\]
from which it is straightforward to deduce that $d_X(p,q) = d_X({\rg}\cdot p, {\rg}\cdot q)$, as asserted.
\end{proof}

Returning to the main line of discussion, let us compute the first variation of $\cY_{\ep}$. Below, and throughout the rest of Section~\ref{sec:existence}, we abuse notation and write $F_A$ for $F_{d + A}$.
\begin{lemm}[First variation of $\mathcal{Y}_{\varepsilon}$]\label{lem: first_variation_SU(2)}
    Given $(A, \Phi) \in X$ and $(a, \phi) \in X$, the first variation of the Yang--Mills--Higgs energy $\mathcal{Y}_{\varepsilon}:X\to\mathbb{R}$ is given by the formula
\begin{align}
(\delta \cY_{\ep})_{(A, \Phi)}(a, \phi) &= \frac{d}{dt}\Big|_{t = 0}\cY_{\ep}(A + ta, \Phi + t\phi)\nonumber\\
&= 2\ep^2 \bangle{F_{A}, da + [A, a]}_{L^2} + 2\bangle{(d + A)\Phi, (d + A)\phi + [a, \Phi]}_{L^2} \nonumber\\
&\quad\quad + \frac{\lambda}{\ep^2}\bangle{(|\Phi|^2 - 1)\Phi,\phi}_{L^2}.\label{eq:first_var_repeat}
\end{align}
\end{lemm}
\begin{proof}
With the help of the formulas 
\begin{align}
F_{A + ta} &= F_{A} + t(da + [A, a]) + \frac{t^2}{2}[a,a],\quad\text{and}\nonumber\\
(d + A + ta)(\Phi + t\phi) &= (d\Phi + [A, \Phi]) + t (d\phi + [A,\phi] + [a,\Phi])+t^2[a,\phi],\nonumber
\end{align}
together with the $3$-dimensional Sobolev embedding $W^{1,2}\hookrightarrow L^6$, one 
sees that $\|F_{A + ta}\|_{L^2}^2$ and $\|(d + A + ta)(\Phi + t\phi)\|_{L^2}^2$, as well as $\|(1 - |\Phi + t\phi|^2)\|_{L^2}^2$ for that matter, are each a quartic polynomial in $t$, with coefficients being the integral of $L^1$-functions. Differentiating at $t = 0$, we find that
    \begin{align}
        \frac{d}{dt}\Big|_{t=0}\varepsilon^2\|F_{A + ta}\|_{L^2}^2 &= 2 \varepsilon^2 \langle  F_{A},da + [A, a]\rangle_{L^2}, \nonumber\\
        \frac{d}{dt}\Big|_{t=0}\|{(d + A + ta)}(\Phi + t\phi)\|_{L^2}^2 &= 2\langle {(d + A)}\Phi, (d+ A)\phi  + [a,\Phi]\rangle_{L^2},\quad\text{and} \nonumber\\
        \frac{d}{dt}\Big|_{t=0}\frac{\lambda}{4\varepsilon^2}\|(1-|\Phi + t\phi|^2)\|_{L^2}^2 &= -\frac{\lambda}{\varepsilon^2}\langle (1-|\Phi|^2)\Phi,\phi\rangle_{L^2}. \nonumber
    \end{align} Summing the above gives the desired result.
\end{proof}

Given $(A, \Phi) \in X$, the norm of $(\delta\cY_{\ep})_{(A, \Phi)}$ is defined by duality. That is, 
\[
\|(\delta \cY_{\ep})_{(A, \Phi)}\| : = \sup\big\{ (\delta \cY_{\ep})_{(A, \Phi)}(a, \phi)\ \big|\ (a, \phi) \in X,\ \ \|(a, \phi)\|_{T_{(A, \Phi)}X} \leqslant 1 \big\}.
\]
With the help of Lemma~\ref{lemm:tangent-norm-equivalent}, we see that $\|(\delta \cY_{\ep})_{(A, \Phi)}\| < \infty$. The following obvious remark will be useful later.
\begin{rmk}\label{rmk:first_var_gauge_symmetry}
Given any ${\rg} \in W^{2, 2}(M; SU(2))$, we note that 
\[
(\delta \cY_{\ep})_{(A, \Phi)}(a, \phi) = (\delta \cY_{\ep})_{{\rg}\cdot (A, \Phi)}({\rg}a{\rg}^{-1}, {\rg}\phi{\rg}^{-1}).
\]
Recalling also~\eqref{eq:TX_norm_invariance}, we infer that 
\[
\|(\delta \cY_{\ep})_{(A, \Phi)}\| = \|(\delta \cY_{\ep})_{{\rg}\cdot (A, \Phi)}\|.
\]
\end{rmk}
Next, we recall two well-known important properties of $\cY_{\ep}$ in three dimensions, which rely heavily on Uhlenbeck's work \cite{uhlenbeck1982connections}. The first is that critical points are smooth up to change of gauge (Proposition~\ref{prop:smoothness_of_weak_solution}). The second is that the Palais--Smale condition holds, again up to change of gauge (Proposition~\ref{prop:PS_up_to_gauge}). 
\begin{prop}[\cite{Jaffe-Taubes}, Theorem V.2.4]
\label{prop:smoothness_of_weak_solution}
Suppose $(A, \Phi) \in X$ is a critical point of $\cY_{\ep}$. That is, assume that 
\[
(\delta\cY_{\ep})_{(A, \Phi)}(a, \phi) = 0, \text{ for all }(a, \phi) \in X.
\]
Then there exists $\rg \in W^{2, 2}(M; SU(2))$ such that $(\widetilde{A}, \widetilde{\Phi}): = \rg \cdot (A, \Phi)$ is smooth, and consequently $(d + \widetilde{A}, \widetilde{\Phi})$ satisfies~\eqref{eq: 2nd_order_crit_pt_intro} in the classical sense.
\end{prop}
\begin{proof}
See Chapter V of~\cite{Jaffe-Taubes}, especially Theorem V.1.1 and Theorem V.2.4. 
\end{proof}

\begin{prop}
\label{prop:PS_up_to_gauge}
Fix $\ep > 0$. Let $(A_i, \Phi_i)$ be a sequence in $X$ such that $\cY_\ep(A_i, \Phi_i) \leqslant \Lambda$ for all $i$, and that 
\begin{equation}\label{eq:PS_vanishing}
\lim_{i \to \infty}\|(\delta\cY_{\ep})_{(A_i, \Phi_i)}\| = 0.
\end{equation}
Then, up to taking a subsequence, there exist $\rg_i \in W^{2, 2}(M; SU(2))$ such that $\rg_i \cdot (A_i, \Phi_i)$ converges strongly in $W^{1,  2}$ to a critical point $(A, \Phi)$ of $\cY_\ep$.
\end{prop}
\begin{proof}
Although this result should be known by experts in the field, and can be inferred from the proofs of similar statements such as~\cite[Theorem 5.6]{Taubes1982I},~\cite[Proposition 4.5]{Taubes-JDG1984} or~\cite[Theorem 2.3]{Parker-Duke1992}, we did not find a reference proving this exact version in the literature, thus we include a proof here. The argument differs in no essential way from the references just mentioned, and is in fact simpler as we are working over a closed $3$-manifold. Since both $\ep$ and $\lambda$ are fixed, we only consider the case $\ep = \lambda = 1$. The remaining cases require only change of notation. 

To begin, since $M$ is closed there exists $\rho = \rho_M \in (0, \inj(M))$ such that for all $x \in M$ we have on $B_{\rho}(0)$ that 
\begin{equation}\label{eq:PS_metric_bound}
\frac{1}{2}g_{\RR^3} \leqslant \exp_{x}^*g \leqslant 2g_{\RR^3},\ \ |\partial_k(\exp_{x}^*g)_{ij}| \leqslant 1.
\end{equation}
For all $x \in M$, denoting still by $(A_i, \Phi_i)$ their pullbacks to $B_{\rho}(0) \subset \RR^3$ by the exponential map, we find by H\"older's inequality and~\eqref{eq:PS_metric_bound} that
\[
\begin{split}
\int_{B_r(0)} |F_{A_i}|^{\frac{3}{2}} \leqslant\ & C r^{\frac{3}{4}}\big( \int_{B_r(0)}|F_{A_i}|^2 \big)^{\frac{3}{4}} \leqslant C' r^{\frac{3}{4}}\big(\int_{B_r(x)} |F_{A_i}|_g^2 \vol_g \big)^{\frac{3}{4}}.
\end{split}
\]
Taking the $\frac{2}{3}$-th power of each term and using the uniform energy bound gives
\[
\|F_{A_i}\|_{\frac{3}{2}; B_r(0)} + r^{\frac{1}{2}} \|F_{A_i}\|_{2; B_{r}(0)} \leqslant Cr^{\frac{1}{2}}\Lambda^{\frac{1}{2}}.
\]
Now pick $r_0 \leqslant \rho_M$ such that 
\[
Cr_0^{\frac{1}{2}}\Lambda^{\frac{1}{2}} < \ep_{\text{gauge}},
\]
where $\ep_{\text{gauge}}$ is the threshold in Proposition~\ref{prop:continuous_change}, and choose a finite cover of $M$ consisting of $B_{\frac{r_0}{2}}(x_1), \cdots, B_{\frac{r_0}{2}}(x_L)$. Then, given $\alpha \in \{1, \cdots, L\}$, there exists for each $i \in \NN$ some $\mathrm{g}^\alpha_i \in W^{2, 2}(B_{r_0}(x_{\alpha}); SU(2))$ such that upon defining
\[
\widetilde{A}_i^\alpha =  \mathrm{g}^\alpha_i d(\mathrm{g}^\alpha_i)^{-1} +  \mathrm{g}^\alpha_i A_i  (\mathrm{g}^\alpha_i)^{-1},\ \ \widetilde{\Phi}_i^\alpha = \mathrm{g}^\alpha_i \Phi_i  (\mathrm{g}^\alpha_i)^{-1},
\]
and also identifying $\rg_{i}^{\alpha}$ and $(\widetilde{A}_i^{\alpha}, \widetilde{\Phi}_{i}^{\alpha})$ with their pullbacks to $B_{r_0}(0)$ by $\exp_{x_\alpha}$, we have that $\widetilde{A}_i^\alpha$ satisfies on $B_{r_0}(0)$ a suitably rescaled version of the condition (U) from Appendix~\ref{sec:Coulomb}. In particular 
\[
\|\widetilde{A}_i^{\alpha}\|_{1, 2; B_{r_0}(0)} \leqslant C\|F_{A_i}\|_{2; B_{r_0}(0)} \leqslant C\Lambda^{\frac{1}{2}},
\]
which is independent of $i$ and $\alpha$. (Here and below, since $r_0$ is fixed, we do not mark explicitly the dependence of constants on $r_0$.) On the other hand, using the potential term in the functional, we see with the help of Young's inequality that
\[
\begin{split}
C\Lambda \geqslant\ &  C\int_{B_{r_0}(0)} (1 - |\Phi_i |^2)^2 \vol_{\exp_{x_\alpha}^*g} \geqslant \int_{B_{r_0}(0)} (1 - |\Phi_i |^2)^2\\
=\ & \int_{B_{r_0}(0)} (1 - |\widetilde{\Phi}_i^{\alpha}|^2)^2 = \int_{B_{r_0}(0)} (|\widetilde{\Phi}_i^{\alpha}|^4 - 2|\widetilde{\Phi}_i^{\alpha}|^2 + 1) \geqslant \int_{B_{r_0}(0)} \frac{|\widetilde{\Phi}_i^{\alpha}|^4}{2} - 1,
\end{split}
\]
and hence 
\[
\|\widetilde{\Phi}_i^{\alpha}\|_{4; B_{r_0}(0)} \leqslant C_{\Lambda}.
\]
Also, from the gradient term we have
\[
\begin{split}
\int_{B_{r_0}(0)} |d\widetilde{\Phi}_i^{\alpha}|^2 \leqslant\ & 2\int_{B_{r_0}(0)}|(d + \widetilde{A}_i^{\alpha})\widetilde{\Phi}_i^{\alpha}|^2 + 2\int_{B_{r_0}(0)} |\widetilde{A}_i^{\alpha}|^2 |\widetilde{\Phi}_i^{\alpha}|^2\\
\leqslant\ & C\int_{B_{r_0}(0)} |(d + \widetilde{A}_i^{\alpha})\widetilde{\Phi}_i^{\alpha}|_{\exp_{x_a}^*g}^2 \vol_{\exp_{x_a}^*g} + C\|\widetilde{A}_i^{\alpha}\|_{4; B_{r_0}(0)}^2 \|\widetilde{\Phi}_i^{\alpha}\|_{4; B_{r_0}(0)}^2\\
\leqslant\ & C\int_{M}|(d + A_i)\Phi_i|^2_g \vol_g +  C_{\Lambda} \leqslant C_{\Lambda}',
\end{split}
\]
where we used the fact that 
\[
(d + \widetilde{A}_i^{\alpha})\widetilde{\Phi}_i^{\alpha} = \mathrm{g}^\alpha_i \cdot (d + A_i)\Phi_i \cdot (\mathrm{g}^\alpha_i)^{-1},
\]
the $L^4$-bound on $\widetilde{\Phi}^{\alpha}_i$ and $W^{1, 2}$-bound on $\widetilde{A}_{i}^{\alpha}$ just established, and Sobolev embedding. To summarize, up to now we have shown that 
\begin{equation}\label{eq:PS_uniform_bound}
\|\widetilde{A}_{i}^{\alpha}\|_{1, 2; B_{r_0}(0)} + \|\widetilde{\Phi}_i^{\alpha}\|_{1, 2; B_{r_0}(0)} \leqslant C_{\Lambda}.
\end{equation}
Since we are in dimension $3$ and the covering is finite, up to taking successive subsequences we can assume that, for each $\alpha \in \{1, \cdots, L\}$ the sequence $(\widetilde{A}_i^{\alpha}, \widetilde{\Phi}_i^{\alpha})$ converges weakly in $W^{1, 2}(B_{r_0}(0))$ and strongly in $L^4(B_{r_0}(0))$ as $i \to \infty$. Below we use the assumption~\eqref{eq:PS_vanishing} to upgrade this to strong $W^{1, 2}$-convergence. Since the argument is the same for each $\alpha$, we only consider $\alpha = 1$ and drop the superscript ``$\alpha$'' from the notation. 

Let $\zeta$ be a cut-off function such that $\zeta = 1$ on $B_{\frac{r_0}{2}}(0)$ while $\zeta = 0$ outside of $B_{\frac{3r_0}{4}}(0)$. Also, for $i, j$, we define 
\[
a = a_{i, j} =  \zeta^2 (\widetilde{A}_i - \widetilde{A}_j),\ \ v = v_{i, j} = \zeta^2 (\widetilde{\Phi}_i - \widetilde{\Phi}_j).
\]
Using~\eqref{eq:PS_uniform_bound} and arguing as in the proof of Lemma~\ref{lemm:tangent-norm-equivalent}, we see that for all $i$, $j$, and $k$, there holds
\[
\|a_{i, j}\|_{2; B_{r_0}} + \|(d + \widetilde{A}_k)a_{i, j}\|_{2; B_{r_0}} + \|v_{i, j}\|_{2; B_{r_0}} + \|(d + \widetilde{A}_k)v_{i, j}\|_{2; B_{r_0}} \leqslant C_{\Lambda},
\]
from which we deduce upon recalling the definition of $\widetilde{A}_k$ that
\begin{equation}\label{eq:PS_test_function_bound_in_charts}
\begin{split}
\|\mathrm{g}_k^{-1} a_{i, j}\mathrm{g}_{k}\|_{2; B_{r_0}} +\ & \|(d + A_k)(\mathrm{g}_k^{-1} a_{i, j}\mathrm{g}_k)\|_{2; B_{r_0}}\\
&+ \|\mathrm{g}_k^{-1} v_{i, j} \mathrm{g}_k\|_{2; B_{r_0}} + \|(d + A_k)(\mathrm{g}_k^{-1} v_{i, j}\mathrm{g}_k)\|_{2; B_{r_0}} \leqslant C_{\Lambda}.
\end{split}
\end{equation}
Now, thanks to the cutting off, the pullbacks of $(\mathrm{g}_k^{-1} a_{i, j}\mathrm{g}_k , \mathrm{g}_k^{-1} v_{i, j}\mathrm{g}_k )$ via $\exp_{x_1}^{-1}$ to $B_{r_0}(x_1)$, which we denote with the same letters, extend to all of $M$ and give elements of $X$. Further, by~\eqref{eq:PS_test_function_bound_in_charts} and the bounds~\eqref{eq:PS_metric_bound} on the metric, we get
\[
\|(\mathrm{g}_k^{-1} a_{i, j}\mathrm{g}_k, \mathrm{g}_k^{-1} v_{i, j}\mathrm{g}_k)\|_{T_{(A_k, \Phi_k)}X} \leqslant C_{M, \Lambda}.
\]
Recalling the assumption~\eqref{eq:PS_vanishing}, and then using the first variation formula~\eqref{eq:first_var_repeat} together with the gauge invariance properties of $\delta\cY_{\ep}$ in Remark~\ref{rmk:first_var_gauge_symmetry}, we see upon taking $k = i$ that the following integrals converge to $0$ as $i, j \to \infty$:
\[
\begin{split}
I_{i, j} : =\ & \int_{B_{r_0}(0)}\bangle{F_{\widetilde{A}_i}, da_{i, j} + [\widetilde{A}_i, a_{i, j}]}_g + \bangle{(d + \widetilde{A}_i)\widetilde{\Phi}_i, (d + \widetilde{A}_i)v_{i, j} + [a_{i, j}, \widetilde{\Phi}_i]}_g  \vol_g \\
& + \int_{B_{r_0}(0)}\frac{|\widetilde{\Phi}_{i}|^2 - 1}{2}\bangle{\widetilde{\Phi}_i, v_{i, j}} \vol_g.
\end{split}
\]
Taking instead $k = j$, we see that if we define $J_{i, j}$ to be the same integral as above with $(\widetilde{A}_i, \widetilde{\Phi}_i)$ replaced by $(\widetilde{A}_j, \widetilde{\Phi}_j)$, then 
\[
\lim_{i, j \to \infty}J_{i, j} = 0.
\]
By a direct computation we have
\[
I_{i, j} - J_{i, j} = P_{1, i, j} + P_{2, i, j} + Q_{1, i, j} + Q_{2, i, j} + R_{i, j},
\]
where
\[
\begin{split}
P_{1, i, j} =\ & \int_{B_{r_0}(0)} \bangle{F_{\widetilde{A}_i} - F_{\widetilde{A}_j}, \zeta^2(d\widetilde{A}_i - d\widetilde{A}_j) + 2\zeta d\zeta \wedge (\widetilde{A}_i - \widetilde{A}_j) }_g \vol_g,\\
P_{2, i, j} =\ & -\int_{B_{r_0}(0)} \zeta^2\bangle{ [\widetilde{A}_i, F_{\widetilde{A}_i}]- [\widetilde{A}_j, F_{\widetilde{A}_j}], \widetilde{A}_i - \widetilde{A}_j}_g \vol_g\\& - \int_{B_{r_0}(x_0)} \zeta^2 \bangle{ 
[(d + \widetilde{A}_i)\widetilde{\Phi}_i, \widetilde{\Phi}_i] - [(d + \widetilde{A}_j)\widetilde{\Phi}_j, \widetilde{\Phi}_j], \widetilde{A}_i - \widetilde{A}_j }_{g}\vol_{g}\\
Q_{1, i, j} =\ & \int_{B_{r_0}(0)} \bangle{(d + \widetilde{A}_i)\widetilde{\Phi}_i - (d + \widetilde{A}_j)\widetilde{\Phi}_j, \zeta^2(d\widetilde{\Phi}_i - d\widetilde{\Phi}_j) + 2\zeta d\zeta\otimes (\widetilde{\Phi}_i - \widetilde{\Phi}_j)}_g \vol_g\\
Q_{2, i, j} =\ & -\int_{B_{r_0}(0)} \zeta^2\bangle{ [\widetilde{A}_i, (d + \widetilde{A}_i)\widetilde{\Phi}_i] - [\widetilde{A}_j, (d + \widetilde{A}_j)\widetilde{\Phi}_j], \widetilde{\Phi}_i - \widetilde{\Phi}_j }_{g} \vol_{g}\\
R_{i, j} =\ & \frac{1}{2}\int_{B_{r_0}(0)} \zeta^2\bangle{(|\widetilde{\Phi}_{i}|^2 - 1)\widetilde{\Phi}_i - (|\widetilde{\Phi}_{j}|^2 - 1)\widetilde{\Phi}_j, \widetilde{\Phi}_i - \widetilde{\Phi}_j} \vol_g.
\end{split}
\]
Using the fact that the sequence $(\widetilde{A}_k, \widetilde{\Phi}_k)$ is bounded in $W^{1, 2}$ and converges strongly in $L^4$, we find after a straightforward computation that
\[
\int_{B_{r_0}(0)} \zeta^2 |d\widetilde{A}_i -d\widetilde{A}_j|_g^2 \vol_g + \int_{B_{r_0}(0)} \zeta^2 |d\widetilde{\Phi}_i - d\widetilde{\Phi}_j|^2_{g} \vol_g \leqslant (I_{i, j} - J_{i, j}) + o_{i, j}(1).
\]
Recalling~\eqref{eq:PS_metric_bound} and the fact that both $I_{i, j}$ and $J_{i, j}$ tend to $0$ as $i, j \to \infty$, we get
\[
\int_{B_{r_0}(0)}\zeta^2 \big( |d\widetilde{A}_i - d\widetilde{A}_j|^2 + |d\widetilde{\Phi}_i - d\widetilde{\Phi}_j|^2 \big) \to 0 \text{ as }i, j \to \infty.
\]
Using again the strong $L^4$-convergence of the sequence, and also recalling that $d^*\widetilde{A}_k = 0$ on $B_{r_0}(0)$ for all $k$, we deduce further that 
\[
\lim_{i, j \to \infty}\big( \|d(\zeta \widetilde{A}_i - \zeta \widetilde{A}_j)\|_{2; B_{r_0}(0)}  + \|d^*(\zeta \widetilde{A}_i - \zeta \widetilde{A}_j)\|_{2; B_{r_0}(0)}  + \|d(\zeta \widetilde{\Phi}_i - \zeta \widetilde{\Phi}_j)\|_{2; B_{r_0}(0)}\big) = 0.
\]
This proves that $(\widetilde{A}_k, \widetilde{\Phi}_k)$ converges strongly in $W^{1, 2}$ on $B_{\frac{r_0}{2}}(0)$. Repeating this argument shows that $\rg_{i}^{\alpha} \cdot (A_{i}, \Phi_{i})$ converges strongly in $W^{1, 2}(B_{\frac{r_0}{2}}(x_{\alpha}))$ for each $\alpha = 1, \cdots, L$, and it is standard to deduce that, after passing to a further subsequence if needed, we have
\[
\rg_{i}^{\alpha} (\rg_{i}^{\beta})^{-1} \text{ converges strongly in } (C^0 \cap W^{2, 2})(B_{\frac{r_0}{2}}(x_{\alpha}) \cap B_{\frac{r_0}{2}}(x_{\beta})),
\]
whenever $B_{\frac{r_0}{2}}(x_{\alpha}) \cap B_{\frac{r_0}{2}}(x_{\beta}) \neq \emptyset$. We may then follow the patching argument in, for instance, \cite[Lemma 7.2]{wehrheim2004uhlenbeck}, to obtain $\mathrm{g}_i \in W^{2, 2}(M; SU(2))$ such that $\mathrm{g}_i \cdot (A_i, \Phi_i)$ converges strongly in $W^{1, 2}$ on $M$. By the first variation formula~\eqref{eq:first_var_repeat}, the assumption~\eqref{eq:PS_vanishing}, and Lemma~\ref{lemm:Finsler}, the limit must be a critical point.
\end{proof}
\subsection{Min-max construction of critical points}\label{subsec:min_max_construction}
We next define the admissible families of configurations for use in the min-max construction. Specifically, let $B^3$ be the closed unit ball in $\RR^3$ and define
\[
\Gamma = \{H \in C^0(B^3; X)\ |\ H(y) = (0, y), \text{ for all }y \in \partial B^3\},
\]
where in viewing $(0, y)$ as an element of $X$ we have fixed an identification of $\fsu(2)$ with $\RR^3$, and understood $y$ as a constant function from $M$ to $\mathfrak{su}(2)$. Then $\Gamma$ is a non-empty collection since $H: y \mapsto (0, y)$ lies in it, and we may therefore define
\[
\omega_\ep = \inf_{H \in \Gamma}\big[ \sup_{y \in B^3}\cY_{\ep}(H(y)) \big]. 
\]
\begin{prop}\label{prop:min-max_lower_bound}
For $\ep$ sufficiently small depending on $M$, we have
\[
\omega_{\ep} \gtrsim \min\{1,\lambda\}\cdot\ep.
\]
\end{prop}
\begin{proof}
As in the beginning of the proof of Proposition~\ref{prop:PS_up_to_gauge}, since $M$ is closed there exists $\rho = \rho_M \in (0, \inj(M))$ such that for all $x \in M$, we have on $B_{\rho}(0) \subset \RR^3$ that 
\begin{equation}\label{eq:min-max_lower_metric_compare}
\frac{1}{2}g_{\RR^3} \leqslant \exp_{x}^*g \leqslant 2g_{\RR^3},\ \ |\partial_k(\exp_{x}^*g)_{ij}| \leqslant 1.
\end{equation}
Below we require that 
\[
\ep < \rho_M.
\]
To prove the asserted lower bound on $\omega_\ep$, suppose $H \in \Gamma$ is such that 
\begin{equation}\label{eq:small_sweepout}
\sup_{y \in B^3}\cY_{\ep}(H(y)) < \ep.
\end{equation}
(If there are no such $H$ in $\Gamma$ then we are done.) In particular, writing $H(y)$ as $(A_y, \Phi_y)$, we have
\[
\int_{M}|F_{A_y}|_g^2\vol_g < \frac{1}{\ep} \text{ for all }y \in B^3,
\]
so that for all $x \in M$, $r \leqslant \rho$ and $y \in B^3$ we have by~\eqref{eq:min-max_lower_metric_compare} and H\"older's inequality that
\begin{equation}\label{eq:omega_lowerbound_F_estimate}
\begin{split}
\int_{B_r(0)} |F_{A_y}|^{\frac{3}{2}} \leqslant \ &  (\vol_{g_{\RR^3}}(B_r))^{\frac{1}{4}} \big( \int_{B_r(0)} |F_{A_y}|^2 \big)^{\frac{3}{4}} \leqslant C_1 r^{\frac{3}{4}}\big(\int_{B_r(0)} |F_{A_y}|_{\exp_x^* g}^{2} \vol_{\exp_x^*g}\big)^{\frac{3}{4}} \\
=\ &  
C_1 r^{\frac{3}{4}}\big(\int_{B_r(x)} |F_{A_y}|^{2}_g \vol_g \big)^{\frac{3}{4}} \leqslant C_1\big(\frac{r}{\ep} \big)^{\frac{3}{4}}.
\end{split}
\end{equation}
Here, for brevity we have written $A_y$ for $\exp_x^*A_y$. Fixing for the rest of the proof some $x_0 \in M$, and also letting $r = \theta \cdot \ep$ with
\[
0< \theta \leqslant \min\{\frac{1}{2}, \big(\frac{\ep_{\text{gauge}}}{2C_1}\big)^{\frac{4}{3}}\},
\]
so that in particular $r < \rho_M$, we see from~\eqref{eq:omega_lowerbound_F_estimate} and the scaling-invariance of the $L^{\frac{3}{2}}$-norm of the curvature that Proposition~\ref{prop:gauge_fixing_family} is applicable to the rescaled connections $\{r A_y(r\  \cdot)\}_{y \in B^3}$, giving us, upon scaling back, a continuous map 
\[
\mathrm{g}: B^3 \to W^{2, 2}(B_r(0); SU(2))
\]
such that $\mathrm{g}(y)\equiv \id$ for all $y\in \partial B^3$ and that $\widetilde{A}_y: = \mathrm{g}(y) \cdot (A_y|_{B_r(0)})$ satisfies condition (U), suitably scaled. Consequently, there holds
\begin{equation}\label{eq:Coulomb_estimate_scaled}
r^{-1}\|\widetilde{A}_y\|_{\frac{3}{2}; B_r(0)} + \|D \widetilde{A}_y\|_{\frac{3}{2}; B_r(0)} \leqslant 2C_{\text{hodge}} \|F_{A_y}\|_{\frac{3}{2}; B_r(0)} < C_2 \theta^{\frac{1}{2}}.
\end{equation}
Moreover, letting also
\[
\widetilde{\Phi}_y = \mathrm{g}(y) \cdot (\Phi_y|_{B_r(0)} ) \cdot \mathrm{g}(y)^{-1},
\]
where as above we have written $\Phi_y$ for $\exp_{x_0}^*\Phi_y$, we see that $\widetilde{\Phi}_y \equiv y$ on $B_r(0)$ for all $y \in \partial B^3$. Now if it happened that
\[
\fint_{B_r(0)}\widetilde{\Phi}_y \neq 0, \text{ for all }y \in B^3,
\]
then letting
\[
h(y) = \frac{\fint_{B_r(0)}\widetilde{\Phi}_y}{\Big| \fint_{B_r(0)}\widetilde{\Phi}_y \Big|},\ \ y \in B^3
\]
defines a retraction of $B^3$ onto $S^2$, which is a contradiction. Therefore there exists $y_0 \in B^3$  such that $\fint_{B_r(0)}\widetilde{\Phi}_{y_0} = 0$. Below we fix this $y_0$ and write $(\widetilde{A}, \widetilde{\Phi})$ for $(\widetilde{A}_{y_0}, \widetilde{\Phi}_{y_0})$. To compare $(d + \widetilde{A})\widetilde{\Phi}$ with $d\widetilde{\Phi}$, we observe that by H\"older's inequality, as well as suitably scaled versions of the Sobolev inequalities $W^{1, \frac{3}{2}} \to L^3$ and $W^{1, 2} \to L^6$, and the usual Poincar\'e inequality for $W^{1, 2}$-functions with zero average, we have
\[
\begin{split}
\int_{B_r(0)}|[\widetilde{A}, \widetilde{\Phi}]|^2 \leqslant\ & \int_{B_r(0)} |\widetilde{A}|^2 |\widetilde{\Phi}|^2 \leqslant \|\widetilde{A}\|_{3; B_r(0)}^{2} \|\widetilde{\Phi}\|_{6; B_r(0)}^2\\
\leqslant\ & C\big( r^{-1}\|\widetilde{A}\|_{\frac{3}{2}; B_r(0)} + \|D\widetilde{A}\|_{\frac{3}{2}; B_r(0)} \big)^2 \big(\int_{B_r(0)}r^{-2}|\widetilde{\Phi}|^2 + |d\widetilde{\Phi}|^2\big)\\
\leqslant\ & C'(C_2 \theta^{\frac{1}{2}})^2 \int_{B_r(0)} |d\widetilde{\Phi}|^2 = : C_3 \theta \int_{B_r(0)}|d\widetilde{\Phi}|^2,
\end{split}
\]
where for the last inequality we used~\eqref{eq:Coulomb_estimate_scaled}. Requiring further that 
\[
\theta \leqslant \frac{1}{4(C_3 + 1)},
\]
we infer that
\begin{equation}\label{eq:min-max_lower_bound_Neumann_eigenvalue}
\begin{split}
\int_{B_r(0)} |(d + \widetilde{A})\widetilde{\Phi}|^2 \geqslant\ & \frac{1}{2}\int_{B_r(0)}|d\widetilde{\Phi}|^2 - \int_{B_r(0)} |[\widetilde{A}, \widetilde{\Phi}]|^2 \\
\geqslant\ & \frac{1}{4}\int_{B_r(0)}|d\widetilde{\Phi}|^2 \geqslant \frac{c_N}{4r^2} \int_{B_r(0)}|\widetilde{\Phi}|^2,
\end{split}
\end{equation}
where in the last step we used again the Poincar\'e inequality for $W^{1, 2}$-functions with zero average, and $c_N$ stands for the lowest positive Neumann eigenvalue of the standard Laplacian on $B_1(0) \subset \RR^3$. Adding the potential term and recalling that $r = \theta \ep$ gives
\[
\int_{B_r(0)} |(d + \widetilde{A})\widetilde{\Phi}|^2 + \frac{\lambda}{4\ep^2}(1 - |\widetilde{\Phi}|^2)^2 \geqslant \frac{1}{4\ep^2}\int_{B_r(0)} \frac{c_N}{\theta^2}|\widetilde{\Phi}|^2 + \lambda(1 - |\widetilde{\Phi}|^2)^2.
\]
Further decreasing $\theta$, if necessary, so that
\[
2\theta^2 \leqslant c_N,
\]
and also using the obvious estimate $2t + (1-t)^2\geqslant 1$, we arrive at 
\begin{equation}\label{eq:min-max_lower_bound_almost}
\begin{split}
\int_{B_r(0)} |(d + \widetilde{A})\widetilde{\Phi}|^2 + \frac{\lambda}{4\ep^2}(1 - |\widetilde{\Phi}|^2)^2 
\geqslant\ & \frac{\min\{1, \lambda\}}{4\ep^2}\int_{B_r(0)} 2|\widetilde{\Phi}|^2 + (1 - |\widetilde{\Phi}|^2)^2\\
\geqslant\ & \frac{\min\{1, \lambda\}}{4\ep^2}\cdot \mathrm{vol}_{g_{\mathbb{R}^3}}(B_r)\\
=\ & c' \cdot \frac{\min\{1, \lambda\}}{\ep^2} (\theta \ep)^3 = c' \cdot \min\{1, \lambda\}\cdot \theta^3\ep,
\end{split}
\end{equation}
To continue, recall the metric comparison~\eqref{eq:min-max_lower_metric_compare}, and the basic fact that, on $B_r(0)$, we have
\[
d\widetilde{\Phi} + [\widetilde{A}, \widetilde{\Phi}] = \mathrm{g}(y) \cdot (d + A_{y_0})\Phi_{y_0}\cdot \mathrm{g}(y)^{-1}.
\]
It follows that
\begin{equation}\label{eq:min-max_lower-bound_almost_3}
\begin{split}
\cY_{\ep}(A_{y_0}, \Phi_{y_0}; B_r(x_0))\geqslant \ &\int_{B_r(x_0)} \big(|(d + A_{y_0})\Phi_{y_0}|_g^2 + \frac{\lambda}{4\ep^2}(1 - |\Phi_{y_0}|^2)^2 \big)\vol_g \\
= \ & \int_{B_r(0)}\big( |(d + A_{y_0})\Phi_{y_0}|_{\exp_{x_0}^* g}^2 + \frac{\lambda}{4\ep^2}(1 - |\Phi_{y_0}|^2)^2 \big) \vol_{\exp_{x_0}^*g}\\
\geqslant\ & c\int_{B_r(0)} \big( |(d + \widetilde{A})\widetilde{\Phi}|_{g_{\RR^3}}^2 + \frac{\lambda}{4\ep^2}(1 - |\widetilde{\Phi}|^2)^2 \big)\vol_{g_{\RR^3}}. 
\end{split}
\end{equation}
Therefore, combining the estimates~\eqref{eq:min-max_lower_bound_almost} and~\eqref{eq:min-max_lower-bound_almost_3}, we have shown that
\[
\sup_{y \in B^3}\cY_{\ep}(H(y)) \geqslant \min\{1,\lambda\}
    \cdot cc'\theta^3 \ep,
\] whenever $H \in \Gamma$ satisfies~\eqref{eq:small_sweepout}. In other words, 
\begin{equation}\label{ineq: lower_bound_omega_ep}
\omega_{\ep} \geqslant \min\left\{\ep, \min\{1,\lambda\}\cdot cc'\theta^3\ep\right\}  \geqslant \min\left\{1, cc'\theta^3\right\}\cdot\min\{1,\lambda\}\cdot \ep.
\end{equation} This completes the proof. 
\end{proof} 
We now proceed to establish an upper bound for $\omega_\ep$ in terms of $\ep$ (and $\lambda$).  We adopt much of the notation and basically follow the arguments in Section 3 of~\cite{Stern2021} and Section 7 of~\cite{Pigati-Stern2021}, making small changes here and there as needed, since we are working with $SU(2)$ as opposed to $U(1)$. First we recall the following construction from Stern's work on the Ginzburg--Landau equations~\cite{Stern2021}. 
\begin{lemm}\label{lem: stern_lipschitz_map}
There exists a Lipschitz function $f: M \to \RR^3 \simeq \fsu(2)$ and some constant $C > 0$ such that 
\begin{equation}\label{eq:min-max_upper_preimage}
\cH^{0}(f^{-1}(y)) \leqslant C \text{ for all }y \in \RR^3,
\end{equation}
and that 
\begin{equation}\label{eq:min-max_upper_Jacobian}
C^{-2} \leqslant \det((df)(df)^T ), \text{ almost everywhere on }M,
\end{equation}
where the transpose of $df$ is taken with respect to $g$ on $TM$ and the flat metric on $\RR^3$.
\end{lemm}
\begin{proof}
Suppose $M$ is isometrically embedded in some Euclidean space. As noted in~\cite{Stern2021}, there exists a finite simplicial complex $K$ in some $\RR^L$ and a bi-Lipschitz map
\[
\Psi: M \to |K|,
\]
where $|K| \subset \RR^L$ denotes the union of all the simplices in $K$. For each $d$-simplex $\Delta \in K$ ($d = 0, 1, 2, 3$), we let $V(\Delta)$ denote the $d$-dimensional subspace of $\RR^L$ parallel to $\Delta$. That is, 
\[
V(\Delta) = \Span\{v - w\ |\ v, w \in \Delta\}.
\]
Then we may find an $(L-3)$-plane $P$ in $\RR^L$ such that 
\[
P \cap V(\Delta) = \{0\}, \text{ for all }\Delta \in K.
\]
Let $\pi: \RR^L \to P^{\perp}$ denote orthogonal projection onto the $3$-plane $P^{\perp}$, which we from now on identify with $\RR^3$. By the positioning we arranged, for all $y \in \RR^3$ and $\Delta \in K$ the pre-image $\pi^{-1}(y) \cap \Delta$ contains at most one point. Moreover, for each $3$-simplex $\Delta\in K$, the restriction $\pi|_{V(\Delta)}$ is invertible. Combining these observations with the fact that $K$ consists of only finitely many simplices, we conclude there exists $C > 0$ such that 
\begin{equation}\label{eq:simplex_preimage}
\cH^0(\pi^{-1}(y) \cap |K|) \leqslant C, \text{ for all }y \in \RR^3,
\end{equation}
and that for each $3$-simplex $\Delta \in K$, 
\begin{equation}\label{eq:simplex_jacobian}
C^{-1} \leqslant |\det(\pi|_{V(\Delta)})|.
\end{equation}
Now define the composition
\[
f = \pi \circ \Psi: M \to \RR^3.
\]
Then clearly we get~\eqref{eq:min-max_upper_preimage} with the same $C$ as in~\eqref{eq:simplex_preimage}. On the other hand, given a $3$-simplex $\Delta$ in $K$, we denote the interior of $\Delta$ by $\mathring{\Delta}$, and let $D_{\Delta}$ be the set of points in $\mathring{\Delta}$ where $\Psi^{-1}$ is differentiable. Similarly, we let $E_{\Delta}$ be the set of points in $\Psi^{-1}(\mathring{\Delta}) \subset M$ where $\Psi$ is differentiable. Since $\Psi^{-1}$ and $\Psi$ are both Lipschitz maps, we see that
\[
B_{\Delta} := \Psi^{-1}(\mathring{\Delta} \setminus D_{\Delta}) \cup \big(\Psi^{-1}(\mathring{\Delta})\setminus E_{\Delta}\big) \text{ has }\cH^3\text{-measure zero}.
\]
For all $p \in \Psi^{-1}(\mathring{\Delta}) \setminus B_{\Delta}$, we have that $\Psi$ is differentiable at $p$ and $\Psi^{-1}$ is differentiable at $\Psi(p)$. The chain rule applied to $\big(\Psi^{-1}|_{\Delta}\big)\circ \big(\Psi|_{\Psi^{-1}(\Delta)}\big) = \Id_{\Psi^{-1}(\Delta)}$ gives
\[
d(\Psi^{-1})_{\Psi(p)} \circ (d\Psi)_p = \Id_{T_p M},
\]
where we emphasize that the linear maps involved have the following domains and targets:
\[
(d\Psi)_p : T_p M \to V(\Delta),\ \ d(\Psi^{-1})_{\Psi(p)}: V(\Delta) \to T_p M.
\]
By the above relation, and the inequality
\[
|\det A| = (\det(A^TA))^\frac{1}{2} \leqslant \big( \frac{\tr A^T A}{3} \big)^{\frac{3}{2}}
\] applied to the matrix representation of $d(\Psi^{-1})_{\Psi(p)}$ with respect to orthonormal bases of $V(\Delta)$ and $T_p M$, we get
\[
\big|\det((d\Psi)_p)\big| = \frac{1}{\big|\det(d(\Psi^{-1})_{\Psi(p)})\big|} \geqslant \frac{1}{C([\Psi^{-1}]_{\text{Lip}})^3}.
\]
Combining this with~\eqref{eq:simplex_jacobian} gives that $df_p: T_p M \to \RR^3$ satisfies
\begin{equation}\label{eq:Jac_lower_on_good_set}
\det((df_p) (df_p)^T)  = \big( \det d\Psi_p \big)^2 \big(\det\pi|_{V(\Delta)}\big)^2 \geqslant C^{-2},
\end{equation}
for all $p \in \Psi^{-1}(\mathring{\Delta})\setminus B_{\Delta}$,
with a constant $C$ which does not depend on $\Delta$. To finish, let $K^3$ denote the collection of $3$-simplices in $K$ and note that
\[
M \setminus \big( \cup_{\Delta \in K^3} (\Psi^{-1}(\mathring{\Delta}) \setminus B_{\Delta}) \big) \subset (\cup_{\Delta \in K^3}B_{\Delta})  \bigcup  \big( M \setminus \cup_{\Delta \in K^3}\Psi^{-1}(\mathring{\Delta})\big),
\]
the right-hand side being a set of $\cH^3$-measure zero. Combining this with~\eqref{eq:Jac_lower_on_good_set}, we get~\eqref{eq:min-max_upper_Jacobian}.
\end{proof}
\begin{prop}\label{prop:min-max_upper_bound}
For $\ep \in (0, \frac{1}{2})$, we have $\omega_{\ep} \lesssim_M \max\{1,\lambda\}\cdot\ep$.
\end{prop}
\begin{proof}
The proof is adapted from Section 7 of~\cite{Pigati-Stern2021}. Throughout this proof, we fix $\ep \in (0, \frac{1}{2})$. It is enough to produce some $H \in \Gamma$ for which $\sup_{y \in B^3}\cY_\ep(H(y)) \lesssim_{M} \max\{1,\lambda\}\cdot\ep$. Let $f$ be the Lipschitz map produced by Lemma \ref{lem: stern_lipschitz_map}. From~\eqref{eq:min-max_upper_preimage} and~\eqref{eq:min-max_upper_Jacobian}, as well as the co-area formula, we see that for all $y \in \RR^3$ and $r > 0$, there holds
\begin{equation}\label{eq:min-max_upper_volume_1}
\begin{split}
\Vol_g(f^{-1}(B_r(y))) \leqslant\ & C\int_{f^{-1}(B_r(y))} \sqrt{\det((df)(df)^T )} \vol_g\\
=\ & C\int_{B_r(y)} \cH^0(f^{-1}(z) \cap M) dz \leqslant K_1 r^3,
\end{split}
\end{equation}
and, by a similar reasoning,
\begin{equation}\label{eq:min-max_upper_volume_2}
\begin{split}
\int_{M \setminus f^{-1}(B_r(y))} \frac{1}{|f(x) - y|^4} \vol_g \leqslant C\int_{\RR^3 \setminus B_r(y)} \frac{dz}{|z - y|^4} \leqslant K_2r^{-1}.
\end{split}
\end{equation}
By using a partition of unity and mollifying in coordinate charts, we obtain a sequence of smooth maps $f_i: M \to \RR^3$ such that
\begin{equation}\label{eq:uniform_lipschitz_bound}
\|f_i - f\|_{\infty} \to 0,\ \ \|df_i\|_{\infty} \leqslant C_M(\|f\|_{\infty} + [f]_{\text{Lip}}) =: K_3.
\end{equation}
Fix $i_0 \in \NN$ such that 
\[
\|f_i - f\|_{\infty} < \ep^8 \text{ for all }i \geqslant i_0.
\]
By the triangle inequality, for all $y \in \RR^3$, there holds
\[
f_i^{-1}(B_{\ep}(y)) \subset f^{-1}(B_{2\ep}(y)),\ \ \ M \setminus f_i^{-1}(B_{\ep}(y)) \subset M \setminus f^{-1}(B_{\frac{\ep}{2}}(y)).
\]
With the help of the second inclusion, we see that everywhere on $M \setminus f_i^{-1}(B_{\ep}(y))$ there holds
\[
\begin{split}
\Big|\frac{1}{|f_i(x) - y|^4} - \frac{1}{|f(x) - y|^4}\Big| \leqslant\ & \sup\big\{ \big|\frac{1}{t^4} - \frac{1}{s^4}\big|\ \big|\ t, s \geqslant \frac{\ep}{2},\ |t - s| \leqslant \|f - f_i\|_{\infty} \big\}\\
\leqslant\ & \frac{4}{(\ep/2)^5}\|f - f_i\|_{\infty} < 128\ep^3.
\end{split}
\]
Combining these with~\eqref{eq:min-max_upper_volume_1},~\eqref{eq:min-max_upper_volume_2}, and letting
\[
h: = f_{i_0},
\]
we have for all $y \in \RR^3$ that
\begin{equation}\label{eq:min-max_upper_volume_1'}
\Vol_g(h^{-1}(B_{\ep}(y)))\leqslant \Vol_g(f^{-1}(B_{2\ep}(y))) \leqslant 8K_1\ep^3,
\end{equation}
and that
\begin{equation}\label{eq:min-max_upper_volume_2'}
\begin{split}
\int_{M \setminus h^{-1}(B_{\ep}(y))}\frac{1}{|h(x)- y|^4}\vol_g \leqslant\ & \int_{M \setminus h^{-1}(B_{\ep}(y))}\frac{1}{|f(x)- y|^4}\vol_g + 128\ep^3\cdot \Vol_g(M)\\
\leqslant\ &  \int_{M \setminus f^{-1}(B_{\frac{\ep}{2}}(y))}\frac{1}{|f(x)- y|^4}\vol_g + 128\ep^3\cdot \Vol_g(M)\\
\leqslant\ & (2K_2 + C_M)\ep^{-1}.
\end{split}
\end{equation}
Next we let $\rho: [0, \infty) \to [0, 1]$ be a smooth function such that
\[
\rho(t) = t \text{ for }t \leqslant \frac{2}{3}, \ \ \rho(t) = 1 \text{ for }t \geqslant 1,\ \ 0 \leqslant \rho'(t) \leqslant 5 \text{ everywhere }.
\]
(Such a $\rho$ can be produced as follows: one starts with a smooth function $\zeta:\RR \to [0, 1]$ with 
\[
\zeta(t) = 0 \text{ for }t \leqslant \frac{2}{3},\ \ \zeta(t) = 1 \text{ for }t \geqslant 1,\ \ 0 \leqslant \zeta' \leqslant 4 \text{ everywhere}, 
\]
and then set $\rho(t) =(1-\zeta(t)) t + \zeta(t)$.) Letting $R(x) = \rho(|x|)\frac{x}{|x|}$, we see that it satisfies 
\[
R(x) = \left\{
\begin{array}{ll}
x, & \text{ if }|x| \leqslant \frac{2}{3},\\
\frac{x}{|x|}, & \text{ if }|x| > 1,
\end{array}
\right.
\]
and is smooth on all of $\RR^3$. Also, for $|y| < 1$ we let $a(y) = \frac{-y}{1 - |y|}$. 

With these building blocks, we define a family $\{\varphi_y\}_{y \in B^3}$ of maps $\varphi_y: M \to \RR^3 \simeq \fsu(2)$ by
\[
\varphi_{y}(x) = \left\{
\begin{array}{ll}
R(\frac{h(x) - a(y)}{\ep}), & \text{ if }|y| < 1,\\
y,& \text{ if }|y| = 1.
\end{array}
\right.
\]
Since $h(M)$ is a bounded subset of $\RR^3$ and $R$ along with its derivatives of all orders are uniformly continuous on compact subsets of $\RR^3$, and since the derivatives of $h$ of all orders are bounded on $M$, we see with the help of the chain rule that $\varphi_y$ varies smoothly when $y$ varies in $\Inte(B^3)$. On the other hand, again since $h(M)$ is bounded, when $y \in \Inte(B^3)$ is sufficiently close to $\partial B^3$ we must have 
\[
h(M) \cap B_{\ep}(a(y)) = \emptyset,
\]
in which case
\[
\varphi_y(x) = \frac{h(x) - a(y)}{|h(x) - a(y)|} = \frac{y + (1 - |y|)h(x)}{|y + (1-|y|)h(x)|}, \text{ for all }x \in M.
\]
From this, and again using the fact that $h$ and all its derivatives are bounded on $M$, it is not hard to see that whenever $(y_i) \subset B^3$ converges to some $y_0 \in \partial B^3$, the maps $\varphi_{y_i}$ converge smoothly to the constant map $y_0$. 

Next, by a direct computation, for all $y \in \Inte(B^3)$ we have
\begin{equation}\label{eq:min-max_upper_dphi}
d\varphi_y = \frac{1}{|h - a(y)|}\big( dh - \bangle{dh, \frac{h - a(y)}{|h 
- a(y)|}} \frac{h - a(y)}{|h - a(y)|}\big), \text{ on }M \setminus h^{-1}(B_{\ep}(a(y))),
\end{equation}
and in particular we have on $M \setminus h^{-1}(B_{\ep}(a(y)))$ that
\begin{equation}\label{eq:min-max_upper_dphi_trans}
\bangle{d\varphi_y, \varphi_y} = 0, \ \ \ |d\varphi_y| \leqslant C_{K_3} |h - a(y)|^{-1},
\end{equation}
where the inequality follows from~\eqref{eq:uniform_lipschitz_bound}. On the other hand, on $h^{-1}(B_{\ep}(a(y)))$, again using~\eqref{eq:uniform_lipschitz_bound}, we find that there holds
\begin{equation}\label{eq:min-max_upper_dphi_estimate}
|d\varphi_y| \leqslant \|dR\|_{\infty; B_1} \cdot \|dh\|_{\infty} \cdot \ep^{-1} \leqslant C_{R, K_3}\ep^{-1}.
\end{equation}
To define the family of connections to go with $\{\varphi_y\}$, note that for $x \in M \setminus h^{-1}(B_{\ep}(a(y)))$, from~\eqref{eq:min-max_upper_dphi_trans} and the fact that $|\varphi_y(x)| = 1$, we get
\[
d\varphi_y(x) = -[[d\varphi_y(x), \varphi_y(x)], \varphi_y(x)].
\]
Thus, if we define a family of $\fsu(2)$-valued $1$-forms $\{B_y\}_{y \in B^3}$ by
\[
B_y = [d\varphi_y, \varphi_y],
\]
which varies smoothly as $y$ varies in $B^3$, 
then we have for all $y \in \Inte(B^3)$ that
\begin{equation}\label{eq:min-max_upper_parallel}
(d + B_y)\varphi_y = 0 \text{ on }M \setminus h^{-1}(B_{\ep}(a(y))).
\end{equation}
Moreover, the curvature of $B_y$ satisfies
\[
|F_{B_y}| \leqslant C|d\varphi_y|^2.
\]
Letting $H(y) = (B_y, \varphi_y)$, we see that $H \in C^0(B^3; X)$ and moreover $H(y) = (0, \text{ const. }y)$ for all $y \in \partial B^3$. That is, $H \in \Gamma$. To estimate $\cY_{\ep}(H(y))$, it suffices to consider only $|y| < 1$, since otherwise $\cY_{\ep}(H(y)) = 0$. On $M \setminus h^{-1}(B_{\ep}(a(y)))$, upon recalling~\eqref{eq:min-max_upper_parallel} and the fact that $|\varphi_y| = 1$, and also using the estimate in~\eqref{eq:min-max_upper_dphi_trans} to bound the curvature term, we have
\begin{equation}\label{eq:min-max_upper_1}
\begin{split}
\cY_{\ep}(B_y, \varphi_y; M\setminus h^{-1}(B_{\ep}(a(y))))  = \ &\ep^2\int_{M \setminus h^{-1}(B_{\ep}(a(y)))} |F_{B_y}|^2 \vol_g \\
\leqslant\ & C\ep^2\int_{M \setminus h^{-1}(B_{\ep}(a(y)))} |d\varphi_y|^4 \vol_g\\
\leqslant\ & C_{K_3}\ep^2 \int_{M \setminus h^{-1}(B_{\ep}(a(y)))} |h(x) - a(y)|^{-4}\vol_g\\
\leqslant \ & C_{K_2, K_3, M} \cdot \ep,
\end{split}
\end{equation}
where we used~\eqref{eq:min-max_upper_volume_2'} for the last inequality. On the other hand, on $h^{-1}(B_{\ep}(a(y)))$ we use~\eqref{eq:min-max_upper_dphi_estimate} to get
\[
\ep^2|F_{B_y}|^2 \leqslant C_{R, K_3}\ep^{-2},\ \ |d\varphi_y|^2 + |B_y|^2 |\varphi_y|^2 \leqslant C_{R, K_3}\ep^{-2},
\]
while the potential term $\frac{\lambda}{4\ep^2}(1-|\varphi_y|^2)^2$ we simply estimate from above by $\lambda\cdot\ep^{-2}$, noting that $|\varphi_y|\leqslant 1$ by construction. Putting these pointwise bounds together gives 
\[
e_{\ep}(B_y, \varphi_y) \leqslant C_{R, K_3}\cdot\max\{1,\lambda\}\cdot\ep^{-2} \text{ on }h^{-1}(B_{\ep}(a(y))).
\]
Upon recalling~\eqref{eq:min-max_upper_volume_1'}, we arrive at
\[
\begin{split}
\cY_{\ep}(B_y, \varphi_y; h^{-1}(B_{\ep}(a(y)))) \leqslant\ &  C_{R, K_1, K_3}\cdot\max\{1,\lambda\}\cdot\ep,
\end{split}
\]
Adding this to~\eqref{eq:min-max_upper_1} gives 
\[
\cY_\ep(H(y)) \leqslant C_{R, M, K_1, K_2, K_3}\cdot\max\{1,\lambda\}\cdot\ep, \text{ for all }y \in B^3 \text{ and }\ep \in (0, \frac{1}{2}),
\]
with $C$ being independent of both $y$ and $\ep$. The proof is complete.
\end{proof}
\begin{prop}\label{prop:existence_critical_point}
For $\ep$ sufficiently small depending on $M$, the min-max value $\omega_{\ep}$ is a critical value of $\cY_{\ep}$.
\end{prop}
\begin{proof}
Having shown that $\omega_{\ep} > 0$, or more relevantly that there is a positive distance between $\omega_{\ep}$ and $(\cY_{\ep} \circ H )(\partial B^3)$ that is uniform over all $H \in \Gamma$, the existence of a critical point at the level $\omega_{\ep}$ is a consequence of the Palais--Smale condition and a standard gradient flow argument. Specifically, since the functional $\cY_{\ep}$ is $C^1$ on the Hilbert space $X$, and since, by Lemma~\ref{lemm:Finsler}, the norms $\{\|\cdot\|_{T_{(A, \Phi)}X}\}_{(A, \Phi) \in X}$ form a Finsler structure on the tangent bundle $TX$ according to the definition given in~\cite[Chapter II.3]{StrBook}, standard theory shows that the functional possesses a pseudo-gradient vector field~\cite[Lemma II.3.9]{StrBook}, that is, a locally Lipschitz map 
\[
v: \widetilde{X}:= \{p \in X\ |\ (\delta\cY_{\ep})_{p} \neq 0\} \to X
\]
having the following two properties:
\vskip 1mm
\begin{enumerate}
\item[(pg1)] $\|v(p)\|_{T_{p}X} \leqslant 2\cdot \min\{1, \|(\delta\cY_{\ep})_p\|\}$, for all $p \in \widetilde{X}$.
\vskip 1mm
\item[(pg2)] $(\delta\cY_{\ep})_{p}(v(p)) \geqslant \min\{1, \|(\delta\cY_{\ep})_p\|\} \cdot \|(\delta \cY_{\ep})_p\|$, for all $p \in \widetilde{X}$.
\end{enumerate}
Now suppose, towards a contradiction, that there exist $\eta, \beta \in (0, 1)$ with $\eta < \frac{1}{8}\min\{\beta^2, \omega_{\ep}\}$ such that 
\begin{equation}\label{eq:min-max_contradiction}
\|(\delta \cY_{\ep})_{p}\| \geqslant \beta, \text{ whenever } |\cY_{\ep}(p) - \omega_{\ep}| \leqslant 3\eta.
\end{equation}
Let $\varphi: \RR \to [0, 1]$ be a smooth cutoff function such that
\[
\varphi(t) = 1 \text{ if }|t - \omega_{\ep}| \leqslant \eta,\ \ \varphi(t) = 0 \text{ if }|t - \omega_{\ep}| \geqslant 2\eta,
\]
and define 
\[
v_1(p) = 
\left\{
\begin{array}{ll}
-\varphi(\cY_{\ep}(p))v(p), & \text{ if }p \in \widetilde{X}\\
0, & \text{ otherwise. }
\end{array}
\right.
\]
Under the assumption~\eqref{eq:min-max_contradiction}, critical points of $\cY_{\ep}$ occur outside $\cY_{\ep}^{-1}(\supp(\varphi))$, and thus $v_1$ is a locally Lipschitz vector field on all of $X$. Moreover, by (pg1) and the completeness of $X$, it has a globally defined flow $\Psi:[0, \infty) \times X \to X$. From the definition of $v_1$ and $\varphi$, we see that 
\begin{equation}\label{eq:min-max_stay_put}
\Psi(t, p) = p, \text{ for all $p$ such that }\cY_{\ep}(p) \not\in  [\omega_{\ep} - 2\eta, \omega_{\ep} + 2\eta].
\end{equation}
Also, differentiating and using (pg2), we see that $t\mapsto \cY_{\ep}(\Psi(t, p))$ is non-increasing for all $p \in X$. Next, take $H \in \Gamma$ such that 
\[
\sup_{y \in B^3}\cY_{\ep}(H(y)) < \omega_\ep + \eta,
\]
and define
\[
\widetilde{H}(y) = \Psi(1, H(y)).
\]
Then $\widetilde{H}: B^3 \to X$ is still a continuous map. Moreover, for all $y \in \partial B^3$, since $\cY_{\ep}(H(y)) = 0 < \omega_{\ep} - 2\eta$, we have by~\eqref{eq:min-max_stay_put} that
\[
\widetilde{H}(y) = H(y) = (0, \text{const. } y).
\]
That is, $\widetilde{H} \in \Gamma$. By the continuity of $\cY_{\ep} \circ \widetilde{H}$, there exists $y_0 \in B^3$ such that
\[
\cY_{\ep}(\widetilde{H}(y_0))  = \sup_{y \in B^3}\cY_{\ep}(\widetilde{H}(y)) \geqslant \omega_{\ep},
\]
where the last inequality follows from the definition of $\omega_\ep$ and the fact that $\widetilde{H} \in \Gamma$. We are now ready to draw a contradiction. Below, for simplicity of notation we define
\[
p_0 = H(y_0),\ \ p_t = \Psi(t, p_0).
\]
Since
\[
\cY_{\ep}(p_1) \geqslant \omega_\ep,\ \ \cY_{\ep}(p_0) < \omega_{\ep} + \eta
\]
and since $\cY_{\ep}(p_t)$ is non-increasing with respect to $t$, we deduce that
\[
\cY_{\ep}(p_t) \in [\omega_{\ep}, \omega_{\ep} + \eta], \text{ for all }t \in [0, 1].
\]
In particular $\varphi(\cY_{\ep}(p_t)) = 1$ for all $t \in [0, 1]$, which together with (pg2) and the assumption~\eqref{eq:min-max_contradiction} gives
\[
\frac{d}{dt}\big(\cY_{\ep}(p_t)\big) = -(\delta\cY_{\ep})_{p_t}(v(p_t)) \leqslant -\beta^2, \text{ for all }t \in [0, 1].
\]
Feeding this into the fundamental theorem of calculus gives
\[
\begin{split}
\omega_{\ep} \leqslant\ & \cY_{\ep}(p_1) = \cY_{\ep}(p_0) + \int_{0}^{1}\frac{d}{dt}\big(\cY_{\ep}(p_t) \big) dt\\
<\ & \omega_\ep + \eta - \beta^2,
\end{split}
\]
which gives $\beta^2 < \eta$, contradicting our assumption that $\eta < \frac{\beta^2}{8}$. Therefore~\eqref{eq:min-max_contradiction} must fail whenever $\eta, \beta \in (0, 1)$ are such that $\eta < \frac{1}{8}\min\{\beta^2, \omega_{\ep}\}$. Letting $\beta = \frac{1}{i}$ and $\eta = \frac{1}{16 i^2}$, then for sufficiently large $i$ we obtain some $(A_i, \Phi_i) \in X$ such that 
\[
|\cY_{\ep}(A_i, \Phi_i) - \omega_{\ep}| \leqslant \frac{3}{16i^2},\ \ \|(\delta \cY_{\ep})_{(A_i, \Phi_i)}\| \leqslant \frac{1}{i}.
\]
By Proposition~\ref{prop:PS_up_to_gauge}, up to taking a subsequence and changing gauge, $(A_i, \Phi_i)$ converges strongly in $X$ to some critical point $(A, \Phi)$. Passing to the limit in the first inequality above shows that 
$\cY_{\ep}(A, \Phi) = \omega_{\ep}$. This finishes the proof.
\end{proof}
As a direct consequence of the results in this section, more specifically Propositions \ref{prop:min-max_lower_bound}, \ref{prop:min-max_upper_bound} and \ref{prop:existence_critical_point}, as well as Proposition~\ref{prop:smoothness_of_weak_solution}, we get our first main result stated in the introduction as Theorem \ref{thm: existence}. Moreover, as explained in \S\ref{subsec: main_results}, under the additional assumption that $b_1(M)=b_2(M)=0$, we are able to guarantee that the solutions produced by Theorem \ref{thm: existence} are actually irreducible; see Theorem \ref{thm: irred}. Ultimately, the application of such result to, say, the $3$-sphere $\mathbb{S}^3$, will lead in Section \ref{sec:asymptotic}, after we study the energy concentration behavior of critical points when $\ep\to 0$, to the existence of non-trivial critical points of $\cY_{1}$ on the flat $3$-dimensional Euclidean space $\mathbb{R}^3$, for any $\lambda>0$; see Proposition \ref{prop:existence_R3}.

\section{A priori estimates for critical points}\label{sec:estimates}
In this section we establish various a priori estimates for smooth solutions $(\nabla, \Phi)$ to the Yang--Mills--Higgs equations~\eqref{eq: 2nd_order_crit_pt_intro}:
\[
\begin{cases}
\varepsilon^2 d_{\nabla}^{\ast}F_{\nabla} = [\nabla\Phi,\Phi],\\
\nabla^{\ast}\nabla\Phi = \frac{\lambda}{2\varepsilon^2}(1-|\Phi|^2)\Phi,
\end{cases}
\]
where $\lambda$ is a positive constant. To introduce the quantities to be estimated, given $(\nabla,\Phi)\in \mathscr{A}(E)\times\Gamma(\mathfrak{su}(E))$, recall that
\[
e_{\ep} = e_{\ep}(\nabla, \Phi) = \ep^2 |F_{\nabla}|^2 + |\nabla\Phi|^2 + \frac{\lambda w^2}{\ep^2},
\]
where $w = \frac{1 - |\Phi|^2}{2}$. Next, for $k \in \NN \cup \{0\}$, we define
\begin{equation}\label{eq:Psi_Theta_definition}
\begin{split}
\Psi_k =\ & (\ep^2|\nabla^k F_{\nabla}|^2 + |\nabla^{k + 1}\Phi|^2)^{\frac{1}{2}},\\
\Theta_k =\ & (\ep^2|\nabla^k [F_{\nabla}, \Phi]|^2 + |\nabla^{k}[\nabla\Phi, \Phi]|^2)^{\frac{1}{2}}.
\end{split}
\end{equation}
In addition, away from the zero locus $Z(\Phi)$ of $\Phi$, we let
\begin{equation}\label{eq:Psi_perp_definition}
\begin{split}
\Psi_k^\perp =\ & (\ep^2|(\nabla^k F_{\nabla})^{\perp}|^2 + |(\nabla^{k + 1}\Phi)^{\perp}|^2)^{\frac{1}{2}}\\
=\ & |\Phi|^{-1}(\ep^2|[\nabla^k F_{\nabla}, \Phi]|^2 + |[\nabla^{k + 1}\Phi, \Phi]|^2)^{\frac{1}{2}}.
\end{split}
\end{equation}
Notice that 
\begin{equation}\label{eq:Theta_by_Psi_basic}
\Theta_0 \leqslant \Psi_0 |\Phi|,
\end{equation}
which on $M \setminus Z(\Phi)$ can be refined to 
\begin{equation}\label{eq:perp_by_Theta_basic}
\Psi_0^\perp = |\Phi|^{-1} \Theta_0.
\end{equation}
In the following, in \S\ref{subsec:identities-and-diff-ineqs}, we derive a number of formulas and inequalities involving the functions just defined. In \S\ref{subsec:coarse-estimates} we obtain basic estimates for $\Psi_0$, $w$, as well as $\Psi_k$ for $k \in \NN$. In \S\ref{subsec:exp-decay}, under additional smallness conditions, we prove exponential decay estimates for $\Theta_k$ and $\Psi_k^{\perp}$. Taking advantage of the fact that $\lambda > 0$, we also get similar estimates for $w$ and $\nabla^{k + 1}\Phi$. In \S\ref{subsec:improved-estimates}, still under suitable smallness assumptions, we show how the estimates of the previous section feed back into the proofs in \S\ref{subsec:coarse-estimates} to give improved estimates on $\Psi_k$. In \S\ref{subsec:local-convergence} we apply the estimates to sequences of solutions and establish a convergence result that will factor into the proofs of Theorems~\ref{thm: asymptotic} and~\ref{thm: bubbling} later. In \S\ref{subsec:consequences-estimates} we collect some other consequences of the estimates that are useful elsewhere in the paper.
\subsection{Preliminaries}\label{subsec:identities-and-diff-ineqs}
For this section, the dimension of $M$ is irrelevant and we denote it by $n$. We begin by mentioning, without proof, some other direct consequences of the definitions~\eqref{eq:Psi_Theta_definition} and~\eqref{eq:Psi_perp_definition}.
\begin{lemm}\label{lemm:Psi_Theta_basic_relations}
Given $(\nabla,\Phi)\in \mathscr{A}(E)\times\Gamma(\mathfrak{su}(E))$, we have for all $k \in \NN \cup \{0\}$ that
\begin{equation}\label{eq:Psi_Theta_schwarz}
\big|d(\Psi_k^2)\big| \leqslant 2 \Psi_k \Psi_{k + 1},\ \ \ \big|d(\Theta_k^2)\big| \leqslant 2 \Theta_k \Theta_{k + 1}.
\end{equation}
Moreover, the following identities hold:
\begin{align}
\Psi_{k + 1}^2 + \Delta \big(\frac{\Psi_k^2}{2} \big) =\ & \bangle{\nabla^*\nabla \nabla^{k + 1}\Phi, \nabla^{k + 1}\Phi} + \ep^2 \bangle{\nabla^* \nabla \nabla^k F_{\nabla}, \nabla^k F_{\nabla}}, \label{eq:Psi_laplacian} \\
\Theta_{k + 1}^2 + \Delta \big(\frac{\Theta_k^2}{2} \big) = \ & \bangle{\nabla^*\nabla \nabla^{k}[\nabla\Phi, \Phi], \nabla^{k}[\nabla\Phi, \Phi]} \nonumber\\
&+ \ep^2 \bangle{\nabla^* \nabla \nabla^k [F_{\nabla}, \Phi], \nabla^k [F_{\nabla}, \Phi]}. \label{eq:Theta_laplacian}
\end{align}
\end{lemm}
\begin{proof}
Straightforward computation. The details are omitted.\\
\end{proof}
Next, we compute in Lemma~\ref{lemm:laplacian_formulas_with_lambda} the Laplacian of several key quantities using the equations~\eqref{eq: 2nd_order_crit_pt_intro}, and then establish in Lemma~\ref{lemm:basic_differential_ineqs} a number of basic differential inequalities. To simplify notation we sometimes write $F$ for $F_{\nabla}$. Also, when we find it more convenient to express a tensor in terms of local orthonormal frames, we use subscripts to denote the components.
\begin{lemm}[See also \cite{Jaffe-Taubes}, Chapter IV.9]
\label{lemm:laplacian_formulas_with_lambda}
Suppose $(\nabla, \Phi)$ is a smooth solution of~\eqref{eq: 2nd_order_crit_pt_intro}. Then we have the following identities. 
\begin{subequations}
\begin{align}
\Delta w =\ & |\nabla\Phi|^2 - \frac{\lambda w}{\ep^2} |\Phi|^2, \label{eq:w_laplacian_with_lambda}\\
\nabla^*\nabla \nabla \Phi =\ & \lambda \frac{w - |\Phi|^2}{\ep^2}\nabla\Phi - \frac{1 - \lambda}{\ep^2}[\Phi, [\nabla\Phi, \Phi]] -2 [F_{ki}, \nabla_{k}\Phi] - \Ric_{ki}\nabla_{k}\Phi, \label{eq:nablaPhi_laplacian_with_lambda}\\
\nabla^*\nabla (\ep F) =\ & -\frac{1}{\ep^2}[\Phi, [\ep F, \Phi]] - \frac{2}{\ep}[\nabla_i\Phi, \nabla_j\Phi] - \frac{2}{\ep}[\ep F_{ki}, \ep F_{kj}] - \cR_2(\ep F), \label{eq:F_laplacian_with_lambda}\\
\nabla^*\nabla [\nabla\Phi, \Phi] =\ & -2[\nabla_{k, i}^2\Phi, \nabla_{k}\Phi] + \frac{2\lambda w - |\Phi|^2}{\ep^2}[\nabla\Phi, \Phi] \nonumber\\
&- \frac{2}{\ep}[[\ep F_{ki}, \nabla_k\Phi], \Phi] - \Ric_{ki}[\nabla_k\Phi, \Phi], \label{eq:nablaPhi_trans_laplacian_with_lambda}\\
\nabla^*\nabla[\ep F, \Phi] =\ & -2[\nabla_{k} (\ep  F)_{ij}, \nabla_{k}\Phi] + \frac{\lambda w - |\Phi|^2}{\ep^2}[\ep F,\Phi] - \frac{2}{\ep}[[\nabla_i\Phi, \nabla_j\Phi], \Phi]\nonumber\\
&- \frac{2}{\ep}[[\ep F_{ki}, \ep F_{kj}], \Phi] - \cR_2([\ep F, \Phi]). \label{eq:F_trans_laplacian_with_lambda}
\end{align}
\end{subequations}
\end{lemm}
\begin{proof}
For~\eqref{eq:w_laplacian_with_lambda}, by the definition of $w$ and the second equation in~\eqref{eq: 2nd_order_crit_pt_intro}, we have
\[
\begin{split}
\Delta w =\ & -\Delta(\frac{|\Phi|^2}{2}) = |\nabla \Phi|^2 - \bangle{\Phi, \nabla^*\nabla\Phi} = |\nabla\Phi|^2 - \frac{\lambda w}{\ep^2}|\Phi|^2,
\end{split}
\]
which is the desired identity.

For~\eqref{eq:nablaPhi_laplacian_with_lambda}, again using the second equation in~\eqref{eq: 2nd_order_crit_pt_intro}, we have 
\[
\begin{split}
\Delta_{\nabla}\nabla\Phi =\ & d_{\nabla}(\frac{\lambda w}{\ep^2}\Phi) + d_{\nabla}^*[F_{\nabla}, \Phi]\\
= \ & -\frac{\lambda}{\ep^2}\bangle{\Phi, \nabla\Phi}\Phi + \frac{\lambda w}{\ep^2}\nabla\Phi + [d_{\nabla}^* F_{\nabla}, \Phi] - [F_{ki}, \nabla_k\Phi].
\end{split}
\]
The second-to-last term can be simplified by using the first equation in~\eqref{eq: 2nd_order_crit_pt_intro} as follows:
\[
\begin{split}
[d_{\nabla}^* F_{\nabla}, \Phi] =\frac{1}{\ep^2}[[\nabla\Phi, \Phi], \Phi] =\ & -\frac{\lambda}{\ep^2}[\Phi, [\nabla\Phi, \Phi]] - \frac{1 - \lambda}{\ep^2}[\Phi, [\nabla\Phi, \Phi]]\\
=\ & -\frac{\lambda |\Phi|^2}{\ep^2}\nabla\Phi + \frac{\lambda}{\ep^2}\bangle{\Phi, \nabla\Phi}\Phi - \frac{1 - \lambda}{\ep^2}[\Phi, [\nabla\Phi, \Phi]].
\end{split}
\]
Substituting this into the previous computation and observing a cancellation leads to
\[
\Delta_{\nabla}\nabla\Phi = \lambda \frac{w - |\Phi|^2}{\ep^2}\nabla\Phi - \frac{1 - \lambda}{\ep^2}[\Phi, [\nabla\Phi, \Phi]] - [F_{ki}, \nabla_{k}\Phi],
\]
which gives~\eqref{eq:nablaPhi_laplacian_with_lambda} when combined with the Weitzenb\"ock formula~\eqref{eq: 1form_Weitzenbock}. 

For~\eqref{eq:F_laplacian_with_lambda}, we note by the first equation in~\eqref{eq: 2nd_order_crit_pt_intro} and the Bianchi identity that
\[
\begin{split}
\Delta_{\nabla}F_{\nabla} =\ep^{-2} d_{\nabla}[\nabla\Phi, \Phi]
=\ & \frac{1}{\ep^2}[[F_{\nabla}, \Phi], \Phi] - \frac{2}{\ep^2}[\nabla_i\Phi, \nabla_j\Phi],
\end{split}
\]
and we are done upon recalling~\eqref{eq: 2form_Weitzenbock}. 

For~\eqref{eq:nablaPhi_trans_laplacian_with_lambda}, we begin with
\[
\begin{split}
\nabla^*\nabla[\nabla\Phi, \Phi] =\ & -2[\nabla_k(\nabla\Phi), \nabla_k\Phi] + [\nabla^*\nabla(\nabla\Phi), \Phi] + \frac{\lambda w}{\ep^2}[\nabla\Phi, \Phi],
\end{split}
\]
where we used the second equation in~\eqref{eq: 2nd_order_crit_pt_intro} to get the last term. 
Replacing $\nabla^*\nabla \nabla\Phi$ by the right-hand side of~\eqref{eq:nablaPhi_laplacian_with_lambda} and observing with the help of~\eqref{eq: triple_cross_product} that
\begin{equation}\label{eq:bracket_with_Phi_three_times}
[[\Phi, [S, \Phi]], \Phi] = [|\Phi|^2 S - \bangle{\Phi, S}\Phi, \Phi] = |\Phi|^2[S, \Phi],
\end{equation}
where $S$ is any $\fsu(E)$-valued tensor, we obtain ~\eqref{eq:nablaPhi_trans_laplacian_with_lambda}.

For~\eqref{eq:F_trans_laplacian_with_lambda}, starting instead with 
\[
\nabla^*\nabla[\ep F_{\nabla}, \Phi] = -2[\nabla_k (\ep F_{\nabla}), \nabla_k\Phi] + [\nabla^*\nabla (\ep F_{\nabla}), \Phi] + \frac{\lambda w}{\ep^2} [\ep F_{\nabla}, \Phi],
\]
using~\eqref{eq:F_laplacian_with_lambda} to replace $\nabla^*\nabla (\ep F_{\nabla})$, and applying~\eqref{eq:bracket_with_Phi_three_times} with $S = \ep F_{\nabla}$, we arrive at~\eqref{eq:F_trans_laplacian_with_lambda}. \\
\end{proof}

\begin{lemm}\label{lemm:basic_differential_ineqs}
Let $(\nabla, \Phi)$ be a smooth solution of~\eqref{eq: 2nd_order_crit_pt_intro}. Then we have the following.
\vskip 1mm
\begin{enumerate}
\item[(a)] In the sense of distributions, we have
\begin{equation}\label{eq:laplacian_|w|}
\Delta |w| \leqslant |\nabla\Phi|^2 - \frac{\lambda |\Phi|^2}{\ep^2}|w|.
\end{equation}
\vskip 1mm
\item[(b)] With $\mu = \min\{\lambda, 1\}$, there holds
\begin{equation}\label{eq:nablaPhi_bochner_with_lambda}
\begin{split}
\bangle{\nabla^*\nabla \nabla\Phi, \nabla\Phi} \leqslant\ & \frac{\lambda w - \mu |\Phi|^2}{\ep^2}|\nabla\Phi|^2 - 2 \sum_{i, j}\bangle{[\nabla_i\Phi, \nabla_j\Phi], F_{ij}} - \Ric(\nabla\Phi, \nabla\Phi).
\end{split}
\end{equation}
\vskip 1mm
\item[(c)] With $\Psi_0, \Theta_0$ defined as in~\eqref{eq:Psi_Theta_definition}, we have
\begin{equation}\label{eq:Psi_0_bochner}
\begin{split}
\bangle{\nabla^*\nabla \nabla\Phi, \nabla\Phi} + \ep^2 \bangle{\nabla^*\nabla F_{\nabla}, F_{\nabla}}
 \leqslant\ & -\frac{\Theta_0^2}{\ep^2} + \big( \frac{\lambda w_+}{\ep^2} + C_n|R| \big)\Psi_0^2 
 \\
 &- 3\sum_{i, j}\bangle{[\nabla_i\Phi, \nabla_j\Phi], F_{ij}} - \ep^2\sum_{i, j, k}\bangle{[F_{ki}, F_{kj}], F_{ij}}.
\end{split}
\end{equation}
\vskip 1mm
\item[(d)] Letting $S$ stand for either $\nabla\Phi$ or $\ep F_{\nabla}$, we have on $M \setminus Z(\Phi)$ that
\begin{equation}\label{eq:trans_laplacian_base}
\begin{split}
\bangle{\nabla^*\nabla[S, \Phi], [S, \Phi]} \leqslant\ & C_n|\Phi|^{-1}\Theta_1\Psi_0\Theta_0 +  \frac{2\lambda w_+ - |\Phi|^2}{\ep^2}|[S, \Phi]|^2\\
& + C_n\big( |\Phi|^{-2}\Psi_0^2 + |\Phi|^{-1}\Psi_1 +  \frac{\Psi_0}{\ep} +  |R|\big)\Theta_0^2.
\end{split}
\end{equation}
\end{enumerate}
\end{lemm}
\begin{proof}
For part (a), given any $\delta >0$, we have by~\eqref{eq:w_laplacian_with_lambda} that
\[
\begin{split}
\Delta \sqrt{w^2 + \delta^2} \leqslant\ &\frac{w\Delta w}{\sqrt{w^2 + \delta^2}} \\
\leqslant \ & |\nabla\Phi|^2 - \frac{\lambda |\Phi|^2}{\ep^2}\frac{w^2}{\sqrt{w^2 + \delta^2}}.
\end{split}
\]
Noting that $\frac{w^2}{\sqrt{w^2 + \delta^2}}$ converges pointwise to $|w|$, we deduce with the help of the dominated convergence theorem that~\eqref{eq:laplacian_|w|} holds distributionally. 

To prove part (b), we pair~\eqref{eq:nablaPhi_laplacian_with_lambda} with $\nabla\Phi$ and notice that 
\[
\begin{split}
\bangle{\frac{\lambda w - \lambda |\Phi|^2}{\ep^2}\nabla\Phi - \frac{1 - \lambda}{\ep^2}[\Phi, [\nabla\Phi, \Phi]] , \nabla\Phi} =\ & \frac{\lambda w - \lambda|\Phi|^2}{\ep^2}|\nabla\Phi|^2 - \frac{1 - \lambda}{\ep^2}|[\nabla\Phi, \Phi]|^2\\
=\ & \frac{\lambda w - |\Phi|^2}{\ep^2}|\nabla\Phi|^2  - \frac{\lambda - 1}{\ep^2}|\bangle{\nabla\Phi,\Phi}|^2.
\end{split}
\]
Using the first line when $\lambda \leqslant 1$, and the second line otherwise, we get~\eqref{eq:nablaPhi_bochner_with_lambda}. 

For part (c), we first notice from the previous computation that
\begin{equation}\label{eq:nabla_Phi_bochner_for_Psi0_bochner}
\begin{split}
\bangle{\nabla^*\nabla\nabla\Phi, \nabla\Phi} \leqslant\  & \frac{\lambda w }{\ep^2}|\nabla\Phi|^2 - \frac{1}{\ep^2}|[\nabla\Phi, \Phi]|^2- 2 \bangle{[\nabla\Phi, \nabla\Phi], F_{\nabla}} - \Ric(\nabla\Phi, \nabla\Phi).
\end{split}
\end{equation}
Noting that 
\[
w|\nabla\Phi|^2 \leqslant w_+  |\nabla\Phi|^2 \leqslant w_+ \Psi_0^2,
\]
we get part (c) upon adding to~\eqref{eq:nabla_Phi_bochner_for_Psi0_bochner} the following consequence of pairing~\eqref{eq:F_laplacian_with_lambda} with $\ep F_{\nabla}$:
\begin{equation}\label{eq:F_bochner_with_lambda}
\begin{split}
\ep^2\bangle{\nabla^*\nabla F_{\nabla}, F_{\nabla}} =\ & -\frac{1}{\ep^2}|[\ep F_{\nabla}, \Phi]|^2 -\sum_{i, j} \bangle{[\nabla_i\Phi, \nabla_j\Phi], F_{ij}}\\
& - \ep^2\sum_{i,j,k}\bangle{[F_{ki}, F_{kj}], F_{ij}} - \ep^2\bangle{\cR_2(F_{\nabla}), F_{\nabla}}.
\end{split}
\end{equation}

For part (d), upon pairing~\eqref{eq:nablaPhi_trans_laplacian_with_lambda} and~\eqref{eq:F_trans_laplacian_with_lambda}, respectively, with $[\nabla\Phi, \Phi]$ and $[\ep F_{\nabla}, \Phi]$, we have
\begin{equation}\label{eq:trans_laplacian_base_initial}
\begin{split}
\bangle{\nabla^*\nabla[S, \Phi], [S, \Phi]} \leqslant\ & C_n |[\nabla S, \nabla\Phi]| |[S, \Phi]| + \frac{2\lambda w_+-|\Phi|^2}{\ep^2}|[S, \Phi]|^2\\
&+ \frac{C_n}{\ep}\big|\big[[T_1, T_2],\Phi\big]\big| \big|[S, \Phi]\big| + C_n |R||[S, \Phi]|^2,
\end{split}
\end{equation}
where $T_1, T_2$ can each be either $\nabla\Phi$ or $\ep F_{\nabla}$, and the dimension $n$ affects the constants through the contractions involved in producing some of the terms in~\eqref{eq:nablaPhi_trans_laplacian_with_lambda} and~\eqref{eq:F_trans_laplacian_with_lambda} from tensor products. To continue, by the estimate~\eqref{eq:tensor-bracket-norm-with-decomp-1} from Appendix~\ref{sec:commute} we have
\begin{equation}\label{eq:trans_laplacian_base_quadratic_terms}
\begin{split}
|[[T_1, T_2],\Phi]| \leqslant\  & |[T_1, \Phi]||T_2| + |T_1||[T_2, \Phi]|.
\end{split}
\end{equation}
On the other hand, for the first term on the right-hand side of~\eqref{eq:trans_laplacian_base_initial}, we follow~\cite[Lemma IV.12.2]{Jaffe-Taubes}. Specifically, by~\eqref{eq:tensor-bracket-norm-with-decomp-2} we have
\[
\begin{split}
|[\nabla S, \nabla\Phi]| \leqslant\ & |(\nabla S)^\perp| |\nabla\Phi| + |\nabla S||(\nabla \Phi)^\perp|\\
\leqslant\ & |\Phi|^{-1}|[\nabla S, \Phi]||\nabla\Phi| + |\Phi|^{-1} |\nabla S| |[\nabla\Phi, \Phi]|.
\end{split}
\]
For the first term on the second line, combining the Leibniz rule and~\eqref{eq:tensor-bracket-norm-with-decomp-2} yields
\[
\begin{split}
|[\nabla S, \Phi]| \leqslant \ &|\nabla [S, \Phi]|  + |S^{\perp}| |\nabla\Phi| + |S||(\nabla\Phi)^{\perp}| \\
\leqslant\ & |\nabla[S, \Phi]| + |\Phi|^{-1}\big( |[S, \Phi]||\nabla\Phi| + |S||[\nabla\Phi,\Phi]|\big).
\end{split}
\]
Consequently, 
\begin{equation}\label{eq:trans_laplacian_base_cross_terms}
\begin{split}
|[\nabla S, \nabla\Phi]| \leqslant\ & |\Phi|^{-1}\big|\nabla[S, \Phi]\big||\nabla\Phi| + |\Phi|^{-2}\big( |[S, \Phi]||\nabla\Phi| + |S||[\nabla\Phi,\Phi]|\big) |\nabla\Phi|\\
& + |\Phi|^{-1}|\nabla S| |[\nabla \Phi, \Phi]|.
\end{split}
\end{equation}
Substituting~\eqref{eq:trans_laplacian_base_cross_terms} and~\eqref{eq:trans_laplacian_base_quadratic_terms} back into~\eqref{eq:trans_laplacian_base_initial}, and recalling the definitions~\eqref{eq:Psi_Theta_definition}, we easily get~\eqref{eq:trans_laplacian_base}.\\
\end{proof}
Lemma~\ref{lemm:Psi_Theta_relations} below generalizes~\eqref{eq:Theta_by_Psi_basic} and~\eqref{eq:perp_by_Theta_basic}. On the other hand, the estimates in Lemma~\ref{lemm:bochner_for_derivatives} and Lemma~\ref{lemm:trans_laplacian_diffed} can be viewed as parallels of those in Lemma~\ref{lemm:basic_differential_ineqs}(b)(c)(d), and appear later in induction arguments that lead to estimates on $\Psi_k$ and $\Theta_k$. That said, in contrast to their counterparts in Lemma~\ref{lemm:basic_differential_ineqs}, in stating the inequalities in Lemma~\ref{lemm:bochner_for_derivatives} and Lemma~\ref{lemm:trans_laplacian_diffed} we drop certain non-positive terms on the right-hand side, as they are not needed for our purposes. In any case, the reader is referred to Appendix~\ref{sec:proofs_derivative_formulas} for the standard yet tedious proofs of the next three lemmas. 

\begin{lemm}\label{lemm:Psi_Theta_relations}
For all $m \geqslant 1$ we have
\begin{equation}\label{eq:Psi_Theta_relation}
\Theta_m \leqslant |\Phi| \Psi_m + C_{m}\sum_{i = 0}^{m-1}\Psi_i \Psi_{m-1-i}.
\end{equation}
Moreover, on $M \setminus Z(\Phi)$, we have for $m \geqslant 1$ that
\begin{equation}\label{eq:Theta_perp_relation}
\big| \Psi_m^\perp - |\Phi|^{-1}\Theta_m \big| \leqslant C_{m}|\Phi|^{-1}\sum_{i = 0}^{m-1}\Psi_i^\perp \Psi_{m-1-i}.
\end{equation}
\end{lemm}
\begin{proof}
See Appendix~\ref{sec:proofs_derivative_formulas}.
\end{proof}
\begin{lemm}\label{lemm:bochner_for_derivatives}
Suppose $(\nabla, \Phi)$ is a smooth solution of~\eqref{eq: 2nd_order_crit_pt_intro}. Let $S$ stand for either $\nabla\Phi$ or $\ep F_{\nabla}$, and set
\[
a = \lambda, \text{ if }S = \nabla\Phi;\ \ \ a = 0, \text{ if } S = \ep F_{\nabla}.
\]
Then we have for all $m \geqslant 1$ that
\begin{equation}\label{eq:bochner_for_derivatives}
\begin{split}
\bangle{\nabla^m\nabla^*\nabla S, \nabla^m S} \leqslant\ & \frac{\lambda w_+ }{\ep^2}|\nabla^m S|^2 + \frac{C_{n}}{\ep}\Psi_0\Psi_m|\nabla^m S| + C_{n} |R| |\nabla^m S|^2\\
& + \big(a \cdot (I) + |1-a|\cdot (II) + (III) + (IV)\big) |\nabla^m S|,
\end{split}
\end{equation}
where $(I)$ to $(IV)$ stand for the following expressions:
\begin{equation}\label{eq:bochner_for_derivatives_remainders}
\begin{array}{ll}
&\displaystyle (I) =  \frac{C_{m, n}}{\ep^2}\sum_{i = 0}^{m-1} \Big(\sum_{j+k = m-i} |\nabla^j\Phi||\nabla^k\Phi|\Big) |\nabla^{i}S|,\\
&\displaystyle(II) = \frac{C_{m, n}}{\ep^2}\sum_{i = 0}^{m-1}|\nabla^{m-i}\Phi| \big(|\nabla^i[S, \Phi]| + |[\nabla^i S, \Phi]| \big) +  \frac{C_{m, n}}{\ep^2}\sum_{i=0}^{m-1} |\nabla^i S||[\nabla^{m-i}\Phi, \Phi]|,\\
&\displaystyle (III) = \frac{C_{m, n}}{\ep}\sum_{i = 1}^{m-1}\Psi_i \Psi_{m-i},\ \ \ \ \displaystyle(IV) = C_{m, n}\sum_{i = 0}^{m-1}|\nabla^{m-i}R||\nabla^i S|.
\end{array}
\end{equation}
Moreover, alternative estimates hold in the following special cases:
\vskip 1mm
\begin{enumerate}
\item[(i)] When $S = \nabla\Phi$, the term $C_{n}\ep^{-1}\Psi_0\Psi_m|\nabla^{m + 1} \Phi|$ on the right-hand side of~\eqref{eq:bochner_for_derivatives} can be replaced by $C_{n}\ep^{-1}\Psi_0 |\nabla^{m + 1}\Phi|^2$, in which case $(III)$ should be replaced by 
\[
(III)_{\nabla\Phi} = \frac{C_{m, n}}{\ep}\sum_{i = 0}^{m-1}\Psi_{m-i}|\nabla^{i + 1}\Phi|.
\]
\vskip 1mm
\item[(ii)] On $M \setminus Z(\Phi)$, the term $C_{n}\ep^{-1}\Psi_0\Psi_m|\nabla^m S|$ on the right-hand side of~\eqref{eq:bochner_for_derivatives} can be replaced by 
$C_{n}\ep^{-1}\Psi_0^{\perp}\Psi_m|\nabla^m S|$, in which case $(III)$ is replaced by 
\[
(III)_{M \setminus Z} = \frac{C_{m, n}}{\ep}\sum_{i = 0}^{m-1}\Psi_{m-i}^\perp \Psi_i.
\]
\end{enumerate}
\end{lemm}
\begin{proof}
See Appendix~\ref{sec:proofs_derivative_formulas}.
\end{proof}

\begin{lemm}\label{lemm:trans_laplacian_diffed}
Suppose $(\nabla, \Phi)$ is a smooth solution of~\eqref{eq: 2nd_order_crit_pt_intro}, and again let $S$ stand for either $\nabla\Phi$ or $\ep F_{\nabla}$. Then for all $m \geqslant 1$, we have on $M \setminus Z(\Phi)$ that
\begin{equation}\label{eq:trans_laplacian_diffed}
\begin{split}
\bangle{\nabla^m\nabla^*\nabla[S, \Phi], \nabla^m[S,\Phi]}\leqslant\ & C_{m, n}\Theta_{m + 1}\cdot (|\Phi|^{-1}\Psi_0\Theta_m)\\
&+ \big( \frac{2\lambda w_+}{\ep^2} + C_{m, n} |R| \big) |\nabla^m[S, \Phi]|^2\\
&+ C_{m, n}\big(\frac{|\Phi|\Psi_0}{\ep} + |\Phi|^{-1}\Psi_0^2 + \Psi_1\big) \Psi_m^\perp \Theta_m \\
& + ((\lambda + 1)\cdot(I) + (II) + (III) + (IV))\cdot \Theta_m,
\end{split}
\end{equation}
where we set
\begin{equation}\label{eq:trans_laplacian_diffed_remainders}
\begin{split}
(I)=\ & \frac{C_{m, n}}{\ep^2} \sum_{i = 0}^{m-1} \Big(\sum_{j + k = m-i} |\nabla^j \Phi||\nabla^k\Phi|\Big) \Theta_i,\\
(II) =\ & \frac{C_{m, n}}{\ep}\sum_{i=0}^{m-1}\Big(\sum_{j + k = m-i}  \Psi_j |\nabla^k\Phi|\Big)\Psi_{i}^\perp,\\
(III) =\ & C_{m ,n}\sum_{i = 0}^{m-1} |\nabla^{m-i}R|\Theta_i,\ \ \ \ (IV) = C_{m, n}\sum_{i = 0}^{m-1}\Psi_i^\perp \big( |\Phi|^{-1}\Psi_0\Psi_{m-i} + \Psi_{m+1-i} \big).
\end{split}
\end{equation}
\end{lemm}
\begin{proof}
See Appendix~\ref{sec:proofs_derivative_formulas}.
\end{proof}
\subsection{Coarse estimates}\label{subsec:coarse-estimates}
Here, and for the remainder of Section~\ref{sec:estimates}, we restrict ourselves back to $3$-manifolds. Let $\Omega$ be an open subset of $(M^3, g)$ for which there are constants $\rho_0, A_0, A_1, A_2, \cdots$ such that the following two conditions hold: First, 
\begin{equation}\label{eq:injectivity_radius_for_estimates}
\inj_g(x) \geqslant 2\rho_0 > 0 \text{ for all }x \in \Omega.
\end{equation}
Secondly, with $\rho_0$ as above, for all $k \geqslant 0$ we have
\begin{equation}\label{eq:curvature_bound_for_estimates}
\rho_0^{k + 2}\|(\nabla^g)^k R_g\|_{\infty; \Omega} \leqslant A_k < \infty,
\end{equation}
where $R_g$ denotes the Riemann curvature tensor and $\nabla^g$ the Levi--Civita connection of $g$. By the Hessian comparison theorem, there exist universal constants $c_0\in(0, \frac{1}{4})$ and $C_0 > 0$ such that if 
\begin{equation}\label{eq:radius_rel_curvature}
\big( \frac{\rho}{\rho_0} \big)^2  < \min\{\frac{c_0}{A_0}, 1\} = : \mu_1^2,
\end{equation}
or, equivalently, if
\begin{equation}\label{ineq: rho_1_def}
\rho < \mu_1 \rho_0 =: \rho_1,
\end{equation}
then on any geodesic ball $B_{\rho}(x_0)$ contained in $\Omega$ we have that
\begin{equation}\label{eq:Hessian_comparison_for_estimates}
-C_0A_0 \cdot  \big(\frac{d(\cdot, x_0)}{\rho_0}\big)^2 g \leqslant \mathrm{Hess}(\frac{d(\cdot, x_0)^2}{2}) - g \leqslant C_0A_0 \cdot  \big(\frac{d(\cdot, x_0)}{\rho_0}\big)^2 g,
\end{equation}
\begin{equation}\label{eq:metric_bounds_for_estimates}
\frac{1}{2}g_{\RR^3} \leqslant \exp_{x_0}^* g \leqslant 2 g_{\RR^3},
\end{equation}
and it follows from~\eqref{eq:metric_bounds_for_estimates} that 
\begin{equation}\label{eq:volume_bounds_for_estimates}
\frac{1}{C} \leqslant \frac{\Vol_{g}(B_{\rho}(x_0))}{\rho^3} \leqslant C, 
\end{equation}
for some universal constant $C$. Furthermore, still with $B_{\rho}(x_0) \subset \Omega$ and $\rho < \rho_1$, since geodesics with length at most $2\rho$ between pairs of points in $B_{\rho}(x_0)$ are minimizing (due to the lower bound~\eqref{eq:injectivity_radius_for_estimates}), we see that in particular $\diam B_{\rho}(x_0) = 2\rho$, which in turn implies that if $B_{r}(x) \subset B_{s}(y) \subset \Omega$ with $s < \rho_1$, then $r \leqslant s$. For later use, we also note that if $\rho < \rho_1$, then we have by~\eqref{eq:radius_rel_curvature} that
\begin{equation}\label{eq:curvature_bound_usage}
\rho_0^{-2}A_0 \leqslant \rho^{-2}c_0.
\end{equation}

As a starting point for the estimates in this section, we derive a pointwise bound on $\Psi_0$ and an integral bound on $\Psi_1$ in terms of the integral of $\Psi_0^2$ (Lemma~\ref{lemm:coarse_estimate_base}). We then combine the bound on $\Psi_0$ with the fundamental theorem of calculus to get a pointwise estimate on $|1 - |\Phi||$ in terms of the integral of $\frac{\lambda w^2}{\ep^2}$ (Lemma~\ref{lemm:w_mean_value_estimate}). This estimate on $|1 - |\Phi||$ in turn enables us to begin an induction argument based on Lemma~\ref{lemm:coarse_estimate_base} whereby estimates on $\Psi_k$ are obtained (Proposition~\ref{prop:coarse_estimate}).

\begin{lemm}\label{lemm:coarse_estimate_base}
Suppose $\lambda \in (0, \lambda_0]$ and that $(\nabla, \Phi)$ is a smooth solution of~\eqref{eq: 2nd_order_crit_pt_intro} on $\Omega$ satisfying for some $B_{4\rho}(x_0) \subset \Omega$ with $\rho \in (0, \frac{\rho_1}{4})$ and some $\Lambda> 0$ that
\begin{equation}\label{eq:e_estimate_energy_bound}
\int_{B_{4\rho}(x_0)} \Psi_0^2 \vol_g =: \ep \cdot \eta \leqslant \ep \cdot \Lambda.
\end{equation}
Then, provided $\ep \leqslant \rho$, we have 
\begin{equation}\label{eq:e_estimate}
\|\Psi_0\|_{\infty; B_{\frac{7\rho}{2}}(x_0)} \leqslant C_{\lambda_0, \Lambda}\ep^{-1}\eta^{\frac{1}{2}},
\end{equation}
\begin{equation}\label{eq:Psi_1_estimate}
\int_{B_{3\rho}(x_0)}\Psi_1^2 \leqslant C_{\lambda_0, \Lambda} \ep^{-1} \eta.
\end{equation}
\end{lemm}
\begin{proof}
Noting that $w_+ = \frac{1}{2}(1 - |\Phi|^2)_+ \leqslant \frac{1}{2}$, estimating the two cubic terms on the second line of~\eqref{eq:Psi_0_bochner} in the straightforward manner, and also recalling~\eqref{eq:curvature_bound_for_estimates}, we get on $B_{4\rho}(x_0)$ that
\begin{equation}\label{eq:Psi_0_bochner_for_basic_estimate}
\begin{split}
\Psi_1^2 + \Delta(\frac{\Psi_0^2}{2}) \leqslant\ &  C(\ep^{-2}\lambda + \ep^{-1}\Psi_0 + \rho_0^{-2}A_0)\Psi_0^2,
\end{split}
\end{equation}
where $C$ is a dimensional constant. Given $x \in B_{\frac{7\rho}{2}}(x_0)$, we estimate with the help of~\eqref{eq:e_estimate_energy_bound} and~\eqref{eq:volume_bounds_for_estimates}, as well as the inclusion $B_{\frac{\ep}{2}}(x) \subset B_{4\rho}(x_0)$, that
\begin{equation}\label{eq:Psi_0_bochner_coefficient_L2}
\begin{split}
\|\ep^{-2}\lambda + \ep^{-1}\Psi_0 + \rho_0^{-2}A_0\|_{2; B_{\frac{\ep}{2}}(x)} \leqslant\ & C\ep^{-\frac{1}{2}}\lambda  + \ep^{-1} \cdot (\ep\Lambda)^{\frac{1}{2}} + C\ep^{\frac{3}{2}}\rho_0^{-2}A_0\\
\leqslant\ & C\ep^{-\frac{1}{2}}(\lambda + \Lambda^{\frac{1}{2}} + 1),
\end{split}
\end{equation}
where for the second inequality we used~\eqref{eq:curvature_bound_usage}, and still $C$ is a universal constant. Thanks to this $L^2$-estimate and also the metric bounds~\eqref{eq:metric_bounds_for_estimates}, we may apply Lemma~\ref{lemm:moser-improve}(b) to the differential inequality~\eqref{eq:Psi_0_bochner_for_basic_estimate}, with $n = 3$, $q = 2$ and $r = \frac{\ep}{4}$, to obtain
\[
\Psi_0^2(x) \leqslant C(1 + \lambda + \Lambda^{\frac{1}{2}})^{6} \cdot \ep^{-3}\int_{B_{4\rho}(x_0)}\Psi_0^2 \leqslant C_{\lambda_0, \Lambda}\ep^{-2}\eta,
\]
where the last inequality follows from~\eqref{eq:e_estimate_energy_bound}. This proves~\eqref{eq:e_estimate}. To prove~\eqref{eq:Psi_1_estimate}, we first use~\eqref{eq:e_estimate} and~\eqref{eq:curvature_bound_usage} to deduce from~\eqref{eq:Psi_0_bochner_for_basic_estimate} that 
\begin{equation}\label{eq:Psi_1_inequality}
\Psi_1^2 + \Delta(\frac{\Psi_0^2}{2}) \leqslant C_{\lambda_0, \Lambda}\ep^{-2}\Psi_0^2, \text{ on }B_{\frac{7\rho}{2}}(x_0).
\end{equation} 
Now choose a cut-off function $\zeta$ such that
\[
\zeta = 1 \text{ on }B_{3\rho}(x_0),\ \ \zeta = 0 \text{ outside of }B_{\frac{7\rho}{2}}(x_0),\ \ |\nabla\zeta| \leqslant C \rho^{-1}.
\]
Testing~\eqref{eq:Psi_1_inequality} against $\zeta^2$ and using~\eqref{eq:Psi_Theta_schwarz} gives
\[
\begin{split}
\int_{M} \zeta^2 \Psi_1^2 \leqslant\ & \int_{M} C_{\lambda_0, \Lambda}\ep^{-2}\zeta^2\Psi_0^2  + 2 \zeta|\nabla\zeta|\Psi_0 \Psi_1\\
\leqslant\ & \int_{M} (C_{\lambda_0, \Lambda}\ep^{-2}\zeta^2 + 2|\nabla\zeta|^2)\Psi_0^2 + \frac{1}{2}\int_{M} \zeta^2 \Psi_1^2,
\end{split}
\]
where we used Young's inequality for the second line. Rearranging and recalling the properties of $\zeta$, we obtain
\[
\int_{B_{3\rho}(x_0)} \Psi_1^2 \leqslant C_{\lambda_0, \Lambda}(\ep^{-2} + \rho^{-2})\int_{B_{4\rho}(x_0)}\Psi_0^2,
\]
from which we deduce~\eqref{eq:Psi_1_estimate} upon recalling~\eqref{eq:e_estimate_energy_bound} and the assumption $\ep \leqslant \rho$.
\end{proof}

\begin{lemm}\label{lemm:w_mean_value_estimate}
Under the hypotheses of Lemma~\ref{lemm:coarse_estimate_base}, if we assume in addition that
\begin{equation}\label{eq:e_estimate_potential_bound}
\int_{B_{4\rho}(x_0)} \frac{\lambda w^2}{\ep^2} \vol_g \leqslant \ep \cdot \Lambda',
\end{equation}
then
\begin{equation}\label{eq:w_estimate}
\|1 - |\Phi|\|_{\infty; B_{\frac{13\rho}{4}}(x_0)} \leqslant C_{\lambda_0, \Lambda} \cdot \max\big\{\big(\frac{\Lambda'}{\lambda} \big)^{\frac{1}{4}}, \big(\frac{\Lambda'}{\lambda} \big)^{\frac{1}{7}}\big\}.
\end{equation}
\end{lemm}
\begin{proof}
To prove~\eqref{eq:w_estimate}, we first note from~\eqref{eq:e_estimate} that
\begin{equation}\label{eq:Phi_bound_for_mean_value}
|\nabla \Phi(x)| \leqslant C_1\ep^{-1} \text{ for all }x \in B_{\frac{7\rho}{2}}(x_0),
\end{equation}
for some $C_1$ depending only on $\lambda_0$ and $\Lambda$. Now, take any $x_1 \in B_{\frac{13\rho}{4}}(x_0)$ where $K: = \big| 1 - |\Phi(x_1)| \big| > 0$. Then, letting 
\[
\sigma = \frac{\min\{K, 1\}}{4(1 + C_1)} \cdot \ep,
\]
with $C_1 = C_1(\lambda_0, \Lambda)$ being the constant from~\eqref{eq:Phi_bound_for_mean_value}, we have $B_{\sigma}(x_1) \subset B_{\frac{7\rho}{2}}(x_0)$. Given $x \in B_{\sigma}(x_1)$ and $\delta > 0$, by integrating along the geodesic segment from $x_1$ to $x$, which lies entirely in $B_{\sigma}(x_1)$, and using the bound~\eqref{eq:Phi_bound_for_mean_value}, we have
\[
\begin{split}
\Big|\big(|\Phi(x)|^2 + \delta^2\big)^{\frac{1}{2}} - \big(|\Phi(x_1)|^2 + \delta^2\big)^{\frac{1}{2}}\Big| \leqslant C_1\ep^{-1}\sigma.
\end{split}
\]
Sending $\delta \to 0$ and using the triangle inequality gives
\[
\big| 1 - |\Phi(x)| \big| \geqslant \frac{3K}{4} \text{ for all }x \in B_{\sigma}(x_1).
\]
Raising to the 4th power, integrating over $B_{\sigma}(x_1)$, and using~\eqref{eq:volume_bounds_for_estimates} along with the inequality $(1 - t)^2 \leqslant (1 + t)^2$ for $t \geqslant 0$, we get
\[
\begin{split}
C K^4 \sigma^3 \leqslant\ & \int_{B_{\sigma}(x_1)} (1 - |\Phi|)^4 \leqslant \int_{B_{4\rho}(x_0)} (1 - |\Phi|)^2(1 + |\Phi|)^2 \leqslant \frac{4\ep^3 \Lambda'}{\lambda},
\end{split}
\]
where $C$ is a universal constant, and for the last inequality we used~\eqref{eq:e_estimate_potential_bound}. Recalling the definition of $\sigma$, we infer that 
\[
K^4 \cdot \big(\min\{K, 1\}\big)^3 \leqslant C_{\lambda_0, \Lambda} \cdot \frac{\Lambda'}{\lambda}.
\]
Since $x_1$ is an arbitrary point in $B_{\frac{13\rho}{4}}(x_0)$ with $|1 - |\Phi(x_1)|| > 0$, we conclude that~\eqref{eq:w_estimate} holds.
\end{proof}

As mentioned above, Lemma~\ref{lemm:coarse_estimate_base} provides the base step for an induction argument that is the content of the next proposition. 
\begin{prop}\label{prop:coarse_estimate}
Suppose $\lambda \in (0, \lambda_0]$. Given constants $\Lambda, K_0 > 0$, let $(\nabla, \Phi)$ be a smooth solution of~\eqref{eq: 2nd_order_crit_pt_intro} such that, on some $B_{4\rho}(x_0) \subset \Omega$ with $\rho \in (0, \frac{\rho_1}{4})$, we have
\begin{equation}\label{eq:energy_bound_for_coarse}
\int_{B_{4\rho}(x_0)} \Psi_0^2 \vol_g =: \ep \cdot \eta \leqslant \ep \cdot \Lambda,
\end{equation}
and that 
\begin{equation}\label{eq:Phi_K_0}
\||\Phi|\|_{\infty; B_{3\rho}(x_0)} \leqslant K_0.
\end{equation}
Then, provided also $\ep \leqslant \rho$, there holds for all $k \in \NN$ that
\begin{equation}\label{eq:coarse_estimate}
\|\Psi_k\|_{\infty; B_{2\rho}(x_0)} \leqslant C \ep^{-k-1} \eta^{\frac{1}{2}},
\end{equation}
where $C$ depends only on $k, \Lambda, \lambda_0, K_0$ and the bounds $A_1, \cdots, A_k$ from~\eqref{eq:curvature_bound_for_estimates}.
\end{prop}
\begin{proof}
Define $\rho_k = (2 + 2^{-k})\rho$. We prove by induction that for all $k \in \NN \cup \{0\}$ there exists $C = C(k, \Lambda, \lambda_0, K_0, \{A_i\}_{1 \leqslant i \leqslant k})$ such that 
\begin{subequations}
\begin{align}
\|\Psi_k\|_{\infty; B_{\rho_{k + 1}}(x_0)} \leqslant\ & C \ep^{-k-1}\eta^{\frac{1}{2}}, \label{eq:coarse_estimate_induction_pointwise}\\
\int_{B_{\rho_{k + 1}}(x_0)}\Psi_{k+1}^2 \leqslant\ & C\ep^{-2k -1} \eta. \label{eq:coarse_estimate_induction_L2}
\end{align}
\end{subequations}
Here and below, when $k = 0$, it is understood that $\{A_i\}_{1 \leqslant i \leqslant k} = \emptyset$, and $\sum_{i = 1}^{k}(\cdots)$ should be interpreted as $0$.

The base step, namely the two above estimates for $k = 0$, is already established in Lemma~\ref{lemm:coarse_estimate_base}. For the induction step, we assume that~\eqref{eq:coarse_estimate_induction_pointwise} and~\eqref{eq:coarse_estimate_induction_L2} hold for $k = 0, \cdots, m-1$ for some $m \in \NN$. Letting $S$ denote either $\nabla\Phi$ or $\ep F_{\nabla}$, by  Lemma~\ref{lemm:commutator-estimate} and the inequality~\eqref{eq:tensor-bracket-norm}, we have that
\begin{equation}\label{eq:coarse_estimate_induction_commutator_1}
\begin{split}
\big|[\nabla^*\nabla, \nabla^m] S \big| \leqslant\ & C_{m}\sum_{i = 0}^{m} |[\nabla^{m-i}  F_{\nabla}, \nabla^i S]| + C_{m}\sum_{i=0}^{m}|\nabla^{m-i}R| |\nabla^i S|\\
\leqslant \ & \frac{C_{m}}{\ep}\sum_{i = 0}^{m} |\nabla^{m-i} (\ep F_{\nabla})| |\nabla^i S| + C_{m}\sum_{i=0}^{m} \rho_0^{i-m-2}A_{m-i} |\nabla^i S| \\
\leqslant\ & C_{ m}\big(\frac{\Psi_0}{\ep}+ \rho_0^{-2}A_0\big)\Psi_m + \frac{C_{m}}{\ep}\sum_{i=1}^{m-1}\Psi_{m-i}\Psi_i  + C_{m}\sum_{i = 0}^{m-1}\rho_0^{i-m-2}A_{m-i}\Psi_i,
\end{split}
\end{equation}
where in getting the second line we also used~\eqref{eq:curvature_bound_for_estimates}. From Lemma~\ref{lemm:coarse_estimate_base} and~\eqref{eq:curvature_bound_usage}, we have
\[
\frac{\Psi_0}{\ep} + \rho_0^{-2}A_0 \leqslant C_{\lambda_0, \Lambda}\ep^{-2} \text{ on }B_{3\rho}(x_0).
\]
Substituting this into~\eqref{eq:coarse_estimate_induction_commutator_1}, using the induction hypothesis to bound the terms $\Psi_{m-i}$ in the first summation on the last line, and noting also that $\rho_0^{i-m-2} \leqslant \ep^{i-m-2}$, we deduce that
\begin{equation}\label{eq:coarse_estimate_induction_commutator_2}
\begin{split}
\bangle{[\nabla^*\nabla, \nabla^m] S, \nabla^m S} \leqslant\ & C_{m, \Lambda, \lambda_0}\ep^{-2}\Psi_m^2 + C_{m, \Lambda, \lambda_0, K_0, \{A_{k}\}_{1\leqslant k \leqslant m}}\sum_{i = 0}^{m-1}\ep^{i-m-2}\Psi_i\Psi_m,
\end{split}
\end{equation}
on $B_{\rho_m}(x_0)$. Next we want to estimate $\bangle{\nabla^m \nabla^*\nabla S, \nabla^m S}$ with the help of Lemma~\ref{lemm:bochner_for_derivatives}. As preparation, note that by~\eqref{eq:Theta_by_Psi_basic} and Lemma~\ref{lemm:coarse_estimate_base}, on $B_{3\rho}(x_0)$ we have
\[
\Theta_0 \leqslant C_{\lambda_0, \Lambda} \cdot K_0 \ep^{-1}.
\]
By the first conclusion of Lemma~\ref{lemm:Psi_Theta_relations} and the assumption that~\eqref{eq:coarse_estimate_induction_pointwise} holds up to $k = m-1$, provided $m \geqslant 2$ we have for all $1 \leqslant l \leqslant m-1$ that, on $B_{\rho_m}(x_0)$,
\[
\Theta_l \leqslant C_{m, \Lambda, \lambda_0, K_0, \{A_{k}\}_{1 \leqslant k \leqslant m-1}}\ep^{-l-1}.
\]
The two above inequalities, along with~\eqref{eq:coarse_estimate_induction_pointwise} for $k = 0, \cdots, m-1$, help us bound the term $(II)$ in Lemma~\ref{lemm:bochner_for_derivatives} on $B_{\rho_m}(x_0)$ as follows:
\begin{equation}\label{eq:bochner_for_derivative_II_bound}
\begin{split}
|(II)| \leqslant \ & \frac{C_m}{\ep^2}\sum_{i = 0}^{m-1}(\Theta_{m-1-i} \Psi_i + |\Phi|\Psi_{m-1-i} \Psi_i)\\
\leqslant\ & C_{m, \Lambda, \lambda_0, K_0, \{A_k\}_{1 \leqslant k \leqslant m-1}} \sum_{i=0}^{m-1} \ep^{i-m-2}\Psi_i. 
\end{split}
\end{equation}
Using the induction hypothesis and~\eqref{eq:curvature_bound_for_estimates} to estimate the terms $(I), (III)$, and $(IV)$ in a similar fashion, we infer from Lemma~\ref{lemm:bochner_for_derivatives} that
\begin{small}
\[
\begin{split}
\bangle{\nabla^m\nabla^*\nabla S, \nabla^m S} \leqslant\ & \big( \frac{\lambda w_+}{\ep^2}+ \frac{C\Psi_0}{\ep} + C\rho_0^{-2}A_0 \big)\Psi_m^2+ C_{m, \Lambda, \lambda_0, K_0, \{A_{k}\}_{1 \leqslant k \leqslant m-1}}\sum_{i=0}^{m-1}\ep^{i-m-2}\Psi_i\Psi_m\\
&+ C_{m}\sum_{i = 0}^{m-1}\rho_0^{i - m - 2}A_{m-i}\Psi_i\Psi_m\ \ \text{ on }B_{\rho_m}(x_0).
\end{split}
\]
\end{small}
Using Lemma~\ref{lemm:coarse_estimate_base} and~\eqref{eq:curvature_bound_usage} to estimate $\ep^{-1}\Psi_0$ and $\rho_0^{-2}A_0$, respectively, while also noting that $\lambda w_+ \leqslant \lambda_0$ and that $\rho_0^{i-m-2} \leqslant \ep^{i-m-2}$, we get
\begin{equation}\label{eq:coarse_estimate_induction_diffed}
\begin{split}
\bangle{\nabla^m\nabla^*\nabla S, \nabla^m S} \leqslant\ & C_{\Lambda, \lambda_0}\ep^{-2}\Psi_m^2 + C_{m, \Lambda, \lambda_0, K_0, \{A_{k}\}_{1 \leqslant k \leqslant m}}\sum_{i=0}^{m-1}\ep^{i-m-2}\Psi_i\Psi_m, \text{ on }B_{\rho_m}(x_0),
\end{split}
\end{equation}
Combining~\eqref{eq:coarse_estimate_induction_diffed} and~\eqref{eq:coarse_estimate_induction_commutator_2} and recalling~\eqref{eq:Psi_laplacian} gives
\begin{equation}\label{eq:coarse_estimate_induction_diff_ineq_int}
\Psi_{m + 1}^2 + \Delta(\frac{\Psi_m^2}{2}) \leqslant C_{m,\Lambda,\lambda_0}\ep^{-2}\Psi_m^2 + C_{m, \Lambda,\lambda_0, K_0, \{A_{k}\}_{1 \leqslant k \leqslant m}}\sum_{i=0}^{m-1}\ep^{i-m-2}\Psi_i\Psi_m,
\end{equation}
on $B_{\rho_m}(x_0)$. Using the induction hypothesis again and applying Young's inequality, we obtain, still on $B_{\rho_m}(x_0)$, that
\begin{equation}\label{eq:coarse_estimate_induction_diff_ineq}
\begin{split}
\Psi_{m + 1}^2 + \Delta(\frac{\Psi_m^2}{2}) \leqslant\ &  C_{m,\Lambda,\lambda_0}\ep^{-2}\Psi_m^2 + C_{m, \Lambda, \lambda_0, K_0, \{A_{k}\}_{1 \leqslant k \leqslant m}}\ep^{-m-3}\eta^{\frac{1}{2}}\Psi_m \\
\leqslant\ & C\ep^{-2}\Psi_m^2 + C\ep^{-2m-4}\eta,
\end{split}
\end{equation}
with the constants $C$ having the admissible dependencies. Now we define
\[
r = \min\{\ep, \rho_{m} - \rho_{m + 1}\}.
\]
Then, given $x \in B_{\rho_{m + 1}}(x_0)$, since $B_{r}(x) \subset B_{\rho_m}(x_0)$ and $r \sim_{m} \ep$, upon applying Lemma~\ref{lemm:moser-improve}(b) to~\eqref{eq:coarse_estimate_induction_diff_ineq}, with $q = \infty$, we obtain 
\begin{equation}\label{eq:coarse_induction_after_moser}
\|\Psi_m^2\|_{\infty; B_{\frac{r}{4}}(x)}\leqslant C \big( \ep^{-3}\int_{B_{\rho_m}(x_0)}\Psi_m^2 \vol_g + \ep^{-2m-2}\eta \big).
\end{equation}
Using~\eqref{eq:coarse_estimate_induction_L2} with $k = m-1$ to estimate the right-hand side leads to 
\[
\Psi_m(x)^2 \leqslant C\ep^{-2m-2}\eta.
\]
Since $x \in B_{\rho_{m + 1}}(x_0)$ is arbitrary, we have proved~\eqref{eq:coarse_estimate_induction_pointwise} for $k = m$. To get~\eqref{eq:coarse_estimate_induction_L2} for $k = m$, let $\zeta$ be a cut-off function such that
\[
\zeta = 1 \text{ on }B_{\rho_{m + 1}}(x_0),\ \ \zeta = 0 \text{ outside of }B_{\rho_m}(x_0),\ \ |\nabla \zeta| \leqslant C_{m}\rho^{-1}.
\]
Multiplying~\eqref{eq:coarse_estimate_induction_diff_ineq_int} by $\zeta^2$ and integrating by parts while using~\eqref{eq:Psi_Theta_schwarz} and H\"older's inequality, we get
\[
\begin{split}
\int_{M} \zeta^2\Psi_{m+1}^2 \leqslant\ & C\ep^{-2} \int_{M}\zeta^2\Psi_m^2 + \int_{M}2\zeta|\nabla\zeta| \Psi_m \Psi_{m + 1} \\
&+  C\sum_{i=0}^{m-1}\ep^{i-m-2}\Big(\int_{B_{\rho_m}(x_0)} 
\Psi_i^2 \Big)^{\frac{1}{2}} \Big(\int_{B_{\rho_m}(x_0)} \Psi_m^2\Big)^{\frac{1}{2}},
\end{split}
\]
where the constants $C$ depend only on $m, \Lambda, \lambda_0, K_0, A_1, \cdots, A_m$. Applying Young's inequality in the second integral on the right-hand side and rearranging, and also using the assumption~\eqref{eq:energy_bound_for_coarse} and the induction hypothesis, respectively, to bound the terms corresponding to $i = 0$ and $1 \leqslant i\leqslant m-1$ in the summation, we obtain
\begin{small}
\[
\begin{split}
\int_{B_{\rho_{m + 1}}(x_0)}\Psi_{m + 1}^2 \vol_{g} \leqslant\ & C(\ep^{-2} + \rho^{-2})\int_{B_{\rho_m}(x_0)}\Psi_m^2 \vol_{g}+ C\sum_{i=0}^{m-1} \ep^{i-m-2} \big(  \ep^{-2i + 1}\eta \big)^{\frac{1}{2}}  \big(  \ep^{-2m + 1}\eta \big)^{\frac{1}{2}}\\
\leqslant\ & C(\ep^{-2} + \rho^{-2})\int_{B_{\rho_m}(x_0)}\Psi_m^2 \vol_{g} + C\ep^{-2m-1}\eta.
\end{split}
\]
\end{small}
Using~\eqref{eq:coarse_estimate_induction_L2} with $k = m-1$ once more and noting that $\rho^{-1} \leqslant \ep^{-1}$ gives~\eqref{eq:coarse_estimate_induction_L2} for $k = m$. The proof is now complete.
\end{proof}
\subsection{Exponential decay}\label{subsec:exp-decay}
We continue to work in the setting of \S\ref{subsec:coarse-estimates}. In this section, under suitable smallness conditions, we first establish exponential decay estimates for the transversal components of $F_{\nabla}$ and $\nabla\Phi$. Then, relying on the assumption $\lambda > 0$, we derive similar estimates for $\nabla\Phi$ itself and $w$. The argument closely follows~\cite[Chapters IV.12 and IV.13]{Jaffe-Taubes}. We begin by recalling the following standard barrier construction.
\begin{lemm}\label{lemm:barrier}
Given $L > 0$, there exists a constant $A$ depending only on $L$ such that for any $B_{L\rho}(x_0) \subset \Omega$ with $\rho < \frac{\rho_1}{L}$, letting $r(\cdot)$ denote the geodesic distance to $x_0$, we have, provided also $\ep \leqslant \rho$, that the function $\varphi(x) = e^{\frac{r(x)^2}{A\ep\rho}}$ satisfies
\begin{equation}\label{eq:barrier}
\Delta\varphi \geqslant - \frac{1}{2\ep^2}\varphi, \text{ on }B_{L\rho}(x_0).
\end{equation}
\end{lemm}
\begin{proof}
By a direct computation with the help of~\eqref{eq:Hessian_comparison_for_estimates} we have
\begin{equation}\label{eq:barrier_initial_ineq}
\begin{split}
\Delta \varphi =\ & \varphi \cdot \big( \frac{\Delta (r^2)}{A\ep \rho} - \frac{|d(r^2)|^2}{A^2 \ep^2 \rho^2} \big) \geqslant \varphi \cdot \big( \frac{-2 \cdot \dim M -C\rho_0^{-2}A_0r^2}{A\ep\rho} - \frac{4r^2}{A^2\ep^2\rho^2} \big),
\end{split}
\end{equation}
where $C$ is a dimensional constant. Since $\ep \leqslant \rho$ and since $r \leqslant L\rho < \rho_1 = \mu_1\rho_0$, we have
\[
\frac{1}{A\ep\rho} \leqslant \frac{1}{A\ep^2},\ \  \ \frac{r^2}{\rho^2} \leqslant L^2  ,\ \ \text{ and } \ \frac{r^2 A_0}{\rho_0^2} \leqslant c_0.
\]
Hence we deduce from~\eqref{eq:barrier_initial_ineq} that
\[
\Delta \varphi \geqslant -\varphi \cdot \big( \frac{6}{A\ep^2} + \frac{C \cdot c_0}{A\ep^2} + \frac{4L^2}{A^2\ep^2} \big) = -\frac{\varphi}{\ep^2}\cdot \big( \frac{6 + C \cdot c_0}{A} + \frac{4L^2}{A^2} \big).
\]
It follows that if  
\[
A \geqslant 4 \cdot (6 + C\cdot c_0 + L),
\]
then we get the desired differential inequality for $\varphi$ on $B_{L\rho}(x_0)$.
\end{proof}

As an initial illustration of how smallness assumptions combine with the inequalities in Lemma~\ref{lemm:basic_differential_ineqs}, and as another preliminary result to be used at a later point (Lemma~\ref{lemm:nablaPhi_exp_decay_base}), we record the following improvement of Lemma~\ref{lemm:coarse_estimate_base}.

\begin{lemm}\label{lemm:improved_coarse_estimate_base}
Suppose $\lambda \in (0, \lambda_0]$. There exists a constant $\overline{\eta} \in (0, 1)$, depending only on $\lambda_0$, such that if $(\nabla, \Phi)$ is a smooth solution of~\eqref{eq: 2nd_order_crit_pt_intro} satisfying for some $B_{4\rho}(x_0) \subset \Omega$ with $\rho \in (0, \frac{\rho_1}{4})$ and some $\beta \in (0, \frac{1}{4})$ that 
\begin{equation}\label{eq:small_energy_for_improved_coarse}
\int_{B_{4\rho}(x_0)} \Psi_0^2\vol_g =: \ep \cdot \eta \leqslant \ep \cdot \overline{\eta},
\end{equation}
and that
\begin{equation}\label{eq:w_bound_for_improved_coarse}
\|w\|_{\infty; B_{3\rho}(x_0)} \leqslant \beta,
\end{equation}
then, provided $\ep \leqslant \sqrt{\mu}\rho$, where $\mu = \min\{\lambda, 1\}$, we have 
\begin{equation}\label{eq:Psi_0_improved_bochner}
\Psi_1^2 + \Delta\big( \frac{\Psi_0^2}{2} \big) \leqslant C_{\lambda_0}\big(\frac{\mu w_+}{\ep^2} + \frac{A_0}{\rho_0^2}\big) \Psi_0^2,\text{ on }  B_{3\rho}(x_0),
\end{equation}
and that
\begin{equation}\label{eq:improved_coarse_pointwise}
\|\Psi_0\|_{\infty; B_{\frac{11\rho}{4}}(x_0)} \leqslant C_{\lambda_0}\mu^{\frac{3}{4}} \ep^{-1}\eta^{\frac{1}{2}}.
\end{equation}
\end{lemm}
\begin{proof}
Taking~\eqref{eq:Psi_0_bochner} from Lemma~\ref{lemm:basic_differential_ineqs}(c) and estimating the last two terms using the identity~\eqref{eq: inner_prod_decomp}, we get on $B_{3\rho}(x_0)$ that
\begin{equation}\label{eq:Psi_0_bochner_off_Z}
\begin{split}
\Psi_1^2 + \Delta\big( \frac{\Psi_0^2}{2} \big) \leqslant\ & (\frac{\lambda w_+}{\ep^2} + C\rho_0^{-2}A_0)\Psi_0^2 + \big( C\ep \Psi_0 - |\Phi|^2 \big)\frac{(\Psi_0^{\perp})^2}{\ep^2},
\end{split}
\end{equation}
where the constants $C$ on the right-hand side are universal. Since $\frac{\lambda}{\mu} = \max\{1, \lambda\} \leqslant 1 + \lambda_0$, we have
\begin{equation}\label{eq:Psi_0_bochner_off_Z_leading_term}
\frac{\lambda w_+}{\ep^2} = \frac{\lambda}{\mu} \cdot \frac{\mu w_+}{\ep^2}  \leqslant C_{\lambda_0} \frac{\mu w_+}{\ep^2}.
\end{equation}
On the other hand, by Lemma~\ref{lemm:coarse_estimate_base} and~\eqref{eq:w_bound_for_improved_coarse} we have on $B_{3\rho}(x_0)$ that
\[
C\ep \Psi_0 - |\Phi|^2 \leqslant C_{\lambda_0} \overline{\eta}^{\frac{1}{2}} - (1 - 2\beta)< C_{\lambda_0} \overline{\eta}^{\frac{1}{2}} - \frac{1}{2} < 0,
\]
provided we choose $\overline{\eta} \in (0, 1)$ so that
\[
C_{\lambda_0}\overline{\eta}^{\frac{1}{2}} \leqslant \frac{1}{4}.
\] 
Substituting the previous estimate and~\eqref{eq:Psi_0_bochner_off_Z_leading_term} back into~\eqref{eq:Psi_0_bochner_off_Z}, we get the asserted differential inequality~\eqref{eq:Psi_0_improved_bochner}. Using the bound $w_+ \leqslant \beta < \frac{1}{4}$, the assumption $\ep \leqslant \sqrt{\mu}\rho$, and~\eqref{eq:curvature_bound_usage}, we deduce from~\eqref{eq:Psi_0_improved_bochner} that
\[
\Psi_1^2 + \Delta\big( \frac{\Psi_0^2}{2} \big)  \leqslant C_{\lambda_0}\cdot \mu \ep^{-2} \cdot \Psi_0^2,\text{ on }B_{3\rho}(x_0).
\]
Now, for all $x \in B_{\frac{11\rho}{4}}(x_0)$, we apply Lemma~\ref{lemm:moser-improve}(b) to the above differential inequality, with $q = \infty$ and $r = \frac{\ep}{8\sqrt{\mu}}$, obtaining
\[
\Psi_0^2(x) \leqslant C_{\lambda_0}\mu^{\frac{3}{2}}\ep^{-3}\int_{B_{4\rho}(x_0)} \Psi_0^2.
\]
Since $x \in B_{\frac{11\rho}{4}}(x_0)$ is arbitrary, we deduce~\eqref{eq:improved_coarse_pointwise} upon recalling the assumption~\eqref{eq:small_energy_for_improved_coarse}.
\end{proof}
Next, we establish the exponential decay of $\Theta_k$ and $\Psi_k^{\perp}$. Similar to \S\ref{subsec:coarse-estimates}, we begin with pointwise estimates on $\Theta_0, \Psi_0^{\perp}$ and an integral estimate of $\Theta_1$ (Lemma~\ref{lemm:exp_decay_base}), and then inductively obtain estimates on $\Theta_k$ and $\Psi_k^{\perp}$ (Proposition~\ref{prop:exp_decay}). Between Lemma~\ref{lemm:exp_decay_base} and Proposition~\ref{prop:exp_decay}, we address the issue of verifying the smallness assumption~\eqref{eq:exp_decay_base_small_w} on $w$ in applications (Remark~\ref{rmk:clearing_out}).
\begin{lemm}\label{lemm:exp_decay_base}
Suppose $\lambda \in (0, \lambda_0]$. There exist $\eta_0, \tau_0 \in (0, 1)$, depending only on $\lambda_0$ and $A_1$, with the following property. If $(\nabla, \Phi)$ is a smooth solution of~\eqref{eq: 2nd_order_crit_pt_intro} on $\Omega$ satisfying for some $B_{4\rho}(x_0) \subset \Omega$ with $\rho \in (0, \frac{\rho_1}{4})$ and some $\beta \in (0, \frac{1}{4(\lambda_0 + 1)})$ that
\begin{equation}\label{eq:exp_decay_base_small_energy}
\int_{B_{4\rho}(x_0)} \Psi_0^2\vol_g =: \ep \cdot \eta \leqslant \ep \cdot \eta_0,
\end{equation}
and that
\begin{equation}\label{eq:exp_decay_base_small_w}
\|w\|_{\infty; B_{3\rho}(x_0)} \leqslant \beta,
\end{equation}
and if also $\ep \leqslant \tau_0 \rho$, then the following estimates hold:
\begin{equation}\label{eq:exp_decay_base_pointwise}
\|\Theta_0\|_{\infty; B_{\frac{3\rho}{2}}(x_0)} + \|\Psi_0^\perp\|_{\infty; B_{\frac{3\rho}{2}}(x_0)} \leqslant C \ep^{-1}e^{-a\frac{\rho}{\ep}} \eta^{\frac{1}{2}},
\end{equation}
\begin{equation}\label{eq:exp_decay_base_L2}
\int_{B_{\rho}(x_0)} \Theta_1^2 \leqslant C\ep^{-1}e^{-a\frac{\rho}{\ep}}\eta,
\end{equation}
where in both~\eqref{eq:exp_decay_base_pointwise} and~\eqref{eq:exp_decay_base_L2}, the constant $a$ is dimensional, while $C = C(\lambda_0)$.
\end{lemm}
\begin{proof}
Since $\beta < \frac{1}{4}$, by~\eqref{eq:exp_decay_base_small_w} we have
\begin{equation}\label{eq:exp_decay_base_Phi_two_sided}
\frac{1}{2} \leqslant |\Phi(x)|^2 \leqslant \frac{3}{2},\quad \text{on }B_{3\rho}(x_0).
\end{equation}
In particular, the condition~\eqref{eq:Phi_K_0} in Proposition~\ref{prop:coarse_estimate} is fulfilled with, say, $K_0 = 2$. Also using the fact that $\eta_0 < 1$, we deduce from Lemma~\ref{lemm:coarse_estimate_base} and Proposition~\ref{prop:coarse_estimate} (with $\Lambda = 1$) that
\begin{equation}\label{eq:e_estimate_for_exp_decay_base}
\|\Psi_0\|_{\infty; B_{3\rho}(x_0)} \leqslant C_{\lambda_0}\ep^{-1}\eta^{\frac{1}{2}},
\end{equation}
\begin{equation}\label{eq:Psi_1_estimate_for_exp_decay_base}
\|\Psi_1\|_{\infty; B_{2\rho}(x_0)} \leqslant C_{A_1, \lambda_0}\ep^{-2}\eta^{\frac{1}{2}}.
\end{equation}
Next, using~\eqref{eq:exp_decay_base_small_w} and the assumption $\beta \in(0, \frac{1}{4(1 + \lambda_0)})$ to see that
\begin{equation}\label{eq:exp_decay_small_w_effect}
2\lambda |w| - |\Phi|^2 \leqslant 2\lambda_0\beta - (1-2\beta) = 2(\lambda_0 + 1)\beta - 1 < -\frac{1}{2},
\end{equation}
and applying Young's inequality to the first term on the right-hand side of~\eqref{eq:trans_laplacian_base}, we have on $B_{3\rho}(x_0)$ that
\[
\begin{split}
\bangle{\nabla^*\nabla[S, \Phi], [S,\Phi]} \leqslant\ &  -\frac{1}{2\ep^2}|[S, \Phi]|^2 + \frac{\Theta_1^2}{8}\\
&+ \frac{C}{\ep^2}\big( |\Phi|^{-2}\ep^2\Psi_0^2 + |\Phi|^{-1}\ep^2\Psi_1 + \ep \Psi_0 + \ep^2\rho_0^{-2}A_0 \big)\Theta_0^2.
\end{split}
\]
Incorporating the estimates~\eqref{eq:e_estimate_for_exp_decay_base} and~\eqref{eq:Psi_1_estimate_for_exp_decay_base}, using~\eqref{eq:exp_decay_base_Phi_two_sided} to bound $|\Phi|^{-1}$, and also noting that $\ep^2\rho_0^{-2}A_0  \leqslant \tau_0^2 c_0$, we obtain on $B_{2\rho}(x_0)$ that
\[
\begin{split}
\bangle{\nabla^*\nabla[S, \Phi], [S,\Phi]}  \leqslant\ &  \frac{\Theta_1^2}{8} -\frac{1}{2\ep^2}|[S, \Phi]|^2  + \frac{C_{\lambda_0, A_1}}{\ep^2}\big( \eta_0^{\frac{1}{2}} + \tau_0^2 \big)\Theta_0^2,
\end{split}
\]
and consequently
\begin{equation}\label{eq:exp_decay_base_ineq}
\begin{split}
\Theta_1^2 + \Delta(\frac{\Theta_0^2}{2}) =\ & \bangle{\nabla^*\nabla[\nabla\Phi, \Phi], [\nabla\Phi, \Phi]} + \ep^2 \bangle{\nabla^*\nabla[F_{\nabla}, \Phi], [F_{\nabla}, \Phi]}\\
\leqslant\ & \frac{\Theta_1^2 }{4}- \frac{1}{2\ep^2}\Theta_0^2  +  \frac{C_{A_1, \lambda_0}}{\ep^2}(\eta_0^{\frac{1}{2}} + \tau_0^2)\Theta_0^2.
\end{split}
\end{equation}
Decreasing $\eta_0$ and $\tau_0$ so that
\[
C_{A_1, \lambda_0}(\eta_0^{\frac{1}{2}} + \tau_0^2) < \frac{1}{4}, 
\]
we get after rearranging~\eqref{eq:exp_decay_base_ineq} that
\begin{equation}\label{eq:exp_decay_base_ineq_good}
\Theta_1^2 + \Delta(\Theta_0^2) \leqslant -\frac{1}{2\ep^2}\Theta_0^2, \text{ on }B_{2\rho}(x_0).
\end{equation}
Now let $\varphi = e^{\frac{(d(\cdot, x_0))^2}{A\ep\rho}}$, and note by Lemma~\ref{lemm:barrier} that if $A$ is above a universal threshold, we have
\[
\Delta \varphi \geqslant -\frac{1}{2\ep^2}\varphi,\text{ on }B_{2\rho}(x_0). 
\]
Moreover, by~\eqref{eq:Theta_by_Psi_basic},~\eqref{eq:exp_decay_base_Phi_two_sided} and~\eqref{eq:e_estimate_for_exp_decay_base}, we have
\[
\Theta_0^2 \leqslant |\Phi|^2 \Psi_0^2 \leqslant C_{\lambda_0} \ep^{-2}\eta, \text{ on }B_{3\rho}(x_0).
\]
Thus, we may apply the maximum principle on $B_{2\rho}(x_0)$ to get
\[
\Theta_0^2 \leqslant C_{\lambda_0}\ep^{-2}\eta  \cdot e^{-\frac{4\rho^2}{A\ep\rho}}\varphi = C_{\lambda_0} \ep^{-2}\eta \cdot e^{\frac{d(\cdot, x_0)^2 - 4\rho^2}{A\ep\rho}} \text{ on }B_{2\rho}(x_0).
\]
Restricting to $B_{\frac{3\rho}{2}}(x_0)$ gives the estimate on $\Theta_0$ asserted in~\eqref{eq:exp_decay_base_pointwise}, from which we derive the estimate on $\Psi_0^\perp$ upon recalling~\eqref{eq:perp_by_Theta_basic} and~\eqref{eq:exp_decay_base_Phi_two_sided}. It remains to prove~\eqref{eq:exp_decay_base_L2}. To that end, let $\zeta$ be a cut-off function such that
\[
\zeta = 1 \text{ on }B_{\rho}(x_0),\ \ \zeta = 0 \text{ outside of }B_{\frac{3\rho}{2}}(x_0),\ \ |\nabla \zeta|\leqslant C\rho^{-1}.
\]
Multiplying~\eqref{eq:exp_decay_base_ineq_good} by $\zeta^2$, integrating by parts, and using~\eqref{eq:Psi_Theta_schwarz}, we get
\[
\begin{split}
\int_{M} \zeta^2 \Theta_1^2 \leqslant\ & \int_{M}4\zeta |\nabla\zeta|\Theta_0\Theta_1 \leqslant \frac{1}{2}\int_{M}\zeta^2\Theta_1^2 + 8\int_{M} |\nabla\zeta|^2 \Theta_0^2.
\end{split}
\]
Rearranging and using~\eqref{eq:exp_decay_base_pointwise} and~\eqref{eq:volume_bounds_for_estimates} gives
\[
\begin{split}
\int_{B_\rho(x_0)}\Theta_1^2 \leqslant C \rho^{-2}\int_{B_{\frac{3\rho}{2}}(x_0)}\Theta_0^2\leqslant\ & C_{\lambda_0}\rho^{-2} \rho^{3} \ep^{-2}e^{-2a\frac{\rho}{\ep}}\eta\\ 
=\  & C_{\lambda_0}\ep^{-1} \cdot \frac{\rho}{\ep} \cdot e^{-2a\frac{\rho}{\ep}}\eta.
\end{split}
\]
Using the fact that $\sup_{t \geqslant 0}t e^{-at} \sim a^{-1}$, we arrive at~\eqref{eq:exp_decay_base_L2}.
\end{proof}
\begin{rmk}\label{rmk:clearing_out}
Before proceeding, we pause to describe a couple of scenarios in which~\eqref{eq:exp_decay_base_small_w} is guaranteed to hold. 
\begin{enumerate}
\item[(i)]
Condition~\eqref{eq:exp_decay_base_small_w} can be fulfilled for instance by requiring that $\ep \leqslant \rho$, and that
\begin{equation}\label{eq:bound_1_for_clearing_out}
\int_{B_{4\rho}(x_0)} e_{\ep}(\nabla, \Phi) \vol_g \leqslant \ep \cdot \theta_0\lambda,
\end{equation}
with a sufficiently small $\theta_0 = \theta_0(\lambda_0, \beta) \in (0, 1)$. Indeed, since $\theta_0 < 1$, we may invoke Lemma~\ref{lemm:w_mean_value_estimate} with $\Lambda = \lambda_0$ and $\Lambda' = \theta_0 \lambda$, and the resulting estimate~\eqref{eq:w_estimate} reduces to
\[
\|1 - |\Phi|\|_{\infty; B_{\frac{13\rho}{4}}(x_0)} \leqslant C_{\lambda_0} \theta_0^{\frac{1}{7}} \leqslant \frac{\beta}{2},
\]
provided that $\theta_0 \leqslant \big( \frac{\beta}{2(C_{\lambda_0} + 1)} \big)^7$, in which case~\eqref{eq:exp_decay_base_small_w} follows, since we then have
\[
|w| = |1 - |\Phi|| \cdot \frac{1 + |\Phi|}{2} \leqslant \frac{\beta}{2} \cdot (1 + \frac{\beta}{4}) < \beta \text{ on }B_{\frac{13\rho}{4}}(x_0).
\] 
\vskip 1mm
\item[(ii)] More generally, assuming~\eqref{eq:e_estimate_energy_bound} for some given $\Lambda > 0$, then by~\eqref{eq:w_estimate} we obtain some $\theta'_0 = \theta'_0(\lambda_0, \Lambda, \beta) \in (0, 1)$ so that $\|1 - |\Phi|\|_{\infty; B_{\frac{13\rho}{4}}(x_0)} < \frac{\beta}{2}$ provided in addition that
\begin{equation}\label{eq:bound_2_for_clearing_out}
\int_{B_{4\rho}(x_0)} \frac{\lambda w^2}{\ep^2}\vol_{g} \leqslant \ep \cdot \theta'_0 \lambda,
\end{equation}
and that $\ep \leqslant \rho$, in which case we deduce~\eqref{eq:exp_decay_base_small_w} as in (i).
\end{enumerate}
\end{rmk}
\begin{prop}\label{prop:exp_decay}
Again assume $\lambda \in (0, \lambda_0]$, and let $\eta_0$ and $\tau_0$ be as in Lemma~\ref{lemm:exp_decay_base}. Suppose that $(\nabla,\Phi)$ is a smooth solution of~\eqref{eq: 2nd_order_crit_pt_intro} satisfying~\eqref{eq:exp_decay_base_small_energy} and~\eqref{eq:exp_decay_base_small_w} for some geodesic ball $B_{4\rho}(x_0) \subset \Omega$ with $\rho \in (0, \frac{\rho_1}{4})$ and some $\beta \in (0, \frac{1}{4(\lambda_0 + 1)})$. Suppose also that $\ep \leqslant \tau_0 \rho$. Then for all $k \in \NN$ we have
\begin{equation}\label{eq:exp_decay}
\|\Theta_k\|_{\infty; B_{\frac{\rho}{2}}(x_0)} + \|\Psi_{k}^\perp\|_{\infty; B_{\frac{\rho}{2}}(x_0)} \leqslant C\ep^{-k-1}e^{-a\frac{\rho}{\ep}}\eta^{\frac{1}{2}},
\end{equation}
where $C = C(k, \lambda_0, A_1, \cdots, A_{k+1})$ and $a = a(k)$.
\end{prop}
\begin{proof}
As in the previous proof, we have the bounds~\eqref{eq:exp_decay_base_Phi_two_sided} on $|\Phi|$, and may also apply Lemma~\ref{lemm:coarse_estimate_base} and Proposition~\ref{prop:coarse_estimate} with, for instance, $K_0 = 2$ and $\Lambda = 1$. Also, as we are working under the hypotheses of Lemma~\ref{lemm:exp_decay_base}, its conclusions still stand. Letting $\rho_k = (2^{-1} + 2^{-k-1})\rho$, we shall prove inductively that for all $k \in \NN \cup \{0\}$, we have
\begin{subequations}
\begin{align}
\|\Theta_k\|_{\infty; B_{\rho_{k + 1}}(x_0)} \leqslant\ & C\ep^{-k-1}e^{-a\frac{\rho}{\ep}}\eta^{\frac{1}{2}}, \label{eq:exp_decay_induction_pointwise}\\
\|\Psi_k^\perp\|_{\infty; B_{\rho_{k + 1}}(x_0)} \leqslant\ & C\ep^{-k-1}e^{-a\frac{\rho}{\ep}}\eta^{\frac{1}{2}}, \label{eq:exp_decay_induction_perp}\\
\int_{B_{\rho_{k + 1}}(x_0)}\Theta_{k+1}^2 \leqslant\ & C \ep^{-2k-1} e^{-a\frac{\rho}{\ep}} \eta,\label{eq:exp_decay_induction_L2}
\end{align}
\end{subequations}
where $C = C(k, \lambda_0, A_1, \cdots, A_{k + 1})$ and $a = a(k)$. The base case $k = 0$ follows from Lemma~\ref{lemm:exp_decay_base}. For the induction step, suppose for some $m \geqslant 1$ that the three estimates above hold for $k = 0, \cdots, m-1$. Then upon recalling the inequality~\eqref{eq:Theta_perp_relation} from Lemma~\ref{lemm:Psi_Theta_relations} and estimating its right-hand side using Lemma~\ref{lemm:coarse_estimate_base}, Proposition~\ref{prop:coarse_estimate}, and~\eqref{eq:exp_decay_base_Phi_two_sided}, we have on $B_{2\rho}(x_0)$ that
\begin{equation}\label{eq:exp_decay_perp_by_theta}
\begin{split}
\Psi_m^\perp \leqslant \ & |\Phi|^{-1}\Theta_m + C_{m}|\Phi|^{-1}\sum_{i=0}^{m-1}\Psi_i^{\perp}\Psi_{m-1-i}\leqslant 2\Theta_m + C\sum_{i = 0}^{m-1}\ep^{i-m}\Psi_i^{\perp},
\end{split}
\end{equation}
where $C$ depends only on $m, \lambda_0$, and $\{A_i\}_{1 \leqslant i \leqslant m-1}$. 

Next, as in the proof of Proposition~\ref{prop:coarse_estimate}, we invoke Lemma~\ref{lemm:commutator-estimate} to get, with $S = \nabla\Phi$ or $S = \ep F_{\nabla}$, that
\begin{equation}\label{eq:commutator_applied_to_Theta}
\begin{split}
\big| [\nabla^*\nabla, \nabla^m][S, \Phi] \big| \leqslant\ & C_{m}\sum_{i =0}^{m} (\frac{|\nabla^{m-i}(\ep F_{\nabla})|}{\ep} + |\nabla^{m-i}R|)|\nabla^i[S, \Phi]|\\
\leqslant\ & C_{m}(\frac{\Psi_0}{\ep} + \rho_0^{-2}A_0)\Theta_m + C_{m}\sum_{i=0}^{m-1}(\frac{\Psi_{m-i}}{\ep} + \rho_0^{i-m-2}A_{m-i})\Theta_{i}.
\end{split}
\end{equation}
Using Lemma~\ref{lemm:coarse_estimate_base} and~\eqref{eq:curvature_bound_usage}, we have 
\begin{equation}\label{eq:bound_for_commutator_1}
\frac{\Psi_0}{\ep} + \rho_0^{-2}A_0 \leqslant C_{\lambda_0}\ep^{-2} \text{ on }B_{3\rho}(x_0).
\end{equation}
On the other hand, by Proposition~\ref{prop:coarse_estimate} and the fact that $\ep < \rho_0$, on $B_{2\rho}(x_0)$ there holds
\begin{equation}\label{eq:bound_for_commutator_2}
\frac{\Psi_{m-i}}{\ep} + \rho_0^{i-m-2}A_{m-i} \leqslant C_{m,\lambda_0, A_1, \cdots, A_m} \ep^{i-m-2}, \text{ for }i = 0, \cdots m-1,
\end{equation} Substituting~\eqref{eq:bound_for_commutator_1} and~\eqref{eq:bound_for_commutator_2} back into~\eqref{eq:commutator_applied_to_Theta}, we get on $B_{2\rho}(x_0)$ that
\begin{equation}\label{eq:exp_decay_commutator_induction}
\begin{split}
\bangle{ [\nabla^*\nabla, \nabla^m][S, \Phi], \nabla^m[S, \Phi]} \leqslant\ & C_{m,\lambda_0}\ep^{-2}\Theta_m^2 + C\sum_{i = 0}^{m-1} \ep^{i-m-2}\Theta_i\Theta_m,
\end{split}
\end{equation}
where $C$ depends only on $m, \lambda_0, A_1, \cdots, A_m$. Next we estimate $\bangle{\nabla^m\nabla^*\nabla [S, \Phi], \nabla^m[S, \Phi]}$ using Lemma~\ref{lemm:trans_laplacian_diffed}. To start, we use~\eqref{eq:exp_decay_base_Phi_two_sided}, Lemma~\ref{lemm:coarse_estimate_base} and Proposition~\ref{prop:coarse_estimate} to bound the terms $(I)$ to $(IV)$ in~\eqref{eq:trans_laplacian_diffed} on $B_{2\rho}(x_0)$, which gives
\begin{equation}\label{eq:exp_decay_laplacian_terms_1}
(1 +\lambda) \cdot(I) + (II) + (IV) \leqslant C\sum_{i = 0}^{m-1}\ep^{i-m-2}(\Theta_i + \Psi_i^\perp), \text{ and }
\end{equation}
\begin{equation}\label{eq:exp_decay_laplacian_terms_2}
(III) \leqslant C_m\sum_{i = 0}^{m-1}\rho_0^{i-m-2}A_{m-i}\Theta_i \leqslant 
C\sum_{i = 0}^{m-1} \ep^{i-m-2}\Theta_i,
\end{equation}
where in both estimates, $C$ depends only on $m, \lambda_0, A_1, \cdots, A_{m + 1}$. On the other hand, using in addition~\eqref{eq:curvature_bound_usage}, we get, again on $B_{2\rho}(x_0)$,
\begin{equation}\label{eq:exp_decay_laplacian_terms_3}
\begin{split}
\big(|\Phi|^{-2}\Psi_0^2 + \rho_0^{-2}A_0\big) \Theta_m^2 + \big(\frac{|\Phi|\Psi_0}{\ep} + |\Phi|^{-1}\Psi_0^2 + \Psi_1\big)  \Psi_m^{\perp} \Theta_m\leqslant C_{\lambda_0, A_1} \ep^{-2} \cdot (\Theta_m^2 + \Psi_m^{\perp}\Theta_m ).
\end{split}
\end{equation}
Now, applying Young's inequality to the first term on the right-hand side of~\eqref{eq:trans_laplacian_diffed} and then taking into account~\eqref{eq:exp_decay_laplacian_terms_3},~\eqref{eq:exp_decay_laplacian_terms_2} and~\eqref{eq:exp_decay_laplacian_terms_1}, we get on $B_{2\rho}(x_0)$ that
\[
\begin{split}
\bangle{\nabla^m\nabla^*\nabla[S, \Phi], \nabla^m[S, \Phi]}\leqslant\ & \frac{1}{8}\Theta_{m + 1}^2 +  \frac{2\lambda w_+}{\ep^2}\cdot|\nabla^m[S, \Phi]|^2 + C_{m, \lambda_0, A_1}\ep^{-2}(\Theta_m^2 + \Psi_m^{\perp}\Theta_m)\\
& + C_{m, \lambda_0, A_1, \cdots, A_{m + 1}} \sum_{i = 0}^{m-1}\ep^{i-m-2}(\Theta_i + \Psi_i^{\perp})\Theta_m.
\end{split}
\]
Noting that $2\lambda w_+ \leqslant 2\lambda_0 \beta < \frac{1}{2}$ on $B_{3\rho}(x_0)$, and recalling~\eqref{eq:exp_decay_perp_by_theta}, we arrive at
\begin{equation}\label{eq:exp_decay_laplacian_terms}
\begin{split}
\bangle{\nabla^m\nabla^*\nabla[S, \Phi], \nabla^m[S, \Phi]}
\leqslant\ & \frac{1}{8}\Theta_{m + 1}^2 + C_{m, \lambda_0, A_1}\ep^{-2}\Theta_m^2\\
&+ C\sum_{i = 0}^{m-1} \ep^{i-m-2}(\Theta_i + \Psi_{i}^{\perp})\Theta_m \ \text{ on }B_{2\rho}(x_0),
\end{split}
\end{equation}
where again $C = C(m, \lambda_0, A_1, \cdots, A_{m+1})$. Recalling~\eqref{eq:Theta_laplacian}, we deduce from~\eqref{eq:exp_decay_laplacian_terms},~\eqref{eq:exp_decay_commutator_induction} and a rearrangement that
\begin{equation}\label{eq:exp_decay_induction_diff_ineq_for_integral}
\Theta_{m + 1}^2 + \Delta(\Theta_m^2) \leqslant C\ep^{-2}\Theta_m^2 + C\sum_{i = 0}^{m-1}\ep^{i-m-2}(\Theta_i + \Psi_i^{\perp})\Theta_m\  \text{ on }B_{2\rho}(x_0).
\end{equation}
Applying the induction hypothesis and Young's inequality then leads to
\begin{equation}\label{eq:exp_decay_induction_diff_ineq}
\begin{split}
\Theta_{m + 1}^2 + \Delta(\Theta_m^2) 
\leqslant\ &  C\ep^{-2}\Theta_m^2 + C\ep^{-m-3}e^{-a\frac{\rho}{\ep}}\eta^{\frac{1}{2}}\Theta_m\\
\leqslant\ & C\ep^{-2}\Theta_m^2 + C\ep^{-2m-4}e^{-a\frac{\rho}{\ep}}\eta\ \  \text{ on }B_{\rho_m}(x_0),
\end{split}
\end{equation}
where the constants $C$ and $a$ have only the admissible dependencies. We are now ready to finish the proof. Specifically, let
\[
r = \min\{\ep, \rho_m - \rho_{m + 1}\}.
\]
Then, as in Proposition~\ref{prop:coarse_estimate}, for all $x \in B_{\rho_{m + 1}}(x_0)$, upon applying Lemma~\ref{lemm:moser-improve}(b) with $q = \infty$, we get
\[
\begin{split}
\|\Theta_m^2\|_{\infty; B_{\frac{r}{4}}(x)} \leqslant\ & C\cdot \big( \ep^{-3}\int_{B_{\rho_m}(x_0)} \Theta_m^2 \vol_{g} + \ep^{-2m-2}e^{-a\frac{\rho}{\ep}}\eta \big)\\
\leqslant\ & C \ep^{-2m-2} e^{-a\frac{\rho}{\ep}}\eta,
\end{split}
\]
where in getting the last line we used~\eqref{eq:exp_decay_induction_L2} for $k = m-1$, which is part of the induction hypothesis. This proves~\eqref{eq:exp_decay_induction_pointwise} for $k = m$, from which we immediately get~\eqref{eq:exp_decay_induction_perp} for $k = m$ thanks to~\eqref{eq:exp_decay_perp_by_theta} and the induction hypothesis. Finally, take a cutoff function $\zeta$ such that 
\[
\zeta = 1 \text{ on }B_{\rho_{m + 1}}(x_0),\ \ \zeta = 0 \text{ outside of }B_{\rho_m}(x_0),\ \ |\nabla\zeta|\leqslant C_m\rho^{-1}.
\]
Multiplying~\eqref{eq:exp_decay_induction_diff_ineq} by $\zeta^2$ and integrating by parts while using~\eqref{eq:Psi_Theta_schwarz} and~\eqref{eq:volume_bounds_for_estimates}, we get
\[
\begin{split}
\int_{M} \zeta^2\Theta_{m + 1}^2\vol_{g} \leqslant\ & \int_{M}\big[C\ep^{-2}\zeta^2\Theta_m^2 + 4\zeta|\nabla\zeta|\Theta_m \Theta_{m + 1} \big] \vol_{g} +C\rho_m^3 \cdot \ep^{-2m-4}e^{-a\frac{\rho}{\ep}}\eta \\
\leqslant\ & \frac{1}{2}\int_{M}\zeta^2 \Theta_{m + 1}^2 \vol_{g} + C\ep^{-2}\int_{B_{\rho_m}(x_0)}\Theta_m^2 \vol_{g}+ C \ep^{-2m-1} \cdot \big(\frac{\rho}{\ep} \big)^3 e^{-a\frac{\rho}{\ep}}\eta,
\end{split}
\]
where in passing to the second line we used Young's inequality and also observed that $\rho^{-2} \leqslant \ep^{-2}$. Rearranging, applying~\eqref{eq:exp_decay_induction_L2} with $k = m-1$ once more, and using part of the term $e^{-a\frac{\rho}{\ep}}$ to absorb the factor $(\frac{\rho}{\ep})^3$, we arrive at
\[
\int_{B_{\rho_{m + 1}}(x_0)} \Theta_{m + 1}^2 \leqslant C\ep^{-2m - 1}e^{-a\frac{\rho}{\ep}} \eta.
\]
This proves~\eqref{eq:exp_decay_induction_L2} for $k = m$, and we are done.
\end{proof}
The remainder of this section is devoted to proving exponential decay estimates for $w$ and $\nabla\Phi$. This is achieved by another induction argument (Proposition~\ref{prop:nablaPhi_exp_decay}), and we again single out the base step (Lemma~\ref{lemm:nablaPhi_exp_decay_base}). The assumption $\lambda > 0$ plays an essential role in the process. Also, we shall make frequent use of the following two facts which have already appeared in the previous proofs: first, with $\mu = \min\{\lambda, 1\}$, there holds
\begin{equation}\label{eq:lambda_mu_ratio}
\frac{\lambda}{\mu} = \max\{1, \lambda\} \leqslant 1 + \lambda.
\end{equation}
Secondly, given $k \geqslant 0$ and $a > 0$, we have
\begin{equation}\label{eq:trading_factors}
s^{-k}e^{-\frac{at}{s}} \leqslant C_{k, a}\cdot t^{-k}e^{-\frac{at}{2s}}, \text{ for all }s, t > 0,
\end{equation}
where we can take the constant $C_{k, a}$ to be $\sup_{x \geqslant 0}x^ke^{-\frac{ax}{2}}$.
\begin{lemm}\label{lemm:nablaPhi_exp_decay_base}
Suppose $\lambda \in (0, \lambda_0]$. There exist $\theta_1, \tau_1 \in (0, 1)$, depending only on $\lambda_0$, with the following property. Suppose $(\nabla, \Phi)$ is a smooth solution of~\eqref{eq: 2nd_order_crit_pt_intro} on $\Omega$ satisfying that
\begin{equation}\label{eq:small_energy_for_nablaPhi_decay}
\int_{B_{4\rho}(x_0)} \Psi_0^2\vol_g =: \ep \cdot \eta \leqslant \ep \cdot (\theta_1 \mu),
\end{equation}
for some $B_{4\rho}(x_0) \subset \Omega$ with $\rho \in (0, \frac{\rho_1}{4})$, and that
\begin{equation}\label{eq:nablaPhi_exp_decay_Phi_two_sided}
\|w\|_{\infty; B_{3\rho}(x_0)} \leqslant \beta,
\end{equation}
for some $0 < \beta < \min\{\frac{1}{6}, \frac{1}{2(\lambda_0 + 2)}\}$. Assume also that $\frac{\ep}{\rho} \leqslant \tau_1\sqrt{\mu}$. Then we have the following. First of all, 
\begin{equation}\label{eq:nablaPhi_diff_ineq_for_gap}
\Delta |\nabla\Phi|^2 \leqslant -\frac{\mu}{2\ep^2}|\nabla\Phi|^2 \text{ on }B_{\frac{11\rho}{4}}(x_0).
\end{equation}
Secondly, we have 
\begin{equation}\label{eq:nablaPhi_exp_decay_base}
\|\nabla\Phi\|_{\infty; B_{\frac{5\rho}{2}}(x_0)} \leqslant C(\sqrt{\mu}\rho)^{-1}e^{-a\frac{\sqrt{\mu}\rho }{\ep}}\eta^{\frac{1}{2}},
\end{equation}
\begin{equation}\label{eq:w_exp_decay}
\|w\|_{\infty; B_{2\rho}(x_0)} \leqslant C e^{-a\frac{\sqrt{\mu}\rho}{\ep}},
\end{equation}
\begin{equation}\label{eq:nablaPhi_integral_decay}
\int_{B_{2\rho}(x_0)} |\nabla^2 \Phi|^2 \vol_g \leqslant C (\sqrt{\mu}\rho)^{-1}e^{-a\frac{\sqrt{\mu}\rho}{\ep}}\eta,
\end{equation}
where in all three estimates, $C$ depends only on $\lambda_0$, while $a$ is universal.
\end{lemm}
\begin{proof}
We first require that 
\[
\theta_1 < \overline{\eta},
\]
where $\overline{\eta}$ is the threshold from Lemma~\ref{lemm:improved_coarse_estimate_base}. Then since $\mu \leqslant 1$, assumption~\eqref{eq:small_energy_for_nablaPhi_decay} above implies~\eqref{eq:small_energy_for_improved_coarse} in Lemma~\ref{lemm:improved_coarse_estimate_base}. Noting also that $\beta < \frac{1}{4}$ and $\ep \leqslant \sqrt{\mu}\rho$, we have available the estimate~\eqref{eq:improved_coarse_pointwise}. Next, noting that 
\[
\frac{\mu}{\lambda + 2\mu} = \left\{
\begin{array}{ll}
\frac{1}{3}, &  \text{ if }\lambda \leqslant 1, \\
\frac{1}{\lambda + 2} \geqslant \frac{1}{\lambda_0 + 2}, & \text{ if }1 < \lambda \leqslant \lambda_0, 
\end{array}
\right.
\]
we get $\beta < \frac{\mu}{2(\lambda + 2\mu)}$, so that, by~\eqref{eq:nablaPhi_exp_decay_Phi_two_sided}, 
\[
\lambda w - \mu|\Phi|^2 \leqslant (\lambda + 2\mu)\beta - \mu < -\frac{\mu}{2}  \text{ on }B_{3\rho}(x_0).
\]
Combining this and the estimate~\eqref{eq:improved_coarse_pointwise} with~\eqref{eq:nablaPhi_bochner_with_lambda}, we get on $B_{\frac{11\rho}{4}}(x_0)$ that
\[
\begin{split}
\bangle{\nabla^*\nabla\nabla \Phi, \nabla\Phi} \leqslant\ & -\frac{\mu}{2\ep^2}|\nabla\Phi|^2 + \frac{C}{\ep^2}\big( \ep\Psi_0 + \ep^2\rho_0^{-2}A_0 \big) |\nabla\Phi|^2\\
\leqslant\ & -\frac{\mu}{2\ep^2}|\nabla\Phi|^2 + \frac{C_{\lambda_0}}{\ep^2}\big( \mu^{\frac{3}{4}}(\theta_1 \mu)^{\frac{1}{2}} + \tau_1^2 \mu \big) |\nabla\Phi|^2.
\end{split}
\]
where in getting the second line we also used~\eqref{eq:curvature_bound_usage} and the assumption that $\frac{\ep}{\rho} \leqslant \tau_1\sqrt{\mu}$. Upon requiring that
\[
C_{\lambda_0} \theta_1^{\frac{1}{2}} \leqslant \frac{1}{8},\ \ \ \  C_{\lambda_0}\tau_1^2 \leqslant \frac{1}{8},
\]
we get
\begin{equation}\label{eq:nablaPhi_exp_decay_diff_ineq}
|\nabla^2\Phi|^2 + \Delta(\frac{|\nabla\Phi|^2}{2}) = \bangle{\nabla^*\nabla\nabla\Phi, \nabla\Phi} \leqslant -\frac{\mu}{4\ep^2}|\nabla \Phi|^2 \text{ on }B_{\frac{11\rho}{4}}(x_0).
\end{equation}
In particular this gives~\eqref{eq:nablaPhi_diff_ineq_for_gap}. To continue, since $\ep \leqslant \sqrt{\mu}\rho$, we have by Lemma~\ref{lemm:barrier} (with $\frac{\ep}{\sqrt{\mu}}$ in place of $\ep$) that
\begin{equation}\label{eq:barrier_for_nablaPhi_decay}
(\Delta + \frac{\mu}{2\ep^2})e^{\frac{(d(\cdot, x_0))^2}{A(\ep/\sqrt{\mu}) \rho}} \geqslant 0 \text{ on }B_{\frac{11\rho}{4}}(x_0),
\end{equation}
provided $A$ is above a universal threshold. Recalling also the estimate~\eqref{eq:e_estimate} from Lemma~\ref{lemm:coarse_estimate_base}, which implies in particular that
\[
|\nabla\Phi(x)|^2 \leqslant \Psi_0^2(x) \leqslant C_{\lambda_0}\ep^{-2}\eta, \text{ for all }x \in B_{\frac{11\rho}{4}}(x_0),
\]
we obtain by the maximum principle that
\begin{equation}\label{eq:nablaPhi_exp_decay_epsilon}
|\nabla\Phi|^2 \leqslant C_{\lambda_0}\ep^{-2}\eta \cdot e^{\frac{d(\cdot, x_0)^2 - (11\rho/4)^2}{A(\ep/\sqrt{\mu}) \rho}} \text{ on }B_{\frac{11\rho}{4}}(x_0).
\end{equation}
Restricting this to $B_{\frac{5\rho}{2}}(x_0)$ and using~\eqref{eq:trading_factors}, we obtain~\eqref{eq:nablaPhi_exp_decay_base}. 

To prove~\eqref{eq:w_exp_decay}, note that from~\eqref{eq:nablaPhi_exp_decay_epsilon} and the differential inequality~\eqref{eq:laplacian_|w|}, as well as the lower bound on $|\Phi|^2$ provided by~\eqref{eq:nablaPhi_exp_decay_Phi_two_sided}, we have on $B_{\frac{5\rho}{2}}(x_0)$ that
\[
\Delta |w| \leqslant -\frac{\mu}{2\ep^2}|w| + C_1\ep^{-2}\cdot \theta_1 \mu \cdot e^{-a_1\frac{\sqrt{\mu}\rho}{\ep}},
\]
where $C_1$ depends only on $\lambda_0$, and $a_1$ is a universal constant. From this and~\eqref{eq:barrier_for_nablaPhi_decay}, we deduce
\[
(\Delta + \frac{\mu}{2\ep^2})|w| \leqslant (\Delta + \frac{\mu}{2\ep^2})\Big(e^{\frac{d(\cdot, x_0)^2 - (5\rho/2)^2}{A(\ep/\sqrt{\mu})\rho}} + 2C_1 \theta_1\cdot e^{-a_1\frac{\sqrt{\mu}\rho}{\ep}}\Big)  \text{ on }B_{\frac{5\rho}{2}}(x_0).
\]
Since $|w| \leqslant \beta < 1$ on $B_{3\rho}(x_0)$ by assumption, while the function on the right-hand side above is at least $1$ on $\partial B_{\frac{5\rho}{2}}(x_0)$, we may apply the maximum principle to obtain
\[
|w(x)| \leqslant   e^{\frac{d(x, x_0)^2 - (5\rho/2)^2}{A(\ep/\sqrt{\mu})\rho}} + 2C_1  \cdot \theta_1 \cdot e^{-a_1\frac{\sqrt{\mu}\rho}{\ep}}, \text{ for all }x \in B_{\frac{5\rho}{2}}(x_0).
\]
Recalling once again that $\theta_1 < 1$, we get~\eqref{eq:w_exp_decay} with the asserted constant dependencies upon restricting the above estimate to $B_{2\rho}(x_0)$. 

To prove~\eqref{eq:nablaPhi_integral_decay}, we notice that by~\eqref{eq:nablaPhi_exp_decay_base}, H\"older's inequality, the assumption~\eqref{eq:small_energy_for_nablaPhi_decay}, and the volume estimate~\eqref{eq:volume_bounds_for_estimates}, we have
\begin{equation}\label{eq:prepare_for_nablaPhi_integral}
\begin{split}
\int_{B_{\frac{5\rho}{2}}(x_0)} |\nabla\Phi|^2 \vol_{g} \leqslant\ & C_{\lambda_0}(\sqrt{\mu}\rho)^{-1} e^{-a\frac{\sqrt{\mu}\rho}{\ep}}\eta^{\frac{1}{2}}  \int_{B_{\frac{5\rho}{2}}(x_0)}|\nabla\Phi| \vol_g\\
\leqslant\ & C_{\lambda_0}(\sqrt{\mu}\rho)^{-1} e^{-a\frac{\sqrt{\mu}\rho}{\ep}}\eta^{\frac{1}{2}} \cdot \rho^{\frac{3}{2}} \cdot (\ep\eta)^{\frac{1}{2}} \\
\leqslant\ & C_{\lambda_0} \rho^2  \cdot (\sqrt{\mu}\rho)^{-1}e^{-a\frac{\sqrt{\mu}\rho}{\ep}} \eta,
\end{split}
\end{equation}
where in getting the last line we also used $\ep \leqslant \rho$. Now take a cutoff function $\zeta$ such that 
\[
\zeta = 1 \text{ on }B_{2\rho}(x_0),\ \ \zeta = 0 \text{ outside of }B_{\frac{5\rho}{2}}(x_0),\ \ |\nabla\zeta|\leqslant C\rho^{-1}.
\]
Testing~\eqref{eq:nablaPhi_exp_decay_diff_ineq} against $\zeta^2$ and noting that $\nabla(\frac{|\nabla\Phi|^2}{2}) \leqslant |\nabla\Phi||\nabla^2\Phi|$, we get
\[
\begin{split}
\int_{M}\zeta^2 |\nabla^2\Phi|^2 \leqslant \int_{M} 2\zeta|\nabla\zeta| |\nabla\Phi||\nabla^2\Phi| \leqslant \frac{1}{2}\int_{M}\zeta^2 |\nabla^2 \Phi|^2 + 2\int_{M} |\nabla\zeta|^2|\nabla\Phi|^2.
\end{split}
\]
Rearranging and using~\eqref{eq:prepare_for_nablaPhi_integral} gives
\[
\begin{split}
\int_{B_{2\rho}(x_0)} |\nabla^2 \Phi|^2 \leqslant\ & C\rho^{-2}\int_{B_{\frac{5\rho}{2}}(x_0)} |\nabla\Phi|^2 \leqslant\ C_{\lambda_0} (\sqrt{\mu}\rho)^{-1} e^{-a\frac{\sqrt{\mu}\rho}{\ep}}\eta.
\end{split}
\]
This establishes~\eqref{eq:nablaPhi_integral_decay}.
\end{proof}
\begin{prop}\label{prop:nablaPhi_exp_decay}
Again assume $\lambda \in (0, \lambda_0]$, set $\mu = \min\{\lambda, 1\}$, and let $\theta_1 = \theta_1(\lambda_0)$ and $\tau_1 = \tau_1(\lambda_0)$ be as in Lemma~\ref{lemm:nablaPhi_exp_decay_base}. Suppose that $(\nabla,\Phi)$ is a smooth solution of~\eqref{eq: 2nd_order_crit_pt_intro} on $\Omega$ satisfying~\eqref{eq:small_energy_for_nablaPhi_decay} and~\eqref{eq:nablaPhi_exp_decay_Phi_two_sided} for some geodesic ball $B_{4\rho}(x_0) \subset \Omega$ with $\rho \in (0, \frac{\rho_1}{4})$, and some $0 < \beta < \min\{\frac{1}{6}, \frac{1}{2(\lambda_0 + 2)}\}$. Assume also that $\frac{\ep}{\rho} \leqslant \tau_1\sqrt{\mu}$. Then, we have for all $k \in \NN$ that 
\begin{equation}\label{eq:nablaPhi_exp_decay}
\|\nabla^{k + 1}\Phi\|_{\infty; B_{\rho}(x_0)} \leqslant C (\sqrt{\mu}\rho)^{-k-1}e^{-a\frac{\sqrt{\mu}\rho}{\ep}} \eta^{\frac{1}{2}},
\end{equation}
where $C = C(k, \lambda_0, A_1, \cdots, A_k)$ and $a = a(k)$.
\end{prop}
\begin{proof}
The proof again proceeds by induction. Letting $\rho_{k} = (1 + 2^{-k})\rho$, we shall prove that for all $k \in \NN \cup \{0\}$, we have
\begin{subequations}
\begin{align}
\|\nabla^{k+1}\Phi\|_{\infty; B_{\rho_{k+1}}(x_0)} \leqslant\ & C (\sqrt{\mu}\rho)^{-k-1}e^{-a\frac{\sqrt{\mu}\rho}{\ep}}\eta^{\frac{1}{2}}, \label{eq:nablaPhi_exp_decay_pointwise}\\
\int_{B_{\rho_{k + 1}}(x_0)} |\nabla^{k + 2}\Phi|^2 \leqslant \ & C(\sqrt{\mu}\rho)^{-2k-1} e^{-a\frac{\sqrt{\mu}\rho}{\ep}}\eta, \label{eq:nablaPhi_exp_decay_L2}
\end{align}
\end{subequations}
where $a = a(k)$ and $C = C(k, \lambda_0, \{A_i\}_{1 \leqslant i \leqslant k})$. (Again, when $k = 0$, we take $\{A_i\}_{1 \leqslant i \leqslant k}$ to mean the empty set.) Lemma~\ref{lemm:nablaPhi_exp_decay_base} establishes the base step. For the induction step we assume that both estimates above hold for $k = 0, \cdots, m-1$ for some $m \in \NN$. As in the proofs of Propositions~\ref{prop:coarse_estimate} and~\ref{prop:exp_decay}, upon using Lemma~\ref{lemm:commutator-estimate}, we have
\[
\begin{split}
\bangle{[\nabla^*\nabla, \nabla^m]\nabla\Phi, \nabla^{m + 1}\Phi} \leqslant\ &  C_{m}\big(\frac{\Psi_0}{\ep} + A_0\rho_0^{-2} \big)|\nabla^{m + 1}\Phi|^2\\
&+ C_{m}\sum_{i=0}^{m-1}\big(\frac{\Psi_{m-i}}{\ep}+ A_{m-i}\rho_0^{i-m-2} \big)|\nabla^{i + 1}\Phi| |\nabla^{m + 1}\Phi|.
\end{split}
\]
With the help of~\eqref{eq:curvature_bound_usage}, Lemma~\ref{lemm:coarse_estimate_base} and Proposition~\ref{prop:coarse_estimate}, we deduce that on $B_{2\rho}(x_0)$ there holds
\begin{equation}\label{eq:nablaPhi_decay_commutator}
\begin{split}
\bangle{[\nabla^*\nabla, \nabla^m]\nabla\Phi, \nabla^{m + 1}\Phi} \leqslant\ & C_{m, \lambda_0}\ep^{-2}|\nabla^{m + 1}\Phi|^2 \\
&+ C_{m, \lambda_0, A_1, \cdots, A_m}\sum_{i = 0}^{m-1}\ep^{i-m-2}|\nabla^{i + 1}\Phi| |\nabla^{m + 1}\Phi|.
\end{split}
\end{equation}
Next we turn to estimating $\bangle{\nabla^m \nabla^*\nabla (\nabla\Phi), \nabla^{m + 1}\Phi}$. Upon recalling Lemma~\ref{lemm:bochner_for_derivatives}(i), we have on $B_{3\rho}(x_0)$ that
\begin{equation}\label{eq:nablaPhi_decay_derivative_1}
\begin{split}
\bangle{\nabla^m \nabla^*\nabla(\nabla\Phi), \nabla^{m + 1}\Phi} \leqslant \ &  C\big(\frac{\lambda w_+}{\ep^2} + \frac{\Psi_0}{\ep} + \rho_0^{-2}A_0\big) |\nabla^{m + 1}\Phi|^2\\
& +\big( \lambda \cdot (I) + (1 + \lambda) \cdot (II) + (III)_{\nabla\Phi} + (IV)\big) |\nabla^{m + 1}\Phi|.
\end{split}
\end{equation}
By Lemma~\ref{lemm:coarse_estimate_base} and~\eqref{eq:curvature_bound_usage}, as well as the assumption~\eqref{eq:nablaPhi_exp_decay_Phi_two_sided}, we have
\begin{equation}\label{eq:nablaPhi_decay_derivative_leading}
\frac{\lambda w_+}{\ep^2} + \frac{\Psi_0}{\ep} + \rho_{0}^{-2}A_0 \leqslant C_{\lambda_0} \ep^{-2} \text{ on }B_{3\rho}(x_0).
\end{equation}
Recalling in addition Proposition~\ref{prop:coarse_estimate}, we have on $B_{2\rho}(x_0)$ that
\begin{equation}\label{eq:nablaPhi_decay_remainder_1}
\begin{split}
(I) +(III)_{\nabla\Phi} + (IV) \leqslant\ & C_{m, \lambda_0, A_{1}, \cdots, A_{m}} \sum_{i = 0}^{m-1}\ep^{i-m-2} |\nabla^{i + 1}\Phi|.
\end{split}
\end{equation}
To bound $(II)$, we note that
\[
|\nabla^{i}[\nabla\Phi, \Phi]| \leqslant C_{m}\sum_{j = 0}^{i}|\nabla^{j + 1}\Phi||\nabla^{i-j}\Phi|, \text{ for }i = 0, \cdots, m-1,
\]
and thus
\begin{equation}\label{eq:nablaPhi_decay_remainder_2}
\begin{split}
(II) \leqslant\ & C_m\ep^{-2}\sum_{j = 0}^{m-1}\sum_{i + k = m-j}|\nabla^i\Phi||\nabla^k\Phi| \cdot |\nabla^{j + 1}\Phi| + C_m\ep^{-2}\sum_{i = 0}^{m-1} |\nabla^{m-i}\Phi||\nabla^{i + 1}\Phi| |\Phi| \\
\leqslant\ & C_{m, \lambda_0, \{A_i\}_{1 \leqslant i \leqslant m-1}}\sum_{i=0}^{m-1}\ep^{i-m-2} |\nabla^{i + 1}\Phi| \text{ on }B_{2\rho}(x_0).
\end{split}
\end{equation}
Putting~\eqref{eq:nablaPhi_decay_remainder_2},~\eqref{eq:nablaPhi_decay_remainder_1} and~\eqref{eq:nablaPhi_decay_derivative_leading} back into~\eqref{eq:nablaPhi_decay_derivative_1}, we obtain
\begin{equation}\label{eq:nablaPhi_decay_derivative}
\begin{split}
\bangle{\nabla^m \nabla^*\nabla (\nabla\Phi), \nabla^{m + 1}\Phi} \leqslant\ & C_{\lambda_0}\ep^{-2}|\nabla^{m + 1}\Phi|^2\\
&+ C_{m, \lambda_0, A_1, \cdots, A_m} \sum_{i=0}^{m-1}\ep^{i-m-2}|\nabla^{i+1}\Phi||\nabla^{m+ 1}\Phi| \text{ on }B_{2\rho}(x_0).
\end{split}
\end{equation}
Summing this with~\eqref{eq:nablaPhi_decay_commutator} leads to
\begin{equation}\label{eq:nablaPhi_decay_for_integral}
\begin{split}
|\nabla^{m + 2}\Phi|^2 + \Delta\big( \frac{|\nabla^{m + 1}\Phi|^2}{2} \big) \leqslant\ &  C\ep^{-2}|\nabla^{m + 1}\Phi|^2 + C\sum_{i=0}^{m-1}\ep^{i-m-2}|\nabla^{i+1}\Phi||\nabla^{m+ 1}\Phi|,
\end{split}
\end{equation}
where the constants $C$ have only the admissible dependencies. Applying the induction hypothesis to the terms $|\nabla^{i + 1}\Phi|$ and using~\eqref{eq:trading_factors} a number of times, we deduce that on $B_{\rho_{m}}(x_0)$ there holds
\begin{equation}\label{eq:nablaPhi_decay_diff_ineq}
\begin{split}
|\nabla^{m + 2}\Phi|^2 + \Delta\big( \frac{|\nabla^{m + 1}\Phi|^2}{2} \big) \leqslant\ & C\ep^{-2}|\nabla^{m + 1}\Phi|^2 + C \ep^{-2} (\sqrt{\mu}\rho)^{-m-1} e^{-a\frac{\sqrt{\mu}\rho}{\ep}} \eta^{\frac{1}{2}} \cdot |\nabla^{m + 1}\Phi|\\
\leqslant\ &  C\ep^{-2}|\nabla^{m + 1}\Phi|^2 + C\ep^{-2}(\sqrt{\mu}\rho)^{-2m-2} e^{-a\frac{\sqrt{\mu}\rho}{\ep}}\eta,
\end{split}
\end{equation}
where in getting the second line we used Young's inequality. By the argument leading to~\eqref{eq:coarse_induction_after_moser} in the proof of Proposition~\ref{prop:coarse_estimate}, we have for all $x \in B_{\rho_{m + 1}}(x_0)$ that
\[
\begin{split}
|\nabla^{m + 1}\Phi|^2(x) \leqslant\ & C\ep^{-3}\int_{B_{\rho_{m}}(x_0)} |\nabla^{m + 1}\Phi|^2 \vol_g + C (\sqrt{\mu}\rho)^{-2m-2}e^{-a\frac{\sqrt{\mu}\rho}{\ep}}\eta\\
\leqslant\ & C\ep^{-3}(\sqrt{\mu}\rho)^{-2m + 1} e^{-a\frac{\sqrt{\mu}\rho}{\ep}}\eta + C (\sqrt{\mu}\rho)^{-2m-2}e^{-a\frac{\sqrt{\mu}\rho}{\ep}}\eta,
\end{split}
\]
where the second inequality follows from~\eqref{eq:nablaPhi_exp_decay_L2} for $k = m-1$. From this we deduce~\eqref{eq:nablaPhi_exp_decay_pointwise} for $k = m$ upon using~\eqref{eq:trading_factors} once more. To prove~\eqref{eq:nablaPhi_exp_decay_L2} for $k = m$, we take a cutoff function $\zeta$ such that
\[
\zeta = 1 \text{ on }B_{\rho_{m + 1}}(x_0),\ \ \zeta = 0 \text{ outside of }B_{\rho_m}(x_0),\ \ |\nabla\zeta|\leqslant C_m\rho^{-1}.
\]
Multiplying~\eqref{eq:nablaPhi_decay_for_integral} by $\zeta^2$ and integrating by parts while using H\"older's inequality, we get that
\[
\begin{split}
\int_{M}\zeta^2 |\nabla^{m + 2}\Phi|^2 \leqslant\ & \int_{M} [C\ep^{-2} \zeta^2 |\nabla^{m + 1}\Phi|^2 + 2\zeta |\nabla\zeta| |\nabla^{m + 1}\Phi||\nabla^{m + 2}\Phi|] \\
& + C\sum_{i = 0}^{m-1}\ep^{i-m-2} \|\nabla^{i + 1}\Phi\|_{2; B_{\rho_m}(x_0)} \|\nabla^{m + 1}\Phi\|_{2; B_{\rho_m}(x_0)}.
\end{split}
\]
Applying Young's inequality and rearranging, and then using the induction hypothesis along with~\eqref{eq:small_energy_for_nablaPhi_decay}, we deduce that
\begin{equation}\label{eq:nablaPhi_exp_decay_L2_penultimate}
\begin{split}
\int_{B_{\rho_{m + 1}}(x_0)} |\nabla^{m + 2}\Phi|^2  \leqslant\ & C\ep^{-2} \cdot (\sqrt{\mu}\rho)^{-2m + 1}e^{-a\frac{\sqrt{\mu}\rho}{\ep}}\eta\\
&+ C\sum_{i = 1}^{m-1}\ep^{i-m-2} \big[(\sqrt{\mu}\rho)^{-2i + 1}e^{-a\frac{\sqrt{\mu}\rho}{\ep}}\eta\big]^{\frac{1}{2}} \big[(\sqrt{\mu}\rho)^{-2m + 1}e^{-a\frac{\sqrt{\mu}\rho}{\ep}}\eta\big]^{\frac{1}{2}} \\
& + C\ep^{-m-2} (\ep\eta)^{\frac{1}{2}} \big[(\sqrt{\mu}\rho)^{-2m + 1}e^{-a\frac{\sqrt{\mu}\rho}{\ep}}\eta\big]^{\frac{1}{2}}.
\end{split}
\end{equation}
A few further applications of~\eqref{eq:trading_factors} yields~\eqref{eq:nablaPhi_exp_decay_L2} for $k = m$, and the proof is complete.
\end{proof}
\subsection{Improved estimates}\label{subsec:improved-estimates}
In this section, we show how the results of the previous section feed back into the inductive argument in \S\ref{subsec:coarse-estimates} to yield improved estimates on $\Psi_k$ under smallness assumptions. (See~\cite[Corollary 4.9]{fadel2023asymptotics} for a precedent of this type of argument.) Below we continue to write $\mu$ for $\min\{\lambda, 1\}$, and assume that $\lambda_0$ is an upper bound for $\lambda$. Also, we let $\overline{\eta}$, $(\eta_0$, $\tau_0)$ and $(\theta_1, \tau_1)$ be the thresholds given by Lemma~\ref{lemm:improved_coarse_estimate_base}, Lemma~\ref{lemm:exp_decay_base} and Lemma~\ref{lemm:nablaPhi_exp_decay_base}, respectively. 
\begin{lemm}\label{lemm:improved_estimates_base}
Suppose $(\nabla, \Phi)$ is a smooth solution of~\eqref{eq: 2nd_order_crit_pt_intro} satisfying that
\begin{equation}\label{eq:small_energy_for_improved_estimate}
\int_{B_{4\rho}(x_0)} \Psi_0^2\vol_g  = \ep\cdot \eta \leqslant \ep \cdot \min\{\overline{\eta}, \eta_0, \theta_1 \mu\}.
\end{equation}
for some geodesic ball $B_{4\rho}(x_0) \subset \Omega$ with $\rho \in (0, \frac{\rho_1}{4})$, and that
\begin{equation}\label{eq:small_w_for_improved_estimate}
\|w\|_{\infty; B_{3\rho}(x_0)} \leqslant \beta,
\end{equation}
for some $\beta \in (0, \frac{1}{4(\lambda_0 + 2)})$. Assume also that $\frac{\ep}{\rho} \leqslant \min\{\tau_0, \tau_1\sqrt{\mu}\}$. Then, we have
\begin{equation}\label{eq:improved_estimates_base_point}
\|\Psi_0\|_{\infty; B_{\rho}(x_0)} \leqslant C\rho^{-\frac{3}{2}}\ep^{\frac{1}{2}} \eta^{\frac{1}{2}},
\end{equation}
\begin{equation}\label{eq:improved_estimates_base_integral}
\int_{B_{\frac{\rho}{2}}(x_0)}\Psi_1^2 \leqslant C\rho^{-2}\ep \eta,
\end{equation}
where $C = C(\lambda_0)$ in both estimates.
\end{lemm}
\begin{proof}
Under our current hypotheses, we have $\eta \leqslant \overline{\eta}$, $\beta < \frac{1}{4}$ and $\ep \leqslant \sqrt{\mu}\rho$. Consequently we get from Lemma~\ref{lemm:improved_coarse_estimate_base} the differential inequality
\begin{equation}\label{eq:Psi_0_improved_bochner_again}
\Psi_1^2 + \Delta\big( \frac{\Psi_0^2}{2} \big) \leqslant C_{\lambda_0}\big(\frac{\mu w_+}{\ep^2} + \frac{A_0}{\rho_0^2}\big) \Psi_0^2,\  \text{ on }B_{3\rho}(x_0).
\end{equation}
Estimating $w_+$ using Lemma~\ref{lemm:nablaPhi_exp_decay_base} instead and recalling~\eqref{eq:curvature_bound_usage} leads to
\begin{equation}\label{eq:improved_diff_ineq_base}
\begin{split}
\Psi_1^2 + \Delta\big(\frac{\Psi_0^2}{2} \big) \leqslant \ & C_{\lambda_0}\rho^{-2}\cdot \big[ \big(\frac{\sqrt{\mu} \rho}{\ep}\big)^2\cdot e^{-a\frac{\sqrt{\mu}\rho}{\ep}} + 1 \big]\Psi_0^2 \leqslant C'_{\lambda_0}\rho^{-2}\Psi_0^2,\ \  \text{ on }B_{2\rho}(x_0).
\end{split}
\end{equation}
Given $x\in B_{\rho}(x_0)$, by Lemma~\ref{lemm:moser-improve}(b) with $q = \infty$ and $r = \frac{\rho}{2}$ applied to~\eqref{eq:improved_diff_ineq_base}, we obtain 
\[
\|\Psi_0^2\|_{\infty; B_{\frac{\rho}{4}}(x)} \leqslant C_{\lambda_0}\rho^{-3}\int_{B_{2\rho}(x_0)} \Psi_0^2 \leqslant C_{\lambda_0}\rho^{-3} \ep \eta,
\]
where we used~\eqref{eq:small_energy_for_improved_estimate} to get the second inequality. Since $x \in B_{\rho}(x_0)$ is arbitrary, we get~\eqref{eq:improved_estimates_base_point}. Next, take a cutoff function $\zeta$ such that
\[
\zeta = 1 \text{ on }B_{\frac{\rho}{2}}(x_0),\ \ \zeta = 0 \text{ outside of }B_{\rho}(x_0),\ \ |\nabla\zeta|\leqslant C\rho^{-1}.
\]
Then we have from~\eqref{eq:improved_diff_ineq_base},~\eqref{eq:Psi_Theta_schwarz}, and Young's inequality that
\[
\begin{split}
\int_{M}\zeta^2 \Psi_1^2 \leqslant \ & \int_{M}\big[ C_{\lambda_0}\rho^{-2}\zeta^2 \Psi_0^2 + 2\zeta |\nabla\zeta| \Psi_0 \Psi_1\big] \leqslant C_{\lambda_0}\rho^{-2}\int_{B_{\rho}(x_0)} \Psi_0^2 + \frac{1}{2}\int_{M} \zeta^2 \Psi_1^2.
\end{split}
\]
Rearranging and using~\eqref{eq:small_energy_for_improved_estimate} again gives~\eqref{eq:improved_estimates_base_integral}.
\end{proof}
\begin{prop}\label{prop:improved_estimates}
Under the assumptions of Lemma~\ref{lemm:improved_estimates_base}, for all $k \in \NN \cup \{0\}$ we have
\begin{equation}\label{eq:improved_estimates}
\|\Psi_k\|_{\infty; B_{\frac{\rho}{4}}(x_0)} \leqslant C(\sqrt{\mu}\rho)^{-k-\frac{3}{2}} \ep^{\frac{1}{2}} \eta^{\frac{1}{2}},
\end{equation}
where $C$ depends only on $k, \lambda_0, A_1, \cdots, A_{k + 1}$.
\end{prop}
\begin{proof}
Define 
\[
\rho_k = (2^{-2} + 2^{-k-1})\rho. 
\]
We shall prove by induction that for all $k \in \NN \cup \{0\}$, we have
\begin{subequations}
\begin{align}
\|\Psi_k\|_{\infty; B_{\rho_{k + 1}}(x_0)} \leqslant\ & C(\sqrt{\mu}\rho)^{-k-\frac{3}{2}}\ep^{\frac{1}{2}}\eta^{\frac{1}{2}},\label{eq:improved_estimate_induction_pointwise}\\
\int_{B_{\rho_{k + 1}}(x_0)} \Psi_{k +1}^2 \leqslant\ & C(\sqrt{\mu}\rho)^{-2k-2}\ep\eta,\label{eq:improved_estimate_induction_L2}
\end{align}
\end{subequations}
where $C = C(k, \lambda_0, A_1, \cdots, A_{k + 1})$ and $a = a(k)$. The base case follows from the previous lemma. Next, suppose that the two above estimates hold for $k = 0, \cdots, m-1$. Applying Lemma~\ref{lemm:commutator-estimate} as in the first line of the string of inequalities~\eqref{eq:coarse_estimate_induction_commutator_1} from the proof of Proposition~\ref{prop:coarse_estimate}, but instead using~\eqref{eq:tensor-bracket-norm-with-decomp-2} to estimate the first summation on the right-hand side, we obtain on $B_{3\rho}(x_0)$ that
\begin{equation}\label{eq:improved_estimate_induction_commutator1}
\begin{split}
\big| [\nabla^*\nabla, \nabla^m]S \big| \leqslant\ & \frac{C_{m}}{\ep}\sum_{i=0}^{m}  \Psi_{m-i}^\perp \Psi_i  + C_{m}\sum_{i=0}^{m}\rho_{0}^{i-m-2}A_{m-i}\Psi_i\\
\leqslant\ & C_{m}\big(\frac{\Psi_0^\perp}{\ep} + \rho_0^{-2} A_0 \big)\Psi_m + C_{m}\sum_{i = 0}^{m-1}\big( \frac{\Psi_{m-i}^\perp}{\ep} + \rho_0^{i-m-2}A_{m-i} \big)\Psi_i,
\end{split}
\end{equation}
where as before we let $S$ stand for either $\nabla\Phi$ or $\ep F_{\nabla}$. By Lemma~\ref{lemm:exp_decay_base} together with~\eqref{eq:curvature_bound_usage} and~\eqref{eq:trading_factors}, we have 
\begin{equation}\label{eq:improved_estimate_induction_commutator_leading}
\frac{\Psi_0^\perp}{\ep} + \rho_0^{-2}A_0 \leqslant C_{\lambda_0}\ep^{-2}e^{-a\frac{\rho}{\ep}} + \rho^{-2}c_0 \leqslant C_{\lambda_0}'\rho^{-2}\  \text{ on }B_{\frac{3\rho}{2}}(x_0).
\end{equation}
Using also Proposition~\ref{prop:exp_decay}, we have for $i = 0, \cdots, m-1$ that
\[
\big( \frac{\Psi_{m-i}^\perp}{\ep} + \rho_0^{i-m-2}A_{m-i} \big) \leqslant C\ep^{i-m-2}e^{-a\frac{\rho}{\ep}} + \rho^{i-m-2}A_{m-i}, \text{ on }B_{\frac{\rho}{2}}(x_0),
\]
where $C = C(m, \lambda_0, A_1, \cdots, A_{m + 1})$ and $a = a(m)$. Substituting the two previous estimates into~\eqref{eq:improved_estimate_induction_commutator1} and using~\eqref{eq:trading_factors} again gives
\begin{equation}\label{eq:improved_estimate_induction_commutator}
\begin{split}
\langle[\nabla^*\nabla, \nabla^m]S, \nabla^m S \rangle \leqslant\ & C\rho^{-2}\Psi_m^2 + C\sum_{i=  0}^{m-1}\rho^{i-m-2}\Psi_i \Psi_m, \text{ on }B_{\frac{\rho}{2}}(x_0),
\end{split}
\end{equation}
where $C$ depends only on $m, \lambda_0, A_1, \cdots, A_{m + 1}$. On the other hand, by Lemma~\ref{lemm:bochner_for_derivatives}(ii), we have
\begin{equation}\label{eq:improved_estimate_induction_derivative1}
\begin{split}
\bangle{\nabla^m\nabla^*\nabla S, \nabla^m S} \leqslant\ & C\big(\frac{\lambda |w|}{\ep^2} + \frac{\Psi_0^{\perp}}{\ep} + \rho_0^{-2}A_0\big)\Psi_{m}^2 \\
&+ \big( \lambda \cdot(I) + (1 + \lambda)\cdot(II) + (III)_{M \setminus Z} + (IV) \big) \Psi_m.
\end{split}
\end{equation}
By Lemma~\ref{lemm:exp_decay_base} and Proposition~\ref{prop:exp_decay}, we have on $B_{\frac{\rho}{2}}(x_0)$ that
\begin{equation}\label{eq:improved_estimate_bound_for_II_III}
\begin{split}
(II) + (III)_{M \setminus Z} \leqslant\ & C_m\ep^{-2}\sum_{i = 0}^{m-1} (\Theta_{m-1-i} + |\Phi|\Psi_{m-1-i}^{\perp}) \Psi_i  +  C_m\ep^{-1}\sum_{i = 0}^{m-1}\Psi_{m-i}^{\perp}\Psi_i\\
\leqslant\ & C\sum_{i = 0}^{m-1} \ep^{i-m-2}e^{-a\frac{\rho}{\ep}} \Psi_i\leqslant C\rho^{-2}\sum_{i = 0}^{m-1} \rho^{i-m}\Psi_i,
\end{split}
\end{equation}
for some constant $C$ depending only on $m, \lambda_0, A_1, \cdots, A_{m + 1}$, where for the last inequality we used~\eqref{eq:trading_factors}. Similarly, in the inner summation in term $(I)$, since at least one of $j, k$ is non-zero, we get upon using Lemma~\ref{lemm:nablaPhi_exp_decay_base} and Proposition~\ref{prop:nablaPhi_exp_decay} that, on $B_{\rho}(x_0)$,
\begin{equation}\label{eq:improved_estimate_bound_for_I}
\begin{split}
\lambda \cdot (I) \leqslant C\lambda \ep^{-2}\sum_{i = 0}^{m-1} (\sqrt{\mu}\rho)^{i-m}e^{-a\frac{\sqrt{\mu}\rho}{\ep}}\Psi_i \leqslant\ & C \lambda \sum_{i = 0}^{m-1}(\sqrt{\mu}\rho)^{i-m-2} \Psi_i\\
\leqslant\ & C\rho^{-2}\sum_{i = 0}^{m-1}(\sqrt{\mu}\rho)^{i-m} \Psi_i.
\end{split}
\end{equation}
for some $C = C(m, \lambda_0, A_1, \cdots, A_m)$, where we used~\eqref{eq:trading_factors} and~\eqref{eq:lambda_mu_ratio}, respectively, in getting the second and third inequalities. Next, as in the previous proof, by Lemma~\ref{lemm:nablaPhi_exp_decay_base}, and again using~\eqref{eq:lambda_mu_ratio}, we have
\begin{equation}\label{eq:improved_estimate_induction_w}
\begin{split}
\frac{\lambda |w|}{\ep^2} \leqslant\ & C_{\lambda_0} \frac{\lambda}{\mu \rho^2} \big( \frac{\sqrt{\mu}\rho}{\ep} \big)^{2}e^{-a\frac{\sqrt{\mu}\rho}{\ep}} \leqslant C'_{\lambda_0}\rho^{-2} \text{ on }B_{2\rho}(x_0).
\end{split}
\end{equation}
To continue, we estimate $(IV)$ in the straightforward way using~\eqref{eq:curvature_bound_for_estimates}, substitute it along with~\eqref{eq:improved_estimate_induction_w},~\eqref{eq:improved_estimate_bound_for_I} and~\eqref{eq:improved_estimate_bound_for_II_III} back into~\eqref{eq:improved_estimate_induction_derivative1}, and also recall~\eqref{eq:improved_estimate_induction_commutator_leading}. The result is
\begin{equation}\label{eq:improved_estimate_induction_derivative}
\begin{split}
\bangle{\nabla^m\nabla^*\nabla S, \nabla^m S} \leqslant\ & C\rho^{-2}\Psi_m^2 + C\rho^{-2}\sum_{i = 0}^{m-1} (\sqrt{\mu}\rho)^{i-m}\Psi_i\Psi_m\ \  \text{ on }B_{\frac{\rho}{2}}(x_0),
\end{split}
\end{equation}
with $C = C(m, \lambda_0, A_1, \cdots, A_{m + 1})$. Summing this with~\eqref{eq:improved_estimate_induction_commutator} yields
\begin{equation}\label{eq:improved_estimate_induction_diff_ineq_1}
\begin{split}
\Psi_{m + 1}^2 + \Delta\big( \frac{\Psi_m^2}{2} \big) \leqslant\ & C\rho^{-2}\Psi_m^2 + C\rho^{-2}\sum_{i = 0}^{m-1} (\sqrt{\mu}\rho)^{i-m}\Psi_i\Psi_m\ \  \text{ on }B_{\frac{\rho}{2}}(x_0).
\end{split}
\end{equation}
Invoking the induction hypothesis, we deduce that
\begin{equation}\label{eq:improved_estimate_induction_diff_ineq_2}
\begin{split}
\Delta\big( \frac{\Psi_m^2}{2} \big) \leqslant\ & C\rho^{-2}\Psi_m^2 + C\rho^{-2}(\sqrt{\mu}\rho)^{-m-\frac{3}{2}}\ep^{\frac{1}{2}}\eta^{\frac{1}{2}}\Psi_m\\
\leqslant\ & C\rho^{-2}\Psi_m^2 + C\rho^{-2}(\sqrt{\mu}\rho)^{-2m-3}\ep \eta, \text{ on }B_{\rho_m}(x_0),
\end{split}
\end{equation}
where in getting the second line we used Young's inequality. Given $x \in B_{\rho_{m + 1}}(x_0)$, we apply Lemma~\ref{lemm:moser-improve}(b) with $q = \infty$ and $r = \frac{1}{2}(\rho_m - \rho_{m + 1})$ to the differential inequality~\eqref{eq:improved_estimate_induction_diff_ineq_2}. Then since $r \sim_{m} \rho$ and $B_r(x) \subset B_{\rho_m}(x_0)$, we have
\[
\begin{split}
\|\Psi_m^2\|_{\infty; B_{\frac{r}{2}}(x)} \leqslant \ & C\rho^{-3}\int_{B_{\rho_m}(x_0)} \Psi_m^2 \vol_g + C(\sqrt{\mu}\rho)^{-2m-3}\ep \eta\\
\leqslant\ & C(\sqrt{\mu}\rho)^{-2m-3}\ep \eta,
\end{split}
\]
where the second inequality follows from~\eqref{eq:improved_estimate_induction_L2} for $k = m-1$. Since $x \in B_{\rho_{m + 1}}(x_0)$ is arbitrary, we get~\eqref{eq:improved_estimate_induction_pointwise} for $k = m$. To finish, we take a cutoff function $\zeta$ such that
\[
\zeta = 1 \text{ on }B_{\rho_{m + 1}}(x_0),\ \ \zeta = 0 \text{ outside of }B_{\rho_m}(x_0),\ \ |\nabla\zeta|\leqslant C_m\rho^{-1}.
\]
Similar to the last part of the proof of Proposition~\ref{prop:coarse_estimate}, we test~\eqref{eq:improved_estimate_induction_diff_ineq_1} against $\zeta^2$ and use~\eqref{eq:Psi_Theta_schwarz} as well as H\"older's inequality to get
\[
\begin{split}
\int_{M}\zeta^2 \Psi_{m + 1}^2 \leqslant\ & \int_{M} \big[ C\rho^{-2}\zeta^2 \Psi_m^2 + 2\zeta|\nabla\zeta|\Psi_m \Psi_{m + 1} \big] \\
&+ C\rho^{-2}\sum_{i = 0}^{m-1}(\sqrt{\mu}\rho)^{i-m} \|\Psi_i\|_{2; B_{\rho_m}(x_0)} \|\Psi_m\|_{2; B_{\rho_m}(x_0)},
\end{split}
\]
Applying Young's inequality to the term $\zeta |\nabla\zeta| \Psi_m\Psi_{m + 1}$ and rearranging, and also using the induction hypothesis and the assumption~\eqref{eq:small_energy_for_improved_estimate}, we obtain
\begin{equation}\label{eq:improved_estimate_L2_penultimate}
\begin{split}
&\int_{B_{\rho_{m + 1}}(x_0)} \Psi_{m + 1}^2\\
&\leqslant C\rho^{-2}\int_{B_{\rho_m}(x_0)} \Psi_m^2 +C\rho^{-2}\sum_{i = 0}^{m-1}(\sqrt{\mu}\rho)^{i-m}\cdot [(\sqrt{\mu}\rho)^{-2i}\ep \eta]^{\frac{1}{2}} \cdot [(\sqrt{\mu}\rho)^{-2m}\ep\eta]^{\frac{1}{2}}\\
&\leqslant C\rho^{-2}(\sqrt{\mu}\rho)^{-2m}\ep \eta + C\rho^{-2}(\sqrt{\mu}\rho)^{-2m}\ep \eta.
\end{split}
\end{equation}
From~\eqref{eq:improved_estimate_L2_penultimate} we conclude that~\eqref{eq:improved_estimate_induction_L2} holds for $k = m$ upon noting that $\rho^{-2} \leqslant (\sqrt{\mu}\rho)^{-2}$.
\end{proof}
\subsection{A local convergence result}\label{subsec:local-convergence}
As an application of the estimates obtained thus far, in this section, we first prove a local convergence result when the smallness conditions of the previous section are in effect (Proposition~\ref{prop:local_convergence}). Then, assuming in addition that $M$ is closed, we establish a corollary involving the Hodge decomposition of the longitudinal component of the curvature (Proposition~\ref{prop:convergence_nonharmonic_parts}). Both results play a role in the analysis in Section~\ref{sec:asymptotic}. We first derive some relevant identities.
\begin{lemm}\label{lemm:identities_for_local_convergence}
Let $(\nabla, \Phi)$ be a smooth solution of~\eqref{eq: 2nd_order_crit_pt_intro}. Define the functions 
\[
\xi = \ep^{-1}e_{\ep}(\nabla, \Phi) - \ep \big|  \bangle{F_{\nabla}, \Phi} \big|^2,\ \ \ \ q = 2\bangle{\ast F_{\nabla},\nabla \Phi},
\]
and the real-valued $1$-form
\[
h = \ep^{\frac{1}{2}}\bangle{\ast F_{\nabla}, \Phi}.
\]
Then the following hold.
\vskip 1mm
\begin{enumerate}
\item[(a)] $\xi = 2w \cdot \ep |F_{\nabla}|^2 + \ep^{-1}|\nabla\Phi|^2 + \lambda\ep^{-3}w^2 + \ep |[F_{\nabla}, \Phi]|^2$.
\vskip 1mm
\item[(b)] In terms of a local orthonormal frame, we have
\[
(dh)_{ij} = - \bangle{\ep^{\frac{1}{2}}(\ast F)_i, \nabla_j\Phi} + \bangle{\ep^{\frac{1}{2}}(\ast F)_j, \nabla_i\Phi},\ \ \ \ d^*h = -\bangle{\ep^{\frac{1}{2}}\ast F_{\nabla}, \nabla\Phi}.
\]
\vskip 1mm
\item[(c)] $q = -2\ep^{-\frac{1}{2}}d^*h$.
\end{enumerate}
\end{lemm}
\begin{proof}
Part (a) follows from a simple calculation using~\eqref{eq: bracket_norm} and the relation $2w + |\Phi|^2 = 1$. Next, from the Yang--Mills--Higgs equations~\eqref{eq: 2nd_order_crit_pt_intro} and the Bianchi identity we get
\[
d_\nabla (\ast F_{\nabla}) = \ep^{-2}[*\nabla\Phi, \Phi],\ \ \ d_{\nabla}^* (* F_{\nabla}) = 0.
\]
Consequently, letting $e_1, e_2, e_3$ be a local geodesic frame on $M$, we compute
\[
\begin{split}
\ep^{-\frac{1}{2}}d^*h =\ & -e_i\bangle{(\ast F)_i, \Phi}= \bangle{d_{\nabla}^* \ast F, \Phi} - \bangle{\ast F_{\nabla}, \nabla\Phi} = -\bangle{\ast F_{\nabla}, \nabla\Phi},
\end{split}
\]
and that
\[
\begin{split}
\ep^{-\frac{1}{2}}(dh)_{ij} =\ & e_{i}\bangle{(\ast F)_j, \Phi} -  e_{j}\bangle{(\ast F)_i, \Phi}\\
=\ & \bangle{ (d_{\nabla}\ast F)_{ij}, \Phi}  - \bangle{(\ast F)_i, \nabla_j\Phi} + \bangle{(\ast F)_j, \nabla_i\Phi}\\
=\ & \ep^{-2}\bangle{[(*\nabla\Phi)_{ij}, \Phi], \Phi} - \bangle{(\ast F)_i, \nabla_j\Phi} + \bangle{(\ast F)_j, \nabla_i\Phi}\\
=\ &  - \bangle{(\ast F)_i, \nabla_j\Phi} + \bangle{(\ast F)_j, \nabla_i\Phi},
\end{split}
\]
where to get the last line we used the fact that $\bangle{[\ \cdot\ , \Phi], \Phi} = 0$. This proves (b), from which we immediately get part (c).
\end{proof}

We proceed to describe the setup of the convergence result mentioned above. As before we take $\Omega$ to be an open subset of $M^3$, but suppose that we have a sequence of Riemannian metrics $(g_i)$ that converges smoothly on compact subsets of $\Omega$ to some metric $g$, and that there exist constants $\rho_0, A_0, A_1, \cdots$ such that~\eqref{eq:injectivity_radius_for_estimates} and~\eqref{eq:curvature_bound_for_estimates} hold for $g$ and all $g_i$. These constants in turn determine $\mu_1$ as in~\eqref{eq:radius_rel_curvature}, and we again define
\[
\rho_1 = \mu_1 \rho_0.
\]
The dependence of various quantities on the choice of metric, when we want to emphasize it, will be marked by superscripts or subscripts. Thus, for instance, $B_r^{g_i}(x)$ denotes a geodesic ball with respect to $g_i$. Also, expressions of the form $|\nabla^{k}(\ \cdot\ )|_{g_i}$ signify that the covariant derivative and the norm are taken with respect to $g_i$. For another example, given a configuration $(\nabla, \Phi)$, we use $e_\ep^{g_i}(\nabla,\Phi)$ to denote the quantity defined by~\eqref{eq:e-definition} where the tensor norms are computed using $g_i$, and $\cY_{\ep}^{g_i}(\nabla, \Phi)$ denotes its integral over $M$ with respect to $\vol_{g_i}$.

Now, fix $\lambda > 0$ and suppose $(\ep_i)$ is a sequence of positive numbers converging to $0$ such that for each $i$ we have a solution $(\nabla_i, \Phi_i)$ of~\eqref{eq: 2nd_order_crit_pt_intro} on $(\Omega, g_i)$ with $\ep = \ep_i$. For brevity, we denote $e_{\ep_i}^{g_i}(\nabla_i, \Phi_i)$ simply by $e_{\ep_i}^{g_i}$, and also introduce the functions
\[
w_i = \frac{1}{2}(1 - |\Phi_i|^2),\ \ \ \ \xi_i = \ep_i^{-1}e_{\ep_i}^{g_i} - \ep_i \big|  \bangle{F_{\nabla_i}, \Phi_i} \big|_{g_i}^2,\ \ \ \ q_i = 2\bangle{\ast_{g_i}F_{\nabla_i},\nabla_i \Phi_i}_{g_i},
\]
along with the real-valued $1$-form
\[
h_i = \ep_{i}^{\frac{1}{2}}\bangle{\ast_{g_i}F_{\nabla_i}, \Phi_i},
\]
where $\ast_{g_i}$ denotes the Hodge star operator with respect to $g_i$. 
\begin{prop}\label{prop:local_convergence}
In the setting described in the two preceding paragraphs, suppose $\lambda \in (0, \lambda_0]$ and let $\overline{\eta}$, $(\eta_0, \tau_0)$, and $(\theta_1, \tau_1)$ denote, respectively, the thresholds given by Lemma~\ref{lemm:improved_coarse_estimate_base}, Lemma~\ref{lemm:exp_decay_base} and Lemma~\ref{lemm:nablaPhi_exp_decay_base}. Assume further that for some pre-compact subset $\Omega'\subset \Omega$ and some geodesic ball $B_{96\rho}^{g}(x_0)\subset \Omega'$ with $\rho \in (0, \frac{\rho_1}{96})$, there holds for all $i$ that
\begin{equation}\label{eq:small_energy_for_convergence}
\int_{B_{96\rho}^g(x_0)} \ep_i^2 |F_{\nabla_i}|_{g_i}^2 + |\nabla_i\Phi_i|_{g_i}^2 \vol_{g_i} \leqslant \ep_i \cdot \min\{\overline{\eta}, \eta_0, \theta_1 \mu\}, 
\end{equation}
where $\mu = \min\{\lambda, 1\}$, and that
\begin{equation}\label{eq:small_w_for_convergence}
\|w_i\|_{\infty; B_{72\rho}^{g}(x_0)} \leqslant \beta,
\end{equation}
for some $\beta \in (0, \frac{1}{4(\lambda_0 + 2)})$. Then we have the following.
\vskip 1mm
\begin{enumerate}
\item[(a)] $\xi_i \to 0$ smoothly on $B_{3\rho}^{g}(x_0)$.
\vskip 1mm
\item[(b)] Up to taking a subsequence, the $1$-forms $h_i$ converge smoothly on $B_{2\rho}^{g}(x_0)$, and the limit $h$ is harmonic with respect to $g$ in the sense that $dh = 0$ and $d^*_g h = 0$.
\vskip 1mm
\item[(c)] $q_i \to 0$ smoothly on $B_{3\rho}^{g}(x_0)$.
\end{enumerate}
\end{prop}
\begin{proof}
Without loss of generality we may assume further that, for all $i$, we have 
\begin{equation}\label{eq:small_ep_for_local_convergence}
\frac{\ep_i}{\rho} \leqslant \min\{\tau_0, \tau_1\sqrt{\mu}\},
\end{equation}
and that 
\[
B_{3r/4}^{g_i}(x) \subset B_{r}^{g}(x) \subset B_{5r/4}^{g_i}(x), \text{ whenever }B_r^{g}(x) \subset \Omega' \text{ with }r \in (0, \rho_1).
\]
In particular, 
\begin{equation}\label{eq:local_convergence_ball_inclusion}
B_{54\rho}^{g_i}(x_0) \subset B_{72\rho}^{g}(x_0),\ \ \ B_{72\rho}^{g_i}(x_0) \subset B_{96\rho}^{g}(x_0),\ \ \ B_{3\rho}^{g}(x_0) \subset B_{72\rho/16}^{g_i}(x_0).
\end{equation}
Next we derive the estimates from which the desired conclusions are drawn. To simplify notation, we temporarily drop the subscript $i$ in $\ep_i$, $(\nabla_i, \Phi_i)$, $w_i$, $\xi_i$, $h_i$ and $q_i$. We also drop the reference to the metric $g_i$ in tensor norms, covariant derivatives, and the Hodge star operator. That said, we still denote by $B_r^{g}(x)$ geodesic balls with respect to the limiting metric $g$. As before, the assumption~\eqref{eq:small_w_for_convergence} immediately gives
\begin{equation}\label{eq:Phi_bound_for_local_convergence}
\frac{1}{2} \leqslant |\Phi|^2 \leqslant \frac{3}{2} \text{ on }B^{g}_{72\rho}(x_0).
\end{equation}
Moreover, combining~\eqref{eq:small_w_for_convergence} with the no-concentration assumption~\eqref{eq:small_energy_for_convergence}, the smallness~\eqref{eq:small_ep_for_local_convergence} of $\ep$, and the inclusions~\eqref{eq:local_convergence_ball_inclusion}, we may invoke Lemma~\ref{lemm:exp_decay_base}, Proposition~\ref{prop:exp_decay}, Lemma~\ref{lemm:nablaPhi_exp_decay_base}, Proposition~\ref{prop:nablaPhi_exp_decay}, and Proposition~\ref{prop:improved_estimates} on $B_{72\rho}^{g_i}(x_0)$ to get the following estimates on $B_{3\rho}^{g}(x_0)$ for all $m \in \NN \cup \{0\}$:
\begin{equation}\label{eq:estimates_for_local_convergence}
\begin{split}
\big|\nabla^m[\ep F_{\nabla}, \Phi]\big| + \big|\nabla^m[\nabla \Phi, \Phi]\big| \leqslant\ & C \ep^{-m-1}e^{-a\frac{\rho}{\ep}},\\
w \leqslant\ & C e^{-a\frac{\sqrt{\mu}\rho}{\ep}},\\
|\nabla^{m + 1}\Phi| \leqslant\ & C(\sqrt{\mu}\rho)^{-m-1}e^{-a\frac{\sqrt{\mu}\rho}{\ep}},\\
|\nabla^m (\ep F_{\nabla})| + |\nabla^{m + 1}\Phi| \leqslant\ & C (\sqrt{\mu}\rho)^{-m-\frac{3}{2}}\ep^{\frac{1}{2}},
\end{split}
\end{equation}
where $C = C(m, \lambda_0, A_1, \cdots, A_{m + 1})$ and $a = a(m)$. In any case, $C$ and $a$ do not depend on $i$. From the above estimates together with Lemma~\ref{lemm:identities_for_local_convergence}(a) and~\eqref{eq:trading_factors}, we see that 
\begin{equation}\label{eq:local_convergence_xi_bound}
\begin{split}
\|\xi\|_{\infty; B^g_{3\rho}(x_0)} \leqslant\ & C(\sqrt{\mu}\rho)^{-3} e^{-a \frac{\sqrt{\mu}\rho}{\ep}}.
\end{split}
\end{equation}
Similarly, differentiating $\xi$ and noting that $\nabla^m w = -\frac{1}{2}\nabla^m \big(|\Phi|^2\big)$ for $m \in \NN$, we have
\begin{equation}\label{eq:local_convergence_xi_derivative_bound}
\|\nabla^{m} \xi\|_{\infty; B^g_{3\rho}(x_0)} \leqslant C(\sqrt{\mu}\rho)^{-m-3}e^{-a\frac{\sqrt{\mu}\rho}{\ep}}.
\end{equation}
Next, expressing $h$ as $\ep^{-\frac{1}{2}}\bangle{\ep \ast F_{\nabla}, \Phi}$, we see from~\eqref{eq:Phi_bound_for_local_convergence} and~\eqref{eq:estimates_for_local_convergence} that
\begin{equation}\label{eq:local_convergence_h_bound}
\|h\|_{\infty; B_{3\rho}^{g}(x_0)} \leqslant C(\sqrt{\mu}\rho)^{-\frac{3}{2}}.
\end{equation}
For the covariant derivatives of $h$ we have
\begin{equation}\label{eq:local_convergence_h_derivative_bound}
\begin{split}
\|\nabla^m h\|_{\infty; B_{3\rho}^{g}(x_0)} \leqslant\ & C_m\ep^{-\frac{1}{2}} |\bangle{\nabla^m(\ep F_{\nabla}), \Phi}|  + C_m \sum_{k = 1}^{m} \ep^{-\frac{1}{2}} |\bangle{\nabla^{m-k}(\ep F_{\nabla}), \nabla^{k}\Phi}| \\
\leqslant\ & C(\sqrt{\mu}\rho)^{-m-\frac{3}{2}}.
\end{split}
\end{equation}
Finally, from Lemma~\ref{lemm:identities_for_local_convergence}(b) and~\eqref{eq:estimates_for_local_convergence} we easily get
\begin{equation}\label{eq:local_convergence_hodge_estimates}
\begin{split}
|dh| + |d^*h| \leqslant\ & C\ep^{-\frac{1}{2}} |\ep F_{\nabla}|\cdot|\nabla\Phi|\leqslant C(\sqrt{\mu}\rho)^{-\frac{5}{2}}e^{-a\frac{\sqrt{\mu}\rho}{\ep}}\  \text{ on }B_{3\rho}^{g}(x_0).
\end{split}
\end{equation}
Likewise, differentiating the formulas for $dh$ and $d^*h$ in Lemma~\ref{lemm:identities_for_local_convergence}(b), we estimate
\begin{equation}\label{eq:local_convergence_hodge_derivative_estimates}
\begin{split}
|\nabla^m dh| + |\nabla^m d^*h| \leqslant\ & C_m\sum_{k = 0}^{m} \ep^{-\frac{1}{2}} |\nabla^{m-k}(\ep F_{\nabla})| |\nabla^{k + 1}\Phi|\\
\leqslant\ & C(\sqrt{\mu}\rho)^{-m-\frac{5}{2}}e^{-a\frac{\sqrt{\mu}\rho}{\ep}}, \text{ on }B_{3\rho}^{g}(x_0).
\end{split}
\end{equation}
We now put back the subscript $i$ and notice that $\xi_i \to 0$ uniformly on $B_{3\rho}^{g}(x_0)$ by~\eqref{eq:local_convergence_xi_bound}. Moreover, with $\nabla^{g_i}$ denoting the Levi--Civita connection of $g_i$, we see by~\eqref{eq:local_convergence_xi_derivative_bound} that $|(\nabla^{g_i})^{m}\xi_i|_{g_i} \to 0$ uniformly on $B_{3\rho}^{g}(x_0)$, for all $m \in \NN$. Since $g_i$ converges smoothly to $g$ on $\Omega'$, we conclude that, for all $m \in \NN \cup \{0\}$, 
\[
|(\nabla^{g})^{m}\xi_i|_{g} \to 0 \text{ uniformly on }B_{3\rho}^{g}(x_0).
\]
This proves (a). Similarly, we infer from~\eqref{eq:local_convergence_h_bound} and~\eqref{eq:local_convergence_h_derivative_bound} that $h_i$, together with its covariant derivatives with respect to $g$ of all orders, are bounded uniformly in $i$ on $B_{3\rho}^{g}(x_0)$. From this and a diagonal argument we get a subsequence of $(h_i)$, which we do not relabel, that converges smoothly on $B_{2\rho}^{g}(x_0)$. Taking into account also~\eqref{eq:local_convergence_hodge_estimates} and~\eqref{eq:local_convergence_hodge_derivative_estimates}, and using again the smooth convergence $g_i \to g$, we deduce that the limit must be harmonic with respect to $g$. This proves (b). Finally, recalling the identity in Lemma~\ref{lemm:identities_for_local_convergence}(c) and again using the estimates~\eqref{eq:local_convergence_hodge_estimates} and~\eqref{eq:local_convergence_hodge_derivative_estimates}, we get for all $m \in \NN \cup \{0\}$ that
\[
|(\nabla^{g_i})^{m} q_i|_{g_i} \leqslant C \ep_i^{-\frac{1}{2}}(\sqrt{\mu}\rho)^{-m-\frac{5}{2}} e^{-a\frac{\sqrt{\mu}\rho}{\ep_i}}  \ \text{ on }B_{3\rho}^{g}(x_0).
\]
Using once again the convergence of $g_i$ to $g$ on $\Omega'$, we conclude that $q_i \to 0$ smoothly on $B_{3\rho}^{g}(x_0)$. This finishes the proof.
\end{proof}
\begin{rmk}\label{rmk:smallness_condition_relation}
For later use, we make two simple observations about Proposition~\ref{prop:local_convergence}.
\begin{enumerate}
\item[(i)] Letting $\theta_0 = \theta_0(\lambda_0, \beta)$ denote the constant from Remark~\ref{rmk:clearing_out}(i), with the parameters $\lambda_0, \beta$ being the ones in Proposition~\ref{prop:local_convergence}, we see that if for all $i$ there holds
\[
\int_{B_{96\rho}^{g}(x_0)} e_{\ep_i}^{g_i}(\nabla_i, \Phi_i) \vol_{g_i} \leqslant \ep_i \cdot  \min\{\overline{\eta}, \eta_0,  \theta_0 \mu, \theta_1 \mu\},
\]
then eventually the assumptions~\eqref{eq:small_energy_for_convergence} and~\eqref{eq:small_w_for_convergence} are fulfilled.
\vskip 1mm
\item[(ii)] Under the hypotheses of Proposition~\ref{prop:local_convergence}, upon combining conclusions (a) and (b), we infer that
\[
\ep_i^{-1}e_{\ep_i}^{g_i} \rightarrow |h|_g^2 \ \text{ smoothly on }B_{2\rho}^{g}(x_0),
\]
where $h$ is the limiting harmonic $1$-form from part (b). 
\end{enumerate}
\end{rmk}
Below we continue to work in the setting described just before the statement of Proposition~\ref{prop:local_convergence}, but specialize to the case where $M$ is closed and $g_i = g$ for all $i$. Further, we assume there exists some $\Lambda > 0$ such that
\begin{equation}\label{eq:energy_bound_for_nonharmonic}
\cY_{\ep_i}(\nabla_i, \Phi_i) \leqslant \ep_i \cdot \Lambda, \text{ for all }i.
\end{equation}
In particular, fixing any $\rho \in (0, \frac{\rho_1}{4})$, we see that the assumptions~\eqref{eq:e_estimate_energy_bound} and~\eqref{eq:e_estimate_potential_bound} from Lemma~\ref{lemm:coarse_estimate_base} and Lemma~\ref{lemm:w_mean_value_estimate}, respectively, are fulfilled at any $x_0 \in M$, and hence provided $i$ is so large that $\ep_i \leqslant \rho$, we may invoke~\eqref{eq:w_estimate} and conclude that 
\begin{equation}\label{eq:Phi_global_bound_for_nonharmonic}
\|\Phi_i\|_{\infty; M} \leqslant K_0 = K_0(\Lambda, \lambda, \lambda_0),
\end{equation}
where recall that $\lambda_0$ is the upper bound for $\lambda$ from Proposition~\ref{prop:local_convergence}. Next, we consider the Hodge decomposition of the $1$-forms $h_i = \ep_i^{\frac{1}{2}}\bangle{* F_{\nabla_i}, \Phi_i}$, namely
\begin{equation}\label{eq:Hodge_decomposition_longitudinal_part}
\ep_i^{\frac{1}{2}}\bangle{* F_{\nabla_i}, \Phi_i} = \widetilde{h}_i + df_i + d^*\alpha_i,
\end{equation}
where $f_i \in C^{\infty}(M)$, $\alpha_i \in \Omega^2(M)$, and $\widetilde{h}_i$ is a harmonic $1$-form on $M$. By subtracting from $f_i$ its average over $M$ and also dropping the components in $\ker(d^*)$ from the Hodge decomposition of $\alpha_i$, we can further assume that 
\[
\int_M f_i = 0,\quad d\alpha_i = 0,
\]
and that 
\[
\int_{M}\langle \alpha_i, h\rangle = 0,\quad \text{for all harmonic $2$-forms }h \text{ on }M.
\]
Finally, given $\beta \in (0, \frac{1}{4(\lambda_0 + 2)})$, we let $\theta_0 = \theta_0(\lambda_0, \beta)$ denote the constant in Remark~\ref{rmk:clearing_out}(i).
\begin{prop}\label{prop:convergence_nonharmonic_parts}
Assume that $(M^3, g)$ is closed. Then, in the above setting, we have:
\vskip 1mm
\begin{enumerate}
\item[(a)] Up to taking a subsequence, there exist a function $f$ and a $2$-form $\alpha$, both of class $W^{1, 2} \cap L^6$ on $M$, such that 
\begin{equation}\label{eq:convergence_nonharmonic_parts}
f_i \to f,\ \ \alpha_i \to \alpha, \text{ weakly in $W^{1, 2}$ and strongly in $L^2$}.
\end{equation}
\vskip 1mm
\item[(b)] If, in addition, for some geodesic ball $B_{48\rho}(x_0)$ in $M$ with $\rho \in (0, \frac{\rho_1}{48})$, we have 
\begin{equation}\label{eq:small_energy_for_nonharmonic}
\int_{B_{48\rho}(x_0)} e_{\ep_i}(\nabla_i, \Phi_i) \vol_{g} \leqslant \ep_i \cdot \min\{\overline{\eta}, \eta_0, \theta_0\mu, \theta_1 \mu\}, \text{ for all }i,
\end{equation}
then, along a subsequence, $(f_i)$ and $(\alpha_i)$ converge smoothly on compact subsets of $B_{2\rho}(x_0)$.
\end{enumerate}
\end{prop}
\begin{proof}
As before, we assume, without loss of generality, that $\frac{\ep_i}{\rho} \leqslant \min\{\tau_0, \tau_1\sqrt{\mu}\}$ for all $i$, where $\mu = \min\{\lambda, 1\}$. Since the terms on the right-hand side of~\eqref{eq:Hodge_decomposition_longitudinal_part} are mutually $L^2$-orthogonal, we have 
\begin{equation}\label{eq:longitudinal_hodge_orthogonality_bound}
\|\widetilde{h}_i\|_{2; M}^2 + \|df_i\|_{2; M}^2 + \|d^*\alpha_i\|_{2; M}^2 = \|\ep_i^{\frac{1}{2}}\bangle{* F_{\nabla_i}, \Phi_i}\|_{2;M}^2 \leqslant C_{\lambda, \lambda_0, \Lambda},
\end{equation}
where the inequality follows from~\eqref{eq:energy_bound_for_nonharmonic} and~\eqref{eq:Phi_global_bound_for_nonharmonic}. Applying the Poincar\'e inequality to $f_i$, and basic $L^2$-estimate for $d + d^*$ to $\alpha_i$, we get
\begin{equation}\label{eq:longitudinal_hodge_W12_bound}
\begin{split}
\|f_i\|_{1, 2; M} \leqslant\ & C_{M}\|df_i\|_{2; M}\leqslant C_{M,\lambda, \lambda_0, \Lambda},\\
\|\alpha_i\|_{1, 2; M} \leqslant\ & C_M \|d^*\alpha_i\|_{2; M}\leqslant C_{M, \lambda, \lambda_0, \Lambda},
\end{split}
\end{equation}
from which part (a) follows, the fact that $f$ and $\alpha$ are of class $L^6$ being a consequence of the $3$-dimensional Sobolev embedding $W^{1,2} \hookrightarrow L^6$. For part (b), as mentioned in Remark~\ref{rmk:smallness_condition_relation}(i), the assumption~\eqref{eq:small_energy_for_nonharmonic} puts us in the setting of Proposition~\ref{prop:local_convergence}. In particular, noting that
\[
d^* df_i = d^*(\ep_i^{\frac{1}{2}}\bangle{* F_{\nabla_i}, \Phi_i}), \ \ \ dd^*\alpha_i = d(\ep_i^{\frac{1}{2}}\bangle{* F_{\nabla_i}, \Phi_i}),
\]
and following the argument leading to the estimates~\eqref{eq:local_convergence_hodge_estimates} and~\eqref{eq:local_convergence_hodge_derivative_estimates} in the proof of Proposition~\ref{prop:local_convergence}, we deduce that 
\[
\lim_{i \to \infty}\big(\|d^* df_i\|_{m, 2; B_{3\rho}(x_0)} + \|d d^*\alpha_i\|_{m, 2; B_{3\rho}(x_0)} \big) = 0, \text{ for all }m \in \NN \cup \{0\}.
\]
Combining this with standard interior estimates for $d + d^*$ applied to the forms $df_i$ and $d^*\alpha_i$, and recalling the bound~\eqref{eq:longitudinal_hodge_orthogonality_bound} on their $L^2$-norms, we get for all $m \in \NN \cup \{0\}$ that
\[
\limsup_{i \to \infty}\big( \|df_i\|_{m + 1, 2; B_{\frac{5\rho}{2}}(x_0)} + \|d^*\alpha_i\|_{m + 1, 2; B_{\frac{5\rho}{2}}(x_0)} \big) < \infty.
\]
Then, by the $L^2$-bound on $f_i$ and $\alpha_i$ in~\eqref{eq:longitudinal_hodge_W12_bound}, and the fact that $f_i$ are functions while $d\alpha_i = 0$, we get
\[
\limsup_{i \to \infty}\big(\|f_i\|_{m +2, 2; B_{2\rho}(x_0)} + \|\alpha_i\|_{m +2, 2; B_{2\rho}(x_0)} \big) < \infty, \text{ for all }m \in \NN \cup \{0\}.
\]
Standard arguments then yield a subsequence of $i$'s along which the asserted smooth convergence holds.
\end{proof}
\subsection{Some other global consequences of the estimates}\label{subsec:consequences-estimates}
For use later in the paper, in this section we collect two other consequences of the previous estimates. Below we assume either that $(M^3, g)$ is closed or that it is non-compact with bounded geometry, with further restrictions specified when needed. In particular, there exist constants $\rho_0, A_0, A_1, \cdots$ such that~\eqref{eq:injectivity_radius_for_estimates} and~\eqref{eq:curvature_bound_for_estimates} hold with $\Omega = M$, and we define $\rho_1 = \mu_1 \rho_0$, where $\mu_1$ is given by~\eqref{eq:radius_rel_curvature}.

We begin with a decay property at infinity of solutions to~\eqref{eq: 2nd_order_crit_pt_intro} with finite energy. For our purposes in this paper, we only need to establish the decay qualitatively. See~\cite{fadel2023asymptotics} for more quantitative decay estimates, over asymptotically conical $3$-manifolds, for finite energy critical points of the $SU(2)$ Yang--Mills--Higgs functional without the self-interaction term.
\begin{prop}\label{prop:finite_action_decay}
Suppose $M$ is non-compact with bounded geometry, and let $(\nabla, \Phi)$ be a smooth solution of~\eqref{eq: 2nd_order_crit_pt_intro} on $M$ satisfying 
\begin{equation}\label{eq:finite_action_for_decay}
\cY_{\ep}(\nabla, \Phi) < \infty.
\end{equation}
Assume also that $\ep < \frac{\rho_1}{4}$. Then for all $m \in \NN \cup \{0\}$, we have that $|\nabla^m w|$, $|\nabla^{m + 1}\Phi|$ and $|\nabla^m F_{\nabla}|$ decay uniformly to zero at infinity, in the sense that for all $\alpha > 0$ there exists a compact subset $K$ of $M$ such that
\[
|\nabla^m w| + |\nabla^{m + 1}\Phi| + |\nabla^m F_{\nabla}| < \alpha, \text{ on }M \setminus K.
\]
\end{prop}
\begin{proof} 
Fixing a reference point $p_0 \in M$ and letting
\[
\mu(R) = \ep^{-1}\int_{M \setminus B_R(p_0)}e_{\ep}(\nabla, \Phi)\vol_{g},
\]
we deduce from the finite energy assumption~\eqref{eq:finite_action_for_decay} that $\lim_{R \to \infty}\mu(R) = 0$. Next we choose some $R_0 > 4\ep$ such that $\mu(R_0) < \lambda$, and observe that for all $x_0 \in M \setminus B_{2R_0}(p_0)$ we have 
\[
d(x_0, p_0) - 4\ep \geqslant \frac{1}{2}d(x_0, p_0) \geqslant R_0.
\]
From this we deduce that $B_{4\ep}(x_0) \subset M \setminus B_{\frac{d(x_0, p_0)}{2}}(p_0) \subset M \setminus B_{R_0}(p_0)$, and hence
\begin{equation}\label{eq:small_energy_far_out}
\int_{B_{4\ep}(x_0)}e_{\ep}(\nabla, \Phi) \vol_g \leqslant \ep \cdot \mu(\frac{d(x_0, p_0)}{2}) < \ep \cdot \lambda.
\end{equation}
Lemma~\ref{lemm:coarse_estimate_base} and Lemma~\ref{lemm:w_mean_value_estimate} (with $\lambda_0 = \lambda$ and $\rho = \ep$) then gives, respectively,
\begin{equation}\label{eq:finite_action_e_bound}
\ep^2 |F_{\nabla}|^2 + |\nabla \Phi|^2  \leqslant C\ep^{-2}\mu(\frac{d(x_0, p_0)}{2}) < C\ep^{-2}\lambda \text{ on }B_{3\ep}(x_0),
\end{equation}
and that 
\begin{equation}\label{eq:finite_action_Phi_bound}
|1 - |\Phi|| \leqslant C \lambda^{-\frac{1}{7}} \big[ \mu(\frac{d(x_0, p_0)}{2}) \big]^{\frac{1}{7}} \leqslant C \text{ on }B_{3\ep}(x_0),
\end{equation}
where the constants $C$ depend only on $\lambda$. The latter estimate implies~\eqref{eq:Phi_K_0}, and hence we are permitted to apply Proposition~\ref{prop:coarse_estimate} (again with $\rho = \ep$) to get for all $k \in \NN$ that
\begin{equation}\label{eq:finite_action_coarse_bound}
\ep^2 |\nabla^{k} F_{\nabla}|^2 + |\nabla^{k + 1}\Phi|^2 \leqslant C_{k, \lambda, A_1, \cdots, A_k}\ep^{-2k-2}\mu(\frac{d(x_0, p_0)}{2}), \text{ on }B_{2\ep}(x_0).
\end{equation}
Combining~\eqref{eq:finite_action_coarse_bound},~\eqref{eq:finite_action_Phi_bound}, and~\eqref{eq:finite_action_e_bound}, we get for all $k \in \NN$ that
\begin{equation}\label{eq:finite_action_w_derivative_bound}
|\nabla^k w| \leqslant C_{k, \lambda, A_1, \cdots, A_k} \ep^{-k} [\mu(\frac{d(x_0, p_0)}{2})]^{\frac{1}{2}}, \text{ on }B_{2\ep}(x_0),
\end{equation}
while for $w$ itself we have from~\eqref{eq:finite_action_Phi_bound} that
\[
|w| = |1 - |\Phi|| \cdot \frac{1 + |\Phi|}{2} \leqslant C_{\lambda} \cdot \lambda^{-\frac{1}{7}} \big[ \mu(\frac{d(x_0, p_0)}{2}) \big]^{\frac{1}{7}}, \text{ on }B_{3\ep}(x_0).
\]
Since $\ep$ and $\lambda$ are fixed while $x_0 \in M \setminus B_{2R_0}(p_0)$ is arbitrary, we obtain the desired conclusion from this last estimate together with~\eqref{eq:finite_action_w_derivative_bound},~\eqref{eq:finite_action_coarse_bound}, and~\eqref{eq:finite_action_e_bound}, upon recalling that $d(x_0, p_0) \to \infty$ as $x_0$ escapes compact subsets by the completeness of $M$, and that $\lim_{R \to \infty}\mu(R) = 0$.
\end{proof}

The next result is a direct consequence of the identity~\eqref{eq:w_laplacian_with_lambda} and the strong maximum principle. 
\begin{prop}\label{prop:maximum_principle_for_w}
Suppose either that $(M, g)$ is closed or that it is non-compact with bounded geometry. Let $(\nabla, \Phi)$ be a smooth solution of~\eqref{eq: 2nd_order_crit_pt_intro} on $M$, with $\ep < \frac{\rho_1}{4}$. In the case $M$ is non-compact, we assume in addition that $\cY_{\ep}(\nabla, \Phi) < \infty$. Then we have 
\begin{equation}\label{eq:Phi_bound_from_max_principle}
|\Phi(x)| \leqslant 1, \text{ for all }x \in M.
\end{equation}
Moreover, if equality holds at some point, then $|\Phi| \equiv 1$, in which case necessarily $\nabla\Phi \equiv 0$.
\end{prop}
\begin{proof}
From~\eqref{eq:w_laplacian_with_lambda} we have
\begin{equation}\label{eq:w_diff_ineq}
\big( \Delta + \frac{\lambda|\Phi|^2}{\ep^2} \big)w = |\nabla\Phi|^2 \geqslant 0\  \text{ on }M.
\end{equation}
Assume by contradiction that $w$ becomes negative somewhere. Then thanks to Proposition~\ref{prop:finite_action_decay}, in both the compact and non-compact cases, we get some $x_0 \in M$ such that
\[
0 > w(x_0) = \inf_{x \in M} w(x).
\]
Since $M$ is connected, by~\eqref{eq:w_diff_ineq} and the strong maximum principle we conclude that $w \equiv w(x_0)$. Substituting this back into~\eqref{eq:w_diff_ineq} leads to
\[
(1-2w(x_0))\cdot w(x_0) \geqslant 0,
\]
a contradiction since $w(x_0) < 0$, and hence we must have $w \geqslant 0$ everywhere on $M$. Applying the strong maximum principle once more, we see that if $w$ vanishes somewhere, then $w \equiv 0$, in which case~\eqref{eq:w_diff_ineq} gives $|\nabla\Phi| \equiv 0$. The proof is complete.
\end{proof}

\section{Gap theorems}\label{sec:gap}
This section is dedicated to the proofs of the gap results stated as Theorems~\ref{thm: gap} and \ref{thm: gap_for_R3} in the introduction, as well as to giving a construction of reducible solutions of equation~\eqref{eq: 2nd_order_crit_pt_intro} with non-zero energy, at least in the case where $(M^3,g)$ is a closed $3$-manifold admitting non-zero harmonic $2$-forms; see Lemma~\ref{lem: counter-examp} and Example~\ref{example: counter-examp}.
\vspace*{0.5cm}
\subsection{Proofs of gap theorems}\label{subsec:proofs_of_gap}
Here we adopt the setting of \S\ref{subsec:consequences-estimates}, so that both~\eqref{eq:injectivity_radius_for_estimates} and~\eqref{eq:curvature_bound_for_estimates} hold with $\Omega = M$, and $\rho_1$ is given by~\eqref{ineq: rho_1_def}. 
We continue to define $\mu = \min\{\lambda, 1\}$, and let $\lambda_0$ be an upper bound for $\lambda$. 
Then Lemma~\ref{lemm:improved_coarse_estimate_base} and Lemma~\ref{lemm:nablaPhi_exp_decay_base} produce, respectively, thresholds $\overline{\eta}$ and $(\theta_1,\tau_1)$, all depending only on $\lambda_0$. Also, given $\beta \in (0, \min\{\frac{1}{6}, \frac{1}{2(\lambda_0 + 2)}\})$, we let $\theta_0 = \theta_0(\lambda_0, \beta)$ denote the constant from Remark~\ref{rmk:clearing_out}(i). The next result shows that if a certain non-concentration condition for a critical point $(\nabla,\Phi)$ holds throughout the manifold on balls of a fixed radius, then in fact $(\nabla,\Phi)$ is reducible satisfying \eqref{eq: reducible_pair_intro}.
\begin{prop}\label{prop:concentration_scale}
Suppose that $(M^3, g)$ has bounded geometry, 
and let $(\nabla, \Phi)$ be a smooth solution of~\eqref{eq: 2nd_order_crit_pt_intro} such that for some $\sigma \in (0, \frac{\rho_1}{4})$ there holds
\begin{equation}\label{eq:small_energy_everywhere}
\sup_{x \in M} \int_{B_{4\sigma}(x)}e_{\ep}(\nabla, \Phi) vol_g < \ep\cdot \min\{\overline{\eta}, \theta_0\mu, \theta_1\mu\}.
\end{equation}
When $M$ is non-compact, we assume also that $\cY_{\ep}(\nabla, \Phi) < \infty$. Then, provided $\frac{\ep}{\sigma} \leqslant \tau_1\sqrt{\mu}$, we have that $(\nabla,\Phi)$ is a reducible solution satisfying \eqref{eq: reducible_pair_intro}.
\end{prop}
\begin{proof}
Thanks to~\eqref{eq:small_energy_everywhere} and our assumption on $\frac{\ep}{\sigma}$, at any given $x_0 \in M$ we may apply Remark~\ref{rmk:clearing_out}(i) and Lemma~\ref{lemm:nablaPhi_exp_decay_base} in succession. In particular, everywhere on $M$ we have
\begin{equation}\label{eq:Phi_bound_for_concentration_scale}
\frac{1}{2} \leqslant |\Phi|^2 \leqslant \frac{3}{2}, 
\end{equation}
\begin{equation}\label{eq:nablaPhi_diff_ineq_for_concentration_scale}
\Delta |\nabla\Phi|^2 + \frac{\mu}{2\ep^2} |\nabla\Phi|^2 \leqslant 0.
\end{equation}
Since $\Phi\not\equiv 0$, by Remark \ref{rmk: reducible_sols} we are left to prove that $\nabla\Phi\equiv 0$. We consider separately the compact and the non-compact cases. When $M$ is compact, we simply integrate~\eqref{eq:nablaPhi_diff_ineq_for_concentration_scale} over $M$ to find that $|\nabla\Phi| \equiv 0$. If $M$ is non-compact, we assume by contradiction that $|\nabla\Phi|$ is positive somewhere in $M$. Then since $\cY_{\ep}(\nabla, \Phi) < \infty$ and $\ep \leqslant \sigma < \frac{\rho_1}{4}$, we may invoke Proposition~\ref{prop:finite_action_decay} to deduce the existence of some $x_0 \in M$ such that
\[
0 < |\nabla\Phi(x_0)| = \sup_{x \in M}|\nabla\Phi(x)|.
\]
But then~\eqref{eq:nablaPhi_diff_ineq_for_concentration_scale} and the strong maximum principle forces $|\nabla\Phi| \equiv |\nabla\Phi(x_0)|$, in which case~\eqref{eq:nablaPhi_diff_ineq_for_concentration_scale} gives
\[
\frac{\mu}{2\ep^2}|\nabla\Phi(x_0)|^2 \leqslant 0,
\]
a contradiction. Thus we conclude that $|\nabla\Phi| \equiv 0$ in the non-compact case as well, as we wanted.
\end{proof}
We are now ready to prove:
\begin{thm}[Theorem \ref{thm: gap}]
Suppose $(M^3,g)$ has bounded geometry, and let $\lambda_0$ be an upper bound for $\lambda$. 
Then, there exist constants $\theta_{\mathrm{gap}} = \theta_{\mathrm{gap}}(\lambda_0)$ and $\tau_{\mathrm{gap}} = \tau_{\mathrm{gap}}(\lambda_0, M, g)$ 
such that if 
\[
\ep < \tau_{\mathrm{gap}} \cdot \min\{\sqrt{\lambda}, 1\}
\]
and if $(\nabla, \Phi) \in \mathscr{C}(E)$ is a 
solution of~\eqref{eq: 2nd_order_crit_pt_intro} satisfying 
    \[
    \ep^{-1}\cY_{\ep}(\nabla,\Phi)\leqslant \theta_{\mathrm{gap}} \cdot \min\{\lambda, 1\},
    \] then $(\nabla,\Phi)$ is reducible as in \eqref{eq: reducible_pair_intro}. In particular, if $(M^3,g)$ admits no non-zero $L^2$-bounded harmonic $2$-forms (or, equivalently, $1$-forms), for instance if furthermore we impose either of the following conditions:
\vskip 1mm
    \begin{itemize} {\normalsize 
        \item[(i)] $M$ is closed and $b_1(M)=0$, 
        \vskip 1mm
        \item[(ii)] $M$ is noncompact and $\mathrm{Ric}(g)\geqslant 0$,}
    \end{itemize}  
\vskip 1mm
then in fact $(\nabla,\Phi)$ is trivial, that is, $\cY_{\ep}(\nabla,\Phi)=0$.
\end{thm}
\begin{proof}[Proof of Theorem~\ref{thm: gap}]
            In the notation preceding Proposition \ref{prop:concentration_scale}, we take $\beta = \frac{1}{8(\lambda_0 + 2)}$ and set
            \[
            \theta_{\text{gap}}:=\min\{\overline{\eta}, \theta_0, \theta_1\}\quad\text{and}\quad \tau_{\text{gap}}:=\frac{\rho_1\tau_1}{8}.
            \] Then, whenever $\ep$ and $(\nabla,\Phi)$ satisfy the assumptions of the statement, the smallness bound \eqref{eq:small_energy_everywhere} is satisfied for $\sigma=\frac{\rho_1}{8}$ and by Proposition~\ref{prop:concentration_scale} we conclude that $(\nabla,\Phi)$ satisfies \eqref{eq: reducible_pair_intro}. By Remark \ref{rmk: reducible_sols}, the final conclusion follows.
            
            
           Although it is well known that if $(M,g)$ is complete and $\mathrm{Ric}(g)\geqslant 0$ then $(M,g)$ admits no non-zero $L^2$-bounded harmonic $1$-forms, for completeness purposes we include here a direct proof of the final conclusion under assumption (ii). Assume that $\mathrm{Ric}(g) \geqslant 0$ and $M$ is noncompact. Since $\nabla$ is Yang--Mills and the $2$-form $F_{\nabla}$ takes values in the abelian subbundle $\langle\Phi\rangle\subset\mathfrak{su}(E)$, we can use the Bochner formula~\eqref{eq: 1form_Weitzenbock} to obtain that 
            \begin{equation*}
                \nabla^\ast\nabla (\ast F_\nabla) = - \ast [\ast F_\nabla,\ast F_\nabla ] - \mathrm{Ric}_g(\ast F_\nabla) = - \mathrm{Ric}_g(\ast F_\nabla).
            \end{equation*}
            Taking the inner product of this with $\ast F_\nabla$, we have
            \begin{equation}\label{eq: bochner-gap}
                \frac{1}{2}\Delta|F_\nabla|^2 = - \mathrm{Ric}_g(\ast F,\ast F) - |\nabla \ast F_\nabla|^2\leqslant - |\nabla \ast F_\nabla|^2,
            \end{equation} where in the later inequality we used the assumption $\mathrm{Ric}(g) \geqslant 0$. Thus, by the maximum principle, either $|F_\nabla|$ has no local maxima or it is constant.
            Since $M$ is noncompact, Proposition~\ref{prop:finite_action_decay} applies to show that $|F_\nabla|$ decays uniformly to zero at infinity, so it must attain, and hence be constantly equal to, its maximum value in the interior of $M$. 
            This together with the decay property just noted forces $|F_{\nabla}|$ to be identically equal to $0$.
\end{proof}
\begin{rmk}
    Closed, oriented $3$-manifolds with vanishing first (and second) Betti number(s) are exactly the \emph{rational homology $3$-spheres}. 
    A class of examples are the 
    Lens spaces $L(p,q)$, which include the $3$-sphere $\mathbb{S}^3$ and the real projective $3$-space $\mathbb{RP}^3$.
\end{rmk}

 In the case where $(M, g)$ is the standard Euclidean $3$-space, it turns out that the scaling invariance of the Euclidean metric allows us to drop the smallness 
 condition on the parameter $\ep$ 
 from the hypotheses of Theorem \ref{thm: gap}. To prove this, we first 
 recall a few standard facts about scalings, some of which will only be used later in Section~\ref{sec:asymptotic}.  
 
 Given a scale $\delta>0$, we will write $|\cdot|_{\delta}$, $\vol_{\delta}$ and $\ast_{\delta}$ for the norm, volume form and Hodge star operator, respectively, associated to $\delta^2 g$. Also, we let  $B_r^{\delta}(x)$ and $\cY_{\ep}^{\delta}$ denote the ball and the Yang--Mills--Higgs energy, respectively, with respect to $\delta^2 g$. With these notations, it is straightforward to check the following identities. The details are omitted.
 
 \begin{lemm}\label{lemm:scaling_g_to_delta_g}
     For each $\ep, \delta\in \RR_+$ and $(\nabla,\Phi)\in \mathscr{C}(E)$ we have:
     \begin{enumerate}
         \setlength\itemsep{0.5em}
         \item[(a)] $|F_{\nabla}|_{\delta}^2 = \delta^{-4}|F_{\nabla}|^2$;
         \item[(b)] $|\nabla\Phi|_{\delta}^2 = \delta^{-2}|\nabla\Phi|^2$;
         \item[(c)] $\vol_{\delta} = \delta^3 \vol$;
         \item[(d)] $e_{\delta\ep}^{\delta}(\nabla,\Phi) = \delta^{-2} e_{\ep}(\nabla,\Phi)$;
         \item[(e)] $B_{\delta r}^{\delta}(x) = B_r(x)$;
         \item[(f)] $\displaystyle (\delta\ep)^{-1}\int_{B_{\delta r}^{\delta}(x)} e_{\delta\ep}^{\delta}(\nabla,\Phi)\vol_{\delta} = \ep^{-1}\int_{B_r(x)} e_{\ep}(\nabla,\Phi) \vol$;
         \item[(g)] $\cY_{\delta\ep}^{\delta}(\nabla,\Phi) = \delta\cY_{\ep}(\nabla,\Phi)$.
    \end{enumerate}
\end{lemm}
In particular, it follows that the set $\mathscr{C}(E)$ remains unchanged upon replacing the metric $g$ by $\delta^2 g$. In accordance with the remarks immediately following~\eqref{eq: 2nd_order_crit_pt_intro}, we say a configuration $(\nabla, \Phi) \in \mathscr{C}(E)$ is a critical point of $\cY_{\delta\ep}^{\delta}$ on $(M, \delta^2 g)$ if it satisfies 
\begin{equation}\label{eq: 2nd_order_crit_pt_scaled}
        \begin{cases}
            (\delta\varepsilon)^2 (d_{\nabla})^{\ast_{\delta}}F_{\nabla} = [\nabla\Phi,\Phi],\\
            \nabla^{\ast_{\delta}}\nabla\Phi = \frac{\lambda}{2(\delta\varepsilon)^2}(1-|\Phi|^2)\Phi,
        \end{cases}
    \end{equation} 
    where $(\cdot )^{\ast_{\delta}}$ denotes the formal $L^2$-adjoint computed with respect to $\delta^2 g$ instead.
We then have the following scaling property for the Yang--Mills--Higgs equations: 
\begin{lemm}[Scaling of critical points]\label{lem: scaling}
$(\nabla,\Phi)$ is a critical point for $\cY_{\ep}$ on $(M^3,g)$ if and only if $(\nabla,\Phi)$ is a critical point for $\cY_{\delta\ep}^{\delta}$ on $(M^3,\delta^2 g)$.
\end{lemm}
\begin{proof}
    Notice that $\ast_{\delta} = \delta^{3-2k}\ast$ acting on $k$-forms. Thus, by~\eqref{eq: d_star},
    \begin{align*}
    (d_{\nabla})^{\ast_{\delta}} F_{\nabla} &= \ast_{\delta} d_{\nabla} \ast_{\delta} F_{\nabla} = \delta^{-2}\ast d_{\nabla}\ast F_{\nabla}=\delta^{-2}d_{\nabla}^{\ast}F_{\nabla}.
    \end{align*} Likewise, 
    \begin{align*}
        \nabla^{\ast_{\delta}}\nabla\Phi &= -\ast_{\delta}d_{\nabla}\ast_{\delta}\nabla\Phi =-\delta^{-2}\ast d_{\nabla}\ast\nabla\Phi =\delta^{-2}\nabla^{\ast}\nabla\Phi.
    \end{align*} 
   Putting these together, we easily see that~\eqref{eq: 2nd_order_crit_pt_scaled} is equivalent to~\eqref{eq: 2nd_order_crit_pt_intro}, and the result follows.
\end{proof}
Finally, we come to the gap result on $\RR^3$ promised earlier:
\begin{thm}[Theorem \ref{thm: gap_for_R3}]
Suppose $\lambda \in (0, \lambda_0]$. For any $\ep > 0$, if $(\nabla,\Phi)$ is a smooth solution to~\eqref{eq: 2nd_order_crit_pt_intro} on an $SU(2)$-bundle $E\to\mathbb{R}^3$ over the Euclidean space $(\mathbb{R}^3,g_{\mathbb{R}^3})$, satisfying in addition that
    \[
    \ep^{-1}\cY_{\ep}(\nabla,\Phi)\leqslant \theta_{\mathrm{gap}}\cdot \min\{\lambda, 1\}, 
    \] then in fact
    \[
    \cY_{\ep}(\nabla,\Phi) = 0.
    \] That is, $F_{\nabla}\equiv 0$, $\nabla\Phi \equiv 0$ and $|\Phi|\equiv 1$.
\end{thm}
\begin{proof}
    By Theorem~\ref{thm: gap}, we just need to prove that, in this case, we do not need to assume $\ep <\tau_{\mathrm{gap}}\cdot\min\{\lambda,1\}$. To do this, we first take $\delta_i \downarrow  0$,  let $m_i = \delta_i^{-1}$ and define
    \begin{equation*}
        (\nabla_i,\Phi_i) \coloneqq s^\ast_{m_i}(\nabla,\Phi),
    \end{equation*}
    where $s_{m_i}: \RR^3 \to \RR^3$ denotes the map $x \mapsto m_i x$. By scale invariance we have $g_{\RR^3} = m_i^{-2}s_{m_i}^\ast g_{\RR^3}$, so it follows from Lemma~\ref{lem: scaling} that $(\nabla_i,\Phi_i)$ is a critical point of $\cY_{\delta_i \ep}$ with respect to the Euclidean metric $g_{\RR^3}$. Now, Lemma~\ref{lemm:scaling_g_to_delta_g} and the hypothesis on the energy imply that
    \begin{equation*}
        (\delta_i\ep)^{-1}\cY_{\delta_i\ep}(\nabla_i,\Phi_i) = \ep^{-1}\cY_{\ep}(\nabla,\Phi) \leqslant \theta_{\mathrm{gap}}\cdot \min\{\lambda, 1\}, 
    \end{equation*}
    and hence upon choosing $i\gg 0$ so that $\delta_i\ep < \tau_{\mathrm{gap}}\cdot\min\{\sqrt{\lambda},1\}$, 
    we can invoke Theorem~\ref{thm: gap} to conclude that $\cY_{\delta_i \ep}(\nabla_i,\Phi_i)$ must vanish, 
    which is equivalent to $\cY_\ep(\nabla,\Phi) = 0$. 
\end{proof}
\subsection{Reducible solutions}\label{subsec:reducible_solutions}
In this section we describe a construction of reducible solutions which in particular shows that, at least when $M$ is closed, the requirement that there be no non-zero $L^2$ harmonic $2$-forms on $(M, g)$ is necessary for the last conclusion of Theorem~\ref{thm: gap}. The ingredients involved are all more or less standard, but we believe that putting them together as we do below sheds further light on the gap results proved in the previous section.

As motivation, we recall from Remark \ref{rmk: reducible_sols} that if~\eqref{eq: 2nd_order_crit_pt_intro} admits a reducible solution $(\nabla,\Phi\not\equiv 0)$, then the bundle $E$ must split orthogonally as
            \[
            E = L\oplus L^{\ast},
            \]
            and $\nabla$ reduces to a $U(1)$ Yang--Mills connection on the Hermitian line bundle $L = \ker(\Phi - \frac{\sqrt{-1}}{2}\text{id})$. In a local trivialization of $E$ that respects the above splitting, we have 
            \begin{equation}\label{eq: reducible_pair_in_good_gauge}
            F_{\nabla}=\mathrm{diag}(F_L,-F_L),\ \ \ \Phi = \mathrm{diag}(\frac{\sqrt{-1}}{2},-\frac{\sqrt{-1}}{2}),    
            \end{equation}
            where $F_L$ is the curvature of the reduced connection; in particular, $\frac{\sqrt{-1}}{2}\bangle{F_{\nabla}, \Phi} = F_L$. 
            Next we note that we can reverse this process to construct reducible solutions to \eqref{eq: 2nd_order_crit_pt_intro}.
\begin{lemm}\label{lem: counter-examp}
    Let $L\to M$ be a Hermitian line bundle and let $D$ be a $U(1)$ Yang--Mills connection on $L$. Then there is an $SU(2)$ Yang--Mills connection $\nabla$ on $E \coloneqq L\oplus L^\ast$, and 
    an endomorphism $\Phi\in\mathfrak{su}(E)$, such that $|\Phi|\equiv 1$, $\nabla\Phi \equiv 0$ and $\frac{\sqrt{-1}}{2}F_\nabla = F_D\otimes \Phi$. In particular, the pair $(\nabla,\Phi)$ is a reducible solution satisfying~\eqref{eq: reducible_pair_intro}. 
\end{lemm}
\begin{proof}
     Denoting by $D^\ast$ the dual connection on $L^\ast$, we take $\nabla$ to be the $SU(2)$-connection on $E = L\oplus L^\ast$ given by $D\oplus D^\ast$, that is,
    \begin{equation*}
        \nabla\begin{pmatrix}
            u\\
            \theta
        \end{pmatrix} := \begin{pmatrix}
            D u\\
            D^\ast \theta
        \end{pmatrix},
    \end{equation*} and we define $\Phi\in\mathfrak{su}(E)$ by
    \begin{equation*}
        \Phi \begin{pmatrix}
            u \\ \theta
        \end{pmatrix} := \begin{pmatrix}
            \frac{\sqrt{-1}}{2}u\\-\frac{\sqrt{-1}}{2}\theta
        \end{pmatrix}.
    \end{equation*}
    Note that, by hypothesis, $F_{D} = \sqrt{-1}\omega$ for some harmonic $2$-form $\omega\in\Omega^2(M)$, and thus $\nabla$ is Yang--Mills with
    \begin{equation*}
        F_{\nabla} = \begin{pmatrix}
            F_D & 0\\
            0& -F_D
        \end{pmatrix}.
    \end{equation*}
   Moreover, by construction $|\Phi| \equiv 1$, and we have
    \begin{equation*}
        \nabla\Phi\begin{pmatrix}
            u\\ \theta 
        \end{pmatrix} = \nabla\begin{pmatrix}
            \frac{\sqrt{-1}}{2}u\\-\frac{\sqrt{-1}}{2}\theta
        \end{pmatrix} - \Phi\left(\nabla\begin{pmatrix}
            u\\\theta
        \end{pmatrix}\right) = 0.
    \end{equation*}
    Finally, it is clear that 
    $\frac{\sqrt{-1}}{2}F_\nabla  = F_D\otimes \Phi$.
\end{proof}
\begin{rmk}
    Suppose that $(\nabla,\Phi\not\equiv 0)$ is a reducible solution of \eqref{eq: 2nd_order_crit_pt_intro}. From equation \eqref{eq: reducible_pair_in_good_gauge}, we already noted that  $\sqrt{-1}\bangle{F_{\nabla}, \Phi}$ is equal to two times the curvature $F_L$ of the $U(1)$ Yang--Mills connection to which $\nabla$ reduces on $L$. We now show that we can also explicitly recognize $\sqrt{-1}\bangle{F_{\nabla}, \Phi}$ as the curvature of a $U(1)$ Yang--Mills connection induced by $\nabla$ on a Hermitian line bundle $\cL\cong L\otimes L$.
    
    Recall that the condition $|\Phi| \equiv 1$ implies that we have a global splitting $\mathfrak{su}(E) = \bangle{\Phi}\oplus \bangle{\Phi}^\perp$. Moreover, with the help of~\eqref{eq:transv_part} we see that the operator $J = \mathrm{ad}(\Phi)\rvert_{\bangle{\Phi}^\perp}$ satisfies
            \[
            \bangle{J\xi, J\eta} = \bangle{\xi, \eta},\ \ \   J^2 = -\mathrm{Id},
            \]
            and together with the metric $(\xi, \eta) = \bangle{\xi, \eta} + \sqrt{-1}\bangle{\xi, J\eta}$ turns $\bangle{\Phi}^{\perp}$ into a Hermitian line bundle that we henceforth denote by $\cL$. 
            Next, we use $\nabla$ to define a connection $\nabla^\perp$ on 
            $\cL$ via orthogonal projection:
            \begin{equation*}
                \nabla^\perp\xi \coloneqq \nabla\xi - \bangle{\Phi,\nabla\xi}\Phi. 
            \end{equation*}
            Using the fact that $\nabla\Phi \equiv 0$, we conclude that $\nabla^{\perp}$ respects the $U(1)$-structure on $\cL$. Moreover, a straightforward computation, using again the assumption $\nabla\Phi \equiv 0$, shows that its curvature is given by 
            \begin{equation}
            F_{\cL}(\xi) = \bangle{F_{\nabla},\Phi}[\Phi,\xi] = \sqrt{-1}\bangle{F_{\nabla}, \Phi} \cdot \xi,
            \end{equation} as desired.
            
            Finally, we give a direct proof that $\cL \cong L\otimes L$. First, note that since the eigenvalues of $\Phi$ are $\pm\frac{\sqrt{-1}}{2}$, we have that $\Phi^2 = -\frac{1}{4}\id$, which together with~\eqref{eq:transv_part} 
            gives 
            \[
            \Phi \xi = - \xi \Phi,\quad \text{for all }\xi \in \cL.
            \]
            Recalling the identification $L^\ast \cong \ker(\Phi + \frac{\sqrt{-1}}{2}\id)$, we deduce that $\xi$ maps $L^\ast$ to $L$ and $L$ to $L^\ast$. In particular, the assignment
            $\xi \mapsto \xi\rvert_{L^\ast}$ defines a bundle map $\alpha : \cL \to \Hom(L^\ast, L)$. 
            As $\mathrm{rank} (\cL) = \mathrm{rank}(\Hom(L^\ast,L)) = 1$, to conclude that $\alpha$ is an isomorphism, we are left to prove that $\alpha$ is injective. To that end, suppose $\xi \in \cL$ and that $\xi|_{L^*} = 0$. Since $\xi$ interchanges $L$ and $L^*$ as shown in 
            the discussion above, 
            we infer that $(\xi^2)|_{L^*} = 0$ and $(\xi^2)|_{L} = 0$. Thus $|\xi|^2 = -2\mathrm{tr}(\xi^2) =  0$, proving the injectivity of $\alpha$.
\end{rmk}    

Returning to Lemma~\ref{lem: counter-examp}, it is now straightforward to obtain \emph{non-trivial} reducible solutions $(\nabla,\Phi\not\equiv 0)$ to \eqref{eq: 2nd_order_crit_pt_intro} when our Riemannian $3$-manifold $(M, g)$ is closed, oriented, 
and $b_1(M)=b_2(M)\neq 0$. Indeed, in view of the equivalent interpretations of the Betti numbers on closed manifolds, we get in this case a non-zero 
harmonic $2$-form $\frac{\sqrt{-1}}{2\pi}\omega$ on $M$ which is \emph{integral} in the sense that 
its cohomology class lies in the image of the natural homomorphism $f: H^2(M,\mathbb{Z})\to H^2(M,\mathbb{R})$. As such, there is a complex line bundle $L\to M$ so that $f(c_1(L)) = [\frac{\sqrt{-1}}{2\pi}\omega]$, and by the Chern--Weil approach to characteristic classes along with Hodge theory, we obtain a $U(1)$ connection $D$ whose curvature is $F_D=\omega$. In other words, $D$ is a non-flat, $U(1)$ Yang--Mills connection on $L$, and thus generates via Lemma \ref{lem: counter-examp} a non-trivial reducible solution $(\nabla,\Phi\not\equiv 0)$ to \eqref{eq: 2nd_order_crit_pt_intro}.

\begin{example}
\label{example: counter-examp}


For a concrete example, take $(M, g) = (\mathbb{S}^2\times\mathbb{S}^1, g_{\mathbb{S}^2} + d\theta^2)$, where $g_{\mathbb{S}^2}$ is the standard round metric on $\mathbb{S}^2$, and let $p_1: M \to \mathbb{S}^2$ denote projection onto the first factor. 
Then all integral harmonic $2$-forms on $M$ arise as pullbacks via $p_1$ of integral harmonic $2$-forms on $\mathbb{S}^2$, the latter given by
\[
\frac{\sqrt{-1}}{2\pi}\omega_m := \frac{m}{4\pi}\vol_{\mathbb{S}^2},\quad m\in\ZZ.
\] 
Moreover, thanks to the fact that $\mathbb{S}^2$ is simply-connected, for each $m \in \ZZ$, there exists, up to gauge, a unique $U(1)$ Yang--Mills connection whose curvature is $\omega_m$, which can be constructed explicitly as follows (see also~\cite{Kuwabara1982}). Letting $\pi: \mathbb{S}^3\to \mathbb{S}^2$ denote the Hopf fibration, which is a principal $U(1)$-bundle, we write $\Xi\in\Omega^1(\mathbb{S}^3,\mathfrak{u}(1))$ for the canonical connection induced by the orthogonal splitting 
\[
T\mathbb{S}^3 = \ker(d\pi) \oplus (\ker(d\pi) )^{\perp}, 
\]
and consider the irreducible unitary representation $\rho_m: \mathrm{U}(1)\to \mathrm{U(1)}$ given by $z\mapsto z^m$. Then 
$\Xi$ induce a connection $\cD_m$ on the associated line bundle $L_m = S^3\times_{\rho_m}\CC$, and it is straightforward to verify that $F_{\cD_m} = \omega_m$. For $m \neq 0$, the pullback $p_1^* \cD_m$ is then a non-flat, $U(1)$ Yang--Mills connection on $p_1^* L_m \to \mathbb{S}^2 \times \mathbb{S}^1$, and can thus be used in the procedure of Lemma~\ref{lem: counter-examp} to yield a non-trivial, reducible solution $(\nabla,\Phi)$ satisfying \eqref{eq: reducible_pair_intro}. 

\end{example}
\section{Asymptotic analysis of critical points}\label{sec:asymptotic}

In this section, we consider our base manifold $(M^3,g)$ to be an oriented $3$-manifold with bounded geometry (possibly noncompact). In particular, there exist positive constants $\rho_0$ and $A_0, A_1, \cdots$ such that~\eqref{eq:injectivity_radius_for_estimates} and~\eqref{eq:curvature_bound_for_estimates} hold with $\Omega = M$. Henceforth, we let $r_0:=\frac{\rho_1}{4}$, where $\rho_1$ is as in \eqref{ineq: rho_1_def}, so that on any geodesic ball $B_{\rho}(x_0)\subset (M,g)$ with $\rho<\rho_1$ we have the geometric control given by \eqref{eq:Hessian_comparison_for_estimates}, \eqref{eq:metric_bounds_for_estimates} and \eqref{eq:volume_bounds_for_estimates}.

We shall work with a family $(\nabla_{\ep},\Phi_{\ep})$ of critical points for $\cY_{\ep}:\mathscr{C}(E)\to\mathbb{R}$, on a $SU(2)$-bundle $E\to M$, satisfying a uniform (normalized-)energy bound
\begin{equation}\label{ineq: uniform_bound_assumption}
\varepsilon^{-1}\cY_{\ep}(\nabla_{\ep},\Phi_{\ep}) \leqslant \Lambda<\infty,
\end{equation} for some constant $\Lambda>0$, possibly depending on $(M^3,g)$ and the coupling constant $\lambda$ in \eqref{eq: YMH_energy}. For example, the critical points produced by Theorem \ref{thm: existence}, in the case where $M$ is closed, satisfy \eqref{ineq: uniform_bound_assumption} with $\Lambda$ of the form $\max\{1,\lambda\}\cdot\Lambda_M$, where $\Lambda_M$ depends only on $(M^3,g)$.

The purpose of this section is to analyze the family $(\nabla_{\ep},\Phi_{\ep})$ in the limit as $\ep\to 0$, and to prove Theorems \ref{thm: asymptotic} and \ref{thm: bubbling}. In \S\S\ref{subsec:blow-up}-\ref{subsec:assigning_charges} we introduce the energy and charge measures, $\mu_{\ep}$ and $\kappa_{\ep}$, 
associated with the family $(\nabla_{\ep},\Phi_{\ep})$, together with the sets that, loosely speaking, capture their respective concentration behavior, namely the blow-up set $S$ and the asymptotic zero set $Z$. Using various a priori estimates from Section \ref{sec:estimates}, as well as the local convergence result of Proposition \ref{prop:local_convergence}, and some standard measure theory, we prove parts (a) and (b) of Theorem \ref{thm: asymptotic}. Furthermore, using Proposition \ref{prop:concentration_scale}, we also show that, when $M^3$ is closed, a family $(\nabla_{\ep},\Phi_{\ep})$ produced by Theorem \ref{thm: existence} has $S=\emptyset$ if and only if $(\nabla_{\ep},\Phi_{\ep})$ is reducible for all but finitely many $\ep$; in particular, in the situation of Theorem \ref{thm: irred}, we have $S\neq\emptyset$ (see Lemma \ref{lemm: reducible_is_boring} and Example \ref{ex:irred_families_non_trivial_S}). Next, in \S\ref{subsec:hodge_decomp_longit_part}, using mainly Hodge theory and Proposition \ref{prop:convergence_nonharmonic_parts}, together with Lemma \ref{lemm:nice_convergence_off_of_S}, we prove part (c) of Theorem \ref{thm: asymptotic}.

Moving forward, \S\ref{subsec:rescaling} is concerned with the proof that, at every point $x\in S$, the family $(\nabla_{\ep},\Phi_{\ep})$ bubbles off a non-trivial critical point for $\cY_{1}^{g_{\RR^3}}$ over $\RR^3$ with energy bounded by $\Theta(x)$, which proves the first part of Theorem \ref{thm: bubbling}. Here the bulk of the work goes into comparing $\ep$ with the rate of rescaling, which again uses the estimates from Section \ref{sec:estimates} and also the gap result of Proposition \ref{prop:concentration_scale}. Finally, \S\S\ref{subsec:regions}-\ref{subsec:energy_identity} contain the core of the bubbling analysis leading to the final conclusions of Theorem \ref{thm: bubbling}. In this part, we first identify the neck regions between bubbles via a standard procedure, and then, primarily by integrating the exponential decay estimates on $|\nabla_{\ep}\Phi_{\ep}|$ and $1 - |\Phi_{\ep}|$ from \S\ref{subsec:exp-decay} along the radial direction, and combining the result with a local conservation law (Lemma~\ref{lem:conservation_law}) reminiscent of \cite[Corollary II.2.2]{Jaffe-Taubes} to control $|F_{\nabla_{\ep}}|$, we manage to show that eventually the neck regions carry no energy (Proposition~\ref{prop:no_neck}). From there, an iterative argument, guaranteed by the energy gap from Theorem~\ref{thm: gap_for_R3} to end after finitely many steps, is used to express $\Theta(x)$ and $\Xi(x)$ in terms of the energies and charges, respectively, of the bubbles extracted at each stage.

\subsection{The blow-up set and decomposition of the limit measure}\label{subsec:blow-up}
By~\eqref{ineq: uniform_bound_assumption}, the associated Radon measures
\[
\mu_{\ep}:=\ep^{-1} e_{\ep}(\nabla_{\ep},\Phi_{\ep})\mathcal{H}^3
\] have uniformly bounded mass: $\mu_{\ep}(M)\leqslant\Lambda$.
Combining~\eqref{ineq: uniform_bound_assumption} with Proposition~\ref{prop:maximum_principle_for_w}, we also get
\[
\int_{M}\ep \big| \bangle{* F_{\nabla_{\ep}}, \Phi_{\ep}} \big|^2\leqslant \Lambda,\quad \text{ for all $\ep < r_0$}.
\]
Next, letting 
\begin{equation}\label{eq: kappa_definition_Sec_5}
\kappa_{\ep}: = 2\bangle{* F_{\nabla_{\ep}}, \nabla_{\ep}\Phi_{\ep}} \cH^3,
\end{equation}
we have by Schwarz's inequality that
\begin{equation}\label{eq:charge_density_by_energy_density}
\begin{split}
2|\bangle{* F_{\nabla_{\ep}}, \nabla_{\ep}\Phi_{\ep}}| \leqslant 2 |\ep^{\frac{1}{2}}F_{\nabla_{\ep}}||\ep^{-\frac{1}{2}}\nabla_{\ep}\Phi_{\ep}| 
\leqslant \ & \ep^{-1}e_{\ep}(\nabla_{\ep},\Phi_{\ep}),
\end{split}
\end{equation}
which together with~\eqref{ineq: uniform_bound_assumption} implies that $(\kappa_{\ep})$ is a bounded sequence in $\big(C^0_c(M)\big)^*$.

Therefore, after passing to a subsequence, which we do not relabel, we have, first of all, that $(\mu_{\ep})$ weakly* converges to a non-negative Radon measure $\mu$ on $M$ as $\ep\to 0$. That is, 
\begin{equation}\label{eq: weak*_convergence}
\lim_{\ep\to 0}\int_M fd\mu_{\ep} = \int_M f d\mu,\quad \text{ for all } f\in C_c^0(M).
\end{equation}
Secondly, there exists some $1$-form $h$ on $M$ of class $L^2$ such that, as $\ep \to 0$,
\begin{equation}\label{eq: weak_L2_convergnece}
\ep^{\frac{1}{2}}\bangle{* F_{\nabla_{\ep}}, \Phi_{\ep}} \rightharpoonup h, \quad \text{ weakly in $L^2(M)$},
\end{equation}
Thirdly, there exists a Radon measure $\kappa$ on $M$ satisfying
\begin{equation}\label{eq:kappa_by_mu}
\Big|\int_{M} f d\kappa\Big| \leqslant \int_M |f| d\mu,\quad\text{for all $f \in C^0_{c}(M)$},
\end{equation}
such that, as $\ep \to 0$, 
\begin{equation}\label{eq: kappa_weak*_convergence}
\kappa_{\ep} \rightharpoonup \kappa, \quad \text{ in weak* sense on $M$}.
\end{equation}
Hereafter, we reserve the notation $(\nabla_{\ep}, \Phi_{\ep})$ for the subsequence we have just extracted, while further subsequences are typically denoted $(\nabla_{\ep_i}, \Phi_{\ep_i})$, or simply $(\nabla_i,\Phi_i)$. For each $x\in M$, we let
\[
\cR_x:=\{r\in (0,r_0]:\mu(\partial B_r(x))>0\},
\] and note that, since $\mu$ is locally finite, $\cR_x$ is at most countable. 
More importantly, we have
\begin{equation}\label{eq: weak_limit_measure_balls}
\mu(B_r(x)) = \lim_{\ep\to 0} \mu_{\ep}(B_r(x)),\quad \text{for all $r\in (0,r_0]\setminus\cR_x$,}
\end{equation} while for a general open set $U\subset M$ we only have $\mu(U)\leqslant \liminf_{\ep\to 0}\mu_{\ep}(U)$.

Next, let $\lambda_0$ be an upper bound for $\lambda>0$. Then from Lemmas~\ref{lemm:improved_coarse_estimate_base},~\ref{lemm:exp_decay_base} and~\ref{lemm:nablaPhi_exp_decay_base}, as well as Remark~\ref{rmk:clearing_out} with $\beta$ taken to be $\frac{1}{8(\lambda_0 + 2)}$, we get thresholds $\overline{\eta}$, $(\eta_0, \tau_0)$, $(\theta_1, \tau_1)$, and  $\theta_0$, which depend only on $\lambda_0$ and $A_1$. In view of these results, along with Proposition~\ref{prop:local_convergence} and Remark~\ref{rmk:smallness_condition_relation}, we then set 
\begin{equation}\label{def: eta_star}
    \eta_{*} := \min\big\{\overline{\eta}, \eta_0, \theta_0\cdot\min\{\lambda, 1\}, \theta_1\cdot\min\{\lambda, 1\}\big\},
\end{equation}
and define the \textbf{blow-up set} or \textbf{energy concentration set} of the sequence $(\nabla_{\ep},\Phi_{\ep})$ by
\begin{equation}\label{eq: blow-up_set_S}
S:=\bigcap_{0<r\leqslant  r_0}\left\{x\in M: \liminf_{\ep\to 0}\ep^{-1}\int_{B_r(x)}e_{\ep}(\nabla_{\ep},\Phi_{\ep})\geqslant \eta_{\ast} \right\}.
\end{equation} 
For later purposes, it is also convenient to let
\[
\tau_{\ast}:=\min\{\tau_0,\tau_1,\tau_1\sqrt{\lambda}\}.
\]
\begin{lemm}\label{lem: S_finite}
    The following hold:
    \begin{itemize}
        \item[(a)] $S$ is closed.
        
        \item[(b)] In fact $\cH^0(S)\leqslant \eta_{*}^{-1}\Lambda$; in particular, $S$ is finite.
    \end{itemize}
\end{lemm}
\begin{proof}
    (a) Let $(x_j)\subset S$ be a sequence converging to a point $x\in M$. Given $r\in (0,r_0]$ and $0<s<r$ arbitrary, there is some $j_{s,r} \in \NN$ such that $B_s(x_j)\subset B_r(x)$ for all $j\geqslant j_{s,r}$. Therefore, since $(x_j)\subset S$, for all $j\geqslant j_{s,r}$ we have
    \[
    \liminf_{\ep\to 0} \ep^{-1}\int_{B_{r}(x)} e_{\ep} \geqslant \liminf_{\ep\to 0} \ep^{-1}\int_{B_{s}(x_j)} e_{\ep}\geqslant\eta_{*}.
    \] Since $r\in (0,r_0]$ is arbitrary, it follows that $x\in S$. Thus, $S$ is closed.
    \vspace*{0.3cm}

    (b) If $S$ is empty there is nothing to prove. Otherwise, take any finite collection of points $x_1, \ldots, x_N$ in $S$. Then there exists $r \in (0, r_0]$ so that the balls $\{B_{r}(x_j)\}_{j = 1}^{N}$ are disjoint, and we can compute:
    \begin{align*}
        N &\leqslant \sum\limits_{j=1}^{N}  \eta_{*}^{-1}\liminf_{\ep\to 0}\ep^{-1}\int_{B_{r}(x_j)} e_{\ep}\quad\text{(each $x_j$ is in $S$)}\\
        &\leqslant \eta_{*}^{-1}\liminf_{\ep\to 0} \sum\limits_{j=1}^{N} \ep^{-1}\int_{B_{r}(x_j)} e_{\ep}\\
        &= \eta_{*}^{-1}\liminf_{\ep\to 0} \ep^{-1}\int_{\bigcup\limits_{j=1}^{N} B_{r}(x_j)} e_{\ep} \quad\text{(the balls $B_r(x_j)$ are disjoint)}\\
        &\leqslant \eta_{*}^{-1}\Lambda,
    \end{align*} where in the last line we used the uniform bound \eqref{ineq: uniform_bound_assumption}. Having shown that any finite subset of $S$ has cardinality at most $\eta_{*}^{-1}\Lambda$, we conclude that $\cH^0(S)\leqslant\eta_{*}^{-1}\Lambda$, 
    as desired.
\end{proof}

To begin our analysis of the limiting measure $\mu$ in~\eqref{eq: weak*_convergence}, we note the following consequence of Proposition~\ref{prop:local_convergence}, Remark~\ref{rmk:smallness_condition_relation}, and our choice of $\eta_*$.
\begin{lemm}\label{lemm:nice_convergence_off_of_S}
The $1$-form $h$ in~\eqref{eq: weak_L2_convergnece} is smooth and harmonic on $M \setminus S$. Moreover, as $\ep\to 0$ we have the following convergences in $C^{\infty}_{\loc}(M \setminus S)$:
\begin{subequations}
\begin{align}
(1 - |\Phi_{\ep}|^2) \cdot \ep |F_{\nabla_{\ep}}|^2 + \ep^{-1}|\nabla_{\ep}\Phi_{\ep}|^2 +  \lambda\ep^{-3}(1 - |\Phi_{\ep}|^2)^2 + \ep|[F_{\nabla_{\ep}}, \Phi_{\ep}]|^2 & \to 0, \label{eq: xi_converges_to_zero}\\
\ep^{\frac{1}{2}}\bangle{* F_{\nabla_{\ep}}, \Phi_{\ep}} & \to h, \label{eq: convergence_to_h}\\
\bangle{* F_{\nabla_{\ep}}, \nabla_{\ep}\Phi_{\ep}} & \to 0. \label{eq: q_converges_to_zero}
\end{align}
\end{subequations}
Consequently, we also have as $\ep \to 0$ that
\begin{equation}\label{eq: Cinfty_loc_convergence}
\ep^{-1}e_{\ep}(\nabla_{\ep},\Phi_{\ep})  \to |h|^2, \quad \text{ in $C^{\infty}_{\loc}(M \setminus S)$}. 
\end{equation}
\end{lemm}
\begin{proof}
If $x_0 \in M \setminus S$, then from the definition of $S$ and the monotonicity of $r \mapsto \mu(B_{r}(x_0))$, we can find $r \in (0, r_0] \setminus \cR_{x_0}$ such that
\[
\lim_{\ep \to 0} \mu_{\ep}(B_{r}(x_0)) < \eta_*.
\]
Remark~\ref{rmk:smallness_condition_relation}(i) then allows us to invoke Proposition~\ref{prop:local_convergence}, part (b) of which upgrades the weak $L^2$ convergence~\eqref{eq: weak_L2_convergnece} to smooth convergence on $B_{\frac{r}{48}}(x_0)$, and also shows that $h$ is harmonic on $B_{\frac{r}{48}}(x_0)$. On the other hand, conclusions (a) and (c) of Proposition~\ref{prop:local_convergence} yields respectively that~\eqref{eq: xi_converges_to_zero} and~\eqref{eq: q_converges_to_zero} take place, in $C^{\infty}(B_{\frac{r}{48}}(x_0))$. A routine covering argument extends the smooth convergences~\eqref{eq: xi_converges_to_zero},~\eqref{eq: convergence_to_h} and~\eqref{eq: q_converges_to_zero}, as well as the harmonicity of $h$, to compact subsets of $M \setminus S$. Finally,~\eqref{eq: Cinfty_loc_convergence} is a consequence of~\eqref{eq: xi_converges_to_zero},~\eqref{eq: convergence_to_h} and~\eqref{eq: bracket_norm}.
\end{proof}
In view of~\eqref{eq: Cinfty_loc_convergence}, and taking into account that $S$ has zero $\cH^3$-measure (thanks to Lemma \ref{lem: S_finite}), we get by Fatou's lemma and the weak* convergence \eqref{eq: weak*_convergence} that
\begin{equation}\label{ineq: fatou_h_mu}
\int_M f|h|^2 \leqslant \liminf_{\ep\to 0}\int_M f \ep^{-1}e_{\ep}(\nabla_{\ep},\Phi_{\ep}) = \int_M fd\mu,
\end{equation} for all non-negative $f\in C_c^0(M)$. Thus, the linear functional $I:C_c^0(M)\to\mathbb{R}$ given by
\[
I(f):=\int_M fd\mu - \int_M f|h|^2
\] is nonnegative and by the Riesz representation theorem there is a unique nonnegative Radon measure $\nu$ on $M$ such that $I(f)=\int fd\nu$ for all $f\in C_c^0(M)$, so we can write
\begin{equation}\label{eq: decomp_mu}
	\mu = |h|^2 \mathcal{H}^3 + \nu,
\end{equation} where $\nu$ is called the \textbf{defect measure} of the sequence $(\nabla_{\ep},\Phi_{\ep})$. By standard measure theory, we have (see for instance \cite[Theorem 7.2, p.212]{folland2013real}):
\begin{equation}\label{eq: property_nu}
\nu(U) = \sup\{I(f):f\in C_c^0(M)\text{ with $\mathrm{supp}(f)\subset U$ and }\|f\|_{\infty}\leqslant 1\},
\end{equation} for any open subset $U\subset M$. Moreover, we can show the following:
\begin{lemm}\label{lem: h_on_M_supp_nu_in_S}
    The harmonic $1$-form $h$ is smooth on all of $M$, while $\nu$ satisfies $\nu(M\setminus S)=0$. In particular, $\mathrm{supp}(\nu)\subset S$ and $\nu$ is singular with respect to $\cH^3$.
\end{lemm}
\begin{proof}
    We already know that $h \in L^2(M)$, and that it is smooth and harmonic on $M \setminus S$. We next show that
    \begin{equation}\label{eq: h_distributionally_harmonic}
    dh = 0 \text{ and } d^*h = 0 \text{ in the distributional sense on }M.
    \end{equation}
    If $S = \emptyset$, there is nothing to prove. Otherwise, take $x_0 \in S$ and let $U$ be a neighborhood of $x_0$ whose closure contains no other points of $S$. Letting $\zeta: \RR \to [0, 1]$ be a standard cutoff function such that
    \[
    \zeta(t) = 1\quad\text{if }t \leqslant 0,\quad\text{and}\quad  \zeta(t) = 0\quad\text{if }t \geqslant 1, 
    \]
    for $\delta < r_0$ sufficiently small such that $B_{4\delta}(x_0) \subset U$, where $r_0$ is the radius introduced at the start of this section, we define
    \[
    \varphi_{\delta}(x) = 1 - \zeta\big( \frac{d(x, x_0) - \delta}{\delta} \big).
    \]
    Then, given $f \in \Omega^0_{c}(U)$, we have from the smooth harmonicity of $h$ on $U \setminus \{x_0\}$ that
    \[
    0 = \int_{U} \bangle{h, d(\varphi_{\delta} f)} = \int_{U} f \bangle{h, d\varphi_{\delta}} + \int_{U} \varphi_{\delta}\bangle{h, df}.
    \]
    Using the definition of $\varphi_{\delta}$ and the integrability of $|h|^2$, we get upon rearranging the above equation and applying H\"older's inequality that
    \[
    \begin{split}
    \Big| \int_{U} \varphi_{\delta} \bangle{h, df} \Big| =\ & \Big| \int_{U}f\bangle{h, d\varphi_{\delta}}  \Big|\\
    \leqslant\ & \|f\|_{L^{\infty}} \|h\|_{L^2}  \cdot \Big( \int_{B_{2\delta}(x_0) \setminus B_{\delta}(x_0)} |d\varphi_{\delta}|^2 \Big)^{\frac{1}{2}}\leqslant C \delta^{\frac{1}{2}},
    \end{split}
    \]
    for some constant $C$ which does not depend on $\delta$, where in getting the last inequality we also used~\eqref{eq:volume_bounds_for_estimates} and the fact that $|d\varphi_{\delta}| \leqslant \|\zeta'\|_{L^{\infty}}\cdot \delta^{-1}$. Letting $\delta \to 0$ and using the dominated convergence theorem gives
    \[
    \int_{U} \bangle{h, df} = 0, \text{ for all }f \in \Omega^0_{c}(U).
    \]
    Similarly, we have that
    \[
    \int_{U} \bangle{h, d^*\alpha} = 0, \text{ for all }\alpha \in \Omega^2_c(U).
    \]
    Repeating this argument near each point of $S$, and recalling again that $h$ is already smooth and harmonic away from $S$, we get~\eqref{eq: h_distributionally_harmonic} as asserted, and it follows from standard elliptic theory that $h$ extends to a smooth harmonic $1$-form over all of $M$.
    
    Next we show the assertion about $\nu$. Since $M\setminus S$ is an open set, by \eqref{eq: property_nu} it suffices to prove that for every $f\in C_c^0(M)$ with $\mathrm{supp}(f)\subseteq M\setminus S$ and $\|f\|_{\infty}\leqslant 1$, we have
                        \begin{equation}\label{eq: h_on_M_supp_nu_in_S}
                        \lim_{\ep\to 0}\int_M f\ep^{-1}e_{\ep}(\nabla_{\ep},\Phi_{\ep}) = \int_M f \lvert h\rvert^2.
                        \end{equation} Now, by the $C_{\text{loc}}^{\infty}$-convergence \eqref{eq: Cinfty_loc_convergence} on $M\setminus S$, and the fact that $\mathrm{supp}(f)\subseteq M\setminus S$ is compact, when $\ep\to 0$ we have
                        \[
                        f\ep^{-1}e_{\ep}(\nabla_{\ep},\Phi_{\ep})\to f\lvert h\rvert^2\quad\text{uniformly on $M$}.
                        \] Therefore, equation \eqref{eq: h_on_M_supp_nu_in_S} follows.
\end{proof}
The next lemma shows that one can replace $\eta_{\ast}$ by any $\eta\in(0,\eta_{\ast})$ in the definition of $S$:
\begin{lemm}\label{lem: S_etas}
For all $\eta\in (0,\eta_{*})$, we have
        \begin{equation}\label{eq: alternative_S}
        S = \bigcap_{0<r\leqslant  r_0}\left\{x\in M: \liminf_{\ep\to 0}\ep^{-1}\int_{B_r(x)}e_{\ep}(\nabla_{\ep},\Phi_{\ep})\geqslant\eta\right\}=:S_{\eta}.
        \end{equation}
\end{lemm}
\begin{proof}
    By definition, since $\eta<\eta_{*}$, we have $S\subset S_{\eta}$. So to prove \eqref{eq: alternative_S} it suffices to show that if $x_0\in M\setminus S$ then $x_0\in M\setminus S_{\eta}$. To do so, note that $x_0\in M\setminus S$ implies the existence of both $\rho\in (0,r_0]$ and a subsequence $\ep_i\to 0$ such that
    \[
    \ep_i^{-1}e_{\ep_i}(\nabla_{\ep_i},\Phi_{\ep_i}) \to |h|^2\quad\text{uniformly (with all derivatives) in $B_{\rho}(x_0)$}.
    \]
    Now, since $h$ is a harmonic $1$-form, the Bochner formula implies
    \[
    \Delta |h|^2\leqslant 2\|\mathrm{Ric}\|_{\infty}\cdot|h|^2\quad \text{ holds pointwise on }M.
    \] Combining this differential inequality with the fact that $h$ is $L^2$-bounded and applying Lemma \ref{lemm:moser-improve} (b) to the function $|h|^2$ gives $\|h\|_{\infty}^2\leqslant C_M\|h\|_{L^2}^2$. Thus, for all $0<r<\rho$ we have
    \[
    \lim_{i\to\infty}\ep_i^{-1}\int_{B_r(x_0)}e_{\ep_i}=\int_{B_{r}(x_0)}|h|^2\leqslant \|h\|_{\infty}^2\cdot\mathrm{vol}(B_r(x_0))\leqslant C_M\|h\|_{L^2}^2\cdot r^3,
    \] and this can be made as small as one wants by decreasing $r$; in particular, taking $r=r(M,\|h\|_{L^2}, \rho ,\eta)>0$ small enough we get 
    \[
    \liminf_{\ep\to 0} \ep^{-1}\int_{B_r(x_0)}e_{\ep}<\eta,
    \] therefore $x_0\in M\setminus S_{\eta}$ as we wanted.\\
\end{proof}
As a consequence of Lemmas \ref{lem: S_finite}, \ref{lem: h_on_M_supp_nu_in_S} and \ref{lem: S_etas} we get:
\begin{coro}\label{cor: Theta}
    The zero-dimensional densities of the measures $\mu$ and $\nu$ exist and coincide everywhere, defining the function
    \[
    \Theta(x):=\lim_{r\downarrow 0}\mu(B_r(x))=\lim_{r\downarrow 0}\nu(B_r(x))
    \] for all $x\in M$, which satisfies $0\leqslant\Theta(x)\leqslant\Lambda$.  Moreover, $\Theta\colon M\to[0,\infty)$ is upper semicontinuous and $\mathrm{supp}(\Theta)={S}$, with $\Theta(x)\geqslant\eta_{*}$ for all $x\in S$.
\end{coro}
\begin{proof}
Since $\nu(M)\leqslant\mu(M)\leqslant\Lambda<\infty$ and the functions $r\mapsto \nu(B_r(x))$, $r\mapsto\mu(B_r(x))$ are monotone increasing, the decomposition \eqref{eq: decomp_mu} and the fact that $|h|^2\in C^{\infty}(M)$ all together imply the first part of the statement.

 The proof of the upper semicontinuity of $\Theta$ is standard, but we include it here for convenience. Let $(x_j)\subset M$ with $x_j\to x\in M$. Given $r\in (0,r_0]\setminus\cR_x$ and $\delta>0$ arbitrary, for $j\gg 1$ we have $B_r(x_j)\subset B_{r+\delta}(x)$ and then
 \[
 \Theta(x_j)\leqslant \mu(B_r(x_j))
 \leqslant\mu(B_{r+\delta}(x))\leqslant\Lambda<\infty.
 \] Thus,
 \[
 \limsup_{j\to\infty}\Theta(x_j)\leqslant\mu(B_{r+\delta}(x))
 \] and letting $\delta\downarrow 0$ we get
 \[
  \limsup_{j\to\infty}\Theta(x_j)\leqslant\mu(B_{r}(x)).
 \] Hence, letting $r\to 0$ we get
 \[ 
 \limsup_{j\to\infty}\Theta(x_j)\leqslant\Theta(x),
 \] as desired. 
 
 Next, let $x$ be a point where $\Theta(x) > 0$. Then, there is $(r_j)\subset (0,r_0]\setminus\cR_x$, $r_j\downarrow 0$, such that by \eqref{eq: weak_limit_measure_balls}
 \[
 0 < \Theta(x) = \lim_{j\to\infty}\liminf_{\ep\to 0}\mu_{\ep}(B_{r_j}(x)).
 \] Thus, for all sufficiently large $j$ there holds
 \[
\liminf_{\ep\to 0}\ep^{-1}\int_{B_{r_j}(x)}e_{\ep} 
 > \frac{1}{2}\min\{\Theta(x), \eta_{\ast}\} =: \eta \in (0, \eta_{\ast}).
 \] It follows from Lemma \ref{lem: S_etas} that $x\in S$. This proves that $\mathrm{supp}(\Theta)\subset S$, since $S$ is closed by Lemma~\ref{lem: S_finite}. To prove the reverse inclusion and the fact that $\Theta(x)\geqslant\eta_{*}$ for all $x\in S$, we let $x\in S$ and using again a sequence $(r_j)\subset (0,r_0]\setminus\cR_x$, $r_j\downarrow 0$, we see that
 \[
 \Theta(x) = \lim_{j\to\infty}\liminf_{\ep\to 0}\mu_{\ep}(B_{r_j}(x)) \geqslant\eta_{*}>0,
 \] since $\liminf_{\ep\to 0}\mu_{\ep}(B_{r_j}(x))\geqslant\eta_{*}$ for all $j$. This completes the proof.
\end{proof}
\begin{coro}
    $\mathrm{supp}(\nu)=S$ and writing $S=\{x_1,\ldots,x_l\}$ we have
    \begin{equation}\label{eq: nu_sum_deltas}
    \nu = \sum_{k=1}^l \Theta(x_k)\delta_{x_k}.
    \end{equation}
\end{coro}
\begin{proof}
    By Lemma \ref{lem: h_on_M_supp_nu_in_S} we already knew that $\mathrm{supp}(\nu)\subset S$. Now, by Corollary \ref{cor: Theta} we also have that $\Theta$ is the density of $\nu$ and $\mathrm{\supp}(\Theta)=S$; in particular, it follows that $S\subset \mathrm{supp}(\nu)$, and therefore we have the desired equality $\mathrm{supp}(\nu)=S$. Now, using Lemma \ref{lem: S_finite} (b) we can further write $\mathrm{supp}(\nu)=S=\{x_1,\ldots,x_l\}$, where $l:=\cH^0(S)\leqslant\eta_{*}^{-1}\Lambda<\infty$. Thus, for any $A\subset\mathrm{supp}(\nu)$ we have
    \[
    \nu(A) = \sum\limits_{j=1}^l \lim_{r\downarrow 0} \nu(A\cap B_r(x_j))\leqslant \Lambda\cH^0(S\cap A).
    \] Therefore, $\nu\ll\cH^0\lfloor S$ and by the Radon--Nikodym theorem it follows that $\nu=\Theta \cH^0\lfloor S$, that is, equation \eqref{eq: nu_sum_deltas} holds.
\end{proof}
To sum up, so far we showed that there is a harmonic $1$-form $h\in\Omega^1(M)$ such that, as $\ep\to 0$, we have $\ep^{-1}e_{\ep}(\nabla_{\ep},\Phi_{\ep})\to |h|^2$ in $C_{\mathrm{loc}}^{\infty}$ outside the energy concentration set $S$ of $(\nabla_{\ep},\Phi_{\ep})$ defined in~\eqref{eq: blow-up_set_S}, which in turn is a finite set of points with $\cH^0(S)\leqslant\eta_{*}^{-1}\Lambda<\infty$. Moreover, as $\ep\to 0$, we have the following weak* convergence of Radon measures
\begin{equation}\label{eq: convergence_of_energy_measures}
\mu_{\ep}=\ep^{-1}e_{\ep}(\nabla_{\ep},\Phi_{\ep})\cH^3\rightharpoonup \mu=|h|^2\cH^3 + \sum_{x\in S}\Theta(x)\delta_x,
\end{equation} where
\begin{equation}\label{eq: Theta_lower_bound}
    \Theta(x)\geqslant \eta_{\ast},\quad\forall x\in S.
\end{equation}

We finish this subsection with a simple result clarifying when it is possible to have $S\neq\emptyset$ when $M$ is closed and the family of critical points under consideration are the ones produced by Theorem~\ref{thm: existence}. 
\begin{lemm}\label{lemm: reducible_is_boring}
    Suppose $M$ is closed and that we have a family of critical points $(\nabla_{\ep},\Phi_{\ep})$ for $\cY_{\ep}$, for $\ep$ sufficiently small, satisfying the energy regime
\begin{equation}\label{ineq: basic_energy_bounds}
0<\liminf_{\ep\to 0} \ep^{-1}\cY_{\ep}(\nabla_{\ep},\Phi_{\ep})\leqslant \limsup_{\ep\to 0} \ep^{-1}\cY_{\ep}(\nabla_{\ep},\Phi_{\ep})<\infty.
\end{equation}
Suppose further that a subsequence of $(\nabla_{\ep},\Phi_{\ep})$, which we do not relabel, has been chosen so that in particular we are in the situation summarized above.
\begin{itemize}
    \item[(a)] If $h\equiv 0$, then $S\neq\emptyset$.
    \vskip 1mm
    \item[(b)] $S = \emptyset$ if and only if $\nabla_{\ep}\Phi_{\ep}\equiv 0$ for all but finitely many $\ep$.
\end{itemize}
\end{lemm}
\begin{proof}
    (a) Since $M$ is compact, using the weak* convergence \eqref{eq: convergence_of_energy_measures} and the assumption $h\equiv 0$, we see that 
    \[
    \lim_{\ep \to 0} \ep^{-1}\cY_{\ep}(\nabla_{\ep}, \Phi_{\ep}) = \sum\limits_{x\in S}\Theta(x).
    \] Combining this with the positive energy lower bound in \eqref{ineq: basic_energy_bounds} forces $S\neq\emptyset$.
    
    (b) 
    We first prove the forward implication. Thus, assume that $S = \emptyset$. Then we have 
    \[
    \ep^{-1}e_{\ep}(\nabla_{\ep}, \Phi_{\ep})  \to |h|^2 \quad \text{in $C^{\infty}(M)$},
    \]
    so that eventually there holds 
    \begin{equation}\label{eq: convegence_to_h_everywhere_if_no_S}
    \ep^{-1}e_{\ep}(\nabla_{\ep}, \Phi_{\ep}) \leqslant |h|^2 + 1,\quad \text{everywhere on }M.
    \end{equation}
    Since $M$ is compact, smoothness alone guarantees that $|h|$ is bounded on $M$. Letting
    \[
    \sigma = \frac{1}{8}\min\{\rho_1, \Big( \frac{\eta_{*}}{C(\|h\|_{\infty}^2 + 1)} \Big)^{\frac{1}{3}}\},
    \]
    where $C$ is the universal constant coming from~\eqref{eq:volume_bounds_for_estimates}, we infer from~\eqref{eq: convegence_to_h_everywhere_if_no_S} and~\eqref{eq:volume_bounds_for_estimates} that
    \[
    \sup_{x \in M}\int_{B_{4\sigma}(x)} \ep^{-1}e_{\ep}(\nabla_{\ep}, \Phi_{\ep}) \leqslant C(\|h\|_{\infty}^2 + 1)(4\sigma)^3 < \eta_{*}.
    \]
    Since eventually we also have $\ep \leqslant \tau_1\cdot \min\{\sqrt{\lambda}, 1\}\cdot\sigma$, the desired conclusion follows from Proposition~\ref{prop:concentration_scale}.
    
    To prove the reverse implication, note that the energy upper bound in \eqref{ineq: basic_energy_bounds}, together with the assumption $\nabla_{\ep}\Phi_{\ep}\equiv 0$ (for all but finitely many $\ep$), prevents $\Phi_{\ep_i}$ from vanishing identically for any sequence $\ep_i\to 0$ with $\sup_i \ep_i\ll 1$. Then, by Remark \ref{rmk: reducible_sols}, for any such sequence $\ep_i$, we get that $|\Phi_i|\equiv 1$ and $F_{\nabla_i} = \omega_i\otimes\Phi_i$, where $\omega_i:=\langle F_{\nabla_i},\Phi_i\rangle\in\mathscr{H}^2(M)$ is harmonic, for all $i$. So $h_i:=\ep_i^{\frac{1}{2}}\langle\ast F_{\nabla_i},\Phi_i\rangle$ defines a sequence in $\mathscr{H}^1(M)$ with
    \[
    \|h_i\|_{L^2(M)}^2 = \ep_i\|F_{\nabla_i}\|_{L^2(M)}^2 = \ep_i^{-1}\cY_{\ep_i}(\nabla_i,\Phi_i),
    \] and the energy regime \eqref{ineq: basic_energy_bounds} translates to
    \[
    0<\liminf_{i\to\infty}\|h_i\|_{L^2(M)}^2\leqslant \limsup_{i\to\infty}\|h_i\|_{L^2(M)}^2<\infty.
    \] Thus, by Hodge theory/standard elliptic theory, the $h_i$ subconverges smoothly on $M$ to a non-trivial harmonic limit $0\neq h\in\mathscr{H}^1(M)$. In particular, $\ep_i^{-1}e_{\ep_i}(\nabla_i,\Phi_i) = \ep_i|F_{\nabla_i}|^2=|h_i|^2$ converges smoothly to $|h|^2$ on all of $M$. Thus, in this reducible case there is no normalized-energy concentration; that is, $S=\emptyset$. 
\end{proof}
\begin{example}[Irreducible families with $S\neq\emptyset$ on rational homology $3$-spheres]\label{ex:irred_families_non_trivial_S}
    By Theorem \ref{thm: irred}, if $(M^3,g)$ is closed and $b_1(M)=0$ (that is, when $M$ is a rational homology $3$-sphere, for instance $M=\mathbb{S}^3$ or $\mathbb{RP}^3$) then we can produce a family of critical points $(\nabla_{\ep},\Phi_{\ep})$ for $\cY_{\ep}$, for $\ep$ small enough, satisfying the energy regime \eqref{ineq: basic_energy_bounds}, and which furthermore consists of \emph{irreducible} pairs, $\nabla_{\ep}\Phi_{\ep}\neq 0$. Now, the condition $b_1(M)=0$ immediately forces $h\equiv 0$. Thus, by Lemma \ref{lemm: reducible_is_boring} (a), we conclude that we must have $S\neq\emptyset$ for such families.
\end{example}

\begin{rmk}
    An interesting question is whether it is possible to produce \emph{irreducible} critical points for $\cY_{\ep}$ in contexts where $b_1(M)\neq 0$, and if, moreover, one is able to construct a family of irreducible critical points $(\nabla_{\ep},\Phi_{\ep})$, for $\ep$ sufficiently small, so that both $h\neq 0$ \emph{and} $S\neq\emptyset$. The techniques employed by Stern \cite{Stern2021} in Ginzburg--Landau theory might be helpful to answer these questions.
\end{rmk} 

\subsection{The asymptotic zero set}\label{subsec:zero_set}

We define the \textbf{asymptotic zero set} of $(\nabla_{\ep},\Phi_{\ep})$ to be the accumulation point set of the zeros of the Higgs fields $\Phi_{\varepsilon}$ as $\ep\to 0$: 
\begin{equation}\label{eq: zero_set_Z}
Z:=\bigcap_{\kappa>0}\overline{\bigcup_{0<\ep<\kappa}\Phi_{\ep}^{-1}(0)}.
\end{equation} Recalling $w_{\ep}=\frac{1}{2}(1-|\Phi_{\ep}|^2)$, for any $\beta\in(0,\frac{1}{2})$ we define also the sets
\[
Z_{\beta}(\Phi_{\ep}):=\{x\in M: w_{\ep}(x)\geqslant \beta\},
\] and let 
\begin{equation}\label{eq: zero_set_beta}
Z_{\beta}:= \bigcap_{\kappa>0}\overline{\bigcup_{0<\ep<\kappa} Z_{\beta}(\Phi_{\ep})}.    
\end{equation}
It is immediate from the definitions that both $Z$ and $Z_{\beta}$ are closed in $M$, and that
\begin{equation}\label{eq: Z_betas_monotone}
\Phi_{\ep}^{-1}(0)\subseteq Z_{\beta'}(\Phi_{\ep})\subseteq Z_{\beta}(\Phi_{\ep}),\quad\text{ for all }0<\beta<\beta'<\frac{1}{2}.
\end{equation} In particular, we have $Z\subseteq Z_{\beta}$ for all $\beta\in (0,\frac{1}{2})$. Furthermore, we can prove the following:
\begin{lemm}\label{lem: Z_contained_in_S}
For any $\beta\in(0,\frac{1}{2})$, one has
\[
Z_{\beta}\subset S.
\] In particular, $Z$ and $Z_{\beta}$ are finite sets.
\end{lemm}
\begin{proof}
    Since $Z\subset Z_{\beta}$, and since $S$ is finite by Lemma \ref{lem: S_finite}, we only need to prove that $ Z_{\beta}\subset S$. Let $\lambda_0$ be an upper bound for $\lambda$. Given $\beta\in(0,\frac{1}{2})$, we shall take $\theta_0(\lambda_0,\beta)$ sufficiently small as in Remark~\ref{rmk:clearing_out}(i) and define
    \begin{equation}\label{def: eta_beta}
    \eta({\beta}):= \min\big\{\eta_{\ast},\theta_0(\lambda_0,\beta)\cdot\min\{\lambda, 1\}\big\}.
    \end{equation} (Note that in this notation we have $\eta_{\ast}=\eta(\frac{1}{8(\lambda_0+2)})$; see \eqref{def: eta_star}.) Since $\eta(\beta)\in(0,\eta_{\ast}]$, it follows from Lemma \ref{lem: S_etas} that $S=S_{\eta(\beta)}$. Thus, if $x_0\in M\setminus S$ then we can find $\rho\in (0,r_0]$ and $\kappa>0$ such that for all $\ep<\kappa$ we have $\ep \leqslant \rho$ and the energy bound \eqref{eq:bound_1_for_clearing_out}, so that by Remark~\ref{rmk:clearing_out} (i) we have $\|w_{\ep}\|_{\infty;B_{\rho}(x_0)}<\beta$ for all $\ep<\kappa$. Therefore, $x_0\in M\setminus Z_{\beta}$ as we wanted.
\end{proof}
Next we want to find the complement of $Z_{\beta}$ inside $S$. First, recall that
\[
e_{\ep}(\nabla_{\ep},\Phi_{\ep}) = \Psi_{0, \ep}^2 + \frac{\lambda w_{\ep}^2}{\ep^2}, 
\] where $\Psi_{0, \ep} = \big(\ep^2|F_{\nabla_\ep}|^2 + |\nabla_{\ep}\Phi_\ep|^2\big)^{\frac{1}{2}}$, as introduced in \eqref{eq:Psi_Theta_definition}. Next, define the energy concentration set of the $\Psi_{0,\ep}$-quantity by
\begin{equation}\label{eq: def_S_1}
    S_{\Psi} := \bigcap_{0<r\leqslant  r_0}\left\{x\in M: \liminf_{\ep\to 0}\ep^{-1}\int_{B_r(x)}\ep^2|F_{\nabla_{\ep}}|^2 + |\nabla_{\ep}\Phi_{\ep}|^2\geqslant\eta_{\ast}\right\}.
\end{equation} 
Since $e_{\ep}(\nabla_{\ep}, \Phi_{\ep})\geqslant \Psi_{0, \ep}^2$, it follows from the definition that $S_{\Psi}\subset S$. Combining this with Lemma \ref{lem: Z_contained_in_S}, we get $S_{\Psi}\cup Z_{\beta}\subset S$. We now prove that the reverse inclusion also holds, provided we restrict $\beta$ to a certain range depending on an upper bound $\lambda_0$ for $\lambda$:
\begin{prop}[Splitting of $S$]\label{prop: splitting_of_S}
    For any $\beta\in(0,\frac{1}{4(\lambda_0+2)})$, one has   
    \[
    S=S_{\Psi}\cup Z_{\beta}.
    \]
\end{prop}
\begin{proof}
    We are left to prove that $M\setminus(S_{\Psi}\cup Z_{\beta})\subset M\setminus S$. Let $x\in M\setminus(S_{\Psi}\cup Z_{\beta})$. Then, we can find a sequence $\ep_i\downarrow 0$, and $\rho\in (0,r_0]$ with $\sup_i \ep_i \leqslant \tau_{\ast}\rho$, such that $B_{3\rho}(x)\subset M\setminus Z_{\beta}(\Phi_{\ep_i})$ and 
 \[
 \ep_i^{-1}\int_{B_{4\rho}(x)}\Psi_{0, \ep_i}^2<\eta_{\ast},
 \] for all $i$. Therefore, by Lemmas \ref{lemm:nablaPhi_exp_decay_base} and \ref{lemm:improved_estimates_base} we get
\begin{align*}
\left\|\frac{w_{\ep_i}}{\ep_i}\right\|_{\infty; B_{\rho}(x)}^2 &\leqslant c\ep_i^{-2} e^{-a'\frac{\sqrt{\mu}\rho}{\ep_i}},\\
 \|\Psi_{0, \ep_i}\|_{\infty; B_{\rho}(x)}^2 &\leqslant c\rho^{-3}\ep_i\eta_{\ast},
\end{align*} where $c=c(\lambda_0)$ while $a'$ is universal, in both estimates.
 Thus, for all $r<\rho$ we have
 \begin{align*}
 \liminf_{\ep\to 0} \ep^{-1}\int_{B_r(x)}e_{\ep}(\nabla_{\ep},\Phi_{\ep}) &\leqslant \liminf_{i\to\infty} \ep_i^{-1}\int_{B_r(x)}(\Psi_{0, \ep_i}^2 + \lambda\left|\frac{w_{\ep_i}}{\ep_i}\right|^2)\\
 &= \liminf_{i\to\infty} \ep_i^{-1}\int_{B_r(x)}\Psi_{0, \ep_i}^2\\
 &\leqslant c\rho^{-3}\eta_{\ast}\cdot\mathrm{vol}(B_r(x))\\
 &\leqslant c_{\lambda_0, M}\cdot\rho^{-3}\eta_{\ast}r^3.
 \end{align*} Thus, by taking $r=r(\lambda_0,M,\rho)>0$ sufficiently small we get
 \[
 \liminf_{\ep\to 0} \ep^{-1}\int_{B_r(x)}e_{\ep}(\nabla_{\ep},\Phi_{\ep})<\eta_{\ast}.
 \] That is, $x\in M\setminus S$ as we wanted. This completes the proof.
\end{proof}
\begin{rmk}
Define the energy concentration set of the nonlinear potential term by
\begin{equation}
    S_{w,\beta} := \bigcap_{0<r\leqslant  r_0}\left\{x\in M: \liminf_{\ep\to 0}\ep^{-1}\int_{B_r(x)}\frac{\lambda w_{\ep}^2}{\ep^2}\geqslant \eta'(\beta)\right\},
\end{equation} where for any given $\beta\in(0,\frac{1}{4(\lambda_0+2)})$, we let $\theta_0'=\theta_0'(\lambda_0,\Lambda,\beta)$ be as in part (ii) of Remark \ref{rmk:clearing_out}, and
    \[
    \eta'(\beta):=\min\{\eta_{\ast},4^{-1}\cdot\theta_0'(\lambda_0,\Lambda,\beta)\cdot\min\{\lambda, 1\}\}.
    \] Since $e_{\ep}(\nabla_{\ep},\Phi_{\ep}) \geqslant \frac{\lambda w_{\ep}^2}{\ep^2}$, it follows from Lemma \ref{lem: S_etas} that $S_{w,\beta}\subset S$. Moreover, it follows from part (ii) of Remark \ref{rmk:clearing_out} that $M\setminus S_{w,\beta}\subset M\setminus Z_{\beta}$. 
    These results combined with Proposition \ref{prop: splitting_of_S} yield $S=S_{\Psi}\cup S_{w,\beta}$. Nevertheless, it is not clear if we have 
    $Z_{\beta} =  S_{w,\beta}$; a priori, it could happen that $Z_{\beta}\subsetneq S_{w,\beta}$. Moreover, it is not clear whether $Z_{\beta}$ (or $S_{w,\beta}$) and $S_{\Psi}$ are disjoint or not in general. One could conjecture, for instance, that $S_{\Psi}\subseteq Z_{\beta}$, 
    which would imply that $S=Z_{\beta}$. From the bubbling analysis in \S\ref{subsec:rescaling}, this could be answered in the affirmative if one could prove that there are no non-trivial critical points for $\cY_1^{g_{\RR^3}}$ with $|\Phi|>0$ everywhere.
\end{rmk}


\subsection{Assigning charges to energy concentration points}\label{subsec:assigning_charges}
Given the results already proved in \S\S\ref{subsec:blow-up}-\ref{subsec:zero_set}, in this section we conclude the proof of parts (a) and (b) of Theorem~\ref{thm: asymptotic}. In what follows we show that, by the same sort of argument as in \S\ref{subsec:blow-up}, but applied to the sequence of measures defined by~\eqref{eq: kappa_definition_Sec_5}:
\[
\kappa_{\ep} = 2\bangle{\ast F_{\nabla_{\ep}}, \nabla_{\ep} \Phi_{\ep}} \cH^3,
\]
together with some standard computations, we can assign an integer (up to a factor of $8\pi$) charge to each point in $S$, thereby proving that the limiting measure $\kappa$ in~\eqref{eq: kappa_weak*_convergence} has the form asserted in Theorem~\ref{thm: asymptotic}(b). In accordance with the notation in Proposition~\ref{prop:local_convergence}, we write 
\[
q_{\ep} = 2\bangle{\ast F_{\nabla_{\ep}}, \nabla_{\ep} \Phi_{\ep}},
\]
and notice that
\begin{equation}\label{eq:q_star_exact}
* q_{\ep} = 2 \langle F_{\nabla_{\ep}} \wedge \nabla_{\ep}\Phi_{\ep} \rangle = 2 d(\bangle{F_{\nabla_{\ep}}, \Phi_{\ep}}).
\end{equation}
Also, letting 
\[
\widehat{\Phi}_{\ep} := \frac{\Phi_{\ep}}{|\Phi_{\ep}|} \text{ on }M \setminus Z(\Phi_{\ep}),
\]
writing\footnote{Recall from \S\ref{subsec:variational_properties} that, after choosing a trivialization, we can assume $E=M\times\mathbb{C}^2$.} $\nabla_{\ep} = d + A_{\ep}$, and then temporarily dropping the subscript $\ep$, we recall the following identity~\cite[Chapter II.5]{Jaffe-Taubes}:
\begin{equation}\label{eq:Taubes_formula_for_degree}
\bangle{\widehat{\Phi}, \frac{1}{2}[d\widehat{\Phi}, d\widehat{\Phi}]} = \bangle{\widehat{\Phi}, \frac{1}{2}[\nabla\widehat{\Phi}, \nabla\widehat{\Phi}]} + d(\bangle{\widehat{\Phi}, A}) - \bangle{F_{\nabla}, \widehat{\Phi}},
\end{equation}
which is a general computation about configurations and involves no critical point equations. 
\begin{prop}\label{prop:assign_charge}
There exists for each $x \in S$ some $\Xi(x) \in 8\pi \ZZ$, with $|\Xi(x)| \leqslant \Theta(x)$, such that, in the sense of Radon measures on $M$, we have as $\ep \to 0$ that
\begin{equation}\label{eq:convergence_kappa_i}
\kappa_{\ep} \rightharpoonup \kappa =\sum_{x \in S} \Xi(x) \delta_{x},
\end{equation}
where $\kappa$ is the weak*-limit from~\eqref{eq: kappa_weak*_convergence}. Moreover, $\Xi(x)\neq 0$ only if $x\in Z\subset S$.
\end{prop}
\begin{proof}
Take any sequence $\ep_i \to 0$ and write 
\[
(\nabla_{i}, \Phi_{i}) := (\nabla_{\ep_i}, \Phi_{\ep_i}),\quad \widehat{\Phi}_{i} : =\widehat{\Phi}_{\ep_i},\quad d + A_i : =  d+ A_{\ep_i}.
\]
Given a compact subset $K \subset M \setminus S$, by Lemma~\ref{lemm:nice_convergence_off_of_S} we have that
\[
q_{\ep_{i}} \to 0 \text{ uniformly on $K$ as $i \to \infty$}.
\]
Consequently $\int_{M}f d\kappa = 0$ whenever $f \in C^0_{c}(M)$ with $\supp(f) \cap S = \emptyset$, which in turn implies that $\supp(\kappa) \subset S$, so that $\kappa$ must have the form
\[
\kappa = \sum_{x \in S} \Xi(x)\delta_{x},
\]
for some real numbers $\{\Xi(x)\}_{x \in S}$. 

To continue, let $d_0 \in (0, r_0]$ be so small that $d_0 < \frac{1}{4}d(x, y)$ for all distinct points $x, y$ in $S$. Given $z_0 \in S$ and $\delta \in (0, d_0)$, we let $\zeta$ be a smooth function such that 
\[
\zeta(z_0) = 1,\ \ 0 \leqslant \zeta \leqslant 1 \text{ on }M, \ \ \supp(\zeta) \subset B_{\delta}(z_0).
\]
Then by~\eqref{eq:kappa_by_mu} and \eqref{eq: convergence_of_energy_measures} we have
\[
\begin{split}
|\Xi(z_0)| = \Big| \int_M \zeta d\kappa \Big| \leqslant \int_M \zeta d\mu 
\leqslant\ & \int_{B_{\delta}(z_0)}|h|^2 \vol_g + \Theta(z_0).
\end{split}
\]
Letting $\delta$ tend to $0$ gives 
\[
 |\Xi(z_0)| \leqslant \Theta(z_0).
\]
Next, we prove that $\Xi(z_0) \in 8\pi \ZZ$ and that $\Xi(z_0)\neq 0$ only if $z_0\in Z$. Let $\varphi:\RR \to [0, 1]$ be a cutoff function such that 
\[
\varphi(t) = 1\quad\text{if }t \leqslant \frac{1}{2},\quad\text{and}\quad\varphi(t) = 0\quad\text{if }t \geqslant 1.
\]
Again taking some $\delta \in (0, d_0)$, we define 
\[
f(x) = \varphi\big(\frac{r(x)}{\delta} \big),
\]
where we have written $r$ for the distance function $d(\cdot, z_0)$. From the identity~\eqref{eq:q_star_exact}, we have after an integration by parts that
\begin{equation}\label{eq:assign_charge_1st_by_parts}
\begin{split}
\int_{M} f q_{\ep_i} \vol_g = 2 \int_{M} f \langle F_{\nabla_i}\wedge \nabla_i\Phi_i\rangle =\ & -2\int_{M} \bangle{F_{\nabla_i}, \Phi_i} \wedge df\\
=\ -2 \int_{B_{2\delta}(z_0)\setminus B_{\frac{\delta}{4}}(z_0)} \bangle{F_{\nabla_i}, \Phi_i} \wedge df,
\end{split}
\end{equation}
where the last equality holds since $f$ takes constant values except on $B_{\delta}(z_0)\setminus B_{\frac{\delta}{2}}(z_0)$. The closure of $B_{2\delta}(z_0)\setminus B_{\frac{\delta}{4}}(z_0)$ being a compact subset of $M \setminus S$, we have by Lemma~\ref{lemm:nice_convergence_off_of_S} that, as $i \to \infty$,
\begin{equation}\label{eq:assign_charge_estimate_on_annulus}
\ep_i^{-\frac{1}{2}}|\nabla_i\Phi_i| + \ep_i^{-\frac{3}{2}}(1 - |\Phi_i|^2) \to 0, \text{ uniformly on }B_{2\delta}(z_0)\setminus B_{\frac{\delta}{4}}(z_0).
\end{equation}
In particular, eventually $B_{2\delta}(z_0)\setminus B_{\frac{\delta}{4}}(z_0)$ is contained in the domain of $\widehat{\Phi}_i$. Also recalling that $|\nabla_i \widehat{\Phi}_i| = |\Phi_i|^{-1}|(\nabla_i\Phi_i)^{\perp}|$, we infer from~\eqref{eq:assign_charge_estimate_on_annulus} that
\begin{equation}\label{eq:assign_charge_Phi_normalized}
\ep_i^{-\frac{1}{2}} |\nabla_i\widehat{\Phi}_i| \to 0 \text{ uniformly on }B_{2\delta}(z_0)\setminus B_{\frac{\delta}{4}}(z_0).
\end{equation}
We now apply~\eqref{eq:Taubes_formula_for_degree} to get from~\eqref{eq:assign_charge_1st_by_parts} that
\begin{small}
\begin{equation}\label{eq:assign_charge_splitting}
\begin{split}
\int_{M} f q_{\ep_i} \vol_g =\ & -2\int_{B_{2\delta}(z_0)\setminus B_{\frac{\delta}{4}}(z_0)} \bangle{F_{\nabla_i}, \widehat{\Phi}_i} \wedge df + 2\int_{B_{2\delta}(z_0)\setminus B_{\frac{\delta}{4}}(z_0)} \bangle{F_{\nabla_i}, (1 - |\Phi_i|)\widehat{\Phi}_i } \wedge df\\
=\ & 2\int_{B_{2\delta}(z_0)\setminus B_{\frac{\delta}{4}}(z_0)} \bangle{\widehat{\Phi}_i, \frac{1}{2}[d\widehat{\Phi}_i, d\widehat{\Phi}_i]} \wedge df - 2\int_{B_{2\delta}(z_0)\setminus B_{\frac{\delta}{4}}(z_0)} \bangle{\widehat{\Phi}_i, \frac{1}{2}[\nabla_i\widehat{\Phi}_i, \nabla_i\widehat{\Phi}_i]} \wedge df\\
& - 2\int_{B_{2\delta}(z_0)\setminus B_{\frac{\delta}{4}}(z_0)} d(\bangle{\widehat{\Phi}_i, A_i}) \wedge df + 2\int_{B_{2\delta}(z_0)\setminus B_{\frac{\delta}{4}}(z_0)} \bangle{F_{\nabla_i}, (1 - |\Phi_i|)\widehat{\Phi}_i } \wedge df\\
=:\ & 2 \cdot \big((I) + (II) + (III) + (IV)\big).
\end{split}
\end{equation}
\end{small}
For $(IV)$ we have by H\"older's inequality and~\eqref{ineq: uniform_bound_assumption} that
\begin{equation}\label{eq:assign_charge_IV}
\begin{split}
|(IV)| \leqslant\ & C\|df\|_{\infty}\int_{B_{2\delta}(z_0)\setminus B_{\frac{\delta}{4}}(z_0)} \ep_i^{\frac{1}{2}}|F_{\nabla_i}| \cdot \ep_i^{-\frac{1}{2}}|1 - |\Phi_i||\\
\leqslant\ & C\|df\|_{\infty} \cdot \Lambda^{\frac{1}{2}} \cdot \Big( \int_{B_{2\delta}(z_0)\setminus B_{\frac{\delta}{4}}(z_0)} \ep_i^{-1}(1 - |\Phi_i|)^2 \vol_{g} \Big)^{\frac{1}{2}} \to 0 \text{ as }i \to \infty,
\end{split}
\end{equation} where the latter convergence follows by first recalling that eventually $|\Phi_i|\leqslant 1$ (by Proposition \ref{prop:maximum_principle_for_w}), thus $1-|\Phi_i|\leqslant 1-|\Phi_i|^2$, and then using \eqref{eq:assign_charge_estimate_on_annulus}.
The integral $(III)$ simply vanishes by Stokes' theorem, since $d(\bangle{\widehat{\Phi}_i, A_i}) \wedge df$ is exact and since $df$ vanishes on the boundary of $B_{2\delta}(z_0)\setminus B_{\frac{\delta}{4}}(z_0)$. Next, for $(II)$ we use~\eqref{eq:assign_charge_Phi_normalized} to get
\begin{equation}\label{eq:assign_charge_II}
\begin{split}
|(II)| \leqslant\ & C\|df\|_{\infty} \int_{B_{2\delta}(z_0)\setminus B_{\frac{\delta}{4}}(z_0)} |\nabla_i\widehat{\Phi}_i|^2\vol_{g}\\
\leqslant\ & C\|df\|_{\infty} \ep_i \cdot \int_{B_{2\delta}(z_0)\setminus B_{\frac{\delta}{4}}(z_0)} \ep_i^{-1}|\nabla_i\widehat{\Phi}_i|^2 \vol_{g} \to 0 \text{ as }i \to \infty.
\end{split}
\end{equation}
To treat $(I)$, we notice from our choice of $f$ that $df = \varphi'(\frac{r}{\delta}) \frac{d r}{\delta}$. Thus, the integrand in $(I)$ becomes
\[
\frac{1}{\delta}\varphi'(\frac{r}{\delta})\bangle{\widehat{\Phi}_i, \frac{1}{2}[d\widehat{\Phi}_i, d\widehat{\Phi}_i]} \wedge d r = \frac{1}{\delta}\varphi'(\frac{r}{\delta})\bangle{\bangle{\widehat{\Phi}_i, \frac{1}{2}[d\widehat{\Phi}_i, d\widehat{\Phi}_i]}, \ast dr} \vol_g.
\]
Noting that $\ast d r$ restricts to the volume form on the geodesic sphere $\partial B_{r}(z_0)$, and applying the coarea formula, we get, since $|\nabla r| = 1$, that
\[
\begin{split}
(I) =\ & \int_{\frac{\delta}{4}}^{2\delta} \frac{1}{\delta}\varphi'(\frac{r}{\delta}) \Big(\int_{\partial B_{r}(z_0)}\bangle{\widehat{\Phi}_i, \frac{1}{2} [d\widehat{\Phi}_i, d\widehat{\Phi}_i]}\Big) dr.
\end{split}
\]
By the mapping degree formula (see for instance~\cite[Chapter 4, \S8]{Guillemin-Pollack}), there exists $N_i \in \ZZ$ such that 
\[
\int_{\partial B_{r}(z_0)}\bangle{\widehat{\Phi}_i, \frac{1}{2}[d\widehat{\Phi}_i, d\widehat{\Phi}_i]} = -4\pi N_i \text{ for all }r \in [\frac{\delta}{4}, 2\delta].
\]
Substituting this back into the above computation for $(I)$ gives
\begin{equation}\label{eq:assign_charge_I}
(I) = -4\pi N_i \cdot (\varphi(2) - \varphi(\frac{1}{4})) = 4\pi N_i.
\end{equation}
Combining this with~\eqref{eq:assign_charge_IV},~\eqref{eq:assign_charge_II} and recalling that $(III)$ vanishes, we get from~\eqref{eq:assign_charge_splitting} that 
\[
\begin{split}
\big| \int_{M} f q_{\ep_i} \vol_g - 8\pi N_i \big| \to 0 \text{ as }i \to \infty.
\end{split}
\]
On the other hand, since $f(z_0) = 1$ and $f$ vanishes outside of $B_{\delta}(z_0)$, we have 
\[
\lim_{i \to \infty} \int_{M} fq_{\ep_i} \vol_g = \Xi(z_0).
\]
It follows that $(N_i)\subset\ZZ$ is eventually constant, and we actually have $\Xi(z_0)=8\pi N_i$ for all $i\gg 1$. This immediately gives $\Xi(z_0) \in 8\pi \ZZ$. Finally, observe from the above argument that we have in fact proved, for any $\delta\in (0,d_0)$, that
\[
\Xi(z_0) = -2\int_{\partial B_{\delta}(z_0)}\bangle{\widehat{\Phi}_i, \frac{1}{2}[d\widehat{\Phi}_i, d\widehat{\Phi}_i]},\quad \text{for all sufficiently large $i$.}
\] In particular, if $\Xi(z_0)\neq 0$, then for any $\delta\ll 1$, we can find a zero of $\Phi_i$ inside the ball $B_\delta(z_0)$ for all $i\gg 1$; hence, $z_0\in Z$. This completes the proof.
\end{proof}
In view of Lemma~\ref{lem: S_finite}, Lemma~\ref{lem: Z_contained_in_S}, the convergence~\eqref{eq: convergence_of_energy_measures}, and Proposition~\ref{prop:assign_charge}, we have established parts (a) and (b) of Theorem~\ref{thm: asymptotic}.
\begin{rmk}
   In the case $(M^3,g)$ is closed, it follows from equation \eqref{eq:q_star_exact} and Stokes' theorem that $\kappa_{\ep}(M)=0$, in which case Proposition \ref{prop:assign_charge} allows us to conclude that
    \[
    \sum\limits_{x\in Z}\Xi(x) = 0.
    \]
\end{rmk}

\subsection{Hodge decomposition of the longitudinal part}\label{subsec:hodge_decomp_longit_part}
For this section, we assume that $M^3$ is closed and prove Theorem~\ref{thm: asymptotic}(c). Specifically, the goal is to use the ideas at the end of \S\ref{subsec:local-convergence}, leading to Proposition \ref{prop:convergence_nonharmonic_parts}, in order to generalize, in some sense, what happens in the proof of Lemma \ref{lemm: reducible_is_boring} (b), and show that the $1$-form $h$ in the convergence of measures~\eqref{eq: convergence_of_energy_measures} is accounted for solely by the harmonic component in the Hodge decomposition of $\ep^{\frac{1}{2}}\langle \ast F_{\nabla_{\ep}}, \Phi_{\ep}\rangle$. 

Let us first recall the Hodge decomposition \eqref{eq:Hodge_decomposition_longitudinal_part}:
\[
\ep^{\frac{1}{2}}\bangle{* F_{\nabla_{\ep}}, \Phi_{\ep}} = \widetilde{h}_{\ep} + df_{\ep} + d^*\alpha_{\ep},
\]
where $f_{\ep} \in C^{\infty}(M)$, $\alpha_{\ep} \in \Omega^2(M)$, $\widetilde{h}_{\ep}\in \mathscr{H}^1(M)$, and we can further assume that 
\[
\int_M f_{\ep} = 0,\quad d\alpha_{\ep} = 0,\quad\text{and}\quad\alpha_{\ep}\in\left(\mathscr{H}^2(M)\right)^{\perp_{L^2}}.
\] 
Next we choose an arbitrary sequence $\ep_i \to 0$ and write
\[
(\nabla_i, \Phi_i) : = (\nabla_{\ep_i}, \Phi_{\ep_i}),\quad \widetilde{h}_{i}: = \widetilde{h}_{\ep_i}, \quad f_{i}: = f_{\ep_i}, \quad \alpha_{i}: = \alpha_{\ep_i}.
\]
Using the uniform normalized-energy bound assumption \eqref{ineq: uniform_bound_assumption}, and the $L^{\infty}$-bound on $\Phi_{i}$ given by Proposition \ref{prop:maximum_principle_for_w} as soon as $\ep_i < r_0$, we have for sufficiently large $i$ that
\begin{equation}\label{eq:longitudinal_hodge_orthogonality_bound_2}
\|\widetilde{h}_i\|_{2; M}^2 + \|df_i\|_{2; M}^2 + \|d^*\alpha_i\|_{2; M}^2 = \|\ep_i^{\frac{1}{2}}\bangle{* F_{\nabla_i}, \Phi_i}\|_{2;M}^2 \leqslant \Lambda.
\end{equation}
Moreover, in the language of the blow-up set $S$ introduced in \S\ref{subsec:blow-up}, we can recast Proposition \ref{prop:convergence_nonharmonic_parts} as follows:
\begin{prop}\label{prop:convergence_nonharmonic_parts_2}
Assume that $(M^3, g)$ is closed. Then, in the above notation, upon passing to a subsequence if needed, we have:
\vskip 1mm
\begin{enumerate}
\item[(a)] There exist a function $f$ and a $2$-form $\alpha$, both of class $W^{1, 2} \cap L^6$ on $M$, such that 
\begin{equation}
f_i \to f,\ \ \alpha_i \to \alpha, \text{ weakly in $W^{1, 2}$ and strongly in $L^2$}.
\end{equation}
\vskip 1mm
\item[(b)] Both $(f_i)$ and $(\alpha_i)$ converge smoothly on compact subsets of $M\setminus S$.
\end{enumerate}
\end{prop}
Along the subsequence produced by Proposition~\ref{prop:convergence_nonharmonic_parts_2}, we turn to the harmonic components $\widetilde{h}_i$ and notice that, as in the proof of Lemma~\ref{lemm: reducible_is_boring} (b), from the uniform $L^2$-bound in~\eqref{eq:longitudinal_hodge_orthogonality_bound_2} and standard elliptic estimates, we may also assume, taking a further subsequence if needed, that there exists some $\widetilde{h} \in \mathscr{H}^1(M)$ such that
\begin{equation}\label{eq: harmonic_component_subsequential_limit}
\widetilde{h}_i \to \widetilde{h} \text{ smoothly on }M.
\end{equation}
We may now state the main result of this section.
\begin{prop}\label{prop:convergence_of_nonharmonic_part}
In the above notation, we have that 
\vskip 1mm
\begin{enumerate}
\item[(a)] $\widetilde{h}$ coincides on $M$ with the harmonic $1$-form $h$ in~\eqref{eq: convergence_of_energy_measures}. Consequently $\widetilde{h}_i$ converges smoothly to $h$ on all of $M$.
\vskip 1mm
\item[(b)] Up to taking a subsequence, $d{f_i}$ and $d^*\alpha_i$ both converge smoothly to $0$ on compact subsets of $M\setminus S$.
\vskip 1mm
\item[(c)] In the sense of Radon measures,
\[
\big(\ep_i^{-1}e_{\ep_i}(\nabla_i, \Phi_i) - |\widetilde{h}_i|^2\big) \cH^3 \rightharpoonup \sum_{x \in S}\Theta(x)\delta_{x}.
\]
\end{enumerate}
\end{prop}
\begin{proof}
As the conclusions of Proposition \ref{prop:convergence_nonharmonic_parts_2} are in effect, we have
\begin{equation}\label{eq:longitudinal_non_harmonic_smooth_convergence_off_S}
{f_i} \to {f}, \ \ \alpha_i \to \alpha \text{ smoothly locally on  }M \setminus S.
\end{equation}
Recalling from Lemma~\ref{lemm:nice_convergence_off_of_S} that $\ep_i^{\frac{1}{2}}\bangle{*F_{\nabla_i}, \Phi_i}$ converge to $h$ smoothly on $M \setminus S$, we get, upon passing to the limit in~\eqref{eq:Hodge_decomposition_longitudinal_part} and rearranging,
\begin{equation}\label{eq:longitudinal_hodge_limit_off_S}
d{f}+ d^*\alpha = h - \widetilde{h} \text{ on }M \setminus S.
\end{equation}
We next use this to prove that 
\begin{equation}\label{eq:d_alpha_d*_beta_vanish}
\int_{M}|d f|^2 + |d^*\alpha|^2 = 0,
\end{equation}
which would give statements (a) and (b) immediately. To that end, using Lemma \ref{lem: S_finite}, we label the (finitely many) points in $S$ by 
\[
S = \{z_1, \cdots, z_N\},
\]
and choose $d_0 \in (0, r_0]$ such that the balls $\{B_{4d_0}(z_i)\}$ are mutually disjoint. Take a cutoff function $\varphi: \RR \to [0, 1]$ such that 
\[
\varphi(t) = 0\quad\text{if }t \leqslant \frac{1}{2},\quad\text{and}\quad \varphi(t) = 1 \text{ if }t \geqslant 1.
\]
Given $\delta \in (0, d_0)$, we let
\[
\zeta(x) = \zeta_{\delta}(x) =  \prod_{i = 1}^{N}\varphi(\frac{d(x, z_i)}{\delta}).
\]
Multiplying~\eqref{eq:longitudinal_hodge_limit_off_S} by $d(\zeta_{\delta} {f})$, which is supported outside of $S$, and integrating over $M$, we get
\begin{equation}\label{eq:longitudinal_d_alpha_test}
0 = \int_{M}\bangle{d{f}, d(\zeta_{\delta} {f})} = \int_{M} \zeta_{\delta} |d{f}|^2 + \sum_{i = 1}^N\int_{B_{\delta}(z_i)\setminus B_{\frac{\delta}{2}}(z_i)} \bangle{d{f}, fd\zeta_{\delta}}.
\end{equation}
For each $i \in \{1, \cdots, N\}$, notice that, since $|d\zeta| \leqslant C\delta^{-1}$, we have by H\"older's inequality that
\begin{equation}\label{eq:d_alpha_cutoff_term}
\begin{split}
\Big|\int_{B_{\delta}(z_i)\setminus B_{\frac{\delta}{2}}(z_i)} \bangle{d{f}, {f}d\zeta}\Big| \leqslant\ & C\delta^{-1} \Big( \int_{B_{\delta}(z_i)\setminus B_{\frac{\delta}{2}}(z_i)} |d{f}|^2 \Big)^{\frac{1}{2}}\Big( \int_{B_{\delta}(z_i)\setminus B_{\frac{\delta}{2}}(z_i)} |{f}|^2 \Big)^{\frac{1}{2}}.
\end{split}
\end{equation}
To bound the integral of $|d{f}|^2$ we simply enlarge the domain to $M$ and use~\eqref{eq:longitudinal_hodge_orthogonality_bound_2} to get
\[
\int_{B_{\delta}(z_i)\setminus B_{\frac{\delta}{2}}(z_i)} |d{f}|^2 \leqslant \int_{M} |d{f}|^2 \leqslant \Lambda.
\] 
On the other hand, by H\"older's inequality applied to the integral of $|{f}|^2$, we get
\[
 \int_{B_{\delta}(z_i)\setminus B_{\frac{\delta}{2}}(z_i)} |{f}|^2 \leqslant C\delta^2 \Big( \int_{B_{\delta}(z_i)\setminus B_{\frac{\delta}{2}}(z_i)} |{f}|^6 \Big)^{\frac{1}{3}}.
\]
Putting these back into~\eqref{eq:d_alpha_cutoff_term} gives
\[
\Big|\int_{B_{\delta}(z_i)\setminus B_{\frac{\delta}{2}}(z_i)} \bangle{d{f}, {f}d\zeta}\Big|  \leqslant C\Lambda^{\frac{1}{2}}\Big( \int_{B_{\delta}(z_i)\setminus B_{\frac{\delta}{2}}(z_i)} |{f}|^6 \Big)^{\frac{1}{6}},
\]
which tends to $0$ as $\delta \to 0$ since $|{f}|^6$ is integrable on all of $M$ by Proposition \ref{prop:convergence_nonharmonic_parts_2} (a). Consequently we deduce upon letting $\delta \to 0$ in~\eqref{eq:longitudinal_d_alpha_test} that 
\[
\int_{M}|d{f}|^2 = 0.
\]
Testing~\eqref{eq:longitudinal_hodge_limit_off_S} by $d^*(\zeta_{\delta}\alpha)$ instead, we get
\[
0 = \int_{M}\zeta_{\delta}|d^*\alpha|^2 - \sum_{i = 1}^N \int_{B_{\delta}(z_i) \setminus B_{\frac{\delta}{2}}(z_i)}\bangle{d^*\alpha, d\zeta_{\delta} \lrcorner \alpha}.
\]
Again using H\"older's inequality with exponents given by $\frac{1}{2} + \frac{1}{3} + \frac{1}{6} = 1$, but this time noting that, by \eqref{eq:longitudinal_hodge_orthogonality_bound_2} and Proposition \ref{prop:convergence_nonharmonic_parts_2} (a), we have
\[
\int_{M}|d^*\alpha|^2 \leqslant \Lambda,\ \ \int_{M}|\alpha|^6 < \infty,
\]
we obtain 
\[
\lim_{\delta \to 0}\int_{M}\zeta_{\delta}|d^*\alpha|^2 = 0.
\]
That is, we have established~\eqref{eq:d_alpha_d*_beta_vanish}.  Going back to~\eqref{eq:longitudinal_hodge_limit_off_S}, we see that consequently $\widetilde{h}$ agrees with $h$ on $M\setminus S$, and thus on all of $M$ by continuity. This proves (a). From the vanishing of $df$ and $d^*\alpha$, and the convergence~\eqref{eq:longitudinal_non_harmonic_smooth_convergence_off_S}, we obtain conclusion (b). Finally, since $\widetilde{h}_i$ converges smoothly to $h$ on all of $M$ by part (a), we get part (c) upon recalling~\eqref{eq: convergence_of_energy_measures}.
\end{proof}
In view of parts (a) and (b) of Proposition~\ref{prop:convergence_of_nonharmonic_part}, we obtain Theorem~\ref{thm: asymptotic}(c). Note that we have in fact demonstrated that for any $\ep_i \to 0$, the sequence $(\widetilde{h}_{\ep_i})$ admits a further subsequence converging smoothly on $M$ to $h$. Thus the full sequence $(\widetilde{h}_{\ep})$ must converge to $h$ in $C^{\infty}(M)$, which combines with~\eqref{eq: convergence_of_energy_measures} to give
\[
\big(\ep^{-1}e_{\ep}(\nabla_{\ep}, \Phi_{\ep}) - |\widetilde{h}_{\ep}|^2\big) \cH^3 \rightharpoonup \sum_{x \in S}\Theta(x)\delta_{x},
\]
as $\ep \to 0$ along the original sequence extracted at the start of \S\ref{subsec:blow-up}. 
\subsection{The rate of rescaling and bubbling}\label{subsec:rescaling}
The remainder of this paper is devoted to the proof of Theorem~\ref{thm: bubbling}. As in the previous subsections, we choose a sequence $\ep_i \to 0$ and set
\[
(\nabla_i, \Phi_i) : = (\nabla_{\ep_i}, \Phi_{\ep_i}).
\]
Also, we assume that $S \neq \emptyset$. If $S$ consists of exactly one point, we let $d_0 := r_0$, where $r_0$ is defined at the start of \S\ref{sec:asymptotic}, while if $S$ contains at least two points, we let $d_0 := \min\{r_0, \frac{d(x_i, x_j)}{16}\}$, where $(x_i, x_j)$ runs through all pairs of distinct points in $S$. 

Now fix a point $z_0\in S$. For all $i \in \NN$ and $t  \in [0, d_0]$ we define
\[
Q_i(t) = \sup_{x \in B_{d_0}(z_0)} \ep_i^{-1}\int_{B_{t}(x)} e_{\ep_i}(\nabla_i, \Phi_i) \vol_{g}.
\]
Notice the following two properties of $Q_i$. First, for each fixed $t$, we have
\[
\liminf_{i \to \infty}Q_{i}(t) \geqslant \liminf_{i \to \infty}\ep_i^{-1}\int_{B_t(z_0)} e_{\ep_i}(\nabla_i, \Phi_i) \vol_g \geqslant \Theta(z_0) \geqslant \eta_{\ast},
\] where in the second and third inequalities we have used \eqref{eq: convergence_of_energy_measures} and \eqref{eq: Theta_lower_bound}, respectively. Secondly, for each fixed $i$, we have
\[
\lim_{t \to 0}Q_i(t) = 0.
\]
Thus we may construct, by a standard argument, a sequence of scales $(t_k)$ converging to zero, a subsequence $(i_k)$ of the $i$'s, and a sequence $(x_k)$ converging in $M$ to $z_0$, such that
\begin{equation}\label{eq:center_and_rate}
\ep_{i_k}^{-1}\int_{B_{t_k}(x_{k})}e_{\ep_{i_k}}(\nabla_{i_k}, \Phi_{i_k})\vol_g = Q_{i_k}(t_k) = \frac{\eta_{\ast}}{2}.
\end{equation}
Below, for brevity we write, by some abuse of notation, 
\[
(\nabla_k, \Phi_k) := (\nabla_{i_k}, \Phi_{i_k}),\quad \ep_k: = \ep_{i_k},
\]
and also assume, dropping finitely many initial terms if needed, that 
\begin{equation}\label{eq: convenient_smallness_conditions}
\ep_k + t_k + d(x_k, z_0) < \frac{d_0}{4},\quad \text{ for all }k.
\end{equation}
Finally, we define the following affine maps from $\RR^3$ to $\RR^3$:
\[
s_{x, t}: y \mapsto x + ty.
\]
\begin{lemm}\label{lemm:rate_of_rescaling}
In the above notation, we have that $\ep_k$ and $t_k$ are comparable. That is, the following hold:
\vskip 1mm
\begin{enumerate}
\item[(a)] $\limsup_{k \to \infty}\frac{\ep_k}{t_k} < \infty$.
\vskip 1mm
\item[(b)] $\liminf_{k \to \infty}\frac{\ep_k}{t_k} > 0$.
\end{enumerate}
\end{lemm}
\begin{proof}
To see part (a), notice that since $\ep_k^{-1}\cY_{\ep_k}(\nabla_k, \Phi_k) \leqslant \Lambda$ by the assumption~\eqref{ineq: uniform_bound_assumption}, we may eventually (as soon as $\ep_k$ becomes small enough) invoke Lemma~\ref{lemm:coarse_estimate_base} and Lemma~\ref{lemm:w_mean_value_estimate} to obtain the following pointwise bound on the energy density:
\[
\|e_{\ep_k}(\nabla_k, \Phi_k)\|_{\infty; M} \leqslant C_{\Lambda,\lambda_0}\ep_k^{-2}.
\]
In particular, for all $\theta \in (0, 1)$ we have
\[
\begin{split}
\ep_k^{-1}\int_{B_{\theta \ep_k}(x_k)} e_{\ep_k}(\nabla_k, \Phi_k) \vol_g \leqslant C_{\Lambda,\lambda_0}\ep_k^{-3} (\theta \ep_k)^3 = C_{\Lambda,\lambda_0}\theta^3.
\end{split}
\]
Choosing $\theta \in (0, 1)$ so that $C_{\Lambda,\lambda_0} \theta^3 < \frac{\eta_{\ast} }{2}$, we deduce from the above and~\eqref{eq:center_and_rate} that $t_k \geqslant \theta\ep_k$. In other words, for sufficiently large $k$ we have
\[
\frac{\ep_k}{t_k} \leqslant \theta^{-1}.
\]
This proves (a). To see part (b), we argue by contradiction and suppose that along a subsequence of $k$'s, which we do not relabel, there holds 
\[
\lim_{k \to\infty}\frac{\ep_k}{t_k} = 0.
\]
Letting $y_k := (\exp_{z_0}|_{B_{2d_0}(0)})^{-1}(x_k)$, so that in particular $|y_k| = d(x_k, z_0) \to 0$ as $k \to \infty$, we define
\[
\widetilde{\ep}_k := \frac{\ep_k}{t_k},\quad \widetilde{g}_k := t_k^{-2}(\exp_{z_0} \circ s_{y_k, t_k})^*g,
\]
and also introduce the rescaled configurations
\[
 (\widetilde{\nabla}_k, \widetilde{\Phi}_k) := (\exp_{z_0} \circ s_{y_k, t_k})^* (\nabla_k, \Phi_k).
\] 
Notice that $\widetilde{g}_k$ is a Riemannian metric on the ball $B_{\frac{2d_0}{t_k}}(-\frac{y_k}{t_k})$ (which contains $B_{\frac{7d_0}{4t_k}}(0)$ by~\eqref{eq: convenient_smallness_conditions}), and that it converges smoothly to $g_{\RR^3}$ on compact subsets as $k \to \infty$. For later use we also note that, by our definition of $d_0$, we have 
\[
\inj_{(\exp_{z_0})^*g}(x)  \geqslant d_0, \text{ for all $x$ in $B_{d_0}(0)$},
\]
and hence for all $k \in \NN$ and $y$ lying in the ball $B_{\frac{d_0}{t_k}}(-\frac{y_k}{t_k})$ (which contains $B_{\frac{3d_0}{4t_k}}(0)$, again by~\eqref{eq: convenient_smallness_conditions}), we see after a straightforward computation that
\begin{equation}\label{eq:g_k_inj}
\inj_{\widetilde{g}_k}(y) \geqslant \frac{d_0}{t_k}, 
\end{equation}
and that 
\begin{equation}\label{eq:g_k_curvature}
\big(\frac{d_0}{t_k}\big)^{m + 2}|(\nabla^{\widetilde{g}_k})^{m}R_{\widetilde{g}_k}|_{\widetilde{g}_k}(y) \leqslant d_0^{m + 2} \|(\nabla^g)^{m}R_g\|_{\infty; B_{2d_0}(z_0)} \leqslant A_m,
\end{equation}
where the second inequality follows since $d_0 \leqslant r_0 < \frac{\rho_0}{4}$. Thus, for each sufficiently large $k$, we see that~\eqref{eq:injectivity_radius_for_estimates} and~\eqref{eq:curvature_bound_for_estimates} are fulfilled with $\Omega = B_{\frac{3d_0}{4t_k}}(0)$ and $\rho_0 = \frac{d_0}{2t_k}$, and with the same curvature bounds $A_0, A_1, \cdots$ as in the beginning of Section~\ref{sec:asymptotic}. On the other hand, by Lemma~\ref{lem: scaling}, we see that $(\widetilde{\nabla}_k, \widetilde{\Phi}_k)$ is a critical point of $\cY_{\widetilde{\ep}_k}^{\widetilde{g}_k}$ on $B_{\frac{2d_0}{t_k}}(-\frac{y_k}{t_k}) \supset B_{\frac{7d_0}{4t_k}}(0)$. Now, notice that by~\eqref{eq:center_and_rate} and the definition of $Q_{i_k}(t_k)$, suitably rescaled (see Lemma \ref{lemm:scaling_g_to_delta_g}), we have that 
\begin{equation}\label{eq:center_and_rate_rescaled}
(\widetilde{\ep}_k)^{-1}\cY_{\widetilde{\ep}_k}^{\widetilde{g}_k}(\widetilde{\nabla}_k, \widetilde{\Phi}_k; B_{1}(y)) \leqslant \frac{\eta_{\ast}}{2}, \text{ for all }y \in B_{\frac{3d_0}{4t_k}}(0),
\end{equation}
with equality holding at $y = 0$. Consequently, for all $R > 0$, by our choice of $\eta_{\ast}$, and since $\widetilde{\ep}_k, t_k \to 0$, eventually we may invoke Proposition~\ref{prop:local_convergence} and Remark~\ref{rmk:smallness_condition_relation} everywhere on $B_R(0)\subset \RR^3$. Taking a sequence of radii $R_i$ tending to infinity and applying a diagonal argument, we obtain a subsequence of $(\widetilde{\nabla}_k, \widetilde{\Phi}_k)$, which we do not relabel, such that 
\[
(\widetilde{\ep}_k)^{-1} e_{\widetilde{\ep}_k}^{\widetilde{g}_k}(\widetilde{\nabla}_k, \widetilde{\Phi}_k) \to |h|_{g_{\RR^3}}^2,\quad \text{ in $C^{\infty}_{\loc}(\RR^3)$},
\]
where $h$ is a harmonic $1$-form with respect to the Euclidean metric. From this, we draw two consequences. First, for all $R > 0$ we have by~\eqref{ineq: uniform_bound_assumption} that
\[
\int_{B_R(0)}|h|^2 \vol_{g_{\RR^3}} = \lim_{k \to \infty} (\widetilde{\ep}_k)^{-1}\cY_{\widetilde{\ep}_k}^{\widetilde{g}_k}(\widetilde{\nabla}_k, \widetilde{\Phi}_k; B_R(0)) \leqslant \Lambda,
\]
so that $h \in L^2(\RR^3)$. Secondly, since equality is attained in~\eqref{eq:center_and_rate_rescaled} at $y =0$,
\[
\int_{B_1(0)}|h|^2 \vol_{g_{\RR^3}} = \lim_{k \to \infty} (\widetilde{\ep}_k)^{-1}\cY_{\widetilde{\ep}_k}^{\widetilde{g}_k}(\widetilde{\nabla}_k, \widetilde{\Phi}_k; B_1(0)) = \frac{\eta_{\ast}}{2} > 0.
\]
As there are no non-trivial harmonic $1$-forms on $\RR^3$ with finite $L^2$-norm, we have arrived at a contradiction.
\end{proof}

\begin{lemm}\label{lemm:rate_of_rescaling_better}
In the setting of Lemma~\ref{lemm:rate_of_rescaling}, we in fact have
\[
\liminf_{k \to \infty}\frac{\ep_k}{t_k} \geqslant \frac{\tau_1}{4}\cdot\min\{1,\sqrt{\lambda}\},
\]
where $\tau_1=\tau_1(\lambda_0)\in (0,1)$ is the threshold given by Lemma \ref{lemm:nablaPhi_exp_decay_base}.
\end{lemm}
\begin{proof}
Having shown in Lemma~\ref{lemm:rate_of_rescaling} that $(\frac{\ep_k}{t_k})$ is a bounded sequence, to prove the assertion, it suffices to show that each of its convergent subsequences has limit at least $\frac{\tau_1}{4}\cdot\min\{1,\sqrt{\lambda}\}$. Suppose that is not the case. Then, up to choosing a subsequence, and also recalling Lemma~\ref{lemm:rate_of_rescaling}(b), we have 
\[
\frac{\ep_k}{t_k} \to \alpha \in (0, \frac{\tau_1}{4}\cdot\min\{1,\sqrt{\lambda}\}) \text{ as }k \to \infty.
\]
Dropping finitely many initial terms we may further assume that
\[
\frac{\ep_k}{t_k} < \frac{\tau_1}{4}\cdot\min\{1,\sqrt{\lambda}\}, \text{ for all }k.
\]
Now define $y_k$, $\widetilde{\ep}_k$, $\widetilde{g}_k$, and the pair $(\widetilde{\nabla}_k, \widetilde{\Phi}_k)$ as in the proof of Lemma~\ref{lemm:rate_of_rescaling}(b). For all $y \in B_{\frac{3d_0}{4t_k}}(0) \subset B_{\frac{d_0}{t_k}}(-\frac{y_k}{t_k})$, by the global energy bound~\eqref{ineq: uniform_bound_assumption}, together with the Higgs field bound given by Proposition \ref{prop:maximum_principle_for_w}, and the coarse estimates in Lemma~\ref{lemm:coarse_estimate_base} and Proposition~\ref{prop:coarse_estimate} applied to the ball $B_{t_k}(y_k + t_k y)$, which is permitted since 
\[
\ep_k < \frac{\tau_1}{4}\cdot\min\{1,\sqrt{\lambda}\}t_k \leqslant \frac{t_k}{4},
\]
we obtain for all $m \in \NN \cup \{0\}$ that, on $B_{\frac{t_k}{2}}(y_k + t_k y)$, there holds
\[
|\Phi_k|\leqslant 1,\ \ \ |(\nabla_{k})^{m + 1}\Phi_k|_{g} \leqslant C_{m, \Lambda,\lambda_0}\ep_k^{-m-1},\ \ \  |(\nabla_k)^m F_{\nabla_k}|_{g} \leqslant  C_{m, \Lambda,\lambda_0}\ep_k^{-m-2}, 
\]
where by abuse of notation we have written $g$ for its pullback under $\exp_{z_0}$. Under scaling, these translate into the following estimates on $B_{\frac{1}{2}}(y)$:
\[
|\widetilde{\Phi}_k|\leqslant 1,\ \ \ |(\widetilde{\nabla}_k)^{m + 1}\widetilde{\Phi}_k|_{\widetilde{g}_k} \leqslant  C_{m, \Lambda,\lambda_0}(\widetilde{\ep}_k)^{-m-1},\ \ \   |(\widetilde{\nabla}_k)^m F_{\widetilde{\nabla}_k}|_{\widetilde{g}_k} \leqslant  C_{m, \Lambda,\lambda_0}(\widetilde{\ep}_k)^{-m-2}.
\]
Since $\widetilde{\ep}_k$ converges to a positive limit, for each $m$ the right-hand side of the latter two estimates are bounded independently of $k$. As $y \in B_{\frac{3d_0}{4t_k}}(0)$ is arbitrary and $\widetilde{g}_k$ is converging smoothly locally to $g_{\RR^3}$, it is standard to deduce from these estimates that each point in $\RR^3$ possesses a neighborhood on which, after locally changing into Coulomb gauges, a convergent subsequence of $(\widetilde{\nabla}_k, \widetilde{\Phi}_k)$ can be extracted. An equally standard patching argument (see for example~\cite[Section 4.4.2]{DK} or~\cite[Chapter 7]{wehrheim2004uhlenbeck}) then yields global gauge transformations over $\RR^3$ so that afterwards $(\widetilde{\nabla}_k, \widetilde{\Phi}_k)$ subconverge smoothly on compact subsets of $\RR^3$ to a limiting configuration $(\nabla, \Phi)$, which must be a solution of~\eqref{eq: 2nd_order_crit_pt_intro} on $(\RR^3, g_{\RR^3})$ with $\ep = \alpha$.

We next verify that the hypotheses of Proposition~\ref{prop:concentration_scale} are fulfilled. First, by the global energy bound~\eqref{ineq: uniform_bound_assumption} again, we have for all $R > 0$ that
\begin{equation}\label{eq:rate_lower_bound_limit_bound}
\alpha^{-1}\cY_{\alpha}^{\RR^3}(\nabla, \Phi; B_R(0)) = \lim_{k \to \infty}(\widetilde{\ep}_k)^{-1}\cY_{\widetilde{\ep}_k}^{\widetilde{g}_k}(\widetilde{\nabla}_k, \widetilde{\Phi}_k; B_R(0)) \leqslant  \Lambda.
\end{equation}
In particular $(\nabla, \Phi)$ has finite $\cY_\alpha^{\RR^3}$-action. Next, for the flat metric $g_{\RR^3}$, the radius $\rho_0$ in~\eqref{eq:injectivity_radius_for_estimates} can be arbitrarily chosen, and we fix, say, $\rho_0 = 4$. The factor $\mu_1$ in~\eqref{eq:radius_rel_curvature}, on the other hand, can be taken to be $1$. Finally, the estimate~\eqref{eq:center_and_rate_rescaled}, passed to the limit as $k \to \infty$, implies that 
\begin{equation}\label{eq:rate_lower_bound_limit_small}
\alpha^{-1}\cY_{\alpha}^{g_{\RR^3}}(\nabla, \Phi; B_1(y)) \leqslant  \frac{\eta_{\ast}}{2} < \eta_{\ast}, \text{ for all }y \in \RR^3,
\end{equation}
with the first inequality becoming an equality at $y = 0$. In particular, we obtain~\eqref{eq:small_energy_everywhere} with $\sigma = \frac{1}{4}$. Since $\alpha < \frac{\tau_1}{4}\cdot\min\{1,\sqrt{\lambda}\}$, we conclude from Proposition~\ref{prop:concentration_scale} that $(\nabla,\Phi)$ satisfies \eqref{eq: reducible_pair_intro}.
Combining this with~\eqref{eq:rate_lower_bound_limit_bound} and the fact that equality holds at $y = 0$ in~\eqref{eq:rate_lower_bound_limit_small}, we infer that $\bangle{F_{\nabla}, \Phi}$ is a non-zero harmonic $2$-form in $\mathbb{R}^3$ with finite $L^2$-norm, which is a contradiction.
\end{proof}

As a consequence of the upper and lower bounds on $\frac{\ep_k}{t_k}$ just established, we can relate the multiplicity $\Theta(x)$ at each point $x \in S$ to the energy threshold provided by our gap theorem. In fact, we have the following result, which in particular gives the first conclusion of Theorem~\ref{thm: bubbling}. 
\begin{prop}\label{prop:multiplicity_lower_bound}
For all $x \in S$ there exists a critical point $(\nabla, \Phi)$ of $\cY_1^{g_{\RR^3}}$ on $\RR^3$ such that 
\begin{equation}\label{eq:non_trivial_first_bubble}
0 < \cY_1^{g_{\RR^3}}(\nabla, \Phi; \RR^3) \leqslant  \Theta(x).
\end{equation}
In particular,
\begin{equation}\label{eq:multiplicity_lower_bound}
\Theta(x) \geqslant \theta_{\mathrm{gap}}\cdot\min\{\lambda,1\},
\end{equation} where $\theta_{\mathrm{gap}}$ is the threshold given by Theorem~\ref{thm: gap}, also appearing in Theorem \ref{thm: gap_for_R3}.
\end{prop}
\begin{proof}
It suffices to prove this just for the blow-up point $z_0$ that we have been working with in this section. The argument at other points in $S$, if any, would be the same. Identifying $B_{2d_0}(z_0)$ with $B_{2d_0}(0) \subset \RR^3$ via the exponential map and writing $\exp_{z_0}^* g$ as $g$, we adopt the notation of the previous proof, except we rescale $(\nabla_k, \Phi_k)$ by $\ep_k$ as opposed to $t_k$. In other words, we set
\[
(\widetilde{\nabla}_k, \widetilde{\Phi}_k) := (s_{y_k, \ep_k})^*(\nabla_k, \Phi_k),\ \ \widetilde{g}_k := \ep_k^{-2}(s_{y_k, \ep_k})^*g.
\]
For all $y \in B_{\frac{3d_0}{4\ep_k}}(0) \subset B_{\frac{d_0}{\ep_k}}(-\frac{y_k}{\ep_k})$, again by 
\eqref{ineq: uniform_bound_assumption} and Proposition \ref{prop:maximum_principle_for_w}, along with Lemma~\ref{lemm:coarse_estimate_base} and Proposition~\ref{prop:coarse_estimate} applied to the original configuration $(\nabla_k, \Phi_k)$ on the ball $B_{4\ep_k}(y_k + \ep_k y)$ we obtain, for $m = 0, 1, \cdots$, the following bounds on $B_{2\ep_k}(y_k + \ep_k y)$:
\[
|\Phi_k|\leqslant 1,\ \ \ |(\nabla_{k})^{m + 1}\Phi_k|_{g} \leqslant C_{m, \Lambda,\lambda_0}\ep_k^{-m-1},\ \ \  |(\nabla_k)^m F_{\nabla_k}|_{g} \leqslant  C_{m, \Lambda,\lambda_0}\ep_k^{-m-2}.
\]
Upon scaling by $\ep_k$, these estimates become simply
\[
|\widetilde{\Phi}_k|\leqslant 1,\ \ \ |(\widetilde{\nabla}_k)^{m + 1}\widetilde{\Phi}_k|_{\widetilde{g}_k} +  |(\widetilde{\nabla}_k)^m F_{\widetilde{\nabla}_k}|_{\widetilde{g}_k} \leqslant  C_{m, \Lambda,\lambda_0}, \text{ on }B_{2}(y).
\]
Since $y \in B_{\frac{3d_0}{4\ep_k}}(0)$ is arbitrary, and again $\widetilde{g}_{k}$ converges in $C^{\infty}_{\loc}(\RR^3)$ to $g_{\RR^3}$, we can argue as in the previous proof to get, after taking a further subsequence and changing gauge if needed, that $(\widetilde{\nabla}_k, \widetilde{\Phi}_k)$ converge smoothly on compact subsets of $\RR^3$ to a solution of~\eqref{eq: 2nd_order_crit_pt_intro} on $(\RR^3, g_{\RR^3})$ with $\ep = 1$. Noting that 
\[
\cY_1^{\widetilde{g}_k}(\widetilde{\nabla}_k, \widetilde{\Phi}_k; B_{\frac{t_k}{\ep_k}}(0)) = \ep_{k}^{-1}\cY_{\ep_k}^g(\nabla_k, \Phi_k; B_{t_k}(y_k)) = \frac{\eta_{\ast}}{2},
\]
and recalling from Lemma~\ref{lemm:rate_of_rescaling} that eventually $\frac{t_k}{\ep_k}$ is bounded away from both $0$ and $\infty$, we infer that
\[
\cY_1^{g_{\RR^3}}(\nabla, \Phi; \RR^3) \geqslant \frac{\eta_{\ast}}{2} > 0.
\]
This proves the first inequality in~\eqref{eq:non_trivial_first_bubble}. On the other hand, given $R > 0$ and $\delta \in (0, d_0)$, since eventually $R\ep_k + |y_k| < \delta$, we see with the help of the local smooth convergence of $(\widetilde{\nabla}_k, \widetilde{\Phi}_k)$ to $(\nabla, \Phi)$ and the scaling relations in Lemmas~\ref{lemm:scaling_g_to_delta_g} and \ref{lem: scaling} that
\[
\begin{split}
\cY_1^{g_{\RR^3}}(\nabla, \Phi; B_R) = \ & \lim_{k \to \infty}\ep_k^{-1}\cY_{\ep_k}^{g}(\nabla_k, \Phi_k; B_{R\ep_k}(y_k))\\
\leqslant \ & \limsup_{k \to \infty}\mu_{\ep_k}(B_{\delta}(z_0)) \leqslant 
 \int_{B_{\delta}(z_0)}|h|^2 \vol_g + \Theta(z_0).
\end{split}
\]
Letting $R \to \infty$ and $\delta \to 0$ gives the second inequality in~\eqref{eq:non_trivial_first_bubble}. Finally, the positivity of $\cY_{1}^{g_{\RR^3}}(\nabla, \Phi)$ together with Theorem~\ref{thm: gap_for_R3} forces
\begin{equation}\label{eq:1st_bubble_nontrivial}
\cY_1^{g_{\RR^3}}(\nabla, \Phi; \RR^3) >\theta_{\mathrm{gap}}\cdot\min\{\lambda,1\},
\end{equation}
which gives~\eqref{eq:multiplicity_lower_bound} when combined with~\eqref{eq:non_trivial_first_bubble}.
\end{proof}
\begin{rmk}\label{rmk:improved_upper_bound_S}
    With $\cR_{x}$ being as defined above~\eqref{eq: weak_limit_measure_balls}, let $r\in (0,r_0]\setminus\big(\bigcup_{x \in S} \cR_x \big)$ be small enough so that $4r<d(x_i,x_j)$ for all pairs $(x_i, x_j)$ of distinct points in $S$. Then, it follows from equation \eqref{eq: weak_limit_measure_balls} and the weak* convergence of measures \eqref{eq: convergence_of_energy_measures} that
    \[
    \lim_{\ep\to 0}\sum\limits_{x\in S}\ep^{-1}\int_{B_r(x)}e_{\ep}(\nabla_{\ep},\Phi_{\ep}) = \sum\limits_{x\in S}\int_{B_r(x)} |h|^2 + \sum\limits_{x\in S} \Theta(x)\geqslant \sum\limits_{x\in S} \Theta(x).
    \] Since the collection $\{B_r(x)\}_{x\in S}$ consists of pairwise disjoint balls, using the uniform energy upper bound assumption \eqref{ineq: uniform_bound_assumption} together with the multiplicity lower bound \eqref{eq:multiplicity_lower_bound} given by Proposition \ref{prop:multiplicity_lower_bound}, we deduce that
    \[
    \Lambda\geqslant \cH^0(S)\cdot \theta_{\mathrm{gap}}\cdot\min\{\lambda,1\},
    \] that is, we get the following improved, universal version of the upper bound in Lemma \ref{lem: S_finite}(b):
    \[
    \cH^0(S)\leqslant \theta_{\mathrm{gap}}^{-1}\cdot\max\{\lambda^{-1},1\}\cdot\Lambda.
    \]
\end{rmk}
Applying the existence part of Proposition~\ref{prop:multiplicity_lower_bound} to the families of irreducible critical points produced by Theorem~\ref{thm: irred} on rational homology $3$-spheres (see Example \ref{ex:irred_families_non_trivial_S}), we obtain critical points of $\cY_{1}^{g_{\RR^3}}$ on $\RR^3$ with positive energy. More precisely, we have the following result.
\begin{prop}\label{prop:existence_R3}
There exists a critical point $(\nabla, \Phi)$ of $\cY_1^{g_{\RR^3}}$ such that 
\[
0 < \cY_1^{g_{\RR^3}}(\nabla, \Phi) < \infty.
\]
\end{prop}
\begin{proof}
Take any closed, oriented Riemannian $3$-manifold $(M^3,g)$ with $b_1(M) = 0$; for example, the round $3$-sphere. By Theorem~\ref{thm: irred}, we can find a sequence $\ep_i \to 0$ and, for each $i$, an irreducible critical point $(\nabla_i, \Phi_i)$ for $\cY_{\ep_i}$, on the trivial $SU(2)$-bundle over $M$, such that 
\begin{equation}\label{eq:existence_R3_two_sided}
1 \lesssim_{\lambda}\ep_i^{-1}\cY_{\ep_i}(\nabla_{i},\Phi_{i}) \lesssim_{\lambda, M} 1, \text{ for all }i.
\end{equation}
Then, as explained in Example \ref{ex:irred_families_non_trivial_S}, for such a sequence we must have $S\neq\emptyset$. Thus, the desired existence result now follows from the first conclusion of Proposition~\ref{prop:multiplicity_lower_bound}.
\end{proof}

\subsection{Extracting bubbles and identifying neck regions}\label{subsec:regions}
We continue to work in the setting of the previous section and adopt the same set of notation. For convenience, we recall the convergences of Radon measures \eqref{eq: convergence_of_energy_measures} and \eqref{eq:convergence_kappa_i}:
\begin{align*}
\mu_{\ep_i} = \ep_{i}^{-1}e_{\ep_i}(\nabla_i, \Phi_i) \vol_g &\rightharpoonup \mu=|h|^2 \cH^3 + \sum_{x \in S}\Theta(x)\delta_x,\\
\kappa_{\ep_i} = 2\bangle{* F_{\nabla_i}, \nabla_i\Phi_i}\vol_g &\rightharpoonup \kappa=\sum_{x \in S}\Xi(x)\delta_{x},
\end{align*}
where for all $x \in S$ there holds
\[
\Xi(x) \in 8\pi \ZZ,\ \ |\Xi(x)| \leqslant  \Theta(x).
\]
Also recall that, by the Schwarz inequality (see~\eqref{eq:charge_density_by_energy_density}), we have 
\begin{equation}\label{eq:charge_by_energy}
2|\bangle{* F_{\nabla_i}, \nabla_i\Phi_i}| \leqslant  \ep_{i}^{-1}e_{\ep_i}(\nabla_i, \Phi_i).
\end{equation}
Thanks to this bound, up to taking a subsequence, we can assume in addition that
\[
\kappa_{\ep_i}^{\pm} : = 2\bangle{* F_{\nabla_i}, \nabla_i\Phi_i}^{\pm} \vol_g \rightharpoonup \kappa^{\pm},\quad \text{ in weak* sense},
\]
with the limiting measures satisfying $\kappa^{\pm} \leqslant  \mu$. Standard theory for non-negative measures (see~\eqref{eq: weak_limit_measure_balls}) then shows that whenever $E \subset M$ is a Borel set with $\mu(\partial E) = 0$, there holds
\begin{equation}\label{eq:kappa_setwise_limit}
\kappa(E) = \kappa^+(E) - \kappa^{-}(E) = \lim_{i \to \infty}\kappa_{\ep_i}^{+}(E) - \kappa_{\ep_i}^{-}(E) = \lim_{i \to \infty}\kappa_{\ep_i}(E).
\end{equation}
Below, we assume that $S \neq\emptyset$ and fix a particular $z_0 \in S$. With $d_0$ defined as in the start of \S\ref{subsec:rescaling}, we again identify $B_{2d_0}(z_0)$ with $B_{2d_0}(0) \subset \RR^3$ via the exponential map, and write $\exp_{z_0}^*g$ simply as $g$. 
Note also the following scaling property pertaining to $\kappa_{\ep_i}$: letting 
\[
(\nabla', \Phi') = s_{y, r}^*(\nabla, \Phi),\ \ g' = r^{-2}s_{y, r}^* g,
\]
we have
\begin{equation}\label{eq:kappa_scaling}
\begin{split}
\langle F_{\nabla'} \wedge \nabla'\Phi'\rangle =\ &  \bangle{\ast_{g'} F_{\nabla'}, \nabla' \Phi'}_{g'} \vol_{g'} = s_{y, r}^*\big(\bangle{\ast_{g} F_{\nabla}, \nabla \Phi}_{g} \vol_{g}\big) = s_{y, r}^*\big( \langle F_{\nabla }\wedge \nabla\Phi\rangle\big).
\end{split}
\end{equation}
With the above preliminaries in mind, we consider
\[
(\widetilde{\nabla}_i, \widetilde{\Phi}_i) = (s_{y_i, \ep_i})^*(\nabla_i, \Phi_i),\ \ \widetilde{g}_i = \ep_i^{-2}(s_{y_i, \ep_i})^*g.
\]
Then, as in the proof of Proposition~\ref{prop:multiplicity_lower_bound}, passing to a further subsequence and changing gauge if necessary, we may assume that $(\widetilde{\nabla}_i, \widetilde{\Phi}_i)$ converges smoothly on compact subsets of $\RR^3$ to a critical point $(\nabla,\Phi)$ of $\cY_1^{g_{\RR^3}}$ which satisfies
\begin{equation}\label{eq:1st_bubble_action_bound}
\theta_{\mathrm{gap}}\cdot\min\{\lambda,1\} \leqslant  \cY_1^{g_{\RR^3}}(\nabla, \Phi; \RR^3) \leqslant  \Theta(z_0).
\end{equation}
We refer to the solution $(\nabla,\Phi)$ as the \emph{top bubble} and define the \emph{energy difference} to be 
\[
\tau(z_0) = \Theta(z_0) - \cY_1^{g_{\RR^3}}(\nabla, \Phi; \RR^3).
\]
Likewise, in view of Proposition~\ref{prop:assign_charge} we introduce what one might call the \emph{charge difference}:
\[
\tau_{\text{charge}}(z_0) = \Xi(z_0) - \cK(\nabla, \Phi),
\]
where $\cK(\nabla,\Phi)$ is ($8\pi$ times) the magnetic charge of the solution $(\nabla,\Phi)$ and is given by the well-known formula (see \eqref{eq: magnetic_charge}):
\[
\cK(\nabla,\Phi) = \lim_{R \to \infty} 2\int_{B_R} \langle F_{\nabla} \wedge \nabla\Phi\rangle.
\]
Notice that 
\begin{equation}\label{eq:tau_definite_drop}
0 \leqslant  \tau(z_0) \leqslant  \Theta(z_0) - \theta_{\mathrm{gap}}\cdot\min\{\lambda,1\}.
\end{equation}
To further analyze what $\tau(z_0)$ is composed of, we would like to realize it as the limit of the energy of $(\nabla_i, \Phi_i)$ on suitably chosen annuli. To that end, we make the following definition.
\begin{defi}\label{defi:transition_region}
A sequence of annuli $(B_{\delta_j}(y_j)\setminus B_{r_j}(y_j))$ in $\RR^3$ is said to determine a \emph{transition region} if the following hold.
\vskip 1mm
\begin{enumerate}
\item[(t1)] $\lim_{j \to \infty}\frac{1}{\delta_j}= \lim_{j \to \infty}\frac{\delta_j}{r_j} = \lim_{j \to \infty}\frac{r_j}{\ep_j} = \infty$.
\vskip 1mm
\item[(t2)] As $j \to \infty$, we have that 
\[
\mu_{\ep_j}(B_{\delta_j}(y_j) \setminus B_{r_{j}}(y_j)) \to \tau(z_0), \ \ \ \kappa_{\ep_j}(B_{\delta_j}(y_j) \setminus B_{r_{j}}(y_j)) \to \tau_{\text{charge}}(z_0).
\]
\vskip 1mm
\item[(t3)] For all $K > 1$, we have 
\[
\lim_{j \to \infty}\mu_{\ep_j}(B_{\delta_j}(y_j) \setminus B_{\frac{\delta_j}{K}}(y_j)) = \lim_{j \to \infty}\mu_{\ep_j}(B_{Kr_j}(y_j) \setminus B_{r_j}(y_j)) = 0.
\]
\end{enumerate}
\end{defi}

\begin{lemm}\label{lemm:neck_analysis_transition_region}
Passing to a subsequence of $(\nabla_i, \Phi_i)$ if necessary, there exist radii $(\delta_j), (r_j)$ such that the sequence of annuli $(B_{\delta_j}(y_j) \setminus B_{r_j}(y_j))$ determines a transition region.
\end{lemm}
\begin{proof}
Since $\ep_i \to 0$ and $y_i \to 0$, and since we have the following convergence of measures:
\[
\kappa_{\ep_i}|_{B_{2d_0}(0)} \to  \Xi(z_0)\delta_0,\ \ \ \mu_{\ep_i}|_{B_{2d_0}(0)} \to |h|^2 \vol_g|_{B_{2d_0}(0)} + \Theta(z_0)\delta_0,
\]
we can find a subsequence $(i_k)$ of $(i)$ such that the following hold
\begin{subequations}
\begin{align}
|y_{i_k}| + \ep_{i_k} \leqslant \ & k^{-6}, \label{eq:annuli_will_degenerate}\\
\Big|\mu_{\ep_{i_k}}(B_{2k^{-1}}(0)) - (\Theta(z_0) + \int_{B_{2k^{-1}}(0)}|h|^2\vol_g)  \Big| \leqslant \ & k^{-1}, \label{eq:multiplicity_captured}\\
\Big|\mu_{\ep_{i_k}}(B_{2k^{-1}}(0)\setminus B_{2^{-1}k^{-3}}(0)) -  \int_{B_{2k^{-1}}(0)\setminus B_{2^{-1}k^{-3}}(0)}|h|^2\vol_g  \Big| \leqslant \ & k^{-1}, \label{eq:no_energy_at_outer_end}\\
\big|\kappa_{\ep_{i_k}}(B_{2k^{-1}}(0)) - \Xi(z_0)  \big| \leqslant \ & k^{-1}, \label{eq:charge_captured}
\end{align}
\end{subequations}
where to get~\eqref{eq:charge_captured} we used~\eqref{eq:kappa_setwise_limit}. 

To continue, we write $\ep_{i_k}, y_{i_k}, \cdots$ as $\ep_k, y_k, \cdots$. From~\eqref{eq:annuli_will_degenerate} we see that eventually 
\begin{equation}\label{eq:recenter_annuli}
B_{\frac{1}{2k^3}}(0) \subset B_{\frac{1}{k^3}}(y_{k}) \subset B_{\frac{1}{k}}(y_{k}) \subset B_{\frac{2}{k}}(0),
\end{equation}
and consequently we get from~\eqref{eq:multiplicity_captured},~\eqref{eq:no_energy_at_outer_end}, and the fact $|h|^2 \in L^1$, that 
\begin{equation}\label{eq:recentered_estimate}
\begin{split}
\mu_{\ep_{k}}(B_{k^{-1}}(y_k)) =\ & \Theta(z_0) + o(1), \\
\mu_{\ep_k}(B_{k^{-1}}(y_k) \setminus B_{k^{-3}}(y_k)) =\  & o(1).
\end{split}
\end{equation}
Further, from~\eqref{eq:charge_captured} and~\eqref{eq:charge_by_energy} we infer that 
\begin{equation}\label{eq:recentered_estimate_for_charge}
\begin{split}
|\kappa_{\ep_k}(B_{k^{-1}}(y_k)) - \Xi(z_0)| \leqslant \ & |\kappa_{\ep_k}(B_{2k^{-1}}(0)) - \Xi(z_0)| + |\kappa_{\ep_k}(B_{2k^{-1}}(0) \setminus B_{k^{-1}}(y_k))|\\
\to\ & 0, \text{ as }k \to \infty.
\end{split}
\end{equation}
This addresses the selection of the outer radii for the desired annuli. To determine the inner radii, we notice that since $(\widetilde{\nabla}_k, \widetilde{\Phi}_k)$ still converges smoothly to $(\nabla, \Phi)$ on compact subsets of $\RR^3$, we may choose a further subsequence $(k_j)$ such that 
\[
\big|\cY_{1}^{\widetilde{g}_{k_j}}(\widetilde{\nabla}_{k_j}, \widetilde{\Phi}_{k_j}; B_{j}(0)) - \cY_{1}^{g_{\RR^3}}(\nabla, \Phi; B_{j}(0)) \big|\leqslant \frac{1}{j},
\]
\[
\big|\cY_{1}^{\widetilde{g}_{k_j}}(\widetilde{\nabla}_{k_j}, \widetilde{\Phi}_{k_j}; B_{j^3}(0)\setminus B_{j}(0)) - \cY_{1}^{g_{\RR^3}}(\nabla, \Phi; B_{j^3}(0)\setminus B_{j}(0)) \big|\leqslant \frac{1}{j},
\]
\[
\Big| 2\int_{B_{j}(0)}\langle F_{\widetilde{\nabla}_{k_j}} \wedge \widetilde{\nabla}_{k_j}\widetilde{\Phi}_{k_j}\rangle - 2\int_{B_{j}(0)}\langle F_{\nabla}\wedge \nabla\Phi\rangle  \Big|\leqslant  \frac{1}{j}.
\]
Since, as $j \to \infty$,
\[
\cY_1^{g_{\RR^3}}(\nabla, \Phi; B_{j}(0)) \to \cY_1^{g_{\RR^3}}(\nabla, \Phi; \RR^3) ,\ \ \ 2\int_{B_{j}(0)}\langle F_{\nabla}\wedge \nabla\Phi\rangle  \to \cK(\nabla,\Phi),
\]
it follows from the above estimates that 
\begin{equation}\label{eq:bubble_captured}
\mu_{\ep_{k_j}}(B_{j \ep_{k_j}}(y_{k_j})) = \cY_1^{g_{\RR^3}}(\nabla, \Phi; \RR^3) + o(1),
\end{equation}
\begin{equation}\label{eq:no_energy_at_inner_end}
\mu_{\ep_{k_j}}(B_{j^3\ep_{k_j}}(y_{k_j}) \setminus B_{j\ep_{k_j}}(y_{k_j}))= o(1),
\end{equation}
and that, recalling also~\eqref{eq:kappa_scaling}, 
\begin{equation}\label{eq:bubble_charge_captured}
\kappa_{\ep_{k_j}}(B_{j\ep_{k_j}}(y_{k_j})) = \cK(\nabla,\Phi) + o(1).
\end{equation}
We now let
\[
\delta_j := \frac{1}{k_j},\ \ r_j := j\ep_{k_j}.
\]
Then it is obvious that
\[
\lim_{j \to \infty}\frac{1}{\delta_j} = \lim_{j \to \infty}\frac{r_j}{\ep_{k_j}} = \infty.
\]
On the other hand, by~\eqref{eq:annuli_will_degenerate} and the fact that $k_j \geqslant j$, we also have
\begin{equation}\label{eq:relation_radii}
\frac{\delta_j}{j^2} \geqslant \frac{1}{k_j^3},\quad \frac{\delta_j}{r_j} \geqslant j^4.
\end{equation}
The second estimate in~\eqref{eq:relation_radii} and the obvious fact mentioned above together verify (t1) from Definition~\ref{defi:transition_region}. Next, subtracting~\eqref{eq:bubble_captured} from the first estimate in~\eqref{eq:recentered_estimate} (with $k_j$ in place of $k$) gives
\[
\mu_{\ep_{k_j}}(B_{\delta_j}(y_{k_j}) \setminus B_{r_j}(y_{k_j})) = \tau(z_0) + o(1).
\]
Likewise, subtracting~\eqref{eq:bubble_charge_captured} from~\eqref{eq:recentered_estimate_for_charge} gives
\[
\kappa_{\ep_{k_j}}(B_{\delta_j}(y_{k_j}) \setminus B_{r_j}(y_{k_j})) = \tau_{\text{charge}}(z_0) + o(1).
\]
This establishes property (t2). Finally, from~\eqref{eq:recentered_estimate} and the first estimate in~\eqref{eq:relation_radii} we get
\[
\lim_{j \to \infty}\mu_{\ep_{k_j}}(B_{\delta_j}(y_{k_j}) \setminus B_{\frac{\delta_j}{j^2}}(y_{k_j})) = 0.
\]
This together with~\eqref{eq:no_energy_at_inner_end} yields (t3).
\end{proof}
Writing $(k_j)$ as $(j)$, our next task is to decompose the annuli $(B_{\delta_j}(y_j) \setminus B_{r_j}(y_j))$ into bubble regions and neck regions. It will be convenient at times to think of annuli as images of cylinders under the maps
\[
\begin{split}
\cE_{y}: \RR \times S^2 \to\ & \RR^3 \setminus \{0\}\\
(t, \xi) \mapsto\ & e^t \xi + y.
\end{split}
\]
Notice the following simple relation:
\begin{equation}\label{eq:dilation_to_translation}
(s_{y, r} \circ \cE_{z})([a, b] \times S^2) = \cE_{s_{y, r}(z)}([a + \log r, b + \log r] \times S^2).
\end{equation}
\begin{defi}\label{defi:bubble_region}
Let $(I_j)$ be a sequence of compact intervals such that $2I_j \subset [\log(r_j), \log(\delta_j)]$ for each $j$ (here $2I_j$ denotes the interval having the same midpoint as $I_j$ but twice the length) and that 
\[
\lim_{j \to \infty}\diam I_j = \infty.
\]
We say that the sequence of annuli $(\cE_{y_j}(I_j \times S^2))$ determines a \emph{bubble region} if there exists, for each $j$, a subinterval $I_j' \subset I_j$ having the same midpoint as $I_j$, such that the following hold:
\vskip 1mm
\begin{enumerate}
\item[(b1)] $\lim_{j \to \infty}\diam I_j' = \lim_{j \to \infty} \dist(I_j', \partial I_j) = \infty$.
\vskip 1mm
\item[(b2)] $\lim_{j \to \infty} \mu_{\ep_j}(\cE_{y_j}((2I_j \setminus I_j') \times S^2)) = 0$.
\vskip 1mm
\item[(b3)] Denote by $m_j$ the common midpoint of $I_j'$ and $I_j$ and set
\[
s_j = s_{y_j, e^{m_j}}: z \mapsto y_j + e^{m_j}z.
\]
Then, as $j \to \infty$, the rescaled measures $s_j^*\mu_{\ep_j}$ and $s_j^*\kappa_{\ep_j}$ converge on compact subsets of $\RR^3 \setminus \{0\}$ to limiting measures $\nu$ and $\gamma$, respectively, and moreover we have
\[
\lim_{j \to \infty}\mu_{\ep_j}(\cE_{y_j}(I_j' \times S^2)) = \nu(\RR^3 \setminus \{0\}),
\]
\[
\lim_{j \to \infty}\kappa_{\ep_j}(\cE_{y_j}(I_j' \times S^2)) = \gamma(\RR^3 \setminus \{0\}).
\]
\vskip 1mm
\item[(b4)] The limiting measures in (b3) have the following forms:
\[
\nu = \sum_{x \in S(\nu)} c(x) \delta_{x},\ \ \gamma = \sum_{x \in S(\nu)} k(x)\delta_{x}
\]
where $S(\nu)$ is a non-empty, finite set of points in $\RR^3 \setminus \{0\}$, and each $c(x)$ is at least $\theta_{\mathrm{gap}}\cdot\min\{\lambda,1\}$, while each $k(x)$ lies in $8\pi \ZZ$.
\end{enumerate}
\end{defi}
\begin{defi}\label{defi:neck_region}
Again let $(I_j = [a_j, b_j])$ be a sequence of compact intervals such that $I_j \subset [\log(r_j), \log(\delta_j)]$ and that
\[
\lim_{j \to \infty}\diam I_j = \infty.
\]
We say that $(\cE_{y_j}(I_j \times S^2))$ determines a \emph{neck region} if the following hold.
\vskip 1mm
\begin{enumerate}
\item[(n1)] For all $L > 0$, we have 
\[
\lim_{j \to \infty}\mu_{\ep_j}(\cE_{y_j}([a_j, a_j + L] \times S^2)) = \lim_{j \to \infty}\mu_{\ep_j}(\cE_{y_j}([b_j - L, b_j] \times S^2)) = 0.
\]
\vskip 1mm
\item[(n2)] We have
\[
\lim_{j \to \infty}\big(\sup_{a_j + \log 2 \leqslant  t \leqslant  b_j - \log 2}\mu_{\ep_j}(\cE_{y_j}([t - \log 2, t + \log 2] \times S^2)) \big)= 0.
\]
\end{enumerate}
One sees that property (n2) implies property (n1). We state (n1) separately only for the sake of convenience.
\end{defi}

The next lemma, consequence of a standard procedure that has appeared many times in the vast literature of bubbling analysis, asserts that either the annuli $(B_{\delta_j}(y_j) \setminus B_{r_j}(y_j))$ produced by Lemma~\ref{lemm:neck_analysis_transition_region} already determines a neck region, or that we can decompose it into bubble regions and neck regions. 
\begin{lemm}\label{lemm:bubble_neck_decomposition}
Up to taking a subsequence of $(j)$, either the sequence $(B_{\delta_j}(y_j) \setminus B_{r_j}(y_j))$ itself determines a neck region, or we can find $N \in \NN$ and, for each $\alpha \in \{1, \cdots, N\}$, a sequence of intervals $I_{\alpha, j} \subset [\log(r_j), \log(\delta_j)]$, such that
\vskip 1mm
\begin{enumerate}
\item[(i)] For all $\alpha \in \{1, \cdots, N\}$, 
\[
\lim_{j \to \infty}\diam I_{\alpha, j} = \lim_{j \to \infty}\dist(I_{\alpha, j}, \{\log r_j, \log \delta_j\}) = \infty.
\]
\vskip 1mm
\item[(ii)] If $\alpha \neq \beta$, then $I_{\alpha, j} \cap I_{\beta, j} = \emptyset$, and moreover 
\[
\lim_{j \to \infty}\dist(I_{\alpha, j}, I_{\beta, j}) = \infty.
\]
\vskip 1mm
\item[(iii)] $(I_{\alpha, j})_{j = 1}^{\infty}$ determines a bubble region, for each $\alpha \in \{1, \cdots, N\}$.
\vskip 1mm
\item[(iv)] Writing
\begin{equation}\label{eq:bubble_complement_neck}
[\log r_j, \log \delta_j] \setminus \big(\cup_{\alpha = 1}^N \Inte(I_{\alpha, j}) \big) = \cup_{\beta = 0}^{N} J_{\beta, j},
\end{equation}
then $(J_{\beta, j})_{j = 1}^{\infty}$ determines a neck region for each $\beta \in \{0, \cdots, N\}$.
\end{enumerate}
\end{lemm}
\begin{proof}
Throughout the proof, by property (i), property (ii), and so forth, we mean the properties listed in the conclusion of the lemma. To start, we notice that if 
\begin{equation}\label{eq:entire_annuli_no_concentration}
\lim_{j \to \infty}\big(\sup_{r \in [2r_j, 2^{-1}\delta_j]} \mu_{\ep_j}(B_{2r}(y_j)\setminus B_{2^{-1}r}(y_j))\big) = 0,
\end{equation}
then $(B_{\delta_j}(y_j) \setminus B_{r_j}(y_j))$ determines a neck region and we are done, so below we suppose~\eqref{eq:entire_annuli_no_concentration} does not hold. Then there exists some $c > 0$ so that, up to taking a subsequence, we can find, for each $j$, some $r_{1, j} \in [2r_j, 2^{-1}\delta_j]$ satisfying
\begin{equation}\label{eq:concentration_somewhere}
\mu_{\ep_j}(B_{2r_{1, j}}(y_j) \setminus B_{2^{-1}r_{1, j}}(y_j) \geqslant c.
\end{equation}
Then by property (t3) of transition regions (Definition~\ref{defi:transition_region}), we must have
\begin{equation}\label{eq:annuli_exhausts}
d_{1, j} := \min\{\log(\frac{\delta_j}{r_{1, j}}), \log(\frac{r_{1, j}}{r_j})\} \to \infty \text{ as } j \to \infty.
\end{equation}
Define $s_{1, j}$ to be the map $z \mapsto y_j + r_{1, j}z$ and let
\[
\ep_{1, j} = \frac{\ep_j}{r_{1, j}},\ \ (\nabla_{1, j}, \Phi_{1, j}) = s_{1, j}^*(\nabla_j, \Phi_j),\ \ g_{1, j} = r_{1, j}^{-2} s_{1, j}^* g.
\]
Then 
\[
\lim_{j \to \infty}\ep_{1, j} = \lim_{j \to \infty}r_{1, j} = 0,\ \ \ g_{1, j} \to g_{\RR^3} \text{ in }C^{\infty}_{\loc}(\RR^3 \setminus \{0\}).
\]
Moreover, $(\nabla_{1, j}, \Phi_{1, j})$ is a critical point of $\cY_{\ep_{1, j}}^{g_{1, j}}$ whose domain contains
\[
\cE_{0}([-d_{1, j}, d_{1, j}] \times S^2) 
= s_{1, j}^{-1}\big(B_{e^{d_{1, j}}r_{1, j}}(y_j) \setminus B_{e^{-d_{1, j}}r_{1, j}}(y_j) \big),
\]
which exhausts $\RR^3 \setminus \{0\}$ as $j \to \infty$. By Lemma~\ref{lem: scaling} and respectively~\eqref{eq:kappa_scaling}, we have
\begin{subequations}
\begin{align}
\ep_{1, j}^{-1}\cY_{\ep_{1, j}}^{g_{1, j}}(\nabla_{1, j}, \Phi_{1, j}; \ \cdot\ ) =\ & s_{1, j}^*\mu_{\ep_j}, \label{eq:rescale_measure_relation}\\
 \langle F_{\nabla_{1, j}} \wedge \nabla_{1, j}\Phi_{1, j}\rangle =\ & s_{1,j}^*(\langle F_{\nabla_j}\wedge \nabla_j\Phi_j\rangle). \label{eq:rescale_charge_relation}
\end{align}
\end{subequations}
From~\eqref{eq:rescale_measure_relation}, we deduce the following uniform energy upper bound:
\begin{equation}\label{eq:bubble_sequence_upper_bound}
\begin{split}
\ep_{1, j}^{-1}\cY_{\ep_{1, j}}^{g_{1, j}}(\nabla_{1, j}, \Phi_{1, j}; B_{d_{1,j}} \setminus B_{d_{1, j}^{-1}} ) \leqslant  \ &\ep_{j}^{-1}\cY_{\ep_j}^{g}(\nabla_j, \Phi_j; B_{\delta_j}(y_j) \setminus B_{r_j}(y_j))\\
=\ & \tau(z_0) + o(1),
\end{split}
\end{equation}
Recalling that $\ep_{1, j} \to 0$, and that $g_{1, j} \to g_{\RR^3}$  in $C^{\infty}_{\loc}(\RR^3 \setminus \{0\})$, we see from this uniform energy upper bound and the analysis in \S\ref{subsec:blow-up} leading to~\eqref{eq: convergence_of_energy_measures} that there exists a finite set $S_{1} \subset \RR^3 \setminus \{0\}$ and a $1$-form $h$ on $\RR^3 \setminus (\{0\} \cup S_1)$ such that 
\begin{equation}\label{eq:bubble_region_harmonic}
dh = 0,\ \ d^*h = 0 \text{ on }\RR^3 \setminus (\{0\} \cup S_1),\ \ \int_{\RR^3}|h|^2 \vol_{g_{\RR^3}} < \infty,
\end{equation}
and that, on compact subsets of $\RR^3 \setminus \{0\}$, up to taking a subsequence,
\[
s_{1, j}^*\mu_{\ep_j} \rightharpoonup |h|^2 \cH^3 + \sum_{x \in S_{1}}c_{x}\delta_x =: \mu_1,
\]
where, by the rescaling argument in \S\ref{subsec:rescaling} and particularly Proposition~\ref{prop:multiplicity_lower_bound}, each $c_x$ is at least $\theta_{\mathrm{gap}}\cdot\min\{\lambda,1\}$. Further, by Proposition~\ref{prop:assign_charge}, we can also arrange that 
\[
s_{1, j}^*\kappa_{\ep_j} \to \sum_{x \in S_1}k_x \delta_x =: \gamma_1,
\]
where $k_x \in 8\pi \ZZ$ for all $x \in S_1$. To continue, we infer from~\eqref{eq:bubble_sequence_upper_bound} that
\begin{equation}\label{eq:bubble_region_1_upper}
\mu_1(\RR^3\setminus \{0\}) \leqslant  \tau(z_0).
\end{equation}
On the other hand, thanks to~\eqref{eq:concentration_somewhere}, we also have 
\begin{equation}\label{eq:concentration_somewhere_scaled}
\mu_1(\RR^3 \setminus \{0\}) \geqslant \limsup_{j \to \infty}\ep_{1, j}^{-1}\cY_{\ep_{1, j}}^{g_{1, j}}(\nabla_{1, j}, \Phi_{1, j}; B_{2}(0) \setminus B_{2^{-1}}(0) ) \geqslant c > 0.
\end{equation}
Noting from~\eqref{eq:bubble_region_harmonic} that $h$ must vanish identically on $\RR^3$, and that consequently $S_{1}$ has to be non-empty by~\eqref{eq:concentration_somewhere_scaled}, we conclude that the limiting measures $\mu_1$ and $\gamma_1$ have the form required by (b4) of Definition~\ref{defi:bubble_region}. In particular, since $c_x \geqslant \theta_{\mathrm{gap}}\cdot\min\{\lambda,1\}$ for all $x \in S_1$, we obtain
\begin{equation}\label{eq:bubble_region_1_lower}
\mu_1(\RR^3 \setminus \{0\}) \geqslant \theta_{\mathrm{gap}}\cdot\min\{\lambda,1\}.
\end{equation}

Next we come to the selection of the intervals. In view of~\eqref{eq:annuli_exhausts}, the convergence of $s_{1, j}^*\mu_{\ep_j}$ and $s_{1, j}^*\kappa_{\ep_j}$ locally on $\RR^3\setminus \{0\}$, and the fact~\eqref{eq:kappa_setwise_limit} noted above, we can extract a subsequence $(j_l)$ of $(j)$ such that the following hold:
\begin{subequations}
\begin{align}
d_{1, j_l} \geqslant\ & l^3, \label{eq:subseq_far_from_boundary}\\
\Big|s_{1,j_l}^*\mu_{\ep_{j_l}}(\cE_0([-l, l] \times S^2)) -\mu_1(\cE_0([-l, l] \times S^2))\Big| \leqslant  \ &\frac{1}{l}, \label{eq:subseq_bubble_region}\\
\Big|s_{1,j_l}^*\mu_{\ep_{j_l}}(\cE_0(([-l^3, l^3] \setminus [-l, l]) \times S^2)) -\mu_1(\cE_0(([-l^3, l^3] \setminus [-l, l]) \times S^2))\Big| \leqslant \ & \frac{1}{l},\label{eq:subseq_negligible_region}\\
\Big|s_{1,j_l}^*\kappa_{\ep_{j_l}}(\cE_0([-l, l] \times S^2)) -\gamma_1(\cE_0([-l, l] \times S^2))\Big| \leqslant  \ &\frac{1}{l}.\label{eq:subseq_charge_bubble_region}
\end{align}
\end{subequations}
Writing $(j_l)$ simply as $(l)$ and recalling~\eqref{eq:dilation_to_translation}, we see that upon defining
\[
(I_1)_{l} = [-l^2 + \log r_{1, l}, l^2 + \log r_{1, l}],
\]
\[
(K_{1})_{l} = [-l^3 + \log r_{1, l}, l^3 + \log r_{1, l}],\ \  (I_1)_{l}' = [-l + \log r_{1, l}, l + \log r_{1, l}],
\]
there holds
\begin{subequations}
\begin{align}
\lim_{l \to \infty} \mu_{\ep_l}(\cE_{y_l}((I_1)_{l}'  \times S^2)) =\ & \mu_1(\RR^3 \setminus \{0\}), \label{eq:bubble_region_1}\\
\lim_{l \to \infty} \kappa_{\ep_l}(\cE_{y_l}((I_1)_{l}'  \times S^2)) =\ & \gamma_1(\RR^3 \setminus \{0\}), \label{eq:charge_bubble_region_1}\\
\lim_{l \to \infty}\mu_{\ep_{l}}(\cE_{y_l}(((K_1)_{l} \setminus (I_1)_{l}' ) \times S^2)) =\ & 0,\label{eq:negligible_region_1}
\end{align}
\end{subequations}
so that (b2) and (b3) in Definition~\ref{defi:bubble_region} are fulfilled. Since obviously 
\[
\lim_{l \to \infty} \diam (I_1)_l' = \lim_{l \to \infty} \dist((I_1)_l', \partial ((I_1)_l)) = \infty,
\]
we have verified property (iii), namely that $((I_1)_{l})$ determines a bubble region. From the definition of $d_{1, l}$ in~\eqref{eq:annuli_exhausts}, and the lower bound~\eqref{eq:subseq_far_from_boundary}, we get
\[
\dist((I_1)_{l}, \{\log r_l, \log \delta_l\}) \geqslant l^3 - l^2 \to \infty \text{ as }l \to \infty,
\]
so that property (i) holds as well. (Property (ii) holds vacuously at this point.)

With the above basic construction at hand, we define $\cD$ to be the set consisting of all $Q \in \NN$ for which there exists a subsequence of $(l)$ (not relabeled), along with sequences of intervals $(I_{1, l}), \cdots, (I_{Q, l})$ in $[\log r_l, \log \delta_l]$, such that properties (i), (ii) and (iii) are fulfilled. By the previous paragraph we see that 
\[
1 \in \cD,
\]
so that $\cD$ is non-empty. Moreover, $\cD$ is bounded from above. Indeed, for each $Q \in \cD$ and for each $l$, since $I'_{1, l}, \cdots, I'_{Q,l}$ are disjoint subsets of $[\log r_l, \log \delta_l]$, we have from properties (b3), (b4), and (t2) (see respectively Definitions~\ref{defi:bubble_region} and~\ref{defi:transition_region}) that 
\[
\begin{split}
Q \cdot \theta_{\mathrm{gap}}\cdot\min\{\lambda,1\} \leqslant \ & \sum_{\alpha = 1}^{Q} \lim_{l \to \infty}\mu_{\ep_l}(\cE_{y_l}(I'_{\alpha, l} \times S^2))\\
\leqslant \ & \limsup_{l \to \infty}\mu_{\ep_l}(B_{\delta_l}(y_l)\setminus B_{r_l}(y_l)) = \tau(z_0).
\end{split}
\]
In particular
\[
Q \leqslant  \frac{\tau(z_0)}{\theta_{\mathrm{gap}}\cdot\min\{\lambda,1\}}, \text{ for all }Q \in \cD.
\]
Since $\cD$ is a non-empty subset of $\NN$, this upper bound allows us to define $N:= \max \cD$. The fact that $N \in \cD$ yields a subsequence of $(l)$, which again we do not relabel, along with sequences of intervals $(I_{1, l}), \cdots, (I_{N, l})$, such that (i), (ii) and (iii) in the conclusion of Lemma~\ref{lemm:bubble_neck_decomposition} hold. We claim that the following sequence converges to $0$ as $l \to \infty$: 
\begin{equation}\label{eq:only_necks_leftover}
\sup\{\mu_{\ep_{l}}(\cE_{y_{l}}(I \times S^2))\ |\ I \subset [\log r_l, \log\delta_l] \setminus (\cup_{\alpha = 1}^N I_{\alpha, l}),\ \ \diam I = \log 4\}.
\end{equation}
Suppose not, then there exists some $b > 0$ so that, taking a further subsequence of $(l)$ if necessary, we get intervals 
\[
\widetilde{I}_{l} = [m_{l} - \log 2, m_{l} + \log 2] \subset [\log r_l, \log\delta_l] \setminus (\cup_{\alpha = 1}^N I_{\alpha, l}),
\]
such that 
\begin{equation}\label{eq:only_necks_leftover_contradiction}
\mu_{\ep_{l}}(\cE_{y_l}(\widetilde{I}_{l} \times S^2)) \geqslant b, \text{ for all }l.
\end{equation}
Then property (t3) (Definition~\ref{defi:transition_region}) forces
\begin{equation}\label{eq:concentration_poertion_far_from_ends}
\lim_{l \to \infty}\dist(m_{l}, \{\log r_l, \log\delta_l\}) = \infty.
\end{equation}
In particular
\[
\frac{e^{m_l}}{\ep_{l}} = \frac{e^{m_l}}{r_l} \cdot \frac{r_l}{\ep_l} \to \infty.
\]
On the other hand, we must also have for large enough $l$ that
\begin{equation}\label{eq:concentration_portion_far_from_bubble}
\widetilde{I}_{l} \cap \frac{3}{2}I_{\alpha, l} = \emptyset \text{ for all }\alpha \in \{1, \cdots, N\}.
\end{equation}
Indeed, if this were not the case, then since $\diam \widetilde{I}_l$ remains constant while
\[
\min_{\alpha \in \{1, \cdots, N\}}\diam I_{\alpha, l} \to \infty \text{ as }l \to \infty,
\]
and since $\widetilde{I}_l$ is disjoint from $\cup_{\alpha = 1}^N I_{\alpha, l}$ to begin with, along a subsequence we would have for each $l$ some $\alpha_l \in \{1, \cdots, N\}$ such that
\[
\widetilde{I}_{l} \subset 2I_{\alpha_l, l} \setminus I_{\alpha_l, l}.
\]
Combining this with~\eqref{eq:only_necks_leftover_contradiction} and (b2) gives
\[
\begin{split}
b \leqslant \ & \mu_{\ep_l}(\cE_{y_l}(\widetilde{I}_l \times S^2))\leqslant  \sum_{\alpha = 1}^{N} \mu_{\ep_l}(\cE_{y_l}((2I_{\alpha, l}\setminus I_{\alpha, l}) \times S^2)) \to 0 \text{ as }l \to \infty,
\end{split}
\]
a contradiction. Thus~\eqref{eq:concentration_portion_far_from_bubble} must hold, which implies that eventually
\[
\dist(\widetilde{I}_l, I_{\alpha, l}) \geqslant \frac{1}{8}\diam I_{\alpha, l}, \text{ for all }\alpha = 1, \cdots, N.
\]
Recalling~\eqref{eq:concentration_poertion_far_from_ends}, we get
\begin{equation}\label{eq:neck_contradiction_exhaustion}
d_{l} : = \dist(m_l, (\cup_{\alpha = 1}^N I_{\alpha, l}) \cup \{\log r_l, \log \delta_l\}) \to \infty \text{ as }l \to \infty.
\end{equation}
Next, in analogy with the steps we took after~\eqref{eq:annuli_exhausts}, we define 
\[
\widehat{\ep_l} = e^{-m_l}\ep_l,\ \ s_{l} = s_{y_l, e^{m_l}},
\]
and also let
\[
(\widehat{\nabla}_l, \widehat{\Phi}_l) = s_l^*(\nabla_{l}, \Phi_{l}),\ \ \widehat{g}_l = e^{-2m_l}s_{l}^*g,
\]
to obtain a sequence of critical points of $\cY_{\widehat{\ep_l}}^{\widehat{g}_l}$ defined at least on $\cE_0([-d_l, d_l] \times S^2)$. To obtain energy upper and lower bounds, notice that given $T > 0$, by~\eqref{eq:neck_contradiction_exhaustion}, we have eventually that
\[
[m_l - T, m_l + T] \subset [\log r_l, \log \delta_l] \setminus (\cup_{\alpha = 1}^N I_{\alpha, l}).
\]
Hence by property (b3) of bubble regions,
\[
\begin{split}
(\widehat{\ep}_l)^{-1}\cY_{\widehat{\ep}_l}^{\widehat{g}_l}(\widehat{\nabla}_l, \widehat{\Phi}_l; B_{e^T}(0) \setminus B_{e^{-T}}(0)) 
=\ & \mu_{\ep_l}(\cE_{y_l}([m_l - T, m_l + T] \times S^2)) \\
\leqslant \ & \mu_{\ep_l}(B_{\delta_l}(y_l) \setminus B_{r_l}(y_l)) - \sum_{\alpha = 1}^{N} \mu_{\ep_l}(\cE_{y_l}(I_{\alpha, l}' \times S^2))\\
\leqslant \ & \tau(z_0) - N\cdot \theta_{\mathrm{gap}}\cdot\min\{\lambda,1\} + o(1).
\end{split}
\]
On the other hand,~\eqref{eq:only_necks_leftover_contradiction} implies
\[
(\widehat{\ep}_l)^{-1}\cY_{\widehat{\ep}_l}^{\widehat{g}_l}(\widehat{\nabla}_l, \widehat{\Phi}_l; B_{2}(0) \setminus B_{2^{-1}}(0)) \geqslant b > 0.
\]
Using these bounds, and recalling that $\widehat{\ep}_l \to 0$ whereas $\widehat{g}_l \to g_{\RR^3}$ smoothly locally on $\RR^3 \setminus \{0\}$, we may apply exactly the same argument by which we treated the sequence $(\nabla_{1, j}, \Phi_{1, j})$ earlier in this proof, and get a subsequence of $(l)$, a non-empty finite subset $\widehat{S} \subset \RR^3\setminus \{0\}$, and for each $x \in \widehat{S}$ some $\widehat{c}_x \geqslant \theta_{\mathrm{gap}}\cdot\min\{\lambda,1\}$ and $\widehat{k}_x \in 8\pi \ZZ$, such that, locally on $\RR^3 \setminus \{0\}$,
\[
s_{l}^*\mu_{\ep_l} \rightharpoonup \sum_{x \in \widehat{S}} \widehat{c}_x \delta_x : = \widehat{\mu},\ \ \ s_{l}^*\kappa_{\ep_l} \rightharpoonup \sum_{x \in \widehat{S}} \widehat{k}_x \delta_x : = \widehat{\gamma}.
\]
In particular $\widehat{\mu}(\RR^3 \setminus \{0\}) \in [\theta_{\mathrm{gap}}\cdot\min\{\lambda,1\}, \infty)$. The reasoning leading to the definition of $(K_1)_{l}, (I_1)_{l}$ and $(I_1)_{l}'$ can then be repeated to yield a further subsequence of $(l)$ and intervals $I_{l}' \subset I_{l} \subset 2I_{l} \subset K_l$, all centered at $m_l$, such that the distance defined in~\eqref{eq:neck_contradiction_exhaustion} satisfies $d_l \geqslant l^3$, that 
\[
\diam I_l' = 2l, \ \ \dist(I_l', \partial I_l) = l^2 - l,\ \ \dist(I_l, (\cup_{\alpha = 1}^N I_{\alpha, l}) \cup \{\log r_l, \log \delta_l\}) \geqslant l^3 - l^2,
\]
and that 
\[
\begin{split}
\lim_{l \to \infty}\mu_{\ep_l}(\cE_{y_l}(I_l' \times S^2))  =\ & \widehat{\mu}(\RR^3\setminus \{0\}), \ \ \ \lim_{l \to \infty}\mu_{\ep_l}(\cE_{y_l}((K_l \setminus I_l') \times S^2))  = 0\\
\lim_{l \to \infty}\kappa_{\ep_l}(\cE_{y_l}(I_l' \times S^2))  =\ & \widehat{\gamma}(\RR^3\setminus \{0\}).
\end{split}
\]
Consequently $(I_{l})$ determines a bubble region, and moreover (i), (ii) are still satisfied after enlarging the collection $(I_{1, l}), \cdots, (I_{N, l})$ to include $(I_l)$. However this implies $N + 1 \in \cD$, which contradicts the maximality of $N$. We conclude that~\eqref{eq:only_necks_leftover} must hold.

Define $J_{0, l}, \cdots, J_{N, l}$ to be the components of $[\log r_l, \log\delta_l] \setminus \big(\cup_{\alpha = 1}^N \Inte(I_{\alpha, l})\big)$. By properties (i) and (ii) we have 
\[
\lim_{l \to \infty}J_{\beta, l} = \infty, \text{ for all }\beta.
\]
Moreover,~\eqref{eq:only_necks_leftover} shows that property (n2) from Definition~\ref{defi:neck_region} holds for $(J_{\beta, l})_{l = 1}^{\infty}$ for each $\beta \in \{0, \cdots, L\}$. Since (n2) implies (n1) as noted in that definition, we have shown that each sequence $(J_{\beta, l})_{l = 1}^{\infty}$ determines a neck region. The proof is complete.
\end{proof}
\subsection{Energy identity}\label{subsec:energy_identity}
In this section we finish the proof of Theorem~\ref{thm: bubbling} by showing that $\Theta(z_0)$ and $\Xi(z_0)$ are respectively equal to the sum of the energy and charge of a finite collection of non-trivial, finite-action critical points of $\cY_1$ on $\RR^3$. 

Returning to the situation right after we used Lemma~\ref{lemm:neck_analysis_transition_region} to produce a sequence $(B_{\delta_j}(y_j) \setminus B_{r_j}(y_j))$ of annuli that determines a transition region, if the first alternative in Lemma~\ref{lemm:bubble_neck_decomposition} prevails, that is, if $(B_{\delta_j}(y_j) \setminus B_{r_j}(y_j))$ already determines a neck region, then we simply use (t2) from Definition~\ref{defi:transition_region} to write 
\begin{equation}\label{eq:multiplicity_no_more_bubbles}
\big|\Theta(z_0) -  \cY_1(\nabla, \Phi; \RR^3) \big| \leqslant  \limsup_{j \to \infty}\mu_{\ep_j}(B_{\delta_j}(y_j) \setminus B_{r_j}(y_j)),
\end{equation}
and, also using~\eqref{eq:charge_by_energy},
\begin{equation}\label{eq:charge_no_more_bubbles}
\begin{split}
\big|\Xi(z_0) -  \cK(\nabla, \Phi) \big| \leqslant \ & \limsup_{j \to \infty}|\kappa_{\ep_j}(B_{\delta_j}(y_j) \setminus B_{r_j}(y_j))| \\
\leqslant \ & \limsup_{j \to \infty}\mu_{\ep_j}(B_{\delta_j}(y_j) \setminus B_{r_j}(y_j)).
\end{split}
\end{equation}
On the other hand, if the second alternative in Lemma~\ref{lemm:bubble_neck_decomposition} holds, then we write $(l)$ and $(I_{\alpha, l}), (J_{\beta, l})$, respectively, for the subsequence of $(j)$ and the intervals produced as a result. For later use, we recapitulate the conclusions of Lemma~\ref{lemm:bubble_neck_decomposition} in more detail than is immediately needed. Recalling that each $(I_{\alpha, l})_{l \in \NN}$ determines a bubble region, we denote by $(I_{\alpha, l}')$ and $(m_{\alpha, l})$ the subintervals and common midpoints in Definition~\ref{defi:bubble_region}, and also let 
\[
d_{\alpha, l} = \frac{1}{2}\diam I_{\alpha, l}.
\]
The limiting measures described in (b3) we write as $\nu_{\alpha}$ and $\gamma_{\alpha}$, while the set $S(\nu_{\alpha})$ and the weights $\{c(x)\}$ and $\{k(x)\}$ in (b4) we denote by $S_{\alpha}$, $\{\Theta_{\alpha}(x)\}$, and $\{\Xi_{\alpha}(x)\}$, respectively. Then, for each $\alpha \in \{1, \cdots, N\}$, letting 
\[
\ep_{\alpha, l} = \frac{\ep_l}{e^{m_{\alpha, l}}},\ \ s_{\alpha, l} = s_{y_l, e^{m_{\alpha, l}}},
\]
and also performing, as has already been done many times, the following rescaling:
\[
(\nabla_{\alpha, l}, \Phi_{\alpha, l}) = s_{\alpha, l}^*(\nabla_l, \Phi_l),\ \ g_{\alpha, l} = e^{-2m_{\alpha, l}}s_{\alpha, l}^* g,
\]
we see that, first of all,
\begin{equation}\label{eq:convergence_of_ep_and_metric}
\lim_{l \to \infty}\ep_{\alpha, l} = 0,\quad \text{ and } \ g_{\alpha, l} \to g \text{ in }C^{\infty}_{\loc}(\RR^3 \setminus \{0\}) \text{ as }l \to \infty.
\end{equation}
Secondly, each $(\nabla_{\alpha, l}, \Phi_{\alpha, l})$ is a critical point of $\cY_{\ep_{\alpha, l}}^{g_{\alpha, l}}$ defined at least on $B_{e^{d_{\alpha, l}}}(0) \setminus B_{e^{-d_{\alpha, l}}}(0)$, with energy bounded from above in the following way:
\begin{equation}\label{eq:energy_identity_bubble_seq_energy}
\begin{split}
(\ep_{\alpha, l})^{-1}\cY_{\ep_{\alpha, l}}^{g_{\alpha, l}}(\nabla_{\alpha, l}, \Phi_{\alpha, l}; B_{e^{d_{\alpha, l}}}(0) \setminus B_{e^{-d_{\alpha, l}}}(0) ) 
\leqslant \ & \tau(z_0) + o(1).
\end{split}
\end{equation}
Thirdly, the set $S_{\alpha}$ and multiplicities $\{\Theta_{\alpha}(x)\}$, $\{\Xi_{\alpha}(x)\}$ arise through the following convergences of measures locally on $\RR^3 \setminus \{0\}$
\begin{equation}\label{eq:group_alpha_bubbles}
s_{\alpha,l}^*\mu_{\ep_l} \rightharpoonup \sum_{x \in S_{\alpha}}\Theta_{\alpha}(x)\delta_{x}, \ \ \ s_{\alpha,l}^*\kappa_{\ep_l} \rightharpoonup \sum_{x \in S_{\alpha}}\Xi_{\alpha}(x)\delta_{x}.
\end{equation}
On the other hand, concerning the energy of the original sequence $(\nabla_l, \Phi_l)$ on the entire transition region, we have 
\[
\begin{split}
\mu_{\ep_l}(B_{\delta_l}(y_l) \setminus B_{r_l}(y_l))  =\ & \sum_{\beta = 0}^{N} \mu_{\ep_l}(\cE_{y_l}(J_{\beta, l} \times S^2)) + \sum_{\alpha = 1}^{N} \mu_{\ep_l}(\cE_{y_l}(I_{\alpha, l} \times S^2)) \\
=\ & \sum_{\beta = 0}^{N} \mu_{\ep_l}(\cE_{y_l}(J_{\beta, l} \times S^2)) + \sum_{\alpha = 1}^{N}\sum_{x \in S_{\alpha}}\Theta_{\alpha}(x) + o(1).
\end{split}
\]
In completely analogous fashion, and using also~\eqref{eq:charge_by_energy}, we have for the measures $\kappa_{\ep_l}$ that 
\[
\begin{split}
\kappa_{\ep_l}(B_{\delta_l}(y_l) \setminus B_{r_l}(y_l))  =\ & \sum_{\beta = 0}^{N} \kappa_{\ep_l}(\cE_{y_l}(J_{\beta, l} \times S^2)) + \sum_{\alpha = 1}^{N} \kappa_{\ep_l}(\cE_{y_l}(I_{\alpha, l} \times S^2)) \\
=\ & \sum_{\beta = 0}^{N} \kappa_{\ep_l}(\cE_{y_l}(J_{\beta, l} \times S^2)) + \sum_{\alpha = 1}^{N}\sum_{x \in S_{\alpha}}\Xi_{\alpha}(x) + o(1).
\end{split}
\]
Recalling property (t2) of transition regions (Definition~\ref{defi:transition_region}) and again using~\eqref{eq:charge_by_energy} gives
\begin{subequations}
\begin{align}
\big|\Theta(z_0) - \big( \cY_1(\nabla, \Phi; \RR^3) + \sum_{\alpha = 1}^N \sum_{x \in S_{\alpha}}\Theta_{\alpha}(x)\big)\big| \leqslant \ & \limsup_{l \to \infty} \sum_{\beta = 0}^{N} \mu_{\ep_l}(\cE_{y_l}(J_{\beta, l} \times S^2)), \label{eq:Theta_decomposition_up_to_neck}\\
\big|\Xi(z_0) - \big( \cK(\nabla, \Phi) + \sum_{\alpha = 1}^N \sum_{x \in S_{\alpha}}\Xi_{\alpha}(x)\big)\big| \leqslant \ & \limsup_{l \to \infty} \sum_{\beta = 0}^{N} \mu_{\ep_l}(\cE_{y_l}(J_{\beta, l} \times S^2)).\label{eq:Xi_decomposition_up_to_neck}
\end{align}
\end{subequations}
Our next task is to show that neck regions carry no energy, so to speak. More specifically, we have the following result.
\begin{prop}\label{prop:no_neck}
Let $(J_l = [a_l, b_l])$ be a sequence of intervals such that $J_l \subset [\log r_l, \log \delta_l]$, that $\lim_{l \to \infty}\diam J_l = \infty$, and that $(J_l)$ determines a neck region in the sense of Definition~\ref{defi:neck_region}. Then in fact
\begin{equation}\label{eq:no_neck}
\lim_{l \to \infty} \mu_{\ep_l}(\cE_{y_l}(J_l \times S^2)) = 0.
\end{equation}
\end{prop}
Before we dive into the proof of this result, we need an auxiliary variational identity, which is a local version of~\cite[Corollary II.2.2]{Jaffe-Taubes}. Below, given $(\nabla, \Phi) \in \mathscr{A}(E) \times \Gamma(\fsu(E))$, we follow~\cite{Jaffe-Taubes} and define the \emph{stress-energy tensor} $T=T_{\varepsilon}(\nabla,\Phi)\in \Gamma(S^2 T^\ast M)$ as the symmetric $2$-tensor given by
\begin{equation*}
    T(v_1,v_2):=2\varepsilon^2\langle v_1\lrcorner F_{\nabla}, v_2\lrcorner F_{\nabla}\rangle + 2\langle \nabla_{v_1}\Phi, \nabla_{v_2}\Phi\rangle - g(v_1,v_2)e_{\varepsilon}(\nabla,\Phi).
\end{equation*} 
Notice that when $(\nabla,\Phi)$ is a solution to the second order equations \eqref{eq: 2nd_order_crit_pt_intro}, a quick computation, which in fact motivates the definition just given, shows that the divergence of the $2$-tensor $T$ vanishes:
\begin{equation*}
    D^{\ast}T = 0,
\end{equation*} 
where $D$ denotes the Levi--Civita connection of $g$.  
\begin{lemm}[Local conservation law]\label{lem:conservation_law}
For any smooth solution $(\nabla, \Phi)$ to the second order equations \eqref{eq: 2nd_order_crit_pt_intro} on a $3$-manifold with bounded geometry, and any geodesic ball $B_r(x)$, we have
\begin{align}\label{eq:conservation_law}
r\int_{\partial B_r(x)}\left(2\varepsilon^2 |\partial_r\lrcorner F_{\nabla}|^2 + 2|\nabla_{\partial_r}\Phi|^2 - e_{\varepsilon}(\nabla,\Phi)\right) &= \int_{B_r(x)} \langle T,\mathrm{Hess}(\frac{1}{2}d(\cdot,x)^2) - g\rangle\\
&\quad + \int_{B_r(x)} \left(\varepsilon^2|F_{\nabla}|^2 - |\nabla\Phi|^2 - \frac{3\lambda w^2}{\ep^2}\right).\nonumber
\end{align}
\end{lemm}
\begin{proof} 
We note right away that
\begin{align}\label{eq: trace_stress_tensor}
    \mathrm{tr}_g(T) = \langle T, g\rangle &= 4\varepsilon^2|F_{\nabla}|^2 + 2|\nabla\Phi|^2 - 3e_{\varepsilon}(\nabla,\Phi)\nonumber\\
    &= \ep^2|F_{\nabla}|^2 -|\nabla\Phi|^2 - \frac{3\lambda w^2}{\ep^2}. 
\end{align} Now, since $(\nabla,\Phi)$ is a solution to the second order equations \eqref{eq: 2nd_order_crit_pt_intro}, by the identity $D^{\ast}T = 0$ mentioned above and the divergence theorem, if $\Omega\subseteq M$ is any precompact open set with smooth boundary $\partial\Omega$, oriented with the outward unit normal vector $\nu$, and if $X$ is a smooth vector field on $\overline{\Omega}$, then 
\[
\int_{\partial\Omega} T(X,\nu) = \int_{\Omega}\langle T,DX\rangle.
\] We use this equation with $\Omega$ equal to the geodesic ball $B_r(x)$ and with $X:=D(\frac{1}{2}r_x^2)$, where in normal coordinates $(x^1,\ldots,x^n)$ on $B_r(x)$, centred at $x$, we let $r_x:=(\sum_i (x^i)^2)^{1/2}=d(\cdot,x)$ be the distance function, whose gradient $Dr_x$ equals the unit radial vector field $\partial_r$, which in turn is the unit outward normal vector $\nu$ in this case, so we get 
\begin{equation*}
    r\int_{\partial B_r(x)}\left(2\varepsilon^2 |\partial_r\lrcorner F_{\nabla}|^2 + 2|\nabla_{\partial_r}\Phi|^2 - e_{\varepsilon}(\nabla,\Phi)\right) = \int_{B_r(x)} \langle T,\mathrm{Hess}(\frac{1}{2}r_x^2) - g\rangle + \int_{B_r(x)} \langle T,g\rangle.
\end{equation*} This together with \eqref{eq: trace_stress_tensor} immediately gives \eqref{eq:conservation_law}.
\end{proof}
\begin{proof}[Proof of Proposition \ref{prop:no_neck}]
Let $\sigma_l = e^{a_l}$ and $\rho_l = e^{b_l}$. Then since $\diam J_l \to \infty$ and $\frac{r_l}{\ep_l} \to \infty$, we have 
\begin{equation}\label{eq:scales}
\lim_{l \to \infty} \frac{\rho_l}{\sigma_l} = \lim_{l \to \infty} \frac{\sigma_l}{\ep_{l}} = \infty.
\end{equation}
Also, by (n1) in Definition~\ref{defi:neck_region} we have 
\begin{equation}\label{eq:no_energy_at_boundaries}
\lim_{l \to \infty} \mu_{\ep_l}(B_{\rho_l}(y_l) \setminus B_{\frac{\rho_l}{4}}(y_l)) = \lim_{l \to \infty} \mu_{\ep_l}(B_{4\sigma_l}(y_l) \setminus B_{\sigma_l}(y_l)) = 0.
\end{equation}
Then by combining the fact that $\delta_l \to 0$, the comparison~\eqref{eq:scales} between the scales, and property (n2) in Definition~\ref{defi:neck_region}, we see that for sufficiently large $l$ there holds
\[
\delta_l < r_0, \ \ \ep_l < \frac{\tau_\ast \sigma_l}{4}, 
\]
and that 
\[
\sup_{2\sigma_l \leqslant  r \leqslant  2^{-1}\rho_l} \mu_{\ep_l}(B_{2r}(y_l) \setminus B_{2^{-1}r}(y_l)) < \eta_{\ast}.
\]
For each such $l$, given $x \in B_{\frac{\rho_l}{2}}(y_l) \setminus B_{2\sigma_l}(y_l)$, we let $r = d(x, y_l)$ and observe that $\ep_l < \frac{\tau_\ast r}{8}$, and that, by the triangle inequality, 
\[
B_{\frac{r}{2}}(x) \subset B_{2r}(y_l) \setminus B_{\frac{r}{2}}(y_l).
\]
By what we have just arranged, this implies that eventually
\[
\ep_l^{-1}\cY_{\ep_l}(\nabla_{l}, \Phi_l; B_{\frac{r}{2}}(x)) < \eta_{\ast},
\]
and we may invoke Remark~\ref{rmk:clearing_out}(i) and Lemma~\ref{lemm:nablaPhi_exp_decay_base} (with $\rho = \frac{r}{8}$) to get that 
\[
\frac{\lambda(1 - |\Phi_l(x)|^2)^2}{\ep_l^2} + |\nabla_l \Phi_l(x)|^2 \leqslant  C_{\lambda}\ep_l^{-2}e^{-a_{\lambda}\frac{d(x, y_l)}{\ep_l}}.
\]
Introducing normal coordinates centered at $y_l$, which satisfy the bounds~\eqref{eq:metric_bounds_for_estimates} on $B_{\delta_l}(y_l)$ since $\delta_l < r_0$, we integrate the above estimate over $ B_{\frac{\rho_l}{2}}(y_l) \setminus B_{2\sigma_l}(y_l)$ to get
\[
\begin{split}
&\ep_l^{-1}\int_{ B_{\frac{\rho_l}{2}}(y_l) \setminus B_{2\sigma_l}(y_l)} \frac{\lambda(1 - |\Phi_l(x)|^2)^2}{\ep_l^2} + |\nabla_l \Phi_l(x)|^2  \vol_g\\
& \leqslant  C_{\lambda} \ep_l^{-1}\int_{2\sigma_l}^{\frac{\rho_l}{2}} r^2 \ep_l^{-2}e^{-a_{\lambda}\frac{r}{\ep_l}}dr = C_{\lambda}\int_{2\frac{\sigma_l}{\ep_l}}^{\frac{\rho_l}{2\ep_l}} s^2 e^{-a_{\lambda}s} ds.
\end{split}
\]
Since $s^2 e^{-a_{\lambda }s}$ is integrable on $[0, \infty)$ and since $\lim_{l \to \infty}\frac{\sigma_l}{\ep_l} = \infty$, we deduce that 
\begin{equation}\label{eq:good_part_of_energy_on_neck}
\ep_l^{-1}\int_{ B_{\frac{\rho_l}{2}}(y_l) \setminus B_{2\sigma_l}(y_l)} \frac{\lambda(1 - |\Phi_l(x)|^2)^2}{\ep_l^2} + |\nabla_l \Phi_l(x)|^2  \vol_g \to 0 \text{ as }l \to \infty.
\end{equation}
It remains to control the curvature term, for which we employ the conservation law~\eqref{eq:conservation_law} with center at $y_l$, which gives for all $r \leqslant  \rho_l$ that 
\begin{equation}\label{eq:conservation_law_for_neck}
\begin{split}
&r\int_{\partial B_r(y_l)}\left(2\ep_l^2 |\partial_r\lrcorner F_{\nabla_l}|^2 + 2|(\nabla_l)_{\partial_r}\Phi_l|^2 - e_{\ep_l}(\nabla_l,\Phi_l)\right)\\
=\ &  \int_{B_r(y_l)} \langle T_l,\mathrm{Hess}(\frac{d(\cdot, y_l)^2}{2}) - g\rangle+ \int_{B_r(y_l)} \ep_l^2 |F_{\nabla_l}|^2 - |\nabla_l\Phi_l|^2 - \frac{3\lambda(1 - |\Phi_l|^2)^2}{4\ep_l^2},
\end{split}
\end{equation} Now, by Fubini's theorem we get for each $l$ some $\rho_l' \in [\frac{\rho_l}{4}, \frac{\rho_l}{2}]$ and $\sigma_l' \in [2\sigma_l, 4\sigma_l]$ such that 
\begin{equation}\label{eq:Fubini_for_neck}
\begin{split}
\rho_l' \int_{\partial B_{\rho_l'}(y_l)} e_{\ep_l}(\nabla_l, \Phi_l) \leqslant  \ & C\int_{B_{\frac{\rho_l}{2}}(y_l) \setminus B_{\frac{\rho_l}{4}}(y_l)} e_{\ep_l}(\nabla_l, \Phi_l),\\
\sigma_l' \int_{\partial B_{\sigma_l'}(y_l)} e_{\ep_l}(\nabla_l, \Phi_l) \leqslant \ & C\int_{B_{4\sigma_l}(y_l) \setminus B_{2\sigma_l}(y_l)} e_{\ep_l}(\nabla_l, \Phi_l).
\end{split}
\end{equation}
Applying~\eqref{eq:conservation_law_for_neck} at $r = \rho_l'$ and $r = \sigma_l'$ and taking the difference, we get
\begin{small}
\[
\begin{split}
\int_{B_{\rho'_l}(y_l) \setminus B_{\sigma'_l}(y_l)} \ep_l^2 |F_{\nabla_l}|^2 =\ & \int_{B_{\rho'_l}(y_l) \setminus B_{\sigma'_l}(y_l)} |\nabla_l\Phi_l|^2 + \frac{3\lambda(1 - |\Phi_l|^2)^2}{4\ep_l^2}\\
& -  \int_{B_{\rho'_l}(y_l) \setminus B_{\sigma'_l}(y_l)}\bangle{T_l, \mathrm{Hess}(\frac{d(\cdot, y_l)^2}{2}) - g} \\
& + \Big(\rho_l' \int_{\partial B_{\rho_l'}(y_l)} -  \sigma_l' \int_{\partial B_{\sigma_l'}(y_l)}\Big)\left(2\ep_l^2 |\partial_r\lrcorner F_{\nabla_l}|^2 + 2|(\nabla_l)_{\partial_r}\Phi_l|^2 - e_{\ep_l}(\nabla_l,\Phi_l)\right).
\end{split}
\]
\end{small}
With the help of~\eqref{eq:Hessian_comparison_for_estimates},~\eqref{eq:Fubini_for_neck}, and the fact that $|\partial_r| = 1$, we infer that
\[
\begin{split}
\int_{B_{\rho'_l}(y_l) \setminus B_{\sigma'_l}(y_l)} \ep_l^2 |F_{\nabla_l}|^2 \leqslant \ & C\int_{B_{\frac{\rho_l}{2}}(y_l) \setminus B_{\frac{\rho_l}{4}}(y_l)} e_{\ep_l}(\nabla_l, \Phi_l) +C\int_{B_{4\sigma_l}(y_l) \setminus B_{2\sigma_l}(y_l)} e_{\ep_l}(\nabla_l, \Phi_l)\\
& + C\rho_{l}^2 \int_{B_{\rho'_l}(y_l) \setminus B_{\sigma'_l}(y_l)} e_{\ep_l}(\nabla_l, \Phi_l) \\
& + C\int_{B_{\frac{\rho_l}{2}}(y_l)\setminus B_{2\sigma_l}(y_l)} |\nabla_l\Phi_l|^2 + \frac{\lambda(1 -|\Phi_l|^2)^2}{\ep_l^2}.
\end{split}
\]
Thus, for sufficiently large $l$, the curvature term involved in the second line of the above estimate can be absorbed to the left-hand side, while the covariant derivative term and potential term can be combined with the third line, and we obtain
\begin{small}
\[
\begin{split}
\ep_l^{-1}\int_{B_{\frac{\rho_l}{4}}(y_l)\setminus B_{4\sigma_l}(y_l)}\ep_l^2 |F_{\nabla_l}|^2 \leqslant \ &  C\ep_l^{-1}\int_{B_{\frac{\rho_l}{2}}(y_l) \setminus B_{\frac{\rho_l}{4}}(y_l)} e_{\ep_l}(\nabla_l, \Phi_l) + C\ep_l^{-1}\int_{B_{4\sigma_l}(y_l) \setminus B_{2\sigma_l}(y_l)} e_{\ep_l}(\nabla_l, \Phi_l)\\
& +  C\ep_l^{-1}\int_{B_{\frac{\rho_l}{2}}(y_l)\setminus B_{2\sigma_l}(y_l)} |\nabla_l\Phi_l|^2 + \frac{\lambda(1 -|\Phi_l|^2)^2}{\ep_l^2}.
\end{split}
\]
\end{small}
Recalling~\eqref{eq:no_energy_at_boundaries} and~\eqref{eq:good_part_of_energy_on_neck}, we finally get
\[
\lim_{l \to \infty}\ep_l^{-1}\int_{B_{\frac{\rho_l}{4}}(y_l)\setminus B_{4\sigma_l}(y_l)}\ep_l^2 |F_{\nabla_l}|^2 = 0.
\]
Combining this with~\eqref{eq:good_part_of_energy_on_neck} and~\eqref{eq:no_energy_at_boundaries} again, we conclude that 
\[
\lim_{l \to \infty} \mu_{\ep_l}(B_{\rho_l}(y_l) \setminus B_{\sigma_l}(y_l)) = 0,
\]
which is exactly~\eqref{eq:no_neck}.
\end{proof}
We conclude the proof of Theorem~\ref{thm: bubbling} with the following result.
\begin{prop}\label{prop:energy_identity}
There exists a finite collection of non-trivial, finite-action critical points of $\cY_1^{g_{\RR^3}}$ on $\RR^3$ whose energy and charge sum up to $\Theta(z_0)$ and $\Xi(z_0)$, respectively.
\end{prop}
\begin{proof}
Thanks to Proposition~\ref{prop:no_neck}, in the case where $(B_{\delta_j}(y_j) \setminus B_{r_j}(y_j))$ already determines a neck region and no more bubbles occur, we get from~\eqref{eq:multiplicity_no_more_bubbles} and~\eqref{eq:charge_no_more_bubbles} that
\begin{equation}\label{eq:energy_identity_no_more_bubbles}
\Theta(z_0) = \cY_1(\nabla, \Phi; \RR^3),\ \ \Xi(z_0) = \cK(\nabla, \Phi),
\end{equation}
and we are done. On the other hand, if further bubble regions exist, then, writing $S_{\alpha}$ as $\{x_{\alpha\beta}\}_{\beta = 1}^{N_{\alpha}}$, we get from~\eqref{eq:Theta_decomposition_up_to_neck},~\eqref{eq:Xi_decomposition_up_to_neck}, and Proposition~\ref{prop:no_neck} that
\begin{subequations}
\begin{align}
\Theta(z_0) =\ & \cY_1(\nabla, \Phi) + \sum_{\alpha = 1}^N \sum_{\beta = 1}^{N_{\alpha}}\Theta_{\alpha}(x_{\alpha\beta}), \label{eq:Theta_decomposition}\\
\Xi(z_0) =\ & \cK(\nabla, \Phi) + \sum_{\alpha = 1}^N \sum_{\beta = 1}^{N_{\alpha}}\Xi_{\alpha}(x_{\alpha\beta}).\label{eq:Xi_decomposition}
\end{align}
\end{subequations}
Notice that 
\[
\theta_{\mathrm{gap}}\cdot\min\{\lambda,1\} \leqslant  \Theta_{\alpha}(x_{\alpha\beta}) \leqslant  \Theta(z_0) - \theta_{\mathrm{gap}}\cdot\min\{\lambda,1\},
\]
for all $\alpha \in \{1, \cdots, N\}$ and $\beta \in \{{1, \cdots, N_{\alpha}\}}$. Now fix a particular $\alpha$ and recall the convergence of measures~\eqref{eq:group_alpha_bubbles}. By the convergence of $\ep_{\alpha, l}$ and $g_{\alpha, l}$ noted in~\eqref{eq:convergence_of_ep_and_metric} and the uniform energy upper bound~\eqref{eq:energy_identity_bubble_seq_energy}, the rescaled sequence $(\nabla_{\alpha,l}, \Phi_{\alpha,l})$ enjoys the same type of a priori estimates satisfied by the original sequence. The idea is now to repeat the argument leading to~\eqref{eq:Theta_decomposition} for each $x_{\alpha\beta}$ ($\beta = 1, \cdots, N_{\alpha}$). We only provide a very rough sketch and omit the details to avoid repetition. 

First, arguing as in Section~\ref{subsec:rescaling}, we see that after rescaling $(\nabla_{\alpha, l}, \Phi_{\alpha, l})$ by $\ep_{\alpha, l}$ at suitably chosen centers $y_{\alpha\beta, l}$ that tend to $x_{\alpha\beta}$, we obtain in the subsequential limit a  critical point of $\cY_1$ on $\RR^3$, denoted $(\nabla_{\alpha\beta}, \Phi_{\alpha\beta})$, which satisfies
\[
\theta_{\mathrm{gap}}\cdot\min\{\lambda,1\} \leqslant  \cY_1^{g_{\RR^3}}(\nabla_{\alpha\beta}, \Phi_{\alpha\beta}; \RR^3) \leqslant  \Theta(x_{\alpha\beta}).
\]
We assign $(\nabla_{\alpha\beta}, \Phi_{\alpha\beta})$ to $x_{\alpha\beta}$ as the top bubble at that point, and define
\[
\tau(x_{\alpha\beta}) = \Theta(x_{\alpha\beta}) - \cY_1^{g_{\RR^3}}(\nabla_{\alpha\beta}, \Phi_{\alpha\beta}; \RR^3),\ \ \ \tau_{\text{charge}}(x_{\alpha\beta}) = \Xi(x_{\alpha\beta}) - \cK(\nabla_{\alpha\beta}, \Phi_{\alpha\beta}).
\]
Notice then that 
\begin{equation}\label{eq:energy_loss_drop_twice}
0 \leqslant  \tau(x_{\alpha\beta}) \leqslant  \Theta(x_{\alpha\beta}) - \theta_{\mathrm{gap}}\cdot\min\{\lambda,1\} \leqslant  \Theta(z_0) - 2\theta_{\mathrm{gap}}\cdot\min\{\lambda,1\},
\end{equation}
where the upper bound has dropped by $\theta_{\mathrm{gap}}\cdot\min\{\lambda,1\}$ compared to~\eqref{eq:tau_definite_drop}. 

Next, starting with the convergence of measures~\eqref{eq:group_alpha_bubbles}, and the local smooth convergence that gives rise to $(\nabla_{\alpha\beta}, \Phi_{\alpha\beta})$, by repeating the proof of Lemma~\ref{lemm:neck_analysis_transition_region} we obtain a sequence of annuli centered at $y_{\alpha\beta, l}$ that determines a transition region in the sense of Definition~\ref{defi:transition_region}, with $\tau(x_{\alpha\beta})$ and $\tau_{\text{charge}}(x_{\alpha\beta})$ in place of $\tau(z_0)$ and $\tau_{\text{charge}}(z_0)$. Following the proofs of Lemma~\ref{lemm:bubble_neck_decomposition} and Proposition~\ref{prop:no_neck}, we can partition the annuli into subannuli that determine respectively bubble regions and neck regions, and show that no energy is left in the limit on the neck regions. The result is that either
\begin{equation}\label{eq:2nd_floor_energy_identity_no_more_bubbles}
\Theta(x_{\alpha\beta}) = \cY_{1}^{g_{\RR^3}}(\nabla_{\alpha\beta}, \Phi_{\alpha\beta}; \RR^3),\ \ \ 
\Xi(x_{\alpha\beta}) = \cK(\nabla_{\alpha\beta}, \Phi_{\alpha\beta}),
\end{equation}
or that there exists some positive integer $N_{\alpha\beta} \leqslant  \frac{\tau(x_{\alpha\beta})}{\theta_{\mathrm{gap}}\cdot\min\{\lambda,1\}}$, and for each $\gamma \in \{1, \cdots, N_{\alpha\beta}\}$ some non-empty finite set $S_{\alpha\beta\gamma} \subset \RR^3 \setminus \{0\}$, such that
\[
\begin{split}
\Theta(x_{\alpha\beta}) = \ & \cY_{1}^{g_{\RR^3}}(\nabla_{\alpha\beta}, \Phi_{\alpha\beta}; \RR^3) + \sum_{\gamma = 1}^{N_{\alpha\beta}} \sum_{x \in S_{\alpha\beta\gamma}} \Theta_{\alpha\beta\gamma}(x),\\
\Xi(x_{\alpha\beta}) = \ & \cK(\nabla_{\alpha\beta}, \Phi_{\alpha\beta}) + \sum_{\gamma = 1}^{N_{\alpha\beta}} \sum_{x \in S_{\alpha\beta\gamma}} \Xi_{\alpha\beta\gamma}(x).
\end{split}
\]
where, for all $\gamma \in \{1, \cdots, N_{\alpha\beta}\}$,
\[
\Theta_{\alpha\beta\gamma}(x) \in [\theta_{\mathrm{gap}}\cdot\min\{\lambda,1\}, \Theta(z_0) - 2\theta_{\mathrm{gap}}\cdot\min\{\lambda,1\}],\ \ \ \Xi_{\alpha\beta\gamma}(x) \in 8\pi\ZZ
\] 
and the set $S_{\alpha\beta\gamma}$ consists of the points where suitable rescalings of $(\nabla_{\alpha, l}, \Phi_{\alpha, l})$ exhibit energy concentration. Repeating this for every $\alpha \in \{1, \cdots, N\}$ and $\beta \in \{{1, \cdots, N_{\alpha}\}}$, we refine~\eqref{eq:Theta_decomposition} and~\eqref{eq:Xi_decomposition} as 
\begin{subequations}
\begin{align}
\Theta(z_0) =\ & \cY_{1}(\nabla, \Phi) + \sum_{\alpha = 1}^N \sum_{\beta = 1}^{N_{\alpha}} \big( \cY_1(\nabla_{\alpha\beta}, \Phi_{\alpha\beta}) + \sum_{\gamma = 1}^{N_{\alpha\beta}}\sum_{x \in S_{\alpha\beta\gamma}}\Theta_{\alpha\beta\gamma}(x) \big), \label{eq:2nd_floor_Theta_decomposition}\\
\Xi(z_0) =\ & \cK(\nabla, \Phi) + \sum_{\alpha = 1}^N \sum_{\beta = 1}^{N_{\alpha}} \big( \cK(\nabla_{\alpha\beta}, \Phi_{\alpha\beta}) + \sum_{\gamma = 1}^{N_{\alpha\beta}}\sum_{x \in S_{\alpha\beta\gamma}}\Xi_{\alpha\beta\gamma}(x) \big).\label{eq:2nd_floor_Xi_decomposition}
\end{align}
\end{subequations}
where it is understood that the term involving $\sum_{\gamma = 1}^{N_{\alpha\beta}}(\cdots)$ is absent in the case that~\eqref{eq:2nd_floor_energy_identity_no_more_bubbles} hold for that particular choice of $\alpha$ and $\beta$. Continuing in this fashion, namely assigning to each $x \in S_{\alpha\beta\gamma}$ a top bubble $(\nabla_{\alpha\beta\gamma}, \Phi_{\alpha\beta\gamma})$ and constructing by Lemma~\ref{lemm:neck_analysis_transition_region} a sequence of annuli determining a transition region, with the energy difference this time satisfying
\[
\tau(x_{\alpha\beta\gamma}) : = \Theta(x_{\alpha\beta\gamma}) - \cY_1(\nabla_{\alpha\beta\gamma}, \Phi_{\alpha\beta\gamma}) \in [0, \Theta(z_0) - 3\theta_{\mathrm{gap}}\cdot\min\{\lambda,1\}],
\]
we see after applying Lemma~\ref{lemm:bubble_neck_decomposition} to the transition annuli and using Proposition~\ref{prop:no_neck} that we can improve~\eqref{eq:2nd_floor_Theta_decomposition} and~\eqref{eq:2nd_floor_Xi_decomposition} by splitting each $\Theta_{\alpha\beta\gamma}(x)$ and $\Xi_{\alpha\beta\gamma}(x)$ in a way similar to~\eqref{eq:Theta_decomposition} and~\eqref{eq:Xi_decomposition}. 

Each time the above argument is repeated, the energy difference is reduced by $\theta_{\mathrm{gap}}\cdot\min\{\lambda,1\}$. Thus after at most $k = \floor{\frac{\Theta(z_0)}{\theta_{\mathrm{gap}}\cdot\min\{\lambda,1\}}} + 1$ iterations we must find ourselves in the first alternative of Lemma~\ref{lemm:bubble_neck_decomposition}, where no further bubble regions arise from the transition annuli, and we get equalities analogous to~\eqref{eq:energy_identity_no_more_bubbles} as opposed to further splittings analogous to~\eqref{eq:Theta_decomposition} and~\eqref{eq:Xi_decomposition}. The end result is that we obtain a finite collection of non-trivial, finite-action critical points of $\cY_1^{g_{\RR^3}}$ on $\RR^3$ whose energy and charge sum up to $\Theta(z_0)$ and $\Xi(z_0)$, respectively.
\end{proof}

\appendix
\section{Standard facts on local Coulomb gauges}\label{sec:Coulomb}
In this appendix, we collect some facts concerning the uniqueness and continuous dependence of local Coulomb gauges that are used in Section~\ref{sec:existence}. All the results mentioned below follow essentially from the classical work of Uhlenbeck \cite{uhlenbeck1982connections} (see also \cite{wehrheim2004uhlenbeck}). 
To fix notation, we denote by $B$ the unit ball in $\RR^n$, and let $G$ be a compact Lie group with Lie algebra $\mathfrak{g}$. Given $p \in (1, \infty)$, recall that if $A \in W^{1, p}(B; \RR^n \otimes \mathfrak{g})$ is a $\mathfrak{g}$-valued $1$-form satisfying $A(\partial_r) = 0$ on $\partial B$ in the trace sense, then for some constant $C_{\text{hodge}} = C_{\text{hodge}}(n, p)$ there holds the a priori estimate
\begin{equation}\label{eq:hodge_Lp_estimate}
\|A\|_{1, p} \leqslant C_{\text{hodge}}\big( \|dA\|_{p} + \|d^* A\|_{p} \big),
\end{equation}
where the norms involved, as well as the $d^*$-oprator on the right-hand side, are taken with respect to the Euclidean metric. 

Next, fixing an exponent $p \in (\frac{n}{2}, n)$, we say that a given $\mathfrak{g}$-valued $1$-form $A \in W^{1, p}(B; \RR^n \otimes \mathfrak{g})$ satisfies \textbf{condition (U)} if 
\begin{equation}\label{eq:Uhlenbeck_gauge}
\left\{
\begin{array}{ll}
d^*A = 0, & \text{ in }B,\\
A(\partial_r ) = 0, & \text{ on }\partial B,\\
\|A\|_{1, q} \leqslant 2C_{\text{hodge}} \|F_{A}\|_{q}, & \text{ for }q = \frac{n}{2} \text{ and }q = p,
\end{array}
\right.
\end{equation}
where, of course, by $F_A$ we mean the $\mathfrak{g}$-valued $2$-form
\[
(F_A)_{ij} = (dA)_{ij} + [A_i, A_j].
\]
By the well-known Uhlenbeck rearrangement argument (using~\eqref{eq:hodge_Lp_estimate} and the Sobolev inequalities), there exists $\eta_{\text{rearr}} = \eta_{\text{rearr}}(n, p)$ such that if $A \in W^{1, p}(B; \RR^n \otimes \mathfrak{g})$ satisfies the first two conditions in~\eqref{eq:Uhlenbeck_gauge}, and if $\|A\|_{1, \frac{n}{2}} < \eta_{\text{rearr}}$, then $A$ also satisfies the last condition. The following uniqueness property is a standard fact, a version of which was left as an exercise in Donaldson--Kronheimer \cite[\S 2.3.9, p. 68]{DK}.
\begin{lemm}\label{lemm:coulomb_uniqueness}
There exists $\ep = \ep(n, p)$ such that given $A \in W^{1, p}(B; \RR^n \otimes \mathfrak{g})$ and $\mathrm{g}_1, \mathrm{g}_2 \in W^{2,p}(B; G)$, if $\|F_A\|_{\frac{n}{2}} < \ep$ and
\[
\mathrm{g}_i \cdot A := \mathrm{g}_i d(\mathrm{g}_i^{-1}) + \mathrm{g}_i A \mathrm{g}_i^{-1}\ (i = 1, 2)
\]
both satisfy the condition $(U)$, then $\mathrm{g}_2 \mathrm{g}_1^{-1}$ is constant on $B$.
\end{lemm}
\begin{proof}
We let $\widetilde{A}_i = \mathrm{g}_i \cdot A$ to save space. Define $\mathrm{g} =\mathrm{g}_2  \mathrm{g}_1^{-1}$, which again lies in $W^{2, p}(B; G)$ since $p > \frac{n}{2}$, it is a standard fact that 
\begin{equation}\label{eq:coulomb_uniqueness_transition}
\widetilde{A}_2 = \mathrm{g}d(\mathrm{g}^{-1}) + \mathrm{g}\widetilde{A}_1 \mathrm{g}^{-1}.
\end{equation}
Since both $\widetilde{A}_1$ and $\widetilde{A}_2$ satisfies~\eqref{eq:Uhlenbeck_gauge}, we see that
\[
d^*(\mathrm{g} d(\mathrm{g}^{-1})) = -d^*(\mathrm{g}\widetilde{A}_1 \mathrm{g}^{-1}) = [\mathrm{g}\widetilde{A}_1 \mathrm{g}^{-1}, \mathrm{g} d(\mathrm{g}^{-1})],
\]
and that 
\[
\mathrm{g}d(\mathrm{g}^{-1})(\partial_r) = 0 \text{ on }\partial B.
\]
Combining these with
\[
d(\mathrm{g} d(\mathrm{g}^{-1}))_{i, j} = -[\mathrm{g}\partial_i (\mathrm{g}^{-1}), \mathrm{g}\partial_j (\mathrm{g}^{-1})],
\]
and recalling~\eqref{eq:hodge_Lp_estimate}, we get
\[
\begin{split}
\|\mathrm{g}d(\mathrm{g}^{-1})\|_{1, p} \leqslant\ & C_{n, p}\big( \||\mathrm{g}d(\mathrm{g}^{-1})| |\widetilde{A}_1| \|_{p} + \||\mathrm{g}d(\mathrm{g}^{-1})|^2\|_{p} \big)\\
\leqslant\ & C_{n, p}\big( \|\widetilde{A}_1\|_{n} + \|\mathrm{g}d(\mathrm{g}^{-1})\|_{n} \big) \cdot \|\mathrm{g}d(\mathrm{g}^{-1})\|_{\frac{np}{n-p}}.
\end{split}
\]
From the relation~\eqref{eq:coulomb_uniqueness_transition}, the Sobolev embedding $W^{1, \frac{n}{2}} \hookrightarrow L^n$, and condition (U), we have 
\[
\|\mathrm{g}d(\mathrm{g}^{-1})\|_{n} \leqslant  \|\widetilde{A}_1\|_{n} + \|\widetilde{A}_2\|_{n} \leqslant  C_n\|F_A\|_{\frac{n}{2}} \leqslant  C_n\ep.
\]
Substituting this back above and using the Sobolev embedding $W^{1, p} \hookrightarrow L^{\frac{np}{n-p}}$ gives
\[
\|\mathrm{g}d(\mathrm{g}^{-1})\|_{1, p} \leqslant  C_{n, p}\ep \cdot \|\mathrm{g}d(\mathrm{g}^{-1})\|_{1, p} \leqslant  \frac{1}{2} \|\mathrm{g}d(\mathrm{g}^{-1})\|_{1, p},
\]
provided $\ep$ is sufficiently small depending on $n$ and $p$. This immediately gives the desired conclusion.
\end{proof}
We next address the issue of locally changing gauge continuously. The following can be gathered from the proof of \cite[Theorem 6.3]{wehrheim2004uhlenbeck}. 
\begin{prop}\label{prop:continuous_change}
There exists $\ep_{\text{gauge}} \in (0, \frac{\eta_{\text{rearr}}}{4C_{\text{hodge}}(n, \frac{n}{2})})$ depending only on $n, p$ such that if $A \in W^{1, p}(B; \RR^n \otimes \mathfrak{g})$ is such that 
\[
\|F_A\|_{\frac{n}{2}} < \ep_{\text{gauge}},
\]
then there exists a unique $\underline{\mathrm{g}} = \underline{\mathrm{g}}_{A} \in W^{2, p}(B; G)$ such that $\underline{\mathrm{g}}(0) = \id$ and that $\underline{\mathrm{g}} \cdot A$ satisfies condition $(U)$. Moreover, the map $A\mapsto \underline{\mathrm{g}}_A$ is continuous from $\cU: = \{A \in W^{1, p}(B; \RR^n \otimes \mathfrak{g})\ |\ \|F_A\|_{\frac{n}{2}} < \ep_{\text{gauge}}\}$ into $W^{2, p}(B; G)$.
\end{prop}
\begin{proof}
That such an $\underline{\mathrm{g}}$ exists provided $\ep_{\text{gauge}}$ is sufficiently small is of course Uhlenbeck's theorem. Uniqueness upon decreasing $\ep_{\text{gauge}}$ if necessary is also standard. (See the previous lemma.) Continuous dependence follows from the implicit function theorem. Specifically, given an arbitrary $A_0 \in \cU$, we define
\[
\mathrm{g}_0 = \underline{\mathrm{g}}_{A_0},\ \ \widetilde{A}_0 = \mathrm{g}_0 \cdot A_0.
\]
Notice that by our choice of $\ep$ and condition (U), 
\[
\|\widetilde{A}_0\|_{1, \frac{n}{2}} < \eta_{\text{rearr}}. 
\]
Since $\mathrm{g}_0(0) = \id$, from the uniqueness property of $\underline{\mathrm{g}}_A$ it is easily seen that whenever $A \in \cU$, we have
\[
\underline{\mathrm{g}}_A = \underline{\mathrm{g}}_{\mathrm{g}_0 \cdot A} \mathrm{g}_0.
\]
Thanks to this relation and the fact that the left action $\mathrm{g}_0 \cdot (\cdot)$ and right multiplication $(\cdot) \mathrm{g}_0$ are, respectively, continuous maps on $\cU$ and $W^{2, p}(B; G)$, to prove that $A \mapsto \underline{\mathrm{g}}_A$ is continuous at $A = A_0$, it suffices to establish its continuity at $A = \widetilde{A}_0$. To that end we define, as in Step 3a of the proof of \cite[Theorem 6.3]{wehrheim2004uhlenbeck}, the Banach spaces
\[
\mathscr{V} = \{u\in W^{2, p}(B; \mathfrak{g})\ |\ \int_B u = 0\},
\]
\[
\mathscr{W} = \{(f, \varphi)\ |\ f \in L^p(B; \mathfrak{g}), \ \ \varphi = h|_{\partial B} \text{ for some } h \in W^{1, p}(B; \mathfrak{g}), \text{ and } \int_B f + \int_{\partial B} \varphi = 0\},
\]
where all the integrals are computed using $g_{\RR^3}$, and we have used $(\cdot)|_{\partial B}$ to denote the Sobolev trace. Also, the space $\mathscr{W}$ is normed by
\[
\|(f, \varphi)\|_{\mathscr{W}} = \|f\|_{p} + \inf\{\|h\|_{1, p}\ |\ h \in W^{1, p}(B; \mathfrak{g}),\ h|_{\partial B} = \varphi\}.
\]
Following the estimates in~\cite[p.100 to p.102]{wehrheim2004uhlenbeck}, we see that upon decreasing $\ep$ (depending only on $n, p$) if necessary, the partial derivative with respect to the first variable of the map
\[
\begin{split}
N:\ & \mathscr{V} \times \cU \to \mathscr{W}\\
\ & (u,A) \mapsto (d^*(e^{u} \cdot A), (e^u \cdot A)(\partial_r)\big|_{\partial B})
\end{split}
\]
is invertible at $(0, \widetilde{A}_0)$. Thus we may apply the implicit function theorem to get $\rho > 0$ and a continuous map $u:B_{\rho}^{1, p}(\widetilde{A}_0) \to \mathscr{V}$ such that $u(\widetilde{A}_0)$ is the constant map zero, and that, for all $A \in B_{\rho}(\widetilde{A}_0)$,
\[
\|e^{u(A)}\cdot A\|_{1, \frac{n}{2}} < \eta_{\text{rearr}},\ \ N(u(A), A) = 0.
\]
In particular $(e^{u(A)}) \cdot A$ satisfies condition (U). Using again the uniqueness property of $\underline{\mathrm{g}}_A$ it is not hard to see that 
\[
e^{-u(A)(0)} \cdot e^{u(A)} = \underline{\mathrm{g}}_A.
\]
Since the left-hand side varies continuously in $W^{2, p}$ as $A$ varies in $W^{1, p}$ near $\widetilde{A}_0$, the same is true of the right-hand side, and we are done.
\end{proof}
A direct consequence of the previous proposition is the following.
\begin{prop}\label{prop:gauge_fixing_family}
Let $Y$ be any metric space and suppose $A:Y \to \cU$ is a continuous map. Then there exists a continuous map $\mathrm{g}: Y \to W^{2, p}(B; G)$ such that $\mathrm{g}(y) \cdot A(y)$ satisfies condition (U) for all $y \in Y$, and that $\mathrm{g}(y)\equiv \id$ whenever $A(y) = 0$.
\end{prop}
\begin{proof}
The result follows from Proposition~\ref{prop:continuous_change} upon taking $\mathrm{g}(y) = \underline{\mathrm{g}}_{A(y)}$. The details are omitted.
\end{proof}

\section{Moser iteration}\label{sec:Moser-iteration}
In this section we record a version of Moser's iteration that is used repeatedly in Section~\ref{sec:estimates}. Let $(M, g)$ be an $n$-dimensional Riemannian manifold with $n \geqslant 3$, and $B_{2r}(x_0)$ a geodesic normal ball on which 
\begin{equation}\label{eq:moser-metric-comparison}
\Lambda^{-1} g_{\RR^n} \leqslant \exp_{x_0}^* g \leqslant \Lambda g_{\RR^n},
\end{equation}
for some $\Lambda > 0$. Here $g_{\RR^n}$ denotes the standard flat metric. Suppose further that $u: B_{2r}(x_0) \to [0, \infty)$ is a non-negative, bounded, Lipschitz function satisfying, in the distributional sense, that
\begin{equation}\label{eq:moser-diff-ineq}
\Delta u \leqslant b u + c,
\end{equation}
where $b, c \in L^q(B_{2r}(x_0))$ for some $q \in (\frac{n}{2}, \infty]$. For all $p_0 \in [1, \infty)$, we define
\[
\gamma_{n, q, p_0} = \left\{
\begin{array}{ll}
\frac{nq}{p_0(2q - n)}, & \text{ if }q < \infty,\\
\frac{n}{2p_0}, & \text{ if }q = \infty.
\end{array}
\right.
\]
\begin{lemm}\label{lemm:moser}
In the above setting, suppose $p_0 > 1$ and let $\tau \in (0, 1)$ be a scaling factor. Then, we have 
\begin{equation}\label{eq:moser-local}
\begin{split}
\|{u}\|_{\infty; B_{\tau r}(x_0)} \leqslant\ & C_{\Lambda, n, q, p_0} (1 - \tau)^{-\frac{n}{p_0}} \big(1 + r^{2 - \frac{n}{q}}\|b\|_{q; B_{r}(x_0)} \big)^{\gamma_{n, q, p_0}}\\
& \times \big(r^{-\frac{n}{p_0}}\|{u}\|_{p_0; B_r(x_0)} + r^{2 - \frac{n}{q}}\|c\|_{q; B_r(x_0)}\big).
\end{split}
\end{equation}
\end{lemm}
\begin{proof}
Every step of the proof is standard, and we include the details only to keep track of how exactly the constants are affected by $\|b\|_{q; B_r(x_0)}$. Below, when the center of a geodesic ball is not specified, it is understood to be centered at $x_0$. Also, all the integrals are taken with respect to the volume form of $g$, which is comparable on $B_{2r}(x_0)$ to the Euclidean volume form due to the assumption~\eqref{eq:moser-metric-comparison}. To begin, let $\varphi:\RR \to [0, \infty)$ be a non-negative smooth function such that
\[
\varphi(t) =
\left\{
\begin{array}{ll}
0 &, \text{ if }t \geqslant  1,\\
1&, \text{ if }t \leqslant 0.
\end{array}
\right.
\]
Given $k \in (0, \infty)$, $\beta \in [p_0, \infty)$, as well as $0 < \sigma < \rho \leqslant r$, we set $p  = \beta - 1$ and define
\[
\widetilde{b} = |b| + \frac{|c|}{k}, \quad \zeta = \varphi(\frac{d(\cdot, x_0)-\sigma}{\rho - \sigma}),\quad v = \zeta^2 \cdot [({u} + k)^p - k^p],
\]
where $d$ denotes the geodesic distance. Testing~\eqref{eq:moser-diff-ineq} against $v$ gives
\begin{equation}\label{eq:moser-first}
\begin{split}
\int_{M} |b|uv + |c|v 
\geqslant\ & \int_{M}  p\zeta^2 ({u} + k)^{p-1} |\nabla u|^2 -  2\zeta [({u} + k)^p - k^p]|\nabla\zeta||\nabla u|\\
\geqslant \ & p\int_{M}  \zeta^2 ({u} + k)^{p-1} |\nabla u|^2 -  2\int_{M}\zeta ({u} + k)^p |\nabla\zeta||\nabla u|\\
\geqslant\ & \frac{p}{2}\int_{M}  \zeta^2 ({u} + k)^{p-1} |\nabla u|^2 - \frac{2}{p}\int_{M} (u + k)^{p + 1}|\nabla \zeta|^2,
\end{split}
\end{equation}
where the last line follows from Young's inequality. Noting that 
\[
(u+k)^{p-1}|\nabla u|^2 = \frac{4}{(p+1)^2} \big| \nabla \big[ (u + k)^{\frac{p + 1}{2}} \big] \big|^{2},
\]
and that
\[
|b|uv + |c|v \leqslant \widetilde{b}(u + k)v \leqslant \widetilde{b}(u + k)^{p + 1}\zeta^2,
\]
we deduce from~\eqref{eq:moser-first} that
\[
\frac{2p}{(p + 1)^2}\int_{M} \zeta^2  \big| \nabla \big[ (u + k)^{\frac{p + 1}{2}} \big] \big|^{2} \leqslant \frac{2}{p}\int_{M}(u + k)^{p + 1}|\nabla \zeta|^2  + \int_{M}\widetilde{b} (u + k)^{p + 1}\zeta^2,
\]
and hence
\begin{equation}\label{eq:moser-tested}
\begin{split}
\int_{M} |\nabla[\zeta({u} + k)^{\frac{p  +1}{2}}]|^2\leqslant \ & 2\int_{M} \zeta^2 \big| \nabla[ ({u} + k)^{\frac{p + 1}{2}} ] \big|^2  + 2\int_{M} |\nabla \zeta|^2 ({u} + k)^{p + 1} \\
\leqslant \ & 2\big[  \big(\frac{1 + p}{p} \big)^2 + 1 
 \big]  \int_{M}|\nabla\zeta|^2 ({u} + k)^{p + 1} + \frac{(1 +p)^2}{p}\int_{M}\widetilde{b} ({u} + k)^{p + 1}\zeta^2\\
\leqslant \ & C_{p_0}\int_{M}|\nabla\zeta|^2 ({u} + k)^{p + 1} + C_{p_0}\cdot p\int_{M}\widetilde{b}({u} + k)^{p + 1}\zeta^2,
\end{split}
\end{equation}
where in passing to the third line we used the fact that $\frac{1 + p}{p} \leqslant  \frac{p_0}{p_0 - 1}$ whenever $p \geqslant  p_0 -1$. 

Next, since the function $\zeta({u} + k)^{\frac{p + 1}{2}}$ is supported in the geodesic ball $B_{2r}(x_0)$, the assumption~\eqref{eq:moser-metric-comparison} allows us to invoke the Euclidean Sobolev inequality accompanying the embedding $W^{1, 2} \hookrightarrow L^{\frac{2n}{n-2}}$ to deduce that
\begin{equation}\label{eq:cpt-Sobolev}
\Big(\int_{M} [\zeta^2({u} + k)^{p + 1}]^{\frac{n}{n-2}}\Big)^{\frac{n-2}{n}} \leqslant  C_{n, \Lambda} \int_{M} |\nabla[\zeta({u} + k)^{\frac{p  +1}{2}}]|^2.
\end{equation}
Hence, upon letting
\[
w = ({u} + k)^{p + 1}\zeta^2,
\]
we get from~\eqref{eq:cpt-Sobolev} and~\eqref{eq:moser-tested} that
\begin{equation}\label{eq:moser-after-sobolev}
\begin{split}
\|w\|_{\frac{n}{n-2}} \leqslant \ & C_{n, \Lambda,p_0} \int_{M}|\nabla\zeta|^2 ({u} + k)^{p + 1} + C_{n, \Lambda, p_0}\cdot p \int_{M}\widetilde{b} w.
\end{split}
\end{equation}
To continue, we set 
\[
\theta = \left\{
\begin{array}{ll}
\frac{n}{2q}, & \text{ if }q < \infty,\\
0, & \text{ if }q = \infty.
\end{array}
\right.
\]
In the case $q < \infty$, by H\"older's inequality, the standard interpolation inequality between $L^p$-norms, and Young's inequality, we have that
\begin{equation}\label{eq:tilde-b-holder-plus-interpolation}
\begin{split}
\int_{M}\widetilde{b}w  \leqslant  \|\widetilde{b}\|_{q; B_r} \cdot \|w\|_{\frac{q}{q-1}} \leqslant \ & \|\widetilde{b}\|_{q; B_r} \cdot \|w\|_{\frac{n}{n-2}}^{\theta} \|w\|_{1}^{1 - \theta}\\
\leqslant\ &  \|\widetilde{b}\|_{q; B_r} \cdot \big( \theta \delta \|w\|_{\frac{n}{n-2}} + (1 - \theta) \delta^{-\frac{\theta}{1 - \theta}} \|w\|_{1} \big),
\end{split}
\end{equation}
with $\delta > 0$ to be determined momentarily. Substituting~\eqref{eq:tilde-b-holder-plus-interpolation} back into~\eqref{eq:moser-after-sobolev}, we obtain upon rearranging that
\[
\begin{split}
(1 - C_{n, \Lambda, p_0} \cdot p \theta \|\widetilde{b}\|_{q; B_r} \cdot\delta)\|w\|_{\frac{n}{n-2}} \leqslant \ & C_{n, \Lambda, p_0}\int_{M}|\nabla\zeta|^2({u} + k)^{p+1} \\
&+ C_{n, \Lambda, p_0} \cdot p \|\widetilde{b}\|_{q; B_r}(1 - \theta)\delta^{-\frac{\theta}{1 - \theta}}\|w\|_1.
\end{split}
\]
Making the choice
\[
\delta = \frac{1}{2C_{n, \Lambda, p_0} \cdot p \theta \cdot( \|\widetilde{b}\|_{q; B_r} + t)},
\]
and then letting $t \to 0^{+}$, we obtain
\[
\begin{split}
\|w\|_{\frac{n}{n-2}} \leqslant \ & 2C_{n, \Lambda, p_0}\int_{M}|\nabla\zeta|^2({u} + k)^{p+ 1} + \big( 2C_{n, \Lambda, p_0}\cdot p \|\widetilde{b}\|_{q; B_r}\big)^{\frac{1}{1 - \theta}}\theta^{\frac{\theta}{1 - \theta}}\|w\|_1. 
\end{split}
\]
Recalling the definition of $w$ and our choice of $\zeta$, and using the fact that $p^{\frac{1}{1 - \theta}} \geqslant (p_0 - 1)^{\frac{1}{1 - \theta}}$, we deduce that
\begin{equation}\label{eq:moser-almost-ready}
\begin{split}
\|({u} + k)^{p + 1}\|_{\frac{n}{n-2}; B_{\sigma}} 
\leqslant  \ & C_{\Lambda, n, q, p_0} \cdot p^{\frac{1}{1 - \theta}}(\rho - \sigma)^{-2}\\
& \times \big[ 1 + (\rho - \sigma)^2 \|\widetilde{b}\|_{q; B_r}^{\frac{1}{1 - \theta}} \big]\|({u} + k)^{p+1}\|_{1; B_{\rho}}.
\end{split}
\end{equation}
On the other hand, when $q = \infty$, we replace~\eqref{eq:tilde-b-holder-plus-interpolation} by 
\begin{equation}\label{eq:tilde-b-simple}
\int_{M}\widetilde{b}\cdot ({u} + k)^{p + 1} \zeta^2\leqslant  \|\widetilde{b}\|_{\infty; B_{r}} \cdot \|({u} + k)^{p + 1}\zeta^2\|_{1}
\end{equation}
to deduce from~\eqref{eq:moser-after-sobolev} that the inequality~\eqref{eq:moser-almost-ready} still holds. At any rate, recalling that $\beta = p + 1$ and letting $\chi = \frac{n}{n-2}$, we get upon taking the $\beta$-th root of both sides of~\eqref{eq:moser-almost-ready} that
\begin{equation}\label{eq:almost-ready-to-iterate}
\|{u} + k\|_{\beta\chi; B_\sigma} \leqslant  C_{\Lambda, n, q, p_0}^{\frac{1}{\beta}} \cdot \beta^{\frac{1}{(1 - \theta)\beta}}(\rho - \sigma)^{-\frac{2}{\beta}} \big[ 1 + (\rho - \sigma)^2 \|\widetilde{b}\|_{q; B_{r}}^{\frac{1}{1 - \theta}} \big]^{\frac{1}{\beta}} \|{u} + k\|_{\beta; B_\rho}.
\end{equation}
For $m \in \NN \cup \{0\}$, we now define 
\[
r_m =  (\tau + \frac{1-\tau}{2^m})r,
\]
and apply~\eqref{eq:almost-ready-to-iterate} with
\[
\beta = p_0\chi^{m},\quad \sigma = r_{m + 1},\quad \rho = r_m,
\]
to obtain, with perhaps a different $C_{\Lambda, n, q, p_0}$, 
\begin{equation}\label{eq:ready-to-iterate}
\begin{split}
\|{u} + k\|_{p_0\chi^{m + 1}; B_{r_{m + 1}}} \leqslant \  &\Big( (C_{\Lambda, n, q, p_0})^{\chi^{-m}} (2^{\frac{2}{p_0}}\chi^{\frac{1}{p_0(1 - \theta)}})^{m \chi^{-m}} [(1 - \tau)r]^{-\frac{2}{p_0}\chi^{-m}}\Big)\\
& \times \Big(\big[ 1 + r^2 \|\widetilde{b}\|_{q; B_{r}}^{\frac{1}{1 - \theta}} \big]^{\frac{1}{p_0}\chi^{-m}} \|{u} + k\|_{p_0\chi^m; B_{r_m}}\Big).
\end{split}
\end{equation}
Recalling the elementary inequality 
\[
1 + t^{\alpha} \leqslant  (1 + t)^{\alpha}, \text{ whenever }t \geqslant  0,\ \alpha \geqslant  1,
\]
we have
\[
1 + r^2 \|\widetilde{b}\|_{q; B_{r}}^{\frac{1}{1 - \theta}} 
\leqslant  \big( 1 + r^{2 - \frac{n}{q}} \|\widetilde{b}\|_{q; B_{r}} \big)^{\frac{1}{1 - \theta}}.
\]
Substituting this back into~\eqref{eq:ready-to-iterate} and iterating, we get for all $m \geqslant  1$ that
\[
\begin{split}
\|{u} + k\|_{p_0\chi^m; B_{r_m}}\leqslant  \ &\Big((C_{\Lambda, n, q, p_0})^{\sum_{i = 0}^{m-1}\chi^{-i}}  \cdot (C_{n, q, p_0})^{\sum_{i = 0}^{m-1}i\chi^{-i}} \cdot [(1-\tau)r]^{-\frac{2}{p_0}\sum_{i = 0}^{m-1}\chi^{-i}}\Big)\\
& \times \Big(\big[ 1 + r^{2-\frac{n}{q}}\|\widetilde{b}\|_{q; B_r} \big]^{\frac{1}{p_0(1-\theta)}\sum_{i = 0}^{m-1}\chi^{-i}}\|{u} + k\|_{p_0; B_r}\Big).
\end{split}
\]
Letting $m \to \infty$ gives
\[
\|{u} + k\|_{\infty; B_{\tau r}} \leqslant  C_{\Lambda, n, q, p_0} [(1 - \tau)r]^{-\frac{n}{p_0}}\big( 1 + r^{2 - \frac{n}{q}}\|\widetilde{b}\|_{q; B_r}  \big)^{\gamma_{n, q, p_0}}(\|{u}\|_{p_0; B_r} + k r^{\frac{n}{p_0}}),
\]
where we used~\eqref{eq:moser-metric-comparison} to estimate $\|k\|_{p_0; B_r}$. Taking $k = r^{2 - \frac{n}{q}}\|c\|_{q; B_r} + \delta$, with $\delta > 0$ to be sent to $0$ in a moment, we find that
\[
r^{2 - \frac{n}{q}}\|\widetilde{b}\|_{q; B_r} \leqslant  r^{2 - \frac{n}{q}}\|b\|_{q; B_r} + 1.
\]
Consequently, 
\[
\begin{split}
\|{u}\|_{\infty; B_{\tau r}} \leqslant \ &  C_{\Lambda, n, q, p_0}(1-\tau)^{-\frac{n}{p_0}}\big(1 + r^{2 - \frac{n}{q}}\|b\|_{q; B_r}\big)^{\gamma_{n, q, p_0}}\\
& \times \big(r^{-\frac{n}{p_0}}\|{u}\|_{p_0; B_{r}} + r^{2 - \frac{n}{q}}\|c\|_{q;B_r} + \delta \big).
\end{split}
\]
Letting $\delta \to 0$ gives the estimate~\eqref{eq:moser-local} we want. 
\end{proof}
As is well-known (see for example~\cite[Section 4.2]{Han-Lin}), with a little bit more work, we can in fact allow $p_0 = 1$ in Lemma~\ref{lemm:moser}. This leads to the next result, which again is entirely standard. We include the proof for the sake of completeness.
\begin{lemm}\label{lemm:moser-improve}
Under the assumptions of Lemma~\ref{lemm:moser}, we have the following.
\vskip 1mm
\begin{enumerate}
\item[(a)] There holds
\begin{equation}\label{eq:moser-local-improved}
\begin{split}
\|{u}\|_{\infty; B_{\tau r}(x_0)} \leqslant\ & C_{\Lambda, n, q, p_0}\cdot \big(1 + r^{2 - \frac{n}{q}}\|b\|_{q; B_r(x_0)} \big)^{\gamma_{n, q, p_0}}\\
&\ \times \big( [(1-\tau)r]^{-\frac{n}{p_0}}\|{u}\|_{p_0; B_r(x_0)} + r^{2 - \frac{n}{q}}\|c\|_{q; B_r(x_0)} \big).
\end{split}
\end{equation}
\vskip 1mm
\item[(b)] The exists a constant $C_{\Lambda, n, q}$ so that the estimate in part (a) holds with $p_0 = 1$.
\end{enumerate}
\end{lemm}
\begin{proof}
For part (a), we let $y$ be any point in $B_{\tau r}(x_0)$ and apply Lemma~\ref{lemm:moser} on the ball $B_{(1-\tau )r}(y)$ with the scaling factor taken to be $\frac{1}{2}$. After absorbing the term $2^{\frac{n}{p_0}}$ in the resulting estimate into $C_{\Lambda, n, q, p_0}$, and using the inclusion $B_{(1-\tau)r}(y) \subset B_r(x_0)$, we infer that
\[
\begin{split}
\|{u}\|_{\infty; B_{(\frac{1-\tau}{2})r}(y)} \leqslant\ & C_{\Lambda, n, q, p_0}\cdot \big(1 + r^{2-\frac{n}{q}}\|b\|_{q; B_r(x_0)} \big)^{\gamma_{n, q, p_0}}\\
&\ \times \big( [(1-\tau)r]^{-\frac{n}{p_0}}\|{u}\|_{p_0; B_{r}(x_0)} + r^{2-\frac{n}{q}}\|c\|_{q; B_r(x_0)} \big).
\end{split}
\]
Since $y \in B_{\tau r}(x_0)$ is arbitrary, we conclude that~\eqref{eq:moser-local-improved} holds. 

Moving to part (b), for convenience we define
\[
A = C_{\Lambda, n, q, 2}(1 + r^{2 - \frac{n}{q}}\|b\|_{q; B_r(x_0)})^{\frac{nq}{2(2q - n)}}, \ \ B = r^{2 - \frac{n}{q}}\|c\|_{q; B_r(x_0)}.
\]
For any $0 < \rho \leqslant r$ and $\lambda \in (0, 1)$, by~\eqref{eq:moser-local-improved} with $p_0 = 2$, followed by the interpolation inequality
\[
\|{u}\|_{2; B_{\rho}(x_0)} \leqslant \|{u}\|_{\infty; B_{\rho}(x_0)}^{\frac{1}{2}} \|{u}\|_{1; B_{\rho}(x_0)}^{\frac{1}{2}},
\]
we have that
\[
\begin{split}
\|{u}\|_{\infty; B_{\lambda \rho}(x_0)} \leqslant\ & A \cdot [(1-\lambda)\rho]^{-\frac{n}{2}}\|{u}\|_{\infty; B_{\rho}(x_0)}^{\frac{1}{2}} \|{u}\|_{1; B_{\rho}(x_0)}^{\frac{1}{2}} +A B\\
\leqslant\ & \frac{1}{4}\|{u}\|_{\infty; B_{\rho}(x_0)} + \frac{A^2 \|{u}\|_{1; B_r(x_0)}}{[(1 - \lambda) \rho]^{n}} + AB.
\end{split}
\]
This being true for all $\rho \in (0, r]$ and $\lambda \in (0, 1)$, we get from Lemma 4.3 of~\cite{Han-Lin} some dimensional constant $c_n$ such that, for all $\tau \in (0, 1)$,
\[
\|{u}\|_{\infty; B_{\tau r}(x_0)} \leqslant c_n \Big( \frac{A^2 \|{u}\|_{1; B_r(x_0)}}{[(1 - \tau) r]^n} + AB \Big).
\]
Recalling the definitions of $A$ and $B$, and noting that $A \leqslant A^2$, we arrive at
\[
\begin{split}
\|{u}\|_{\infty; B_{\tau r}(x_0)} \leqslant\ & C_{\Lambda, n, q}\big(1 + r^{2 - \frac{n}{q}}\|b\|_{q; B_r(x_0)} \big)^{\frac{qn}{2q - n}}\\
&\ \times \big([(1-\tau)r]^{-n}\|{u}\|_{1; B_{r}(x_0)} + r^{2 - \frac{n}{q}}\|c\|_{q; B_r(x_0)} \big),
\end{split}
\]
which is the desired estimate.
\end{proof}
\section{Commuting the rough Laplacian with covariant derivatives}\label{sec:commute}
The main purpose of this appendix is to recall a standard commutator estimate involving the rough Laplacian. The dimension of $M$ is irrelevant here. Thus, we let $E$ be a complex rank-$2$ vector bundle associated with a principal $SU(2)$-bundle over a Riemannian $n$-manifold $M$, and denote by $\mathfrak{su}(E)$ the adjoint bundle of $E$. We remind the reader that by $\mathscr{A}(E)$ we mean the set of $SU(2)$-connections on $E$. Given a section $\Phi$ of $\fsu(E)$, we let $Z(\Phi) = \{x \in M\ |\ |\Phi(x)| = 0\}$. 

Fixing an arbitrary pair $(\nabla, \Phi) \in \mathscr{A}(E) \times \Gamma(\fsu(E))$, for use here and in Appendix~\ref{sec:proofs_derivative_formulas}, we mention a few basic inequalities concerning the norm of $\fsu(E)$-valued tensors, their derivatives with respect to $\nabla$, and their transversal and longitudinal parts with respect to the splitting~\eqref{eq: adjoint_decomp} induced by $\Phi$. 

Given $\fsu(E)$-valued tensors $A$ and $B$ of degree $p$ and $q$, respectively, we write $[A, B]$ for the $\fsu(E)$-valued $(p + q)$-tensor defined by
\[
[A, B]_{i_1, \cdots, i_{p+q}} = [A_{i_1, \cdots, i_p}, B_{i_{p + 1}, \cdots, i_{p+q}}].
\]
Then from~\eqref{eq: bracket_norm} applied to each component, we immediately get 
\begin{equation}\label{eq:tensor-bracket-norm}
\big|[A, B]\big| \leqslant |A||B|.
\end{equation}
By~\eqref{eq: double_bracket_ineq}, still applied component-wise, together with the triangle inequality, we have
\begin{equation}\label{eq:tensor-bracket-norm-with-decomp-1}
\big|[[A, B], \Phi]\big| \leqslant \big|[A, \Phi]\big||B| + |A|\big|[B, \Phi] \big|.
\end{equation}
Likewise, away from $Z(\Phi)$, we have thanks to~\eqref{eq: bracket_ineq} that
\begin{equation}\label{eq:tensor-bracket-norm-with-decomp-2}
\big|[A, B]\big| \leqslant |A^\perp||B| + |A||B^{\perp}|.
\end{equation}

Suppose in addition that $m \in \NN$. Then by Leibniz's rule, the triangle inequality and Schwarz's inequality, we have
\begin{equation}\label{eq:tensor-bracket-derivative-bound}
\sum_{i_1, \cdots, i_m}|\nabla^m_{i_1, \cdots, i_m}[A, B]|^2
\leqslant C_m\sum_{l = 0}^{m}|[\nabla^l A, \nabla^{m-l}B]|^2.
\end{equation}
By a similar argument we also get
\begin{equation}\label{eq:tensor-bracket-derivative-bound-no-head}
\sum_{i_1, \cdots, i_m}|\nabla_{i_1, \cdots, i_m}^m[A, B] - [\nabla_{i_1, \cdots, i_m}^m A, B]|^2 \leqslant C_m \sum_{l = 0}^{m-1}|[\nabla^l A, \nabla^{m-l}B]|^2,
\end{equation}
and, when $m \geqslant 2$,
\begin{equation}\label{eq:tensor-bracket-derivative-bound-no-ends}
\begin{split}
\sum_{i_1, \cdots, i_m}|\nabla_{i_1, \cdots, i_m}^m[A, B] - [A, \nabla_{i_1, \cdots, i_m}^m B]  -\ & [\nabla_{i_1, \cdots, i_m}^m A, B]|^2 \leqslant C_m \sum_{l = 1}^{m-1}|[\nabla^l A, \nabla^{m-l}B]|^2.
\end{split}
\end{equation}

Next, with $R$ denoting the Riemann curvature tensor of $M$, we let $R \cdot A$ be the $(2 + p)$-tensor given by 
\[
\begin{split}
(R_{ij} \cdot A)_{i_1, \cdots, i_p} =\ & (R \cdot A)_{i, j, i_1, \cdots, i_p}:= R_{i,j,i_1,k}A_{k, i_2, \cdots, i_p} + \cdots + R_{i,j,i_p,k}A_{i_1, \cdots, i_{p-1}, k},
\end{split}
\]
and define $(\nabla^m R) \cdot A$, which would be a $(2 + m + p)$-tensor, analogously. Each of the $p$ terms on the right-hand side above being a contraction of $R \otimes A$ with indices permuted, we have for all $m \in \NN \cup \{0\}$ that
\begin{equation}\label{eq:curvature-action-derivative-bound}
|\nabla^m(R \cdot A)| \leqslant C_{n, m, p} \sum_{l = 0}^{m} |\nabla^{l}R| |\nabla^{m  - l}A|,
\end{equation}
and a similar estimate holds when $R$ is replaced by any of its covariant derivatives. 

We now come to the standard commutator estimate mentioned above.
\begin{lemm}\label{lemm:commutator-estimate}
Suppose $(\nabla,\Phi)\in \mathscr{A}(E)\times\Gamma(\mathfrak{su}(E))$ and let $S$ be an $\fsu(E)$-valued $p$-tensor on $M$. Then for all $m \in \NN$, we have the following pointwise estimate:
\begin{equation}\label{eq:commutator-estimate}
|\nabla^m \nabla^* \nabla S - \nabla^*\nabla \nabla^m S| \leqslant C_{n, m, p}\sum_{l = 0}^{m} |[\nabla^l F, \nabla^{m - l}S]| + C_{n, m, p}\sum_{l = 0}^{m} |\nabla^l R| |\nabla^{m-l}S|.
\end{equation}
\end{lemm}
\begin{proof}
We prove by induction on $m \in \NN$ that~\eqref{eq:commutator-estimate} holds for any $\fsu(E)$-valued tensor $S$ of arbitrary degree $p$. For the base case, by a direct computation we have in terms of a local geodesic frame that
\[
\begin{split}
\nabla_i\nabla^*\nabla S =\ & -\nabla^3_{i, k, k}S\\
=\ & -(\nabla^3_{k, i, k}S + [F_{ik}, \nabla_k S] - \nabla_{\Ric(e_i)} S - R_{ik}\cdot \nabla_{k} S)\\
=\ & -\nabla_k \big(\nabla^{2}_{k, i}S + [F_{ik}, S] - R_{ik} \cdot S \big) - [F_{ik}, \nabla_k S] + \nabla_{\Ric(e_i)} S + R_{ik}\cdot \nabla_{k} S\\
=\ & (\nabla^*\nabla \nabla S)_{i} - [\nabla_{k}F_{ik}, S] - 2[F_{ik}, \nabla_{k}S] \\
&+ \nabla_{k}R_{ik} \cdot S + 2 (R_{ik} \cdot \nabla S)_{k, \cdots} - \Ric_{ij}\nabla_{j}S,
\end{split}
\]
where in getting the last equality we used the fact that $\nabla_k(R_{ik} \cdot S) = \nabla_k R_{ik} \cdot S + R_{ik} \cdot \nabla_k S$, along with the identity
\[
(R_{ik} \cdot \nabla S)_{k, \cdots} = R_{i k k j}\nabla_{j}S + R_{ik} \cdot \nabla_k S = \nabla_{\Ric(e_i)} S + R_{ik} \cdot \nabla_k S.
\]
Thus we have shown that
\begin{equation}\label{eq:commutator-identity-1}
\begin{split}
&\nabla \nabla^*\nabla S - \nabla^*\nabla \nabla S\\
=\ & -[\nabla_{e_k} F_{\cdot, e_k}, S] - 2[F_{\cdot, e_k}, \nabla_{e_k}S] + \nabla_{e_k}R_{\cdot, e_k}\cdot S + 2(R_{\cdot, e_k}\cdot \nabla S)_{e_k, \cdots} - \Ric_{\cdot, e_k}\nabla_{e_k}S,
\end{split}
\end{equation} from which we deduce, with the help of~\eqref{eq:curvature-action-derivative-bound}, that
\begin{equation}\label{eq:commutator-estimate-base}
|\nabla \nabla^*\nabla S - \nabla^*\nabla \nabla S| \leqslant C_{n, p}(|[\nabla F, S]| + |[F, \nabla S]| + |\nabla R||S| + |R||\nabla S|).
\end{equation}
This establishes the base step. For the induction step, suppose that for some $m \geqslant 1$, the estimate~\eqref{eq:commutator-estimate} holds for any $\fsu(E)$-valued $p$-tensor, $p$ being arbitrary. We split the commutator term for the $(m+1)$-case as
\[
\begin{split}
&\nabla^{m + 1} \nabla^*\nabla S - \nabla^*\nabla \nabla^{m + 1} S\\
=\ & \underbrace{\nabla^m(\nabla \nabla^*\nabla S - \nabla^*\nabla \nabla S)}_{T_1} + \underbrace{\nabla^m \nabla^*\nabla (\nabla S) - \nabla^*\nabla \nabla^m(\nabla S)}_{T_2}.
\end{split}
\]
The tensor $T_2$ we estimate by applying the induction hypothesis to $\nabla S$:
\begin{equation}\label{eq:commutator-estimate-T2}
|T_2| \leqslant C_{n, m, p}\sum_{l = 0}^{m}|[\nabla^l F, \nabla^{m + 1 - l}S]| + C_{n, m, p}\sum_{l = 0}^{m} |\nabla^l R||\nabla^{m + 1 - l}S|.
\end{equation}
For $T_1$ we use~\eqref{eq:commutator-identity-1} together with standard properties of contractions to see that 
\[
|T_1| \leqslant C_n(|\nabla^m[\nabla F, S]| + |\nabla^m [F, \nabla S]| + |\nabla^m(\nabla R \cdot S)| + |\nabla^m(R \cdot \nabla S)| + |\nabla^m (\Ric\otimes \nabla S)|).
\]
By~\eqref{eq:tensor-bracket-derivative-bound} we have
\begin{equation}\label{eq:commutator-estimate-T1-I}
|\nabla^m [\nabla F, S]| + |\nabla^m [F, \nabla S]| \leqslant C_m \sum_{l = 0}^{m + 1}|[\nabla^l F, \nabla^{m + 1 - l}S]|.
\end{equation}
Repeating the reasoning leading to~\eqref{eq:tensor-bracket-derivative-bound} shows that
\begin{equation}\label{eq:commutator-estimate-T1-II}
|\nabla^m (\Ric\otimes \nabla S)| \leqslant C_{n, m}\sum_{l = 0}^{m} |\nabla^{l}R| |\nabla^{m + 1 - l}S|.
\end{equation}
On the other hand, by~\eqref{eq:curvature-action-derivative-bound} we have
\[
|\nabla^m(\nabla R \cdot S)| + |\nabla^m (R \cdot \nabla S)| \leqslant C_{n, m, p} \sum_{l = 0}^{m+1}|\nabla^{l}R| |\nabla^{m + 1 - l}S|.
\]
Combining this with~\eqref{eq:commutator-estimate-T1-I} and~\eqref{eq:commutator-estimate-T1-II} yields
\[
|T_1| \leqslant C_{n,m, p} \sum_{l = 0}^{m+1}|[\nabla^l F, \nabla^{m + 1 - l}S]| + C_{n, m, p}\sum_{l = 0}^{m + 1} |\nabla^l R| |\nabla^{m + 1 - l}S|.
\]
We complete the inductive step upon recalling~\eqref{eq:commutator-estimate-T2}, so that~\eqref{eq:commutator-estimate} holds for any $m$. 
\end{proof}
\section{Proofs of some derivative formulas}\label{sec:proofs_derivative_formulas}
In this section we give the proofs of Lemma~\ref{lemm:Psi_Theta_relations}, Lemma~\ref{lemm:bochner_for_derivatives} and Lemma~\ref{lemm:trans_laplacian_diffed}. We shall make frequent use of the estimates mentioned at the start of Appendix~\ref{sec:commute}.

\begin{proof}[Proof of Lemma~\ref{lemm:Psi_Theta_relations}]
Let $S$ stand for either $\nabla\Phi$ or $\ep F_{\nabla}$. By~\eqref{eq:tensor-bracket-derivative-bound-no-head}, we have 
\begin{equation}\label{eq:Leibniz-and-Schwarz}
\begin{split}
|\nabla^m[S, \Phi] - [\nabla^m S, \Phi]| \leqslant\ & C_{m}\sum_{i = 0}^{m-1}|[\nabla^i S, \nabla^{m-i}\Phi]|.
\end{split}
\end{equation}
Temporarily setting
\[
A = |\nabla^m[\ep F_{\nabla}, \Phi]|,\ \ B = |\nabla^m[\nabla\Phi, \Phi]|,\ \ C = |[\nabla^m (\ep F_{\nabla}), \Phi]|,\ \ D = |[\nabla^{m + 1}\Phi, \Phi]|,
\]
we see after two applications of the triangle inequality, followed by~\eqref{eq:Leibniz-and-Schwarz}, that
\begin{equation}\label{eq:Psi_Theta_bridge}
\begin{split}
\big| \Theta_m - (C^2 + D^2)^{\frac{1}{2}} \big| =\ & \big| (A^2 + B^2)^{\frac{1}{2}} - (C^2 + D^2)^{\frac{1}{2}} \big|\\
\leqslant\ & (|A - C|^2 + |B - D|^2)^{\frac{1}{2}}\\
\leqslant\ & \big( \big|\nabla^m[\ep F_{\nabla}, \Phi] - [\nabla^m (\ep F_{\nabla}), \Phi]\big|^2 +  \big| \nabla^m[\nabla\Phi, \Phi] - [\nabla^{m}(\nabla\Phi), \Phi] \big|^2\big)^{\frac{1}{2}}\\
\leqslant\ & C_{m}\sum_{i = 0}^{m-1}|[\nabla^i (\ep F_{\nabla}), \nabla^{m-i}\Phi]| + C_{m}\sum_{i = 0}^{m-1}|[\nabla^i (\nabla\Phi), \nabla^{m-i}\Phi]|.
\end{split}
\end{equation}
From this we easily get~\eqref{eq:Psi_Theta_relation} upon using~\eqref{eq:tensor-bracket-norm} and the fact that $(C^2 + D^2)^{\frac{1}{2}} \leqslant |\Phi|\Psi_m$. To establish~\eqref{eq:Theta_perp_relation} on $M \setminus Z(\Phi)$, we again use~\eqref{eq:Psi_Theta_bridge}, observing instead that $(C^2 + D^2)^{\frac{1}{2}} = |\Phi|\Psi_m^\perp$ on $M \setminus Z(\Phi)$, and that, by~\eqref{eq:tensor-bracket-norm-with-decomp-2},
\[
\begin{split}
\sum_{i = 0}^{m-1}\big|[\nabla^i S, \nabla^{m-i}\Phi]\big| 
\leqslant\ & \sum_{i = 0}^{m-1} (\Psi_i^\perp \Psi_{m-1-i} + \Psi_i \Psi_{m-1-i}^\perp) = 2\sum_{i=0}^{m-1}\Psi_i^\perp \Psi_{m-1-i}.
\end{split}
\]
\end{proof}

\begin{proof}[Proof of Lemma~\ref{lemm:bochner_for_derivatives}]
Each term on the right-hand side of~\eqref{eq:nablaPhi_laplacian_with_lambda} and~\eqref{eq:F_laplacian_with_lambda}, omitting some constant factors not depending on $\ep$ or $\lambda$, are of one of the following types: 
\vskip 1mm
\begin{enumerate}
\item[(1)] $a\ep^{-2}(w - |\Phi|^2)S - (1-a)\ep^{-2} [\Phi, [S, \Phi]]$, where we recall that $a = \lambda$ if $S = \nabla\Phi$, whereas $a = 0$ if $S = \ep F_{\nabla}$.
\vskip 1mm
\item[(2)] $\ep^{-1} [T_1, T_2]$ with at most one pair of indices contracted, where each of $T_1$ and $T_2$ can be either $\nabla \Phi$ or $\ep F_{\nabla}$. We denote such terms by $\ep^{-1} T_1 * T_2$. 
\vskip 1mm
\item[(3)] $R \otimes S$ with two pairs of indices contracted, which we abbreviate as $R* S$.
\end{enumerate}
\vskip 1mm
Accordingly, when computing $\nabla^m \nabla^*\nabla S$, we consider each type of terms separately. We first notice by~\eqref{eq:tensor-bracket-derivative-bound-no-head} that
\[
\begin{split}
&\Big|\nabla^m\big[ (w - |\Phi|^2)S\big] - (w - |\Phi|^2)\nabla^{m}S \Big|\\
\leqslant\ &   C_{m}\sum_{i = 0}^{m-1}|\nabla^{m-i}(\frac{1 - 3|\Phi|^2}{2})||\nabla^{i}S|\leqslant C_{m}\sum_{i = 0}^{m-1}\sum_{j + k = m-i}|\nabla^j \Phi| |\nabla^k\Phi| |\nabla^{i}S|,
\end{split}
\]
which implies
\begin{equation}\label{eq:bochner_for_derivatives_I}
\begin{split}
\bangle{\nabla^m\big( a\frac{w - |\Phi|^2}{\ep^2}S \big), \nabla^{m}S} \leqslant\ & a\frac{w - |\Phi|^2}{\ep^2}|\nabla^{m}S|^2 \\
&+  a \cdot \frac{C_{m}}{\ep^2}\sum_{i = 0}^{m-1}\sum_{j + k = m-i}|\nabla^j \Phi| |\nabla^k\Phi| |\nabla^{i}S||\nabla^{m}S|.
\end{split}
\end{equation}
Another application of~\eqref{eq:tensor-bracket-derivative-bound-no-head} gives
\[
\begin{split}
\big| \nabla^m [\Phi, [S, \Phi]] - [\Phi, \nabla^m[S, \Phi]] \big| \leqslant\ & C_{m} \sum_{i = 0}^{m-1}\big|[\nabla^{m-i}\Phi, \nabla^i[S, \Phi]]\big| \leqslant C_{m}\sum_{i=0}^{m-1} \big|\nabla^{m-i}\Phi\big|\big|\nabla^i[S, \Phi]\big|.
\end{split}
\]
On the other hand, the reasoning leading to~\eqref{eq:tensor-bracket-derivative-bound-no-head} also yields
\begin{equation}\label{eq:leading_double_bracket_bound_1}
\begin{split}
\big| [\Phi, \nabla^m[S, \Phi]] - [\Phi, [\nabla^m S, \Phi]] \big| \leqslant\ & C_{m}\sum_{i = 0}^{m-1}\big|[\Phi, [\nabla^i S, \nabla^{m-i}\Phi]]\big|\\
\leqslant\ & C_{m}\sum_{i=0}^{m-1} \big( |[\nabla^i S,\Phi]||\nabla^{m-i}\Phi| + |\nabla^i S||[\nabla^{m-i}\Phi, \Phi]| \big),
\end{split}
\end{equation}
where the second inequality follows from~\eqref{eq:tensor-bracket-norm-with-decomp-1}. Consequently, we get
\begin{small}
\begin{equation}\label{eq:bochner_for_derivatives_II}
\begin{split}
-\frac{1-a}{\ep^2}\bangle{\nabla^m[\Phi, [S, \Phi]], \nabla^m S} 
\leqslant\ & -\frac{1- a}{\ep^2}|[\nabla^m S, \Phi]|^2 + \frac{C_{m}|1-a|}{\ep^2}\sum_{i=0}^{m-1} |\nabla^i S||[\nabla^{m-i}\Phi, \Phi]||\nabla^m S| \\
& +  \frac{C_{m}|1-a|}{\ep^2}\sum_{i = 0}^{m-1}|\nabla^{m-i}\Phi| \big(|\nabla^i[S, \Phi]| + |[\nabla^i S, \Phi]| \big) |\nabla^m S|.
\end{split}
\end{equation}
\end{small}
Next, by~\eqref{eq:tensor-bracket-derivative-bound-no-ends} and the triangle inequality, we have that
\begin{equation}\label{eq:laplacian_quad_leading}
|\nabla^m[T_1, T_2]| \leqslant  |[\nabla^m T_1, T_2]| + |[T_1, \nabla^m T_2]| + C_{m} \sum_{i = 1}^{m-1}|[\nabla^{i}T_1, \nabla^{m-i}T_2]|,
\end{equation}
from which we deduce that
\begin{equation}\label{eq:bochner_for_derivatives_III}
\begin{split}
\frac{1}{\ep}\big|\bangle{\nabla^m(T_1 * T_2), \nabla^m S}\big| \leqslant\ & \frac{C_n}{\ep}|\nabla^m[T_1, T_2]| |\nabla^m S|\\
\leqslant \ & \frac{C_n}{\ep}\Psi_0\Psi_m |\nabla^m S| + \frac{C_{n, m}}{\ep} \sum_{i=1}^{m-1}\Psi_i \Psi_{m-i}|\nabla^m S|,
\end{split}
\end{equation}
where the summation is absent if $m = 1$. Finally, following the argument leading to~\eqref{eq:tensor-bracket-derivative-bound-no-head}, we get that
\[
|\nabla^m(R \otimes S)| \leqslant |R| |\nabla^m S| + C_{m}\sum_{i = 0}^{m-1} |\nabla^{m-i}R||\nabla^i S|,
\]
and hence 
\begin{equation}\label{eq:bochner_for_derivatives_IV}
\big|\bangle{\nabla^m(R * S), \nabla^m S}\big| \leqslant  C_{n}|R| |\nabla^m S|^2 + C_{m, n}\sum_{i=0}^{m-1} |\nabla^{m-i}R||\nabla^i S| |\nabla^m S|.
\end{equation}
Combining~\eqref{eq:bochner_for_derivatives_I},~\eqref{eq:bochner_for_derivatives_II},~\eqref{eq:bochner_for_derivatives_III} and~\eqref{eq:bochner_for_derivatives_IV}, 
and noting that
\begin{equation}\label{eq:bochner_for_derivatives_leading}
\begin{split}
\frac{aw - a|\Phi|^2}{\ep^2}|\nabla^m S|^2 - \frac{1 - a}{\ep^2}|[\nabla^mS, \Phi]|^2 \leqslant\ &\frac{aw}{\ep^2}|\nabla^m S|^2 - \frac{1}{\ep^2} |[\nabla^m S, \Phi]|^2, 
\end{split}
\end{equation}
we infer that~\eqref{eq:bochner_for_derivatives} holds. To establish the alternative estimate asserted in case (i), where $S = \nabla\Phi$ and $a = \lambda$, we note by~\eqref{eq:tensor-bracket-derivative-bound-no-head} that
\[
\begin{split}
\big| \nabla^m[\nabla\Phi, \ep F_{\nabla}] \big| 
\leqslant\ & |\nabla^{m + 1}\Phi||\ep F_{\nabla}| + C_{m}\sum_{i = 0}^{m-1} |[\nabla^{i + 1}\Phi, \nabla^{m-i}(\ep F_{\nabla})]|,
\end{split}
\]
which leads to
\[
\frac{1}{\ep}\big| \bangle{\nabla^m (\ep F_{\nabla}* \nabla\Phi) , \nabla^{m + 1}\Phi}\big| \leqslant \frac{C_n}{\ep}\Psi_0 |\nabla^{m + 1}\Phi|^2 + \frac{C_{m,n}}{\ep}\sum_{i=0}^{m-1} \Psi_{m-i}|\nabla^{i+1}\Phi||\nabla^{m + 1}\Phi|.
\]
In view of~\eqref{eq:nablaPhi_laplacian_with_lambda}, adding this estimate instead of~\eqref{eq:bochner_for_derivatives_III} with~\eqref{eq:bochner_for_derivatives_I},~\eqref{eq:bochner_for_derivatives_II} and~\eqref{eq:bochner_for_derivatives_IV} proves the asserted modifications of~\eqref{eq:bochner_for_derivatives} when $S = \nabla\Phi$. For case (ii), we estimate by~\eqref{eq:laplacian_quad_leading} and~\eqref{eq:tensor-bracket-norm-with-decomp-2} that
\[
\begin{split}
|\nabla^m [T_1, T_2]| \leqslant\ & 2(\Psi_m^{\perp}\Psi_0 + \Psi_m \Psi_0^{\perp}) + C_m \sum_{i = 1}^{m-1}(\Psi_i^{\perp}\Psi_{m-i} + \Psi_i \Psi_{m-i}^{\perp})\\
\leqslant\ & 2\Psi_0^{\perp}\Psi_m + C_m \sum_{i = 0}^{m-1} \Psi_{m-i}^{\perp}\Psi_i.
\end{split}
\]
From this we easily get
\[
\frac{1}{\ep}\big|\bangle{\nabla^m (T_1 * T_2), \nabla^m S}\big| \leqslant \frac{C_n}{\ep} \Psi_0^{\perp} \Psi_m |\nabla^m S| + \frac{C_{n, m}}{\ep}\sum_{i=0}^{m-1} \Psi_{m-i}^{\perp}\Psi_i |\nabla^m S|.
\]
Replacing~\eqref{eq:bochner_for_derivatives_III} by the above estimate proves the asserted modifications of~\eqref{eq:bochner_for_derivatives} on $M \setminus Z(\Phi)$, and the proof of the lemma is complete.
\end{proof}
\begin{proof}[Proof of Lemma~\ref{lemm:trans_laplacian_diffed}]
As above, in addition to $S$, we let $T_1$ and $T_2$ denote tensors that could be either $\nabla\Phi$ or $\ep F_{\nabla}$. To begin, observe that the right-hand sides of~\eqref{eq:nablaPhi_trans_laplacian_with_lambda} and~\eqref{eq:F_trans_laplacian_with_lambda} consist of terms of the following four types:
\vskip 1mm
\begin{enumerate}
\item[(1)] $\ep^{-2}(a\lambda w - |\Phi|^2)[S, \Phi]$, where $a = 1$ if $S = \ep F_{\nabla}$ while $a = 2$ if $S = \nabla\Phi$.
\vskip 1mm
\item[(2)] $\ep^{-1}[T_1 * T_2, \Phi]$.
\vskip 1mm
\item[(3)] $R * [S, \Phi]$.
\vskip 1mm
\item[(4)] $[\nabla S, \nabla\Phi]$ with a pair of indices contracted, which we write as $\nabla S * \nabla\Phi$.
\end{enumerate}
\vskip 1mm
As in the previous proof we treat these one by one. First, by~\eqref{eq:tensor-bracket-derivative-bound-no-head} we have 
\[
\begin{split}
\Big|\nabla^m\big( \frac{a\lambda w  -|\Phi|^2}{\ep^2}[S, \Phi] \big) - \frac{a\lambda w - |\Phi|^2}{\ep^2}\nabla^m[S, \Phi] \Big|\leqslant\ & C_{m}\sum_{i = 0}^{m-1} |\nabla^{m-i}\big(\frac{a\lambda w - |\Phi|^2}{\ep^2} \big)| |\nabla^{i}[S, \Phi]|\\
\leqslant\ & \frac{C_{m}(\lambda + 1)}{\ep^2} \sum_{i=0}^{m-1}\sum_{j + k = m-i} |\nabla^j\Phi||\nabla^k\Phi| \Theta_i,
\end{split}
\] and thus
\begin{equation}\label{eq:trans_laplacian_diffed_I}
\begin{split}
&\bangle{\nabla^m\big( \frac{a\lambda w - |\Phi|^2}{\ep^2}[S, \Phi] \big), \nabla^m [S, \Phi]}\\
\leqslant\ &  \frac{a\lambda w - |\Phi|^2}{\ep^2}|\nabla^m[S, \Phi]|^2 + \frac{C_{m}(\lambda + 1)}{\ep^2} \sum_{i=0}^{m-1}\sum_{j + k = m-i} |\nabla^j\Phi||\nabla^k\Phi| \Theta_i\Theta_m.
\end{split}
\end{equation}
Next, following the proof of~\eqref{eq:tensor-bracket-derivative-bound-no-ends}, we have
\[
\big| \nabla^m[[T_1, T_2], \Phi]\big| \leqslant  \big|[[\nabla^m T_1, T_2], \Phi]\big| + \big| [[T_1, \nabla^m T_2], \Phi] \big| + C_m \sum_{\substack{i + j + k = m \\ i, j \neq m}}\big| [[\nabla^i T_1, \nabla^j T_2], \nabla^k\Phi] \big|.
\]
Recalling~\eqref{eq:tensor-bracket-norm-with-decomp-2}, we deduce that
\[
\begin{split}
\big| \nabla^m[[T_1, T_2], \Phi]\big|\leqslant\ & (\big| [\nabla^m T_1, T_2] \big| + \big| [T_1, \nabla^m T_2] \big|)\cdot |\Phi| + C_m\sum_{\substack{i + j + k = m\\ i, j \neq m}} (\Psi_i^\perp\Psi_j + \Psi_i \Psi_j^\perp) |\nabla^{k}\Phi| \\
\leqslant\ & 2(\Psi_m^\perp\Psi_0  + \Psi_m\Psi_0^\perp )|\Phi| + 2C_m\sum_{\substack{i + j + k = m\\ i, j \neq m}} \Psi_i^\perp\Psi_j  |\nabla^{k}\Phi| \\
\leqslant\ & 2|\Phi|\Psi_0\Psi_m^{\perp}+ C_m\sum_{i = 0}^{m-1}\sum_{j + k = m-i} \Psi_i^\perp\Psi_j |\nabla^{k}\Phi|,
\end{split}
\]
from which we get
\begin{equation}\label{eq:trans_laplacian_diffed_II}
\begin{split}
\frac{1}{\ep}\big|\nabla^m[T_1 * T_2, \Phi]\big| \big|\nabla^m[S, \Phi]\big| \leqslant\ & \frac{C_n}{\ep}|\Phi| \Psi_0 \Psi_{m}^\perp\Theta_m  + \frac{C_{n, m}}{\ep}\sum_{i = 0}^{m-1} \sum_{j + k = m-i} \Psi^\perp_i \Psi_j |\nabla^k\Phi| \Theta_m.
\end{split}
\end{equation}
The curvature terms in~\eqref{eq:nablaPhi_trans_laplacian_with_lambda} and~\eqref{eq:F_trans_laplacian_with_lambda} are handled in exactly the same way as in the previous proof, and we have
\begin{equation}\label{eq:trans_laplacian_diffed_III}
\big|\bangle{\nabla^m(R*[S, \Phi]),\nabla^m[S, \Phi]}\big| \leqslant C_{n}|R| |\nabla^m[S, \Phi]|^2 + C_{m, n}\sum_{i=0}^{m-1}|\nabla^{m-i}R|\Theta_i\Theta_m.
\end{equation}
Finally, for the terms of the form $\nabla S * \nabla\Phi$, we observe by~\eqref{eq:tensor-bracket-derivative-bound} and~\eqref{eq:tensor-bracket-norm-with-decomp-2} that
\[
\begin{split}
|\nabla^m[\nabla S, \nabla\Phi]| \leqslant\ & C_{m}\sum_{i = 0}^{m}|[\nabla^{i + 1}S, \nabla^{m-i + 1}\Phi]|\leqslant C_{m}\sum_{i=0}^{m}\big(\Psi_{i+1}^\perp \Psi_{m-i} + \Psi_{i+1}\Psi_{m-i}^\perp\big)\\
=\ & C_{m}\Big( \sum_{i=1}^{m+1}\Psi_i^\perp \Psi_{m-i + 1} + \sum_{i = 0}^{m}\Psi_{m-i + 1}\Psi_{i}^\perp \Big)\leqslant 2C_{m}\sum_{i = 0}^{m + 1}\Psi_i^\perp \Psi_{m + 1-i}.
\end{split}
\]
Isolating the terms $i = m+1$ and $i = m$ and applying the estimate~\eqref{eq:Theta_perp_relation} with $m+ 1$ in place of $m$ to bound $\Psi_{m + 1}^{\perp}$, we get 
\[
\begin{split}
|\nabla^m[\nabla S,\nabla\Phi]| \leqslant\ & C_{m}\Psi_{m + 1}^\perp \Psi_0 + C_{m} \Psi_m^\perp \Psi_1  + C_{m}\sum_{i=0}^{m-1}\Psi_i^\perp \Psi_{m-i + 1}\\
\leqslant\ & C_{m}\Psi_0 |\Phi|^{-1} \big( \Theta_{m + 1} + \Psi_m^\perp \Psi_0 + \sum_{i = 0}^{m-1}\Psi_i^\perp \Psi_{m-i} \big)\\
& + C_{m}\Psi_1 \Psi_m^\perp + C_{m}\sum_{i=0}^{m-1}\Psi_i^\perp \Psi_{m-i + 1}\\
\leqslant \ & C_{m}\Theta_{m + 1}|\Phi|^{-1}\Psi_0 + C_{m} |\Phi|^{-1}\Psi_0^2\Psi_m^\perp + C_{m}\Psi_1\Psi_m^\perp\\
& + C_{m}\sum_{i = 0}^{m-1} \Psi_i^\perp\big( |\Phi|^{-1}\Psi_0\Psi_{m-i} + \Psi_{m+1-i} \big),
\end{split}
\] from which we deduce that
\begin{equation}\label{eq:trans_laplacian_diffed_IV}
\begin{split}
\big|\bangle{\nabla^m(\nabla S * \nabla\Phi), \nabla^m[S, \Phi]}\big| \leqslant\ & C_n |\nabla^m [\nabla S, \nabla \Phi]| |\nabla^m[S, \Phi]|\\
\leqslant\ & C_{n, m}\Theta_{m + 1} \cdot (|\Phi|^{-1}\Psi_0\Theta_m) + C_{n, m}\big( |\Phi|^{-1}\Psi_0^2 + \Psi_1 \big)\Psi_m^\perp\Theta_m\\
& + C_{n, m}\sum_{i = 0}^{m-1}\Psi_i^\perp \big( |\Phi|^{-1}\Psi_0\Psi_{m-i} + \Psi_{m+1-i} \big) \Theta_m.
\end{split}
\end{equation}
We conclude the proof of~\eqref{eq:trans_laplacian_diffed} by summing~\eqref{eq:trans_laplacian_diffed_I},~\eqref{eq:trans_laplacian_diffed_II},~\eqref{eq:trans_laplacian_diffed_III} and~\eqref{eq:trans_laplacian_diffed_IV}.
\end{proof}


\bibliographystyle{abbrv}
\bibliography{main}
\end{document}